\documentclass[12pt,oneside,reqno]{amsart}
\usepackage[margin=1in]{geometry}
\usepackage{color}
\usepackage{esint,amssymb}
\usepackage{graphicx}
\usepackage{MnSymbol}
\usepackage{mathtools}
\usepackage[colorlinks=true, pdfstartview=FitV, linkcolor=blue, citecolor=blue, urlcolor=blue,pagebackref=false]{hyperref}
\usepackage{microtype}
\usepackage{amsmath}
\usepackage{relsize}
\usepackage{cancel}
\usepackage{xifthen}
\usepackage{verbatim}
\usepackage{amsxtra}
\definecolor{darkgreen}{rgb}{0,0.5,0}
\definecolor{darkblue}{rgb}{0,0,0.7}
\definecolor{darkred}{rgb}{0.9,0.1,0.1}

\newtheorem{theorem}{Theorem}
\newtheorem{proposition}[theorem]{Proposition}
\newtheorem{lemma}[theorem]{Lemma}
\newtheorem{corollary}[theorem]{Corollary}

\theoremstyle{definition}
\newtheorem{remark}[theorem]{Remark}

\newcommand{\lref}[1]{Lemma~\ref{l.#1}}

\newcommand{\cref}[1]{Corollary~\ref{c.#1}}

\newcommand{\eref}[1]{(\ref{e.#1})}

\numberwithin{equation}{section}
\numberwithin{theorem}{section}

\newcommand{\Z}{\mathbb{Z}}
\newcommand{\N}{\mathbb{N}}
\newcommand{\R}{\mathbb{R}}

\newcommand{\T}{\mathbb{T}}

\newcommand{\ep}{\varepsilon}

\newcommand{\test}[1][]{%
\ifthenelse{\equal{#1}{}}{omitted}{given}%
}

\newcommand{\derv}[3]{\partial_x^{\alpha{#1}}\partial_v^{\beta{#2}}Y^{\omega{#3}}}
\newcommand{\der}{\derv{}{}{}}

\renewcommand{\d}[1]{\ensuremath{\operatorname{d}\!{#1}}}
\DeclarePairedDelimiter{\norm}{\lVert}{\rVert}

\newcommand{\jap}[1]{\left\langle {#1} \right\rangle}

\renewcommand{\bar}{\overline}
\renewcommand{\tilde}{\widetilde}

\renewcommand{\part}{\partial}

\usepackage{dsfont}

\def\f {\frac}
\def\rd {\partial}
\def\ls {\lesssim}
\def\de {\delta}
\def\i {\infty}
\def\alp {\alpha}
\def\bt {\beta}
\def\nab {\nabla}
\def\ep {\epsilon}
\def\om {\omega}
\newcommand{\mfC}{\mathfrak C}
\newcommand{\mfc}{\mathfrak c}
\newcommand{\ud}{\mathrm{d}}

\newcommand{\bv}{\langle v\rangle}
\def\la {\langle}
\def\ra {\rangle}
\def\RR {\mathbb R}
\def\th {\theta}
\def\wtE {\widetilde{\mathcal E}}
\def\wtbE {\widetilde{\mathbb E}}

\def\wtD {\widetilde{\mathcal D}}
\def\wtbD {\widetilde{\mathbb D}}

\def\wtG {\widetilde{\mathbb G}}
\def\mfm {\mathfrak m}
\def\uN {\underline{N}}

\begin{document}
\title[The Vlasov--Poisson--Landau system in the weakly collisional regime]{The Vlasov--Poisson--Landau system \\in the weakly collisional regime}
\begin{abstract}
Consider the Vlasov--Poisson--Landau system with Coulomb potential in the weakly collisional regime on a $3$-torus, i.e.
\begin{align*}
\rd_t F(t,x,v) + v_i \rd_{x_i} F(t,x,v) + E_i(t,x) \rd_{v_i} F(t,x,v) = \nu Q(F,F)(t,x,v),\\
E(t,x) = \nabla \Delta^{-1} (\int_{\mathbb R^3} F(t,x,v)\, \ud v - \strokedint_{\mathbb T^3} \int_{\mathbb R^3} F(t,x,v)\, \ud v \, \ud x), 
\end{align*}
with $\nu\ll 1$. We prove that for $\ep>0$ sufficiently small (but independent of $\nu$), initial data which are $O(\ep \nu^{1/3})$-Sobolev space perturbations from the global Maxwellians lead to global-in-time solutions which converge to the global Maxwellians as $t\to \infty$. The solutions exhibit uniform-in-$\nu$ Landau damping and enhanced dissipation.

Our main result is analogous to an earlier result of Bedrossian for the Vlasov--Poisson--Fokker--Planck equation with the same threshold. However, unlike in the Fokker--Planck case, the linear operator cannot be inverted explicitly due to the complexity of the Landau collision operator. For this reason, we develop an energy-based framework, which combines Guo's weighted energy method with the hypocoercive energy method and the commuting vector field method. The proof also relies on pointwise resolvent estimates for the linearized density equation.

\end{abstract}

\author[S. Chaturvedi]{Sanchit Chaturvedi}
\address[Sanchit Chaturvedi]{Department of Mathematics, Stanford University, 450 Jane Stanford Way, Bldg 380, Stanford, CA 94305, USA}
\email{sanchat@stanford.edu}
\author[Jonathan Luk]{Jonathan Luk}
\address[Jonathan Luk]{Department of Mathematics, Stanford University, 450 Jane Stanford Way, Bldg 380, Stanford, CA 94305, USA}
\email{jluk@stanford.edu}
\author[T. Nguyen]{Toan T. Nguyen}
\address[Toan T. Nguyen]{Penn State University, Department of Mathematics, State College, PA 16803, USA}
\email{nguyen@math.psu.edu}
\subjclass[2010]{}
\date{\today}

\maketitle
\parskip = 0 pt

\setcounter{tocdepth}{1}
\tableofcontents

\section{Introduction}

In this paper, we study the \textbf{Vlasov--Poisson--Landau system} for a particle density function $F: [0,\infty) \times \mathbb T^3_x \times \mathbb R^3_v\to [0, \infty)$ on the $3$-torus $\mathbb T^3 := \mathbb R^3/(2\pi \mathbb Z)^3$, which takes the form
\begin{subequations}
\begin{align}
\rd_t F(t,x,v) + v_i \rd_{x_i} F(t,x,v) + E_i(t,x) \rd_{v_i} F(t,x,v) = \nu Q(F,F)(t,x,v), \label{eq:VPL}\\
E(t,x) = -\nabla \phi(t,x),\quad \phi(t,x) = - \Delta^{-1} (\int_{\mathbb R^3} F(t,x,v)\, \ud v - \strokedint_{\mathbb T^3} \int_{\R^3} F(t,x,v)\, \ud v \, \ud x), 
\label{eq:Poisson}
\end{align}
\end{subequations}
where (from now on) repeated lower case Latin indices are summed over $i,j = 1,2,3$, $\strokedint_{\mathbb T^3} := \f 1{(2\pi)^3} \int_{\T^3}$, and $Q$ is the \textbf{Landau collision operator} with \textbf{Coulomb potential} given by 
$$Q(G,F)(t,x,v)  := \rd_{v_i} \int_{\mathbb R^3} \Phi_{ij}(v-v_*)  \{G(t,x,v_*) (\rd_{v_j} F)(t,x,v) - F(t,x,v) (\rd_{v_j}G)(t,x,v_*) \} \, \ud v_*,$$
where
\begin{equation}\label{eq:Phiij}
\Phi_{ij}(z) := \f 1{|z|} \{ \de_{ij} - \f{z_i z_j}{|z|^2} \} ,
\end{equation}
with $\de_{ij}$ being the Kronecker delta. We will work in the \textbf{weakly collisional regime}, i.e.~we will assume $\nu$ in \eqref{eq:VPL} satisfies $\nu \ll 1$, which is relevant in physical situations (see \cite{cMcV2010}). 

The system \eqref{eq:VPL}--\eqref{eq:Poisson} describes the dynamics of electrons in a constant ion background. The electrons both undergo (weak) bilinear collisions and are subject to the mean field force generated by the electrons themselves. It is also of interest to consider the $2$-species analogue of \eqref{eq:VPL}--\eqref{eq:Poisson}, which describes the motion of both the electrons and the ions. We will not explicitly write down that case, though we hope the ideas of this paper will extend to that case.

It is easy to check that the global Maxwellian 
\begin{equation}\label{Maxwellian}\mu(v) := e^{-|v|^2}\end{equation}
 is a steady state solution to \eqref{eq:VPL}--\eqref{eq:Poisson}. For any fixed $\nu>0$, the celebrated work of Guo \cite{Guo12} implies that the global Maxwellian $\mu$ is asymptotically \emph{stable}. For $\nu = 0$, however, the situation is much more subtle. The seminal work of Mouhot--Villani \cite{cMcV2011} showed that the global Maxwellians are stable in an analytic topology via a phase-mixing mechanism known as Landau damping, which causes the electric field to decay rapidly. The same was proven to hold in a sufficiently strong Gevrey topology \cite{jBnMcM2016}; see also a more recent proof in \cite{eGtNiR2020a}. Nevertheless, in a Sobolev topology, Bedrossian showed in \cite{jB2021} that (for a different spatially homogeneous background,) a uniform statement of the stability does \underline{not} hold due to so-called plasma echoes. (See however \cite{eGtNiR2020b}.)

Our goal in this paper is twofold. First, we give a detailed description of the dynamics in the presence of both (collisional) entropic effect and (non-collisional) phase mixing effect. Second, we seek to understand the \emph{threshold} of stability for \eqref{eq:VPL}--\eqref{eq:Poisson}, i.e.~for an appropriate norm $X$ (which will be chosen to be a Sobolev norm) and a $\beta>0$, we want to show
\begin{equation}\label{eq:threshold.diagram}
\|\mathrm{Data}\|_X \ll \nu^{\beta} \implies \mbox{stability}.
\end{equation}
Ideally, we would like to find a $\beta$ that is optimal.

The proof of Guo's result \cite{Guo12} discussed above, when appropriately adapted to the weakly collisional regime, straightforwardly implies a version of \eqref{eq:threshold.diagram} with $\beta = 1$. This is summarized in the following theorem.
\begin{theorem}[Guo \cite{Guo12}]\label{thm:guo}
There exist a ($v$-weighted, $L^2$-based, up to second order derivatives) Sobolev space $X$ and an $\ep_0 >0$ independent of $\nu$ such that if the initial data $F_0$ satisfies
$$\|\f 1{\sqrt{\mu}} (F_0 - \mu)\|_{X} \leq \ep_0 \nu,$$
then there is a unique global smooth solution to \eqref{eq:VPL}--\eqref{eq:Poisson} arising from the given data, which converges to $\mu$ as $t\to +\infty$.
\end{theorem}

To improve the threshold in Theorem~\ref{thm:guo}, one needs to take advantage of the following two mechanisms specific to the $\nu\to 0$ limit:\begin{enumerate}
\item (Enhanced dissipation) Solutions to \eqref{eq:VPL} dissipate energy much faster than that given by the proof of Theorem~\ref{thm:guo}. From \eqref{eq:VPL}, one may expect (say, by comparison with the heat equation) that the solution dissipates energy at an $O(\nu^{-1})$ time scale. However, the transport part shifts the solution to high $v$-frequency, which enhances the dissipation. As a result, after subtracting the average-in-$x$ mode, the solution in fact dissipates energy at an $O(\nu^{-1/3})$ time scale.
\item (Landau damping) When $\nu=0$, the Vlasov--Poisson--Landau system reduces to the Vlasov--Poisson system, which as discussed above has a decay mechanism of Landau damping. One expects that Landau damping persists for small $\nu$, and gives a decay mechanism at an $O(1)$ time, before the collisional effects enter.
\end{enumerate}

To understand exactly how Landau damping enters requires some knowledge of nonlinear Landau damping for Sobolev data. It is by now well-understood, for instance by adapting methods of \cite{cMcV2010}, that for initial data of size $O(\de)$ (with $\de$ small) in a (sufficiently regular) Sobolev topology, the linear Landau damping mechanism drives the nonlinear dynamics for the Vlasov--Poisson system up to a time of $O(\de^{-1})$. It is therefore reasonable to expect $O(\ep \nu^{\f 13})$ to be a natural threshold for the problem \eqref{eq:threshold.diagram}: Landau damping gives a decay mechanism up to time $O(\ep^{-1} \nu^{-\f 13})$, at which point the collisional effect kicks in and dominates due to the enhanced dissipation. (See discussions in \cite{jB2017}.) This is exactly what we obtain in our main theorem.

\begin{theorem}\label{thm:main.intro}
There exist a ($v$-weighted, $L^2$-based, up to eleventh order derivatives) Sobolev space $X$ and an $\ep_0 >0$ independent of $\nu$ such that if the initial data $F_0$ satisfies
$$\|\f 1{\sqrt{\mu}} (F_0 - \mu)\|_{X} \leq \ep_0 \nu^{1/3},$$
then there is a unique global smooth solution to \eqref{eq:VPL}--\eqref{eq:Poisson} arising from the given data,  which converges to $\mu$ as $t\to +\infty$. 

Moreover, the solution exhibits enhanced dissipation and uniform-in-$\nu$ Landau damping.
\end{theorem}

The enhanced dissipation and uniform-in-$\nu$ Landau damping are reflected in the large-time estimates that we prove; see Section~\ref{sec:intro.decay} and Theorem~\ref{t.main}. (Notice that if we only capture enhanced dissipation without exploiting Landau damping in the proof, this would correspond 
to a weaker theorem where the initial data could only be an $O(\epsilon_0 \nu^{\f 23})$-perturbation of the global Maxwellian.)

A very similar result 
was proven in a recent work of Bedrossian for the Vlasov--Poisson--Fokker--Planck system \cite{jB2017}. The Landau collision kernel is more complicated than Fokker--Planck collision kernel in its anisotropy and degeneracy as $|v|\to \infty$, as well as a lack of a spectral gap. More importantly from the point of view of this problem, the linearized Vlasov--Poisson--Landau system around the global Maxwellians cannot be solved explicitly, unlike the corresponding linearized Vlasov--Poisson--Fokker--Planck system. This requires a different approach for the problem.

Given the lack of explicit representation formulas for the linear solution, we rely instead on adaptations of Guo's energy method, but we further need to design the energies so as to capture the phenomena of enhanced dissipation and Landau damping.
\begin{enumerate}
\item (Hypocoercivity) To capture enhanced dissipation, we use the hypocoercive energy method: this is a choice of an energy which incorporates some lower order boundary terms, which in turn generates useful coercive spacetime terms. This idea is known to be well-suited to capture the interaction of the transport and the collision terms \cite{fHfN2004, bHfN2005, cV2009} (see also Section~\ref{sec:hypo.related.works}), and particularly to obtain sharp enhanced dissipation rate in some weakly viscous settings \cite{jBmCZ2017}. See Section~\ref{sec:intro.hypocoercivity}.
\item (Commuting vector fields method) To capture Landau damping, we need quantitative estimates showing that $f:= \f 1{\sqrt{\mu}}( F - \mu)$ behaves like a solution of the transport equation. To achieve this, we commute the equation with the ($t$-weighted) vector field $Y_i = t \rd_{x_i} + \rd_{v_i}$ and bound $Y^\om f$ (and its derivatives) in addition to $f$ itself. This lets one prove transport bounds in the presence of collision, and in fact to take advantage of the coercivity given by collisions. Such a commutating vector field method is inspired by related techniques for nonlinear wave equations, fluid equations and other kinetic models \cite{sC2020b, sC2020a, sK1985, jL2019, Sm16, WeZhZh20}. See Sections~\ref{sec:intro.vector.fields} and \ref{sec:vector.field.method.related.works}.
\item (Resolvent estimates via hypocoercivity and commuting vector field method) Using the hypocoercive energy method itself is difficult to control the nonlocal terms associated with the electric field. Thus, in addition to energy estimates for $f$, we derive an \emph{independent density estimate} (as in \cite{jB2017, jBnMcM2016, eGtNiR2020a, cMcV2011}) for the macroscopic density $\rho := \int_{\mathbb R^3} f \sqrt{\mu} \,\ud v$, proven using the Volterra equation that it satisfies. 

To achieve the density estimates involves (1) proving a resolvent estimate to invert the linear part and (2) bounding the nonlinear contributions. Both of these can be achieved by extending the resolvent estimate and nonlinear analysis in \cite{eGtNiR2020a} in conjunction with obtaining control of the linear Landau flow (see Section~\ref{sec:intro.density}). Importantly, the linear Landau flow no longer features nonlocal terms. Thus, the estimates we need for the linear Landau flow can in turn be achieved by a combination of the hypocoercive energy method and the commuting vector field method.
\end{enumerate}

We further discuss the ideas of the proof in \textbf{Section~\ref{sec:ideas}}. We then turn to related works in \textbf{Section~\ref{sec:related.works}}. 
Finally, we will end our introduction with some discussions on future directions in \textbf{Section~\ref{sec:discussion}} and an outline of the remainder of the paper in \textbf{Section~\ref{sec:outline}}.

\subsection{Idea of the proof}\label{sec:ideas}

\subsubsection{Preliminaries}\label{sec:ideas.prelim} Let $\mu$ be the global Maxwellian \eqref{Maxwellian}. We rewrite the problem for $f$ defined by $F = \mu + \sqrt{\mu} f$ so that the Vlasov--Poisson--Landau system \eqref{eq:VPL}--\eqref{eq:Poisson} becomes
\begin{equation}\label{eq:intro.Vlasov.f}
\rd_t f + v\cdot \nabla_x f + E\cdot \nabla_v f -(E\cdot v) f- 2 (E\cdot v) \sqrt{\mu} + \nu L f = \nu \Gamma(f,f),
\end{equation}
where $E$ is as in \eqref{eq:Poisson}, 
$L$ is the linearized Landau collision operator, which has some coercivity properties, and $\Gamma$ is the nonlinear collisional terms in $f$ (see Section~\ref{sec:f} for precise definitions). The problem is now rephrased to proving boundedness and decay for $f$.

\subsubsection{The Guo's energy method}\label{sec:intro.ideas.Guo}

The starting point of our approach is Guo's work \cite{Guo12} (see Theorem~\ref{thm:guo} above). The general strategy, first devised by Guo and is applicable in many kinetic models of the form $\partial_t f + v\cdot \nabla_x f + \nu  L f = \nu \Gamma(f,f)$, is to find an energy norm $\|\cdot\|_{\mathcal E}$, a dissipation norm $\|\cdot\|_{\mathcal D}$, and a suitable $\theta>0$, $\eta\in \R$ such that 
\begin{equation}\label{eq:intro.naive.energy}
\f{d}{dt} \| f\|_{\mathcal E}^2 + \theta \nu^{\eta} \| f\|_{\mathcal D}^2 \quad \leq \quad 0.
\end{equation}
The control in $\|\cdot\|_{\mathcal D}$ comes from the linear $Lf$ part, and the norms are chosen appropriately to bound the nonlinear term $\nu |\la f, \Gamma(f,f) \ra_{\mathcal E}|\ls \nu^{\eta'} \| f \|_{\mathcal E} \| f\|_{\mathcal D}^2$ so that one can indeed obtain \eqref{eq:intro.naive.energy} with suitable smallness on the initial data.

The proof of such an energy inequality, or even the choice of the norms $\|\cdot\|_{\mathcal E}$ and $\|\cdot\|_{\mathcal D}$, is delicate, and depends on the kinetic model under consideration. The construction in general requires a careful choice of weight functions, as well as using an additional argument to handle the kernel of the linear operator $L$. This type of method is particularly useful for soft potentials (such as the Landau collision operator), since in general the $\|\cdot\|_{\mathcal D}$ norm is weaker than the $\|\cdot\|_{\mathcal E}$ norm in $v$-weights.

We highlight two innovations in the energy introduced by Guo \cite{Guo12} which are specific to the Vlasov--Poisson--Landau system:
\begin{itemize}
\item use of \textbf{$e^{\phi}$ weights} in the energy, where $\phi$ is the electric potential (to handle the costly-in-$v$-moment term $(E\cdot v) f$), and
\item use of \textbf{weights in $\bv$} which are weaker for higher derivatives (to handle simultaneously the weak coercivity of the the dissipation energy for large $\bv$ and commutator terms arising from the linear free streaming term).
\end{itemize}
A more detailed explanation of these weights and their motivations can be found in the introduction of \cite{Guo12}. These will also be featured prominently in our energies.

\subsubsection{Hypocoercivity and energy method}\label{sec:intro.hypocoercivity}
Slightly over-simplifying\footnote{We have suppressed in particular the fact that (1) Guo also incorporated $E$ in the energy (which we will not need, see beginning of Section~\ref{sec:intro.density}) and (2) Guo has stronger weighted in $v$ so as to obtain stretched exponential decay estimates (which we will discuss later in Section~\ref{sec:intro.decay}).} for the moment, the Guo energy in \cite{Guo12}, when adapted to small $\nu$, corresponds to
\begin{equation}\label{eq:intro.Guo.E}
  \| e^{\phi} \bv^{2M- 2|\alp| - 2|\bt|} H_{\alp,\bt}\|_{L^2_{x,v}}^2,\quad H_{\alp,\bt} :=  \nu^{|\bt|} \rd_x^\alp \rd_v^\bt f,
\end{equation}
after appropriately summing up in $\alp$ and $\bt$. (Note the $\bv$ and $e^\phi$ weights are incorporated in the energy, as discussed in Section~\ref{sec:intro.ideas.Guo}.)

Differentiating the energy \eqref{eq:intro.Guo.E} also gives an integrated decay estimate (cf.~\eqref{eq:intro.naive.energy}) which controls, for $H_{\alp,\bt}$ as in \eqref{eq:intro.Guo.E},
\begin{equation}\label{eq:intro.Guo.D}
 \nu \int_0^T \| e^{\phi} \bv^{2M- 2|\alp| - 2|\bt| - \f 12} H_{\alp,\bt} \|_{L^2_{x,v}}^2 \, \ud t + \nu\int_0^T \sum_{|\bt'| = 1} \| e^{\phi}  \bv^{2M- 2|\alp| - 2|\bt| - \f 32} \rd_v^{\bt'} H_{\alp,\bt} \|_{L^2_{x,v}}^2 \, \ud t.
\end{equation}

One reads off from \eqref{eq:intro.Guo.E} that each $\rd_v$ derivative costs $\nu^{-1}$, and by comparing \eqref{eq:intro.Guo.E} and \eqref{eq:intro.Guo.D}  that integration in $t$ also costs $\nu^{-1}$. Heuristically, this means that one expects to deduce from \eqref{eq:intro.Guo.E} and \eqref{eq:intro.Guo.D} that energy decays on a time scale of order $\nu^{-1}$. In particular, this does not capture the enhanced dissipation generated 
by the interaction between the transport and the diffusive terms. (Notice that the second term in \eqref{eq:intro.Guo.D} gives a better (in $\nu$) estimate for the $\rd_v$ derivatives, but without some corresponding estimates for the $\rd_x$ derivatives, it is unclear how that could improve the rate.)

Inspired by \cite{Beck-Wayne2013,jBmCZ2017, fHfN2004, cV2009}, we modify the Guo energy so as to capture enhanced dissipation. Precisely, for every $(\alp,\bt)$, we define\footnote{This is still not yet the actual energy we use, which also includes commutations with $Y_i$; see Section~\ref{sec:intro.vector.fields}.}
 the following energy at the $H^1$ level for $\rd_x^\alp \rd_v^\bt f$:
 \begin{equation}\label{eq:intro.hypocoercive.E}
\begin{split}
 &\: \|\rd_x^\alp \rd_v^\bt f \|_{\mathcal E_{\alp,\bt}}^2 \\
:= &\: A_0\sum_{|\alp'| \leq 1} \| e^{\phi} \la v\ra^{2(M-|\alp|-|\alp'|-|\bt|)} \rd_x^{\alp'} H_{\alp,\bt} \|_{L^2_{x,v}}^2 + \nu^{\f 23} \sum_{|\bt'| = 1} \| e^{\phi}  \la v\ra^{2(M-|\alp|-|\bt|-1)}\rd_v^{\bt'} H_{\alp,\bt} \|_{L^2_{x,v}}^2\\
&\: + \underbrace{\nu^{\f 13} \int_{\mathbb T^3\times \mathbb R^3} e^{2\phi} \la v\ra^{4(M-|\alp|-|\bt|-1)} \nab_x H_{\alp,\bt} \cdot \nab_v H_{\alp,\bt} \, \ud v\, \ud x}_{(*)}.
\end{split}
\end{equation}
where now $H_{\alp,\bt}:= \nu^{|\bt|/3} \rd_x^\alp \rd_v^\bt f$.
For $A_0$ large but fixed, \eqref{eq:intro.hypocoercive.E} is comparable to
\begin{equation}\label{eq:intro.hypocoercive.E.compare}
\sum_{0\leq |\alp'| \leq 1} \| e^{\phi} \la v\ra^{2(M-|\alp|-|\alp'|-|\bt|)} \rd_x^{\alp'} H_{\alp,\bt} \|_{L^2_{x,v}}^2 + \nu^{\f 23} \sum_{|\bt'| = 1} \| e^{\phi}  \la v\ra^{2(M-|\alp|-|\bt|-1)}\rd_v^{\bt'} H_{\alp,\bt} \|_{L^2_{x,v}}^2.
\end{equation}
 The key here is that despite the equivalence of \eqref{eq:intro.hypocoercive.E} and \eqref{eq:intro.hypocoercive.E.compare}, when differentiating the $(*)$ term in \eqref{eq:intro.hypocoercive.E} by $\f{d}{dt}$, a non-negative useful term $\nu^{1/3} \sum_{|\alp'| = 1} \| e^{\phi} \la v\ra^{2(M-|\alp|-|\bt|-2)} \rd_x^{\alp'} H_{\alp,\bt} \|_{L^2_{x,v}}^2$ is generated. As a result, after suppressing some terms, the $\f{d}{dt}$ derivative of $\|\rd_x^\alp \rd_v^\bt f \|_{\mathcal E_{\alp,\bt}}^2$ precisely controls 
\begin{equation}\label{eq:intro.hypocoercive.D}
\begin{split}
 \nu^{1/3} \| \rd_x^\alp \rd_v^\bt f \|_{\mathcal D_{\alp,\bt}}^2 
\gtrsim &\: \nu^{1/3} (\int_0^T \sum_{|\alp'| = 1} \| e^{\phi} \bv^{2M- 2|\alp| - 2|\bt|-2} \rd_x^{\alp'} H_{\alp,\bt} \|_{L^2_{x,v}}^2 \, \ud t \\
&\: \qquad + \nu^{2/3} \int_0^T \sum_{|\bt'| \leq 1} \| e^{\phi}  \bv^{2M- 2|\alp| - 2|\bt| - \f 72} \rd_v^{\bt'} H_{\alp,\bt} \|_{L^2_{x,v}}^2 \, \ud t).
\end{split}
\end{equation} 
In \eqref{eq:intro.hypocoercive.E} and \eqref{eq:intro.hypocoercive.D}, we see that each $\rd_v$ derivative now costs $\nu^{-\f 13}$. Moreover, comparing the $|\alp'|=1$ and $|\bt'| = 1$ terms in \eqref{eq:intro.hypocoercive.E} with those in \eqref{eq:intro.hypocoercive.D} may suggest that the energy for the derivatives of $f$ decay at a time scale of $\nu^{-\f 13}$, which is much earlier than $\nu^{-1}$, i.e.~this energy captures enhanced dissipation. This enhancement is crucial in controlling the nonlinear terms.

Note, however, that if we compare the $|\alp'| = 0$ term in \eqref{eq:intro.hypocoercive.E} with the $|\bt'| = 0$ term in \eqref{eq:intro.hypocoercive.D}, we see that the term has an additional loss of $\nu^{-2/3}$. This is a reflection of the fact that enhanced dissipation only holds after removing the $x$-average mode. 

\subsubsection{Commuting vector fields and Landau damping}\label{sec:intro.vector.fields}

We capture Landau damping using the commuting vector field method, with vector fields adapted to the flow of the transport equation. The advantage of using such a commuting vector field method approach is that we can hope to prove \emph{transport} estimates and largely ignore the collision term because in principle the collision term  gives rise to terms which have a good sign. 

 Let $Y_i = t\rd_{x_i} + \rd_{v_i}$. We will use $Y_i$ as a commuting vector field, together with $\rd_{x_i}$ and $\nu^{\f 13}\rd_{v_i}$, i.e.~we control $\nu^{|\bt|/3} |\rd_{x}^\alp \rd_{v}^\bt Y^\om f|$ in appropriate weighted spaces (with weights allowed to depend on $(\alp,\bt,\om)$) and define more generally $\mathcal E_{\alp,\bt,\om}$ and $\mathcal D_{\alp,\bt,\om}$ spaces (see \eqref{eq:Eabo}--\eqref{eq:Dabo} for details). The significance of $Y$ can be explained as follows:
 \begin{itemize}
 \item For a solution $f_{lin}$ to the linear transport equation $\rd_t f_{lin} + v\cdot \nabla_x f_{lin} = 0$ with regular data, it is easy to see that $|Y^\om f_{lin}|$ is uniformly bounded in time (since $[\rd_t + v\cdot \nabla_x, Y_i] = 0$), despite $Y$ being a $t$-weighted vector field. Thus controlling the $Y$ derivatives of $f$ can be viewed as proving an \emph{asymptotic transport-like estimate}.
 \item Controlling $Yf$ also implies \emph{decay} of averaged quantities of $f$, thus capturing \emph{phase mixing}. For instance, Poincar\'e's inequality gives that for $\rho := \int_{\mathbb R^3} f \sqrt{\mu} \, \ud v$ and for $\strokedint_{\mathbb T^3}$ the average over $\mathbb T^3$, we have
\begin{equation*}
\begin{split}
&\: \| \rho - \strokedint_{\mathbb T^3} \rho \, \ud x \|_{L^2_x}^2 \ls  \|\nabla_x \rho \|_{L^2_x}^2 \ls \sum_{i} \int_{\mathbb T^3} (\int_{\mathbb R^3} \rd_{x_i} f \sqrt{\mu} \, \ud v)^2 \, \ud x \\
\ls &\: t^{-2} \sum_{i} \int_{\mathbb T^3} (\int_{\mathbb R^3} (Y_i f - \rd_{v_i} f) \sqrt{\mu} \, \ud v)^2 \, \ud x \ls t^{-2} (\|Yf \|_{L^2_{x,v}}^2 + \|f \|_{L^2_{x,v}}^2),
\end{split}
\end{equation*}
where we integrated by parts in the last estimate. 
 \item Importantly for understanding phase mixing in the presence of collision, capturing phase mixing by the vector field $Y$ allows one to still \emph{take advantage of the coercivity of the collision term while proving phase mixing}. More precisely, a term such as $\nu L (Yf)$ (where $L$ is the linear Landau collision operator) that arises in the argument for bounding $Yf$ is not treated as an error, but instead we take advantage of the coercivity of $L$ and make use of this term. (There are associated commutator terms, which we will show to be of a lower order.)
 \end{itemize}

\subsubsection{Density estimates}\label{sec:intro.density}
The above ideas would in principle be sufficient to prove enhanced dissipation and Landau damping for the Landau equation (i.e.~without the Poisson part) in the weakly collisional regime. However, the Vlasov--Poisson--Landau system has terms in $E$ (see $E\cdot \nab_v f$, $E\cdot v f$, and $2(E\cdot v) \sqrt{\mu}$ in \eqref{eq:intro.Vlasov.f}), which require an additional idea. 

The linear $E$ term was handled by Guo \cite{Guo12} using a cleverly designed energy which incorporates $E$ 
so that this linear $E$ term is cancelled in the derivation of the energy estimates. Such a strategy seems difficult to implement when at the same time carrying out ideas in Sections~\ref{sec:intro.hypocoercivity} and \ref{sec:intro.vector.fields}. The nonlinear $E\cdot \nab_v f$, if treated using the energy estimates alone, would give a worse threshold compared to $\ep \nu^{1/3}$.

Instead, we follow the general strategy \cite{BaDe85,jBnMcM2016, eGtNiR2020a, cMcV2011} and prove an \emph{independent} estimate for the density 
that does not depend on the energy estimate. These density estimates rely on resolvent bounds on the kernel of the linearized density, which we now explain. 

In the Vlasov--Poisson case, the $k$-th Fourier mode of the density $\rho$ satisfies a Volterra equation
$$\hat\rho_k(t)+\int_0^t K_k^{VP}(t-\tau)\hat{\rho}_k(\tau) \d \tau = \mathcal{N}_k^{VP}(t),$$
where the kernel $K_k^{VP}(t) =  \f 2{|k|^2} \int_{\R^3} i(k\cdot v) e^{-ik\cdot v t} \mu \, \d v$, and
$\mathcal{N}_k^{VP}(t)$ is an error term containing the contributions from the initial data and the nonlinear terms. 

In the Vlasov--Poisson--Landau case, the Volterra equation is less explicit, and the kernel takes the form
$$K_k(t) =  \frac{2}{|k|^2} \int_{\R^3} ik \cdot S_k(t)[v\sqrt{\mu}] \sqrt \mu\d v,$$
where $S_k(t)$ denotes the linear Landau semigroup generated by the fixed mode linear Landau equation $\part_{t} h  + i k \cdot v h  +  \nu L h  = 0$.

To solve the nonlinear Volterra equation, we take the following steps:
\begin{itemize}
\item We first derive pointwise resolvent estimates (cf.~\cite{eGtNiR2020a,  dHKttNfR2019, dHKttNfR2020}), showing that there is a kernel $G_k$ which is rapidly decaying (and thus negligible) such that 
\begin{equation}\label{eq:intro.resolvent}
\hat\rho_k(t) = \mathcal{N}_k(t) + \int_0^t G_k(t-s)\mathcal{N}_k(s) \; \ud s.
\end{equation}
Relying on the resolvent estimate proven for the Vlasov--Poisson case in \cite{eGtNiR2020a}, it essentially suffices to show that $\lim_{\nu \to 0}\| K_k(\cdot) - K_k^{VP}(\cdot) \|_{L^1_t} \to 0$. This in turn can be obtained by energy and vector field methods for the linear Landau flow for all small $\nu$.
\item We then need to control $\mathcal{N}_k(t)$ in \eqref{eq:intro.resolvent} (see \eqref{def-cNk}, \eqref{def-NNN} for the precise terms). 
The most difficult term here comes from the nonlinear contribution $E \cdot \nab_v f$ associated with the Poisson part, which takes the form
$$\sum_{\l\not =0} \int_0^t\widehat{E}_l(\tau) \cdot \int_{\R^3}S_k(t-\tau) [\widehat{\nabla_v f}_{k-l}(\tau)] \sqrt{\mu}\d v\d \tau.$$
(The nonlinear collisional terms are slightly easier.) As above, we control $S_k(t-\tau)$ using the hypocoercive energy method and the commuting vector field method. The bounds we prove give (1) rapid decay in $\la \nu^{1/3} (t-\tau)\ra$, and (2) bounds associated with the $Y$ vector field, which can be viewed as transport-like bounds. Precisely because we obtain transport-like bounds, this gives hope of controlling the nonlinear term by extending ideas from the density estimates for the Vlasov--Poisson system.
\end{itemize}

\subsubsection{Decay estimates}\label{sec:intro.decay}

Once we close the energy estimates, we adapt the methods of Strain--Guo \cite{StGu06, StGu08} to exchange $v$-weights in the energy with decay in the variable $\nu t$. More precisely, following \cite{StGu08}, we additionally introduce $e^{c|v|^2}$ weights in the energy \eqref{eq:intro.hypocoercive.E} so as to obtain energy decay with a stretched exponential rate. In order to avoid the technicalities associated with simultaneously using $e^{c|v|^2}$ weights and commuting with $Y$, we only use $e^{c|v|^2}$ weights when there are no $Y$ commutations. At first this only gives decay of energy without $Y$ commutations, yet a full decay statement can then be achieved by interpolation.

This allows us to obtain the following decay results (see precise statements in Theorem~\ref{t.main}):
\begin{enumerate}
\item Essentially arguing as Strain--Guo, but taking into account the dependence on $\nu$, we prove that the energy decays with an $\exp(-\de (\nu t)^{\f 23})$ rate (for $\de>0$ small).
\item As discussed earlier, there is an enhanced dissipation (which operates at the time scale of $O(\nu^{-\f 13})$ instead of $O(\nu^{-1})$) after removing the zeroth spatial Fourier mode. Instead of explicitly removing the zeroth mode, we prove an enhanced decay estimate by considering an energy in which $f$ has at least one $\rd_x$ derivative. For such an energy, we prove energy decay with a rate $\min \{ \exp(-\de (\nu^{\f 13} t)^{\f 13}), \exp(-\de (\nu t)^{\f 23}) \}$.
\end{enumerate}

In order to obtain the decay estimates, in addition to deriving weighted energy estimates, we also need to propagate the stretched exponential decay in the density estimates (recall Section~\ref{sec:intro.density}). This requires (1) a precise estimate for the resolvent, which incorporates the stretched exponential decay, and (2) a more careful nonlinear analysis. This more precise nonlinear analysis (see for instance the decomposition of the density in \eqref{eq:rho1.bound}--\eqref{eq:rho2.bound}) is devised so that one does not see an analogue of the top-order loss in the energy boundedness argument (see Section~\ref{sec:additional.difficulties}), which is now possible because we are only propagating the stretched exponential decay estimate for the low-order derivatives. (See the beginning of Section~\ref{sec:exp-decay-density} for further remarks on the nonlinear density estimates.)

Once we obtain the enhanced decay rate for the nonzero modes, the density estimate implies that the Fourier modes $\rho_k$ obey Landau damping-type uniform inverse polynomial decay estimates:
$$|\rho_k| \ls \ep \nu^{1/3} (1+|k| + |kt|)^{-N} \min\{e^{-\de(\nu^{1/3} t)^{1/3}}, e^{-\de(\nu t)^{2/3}} \}.$$

\subsubsection{Structure of the energy estimates}

In order to carry out the full scheme described above, we implicitly need that under suitable bootstrap assumptions, we can bound various energies which use only some subsets of commutators.

For instance, for the linear Landau energy estimates used in the density estimates (see Section~\ref{sec:intro.density}), we need to commute the linear Landau equation for each fixed mode with a large number of $Y$ derivatives, but with at most one $\rd_x$ or $\rd_v$ derivatives. This is important for obtaining the correct constants in the estimates.

On the other hand, for the stretched exponential decay (see Section~\ref{sec:intro.decay}), we need an energy without any $Y$ commutations (since, as discussed above, we do not put $Y$ commutations together with $e^{c|v|^2}$ weights). For the decay of the full solution, we use only the hypocoercive energy without any additional commutations. For the enhanced dissipation, we need to remove the $k = 0$ $x$-Fourier mode. For this purpose, we consider the hypocoercive energy with exactly one additional $\rd_x$ commutation.

To propagate the boundedness of energies with only suitable subsets of commutators, we define an energy $\mathbb E_{N_{\alp}^{low},N_{\alp,\bt},N_\bt,N_\om}$, which is a sum of appropriate $\mathcal E_{\alp,\bt,\om}$ energies. The parameters $N_{\alp}^{low},N_{\alp,\bt},N_\bt,N_\om$ describe the commutators used: they depend not only on the total number of commutators, but also on various upper and lower bounds on each type of commutators used; see \eqref{eq:combined.norms}. 

\subsubsection{Additional difficulties}\label{sec:additional.difficulties} While we have already described the main conceptual difficulties, the even more interesting difficulties lie in the technicalities. We highlight a few technical issues here.\\

\textbf{Asymmetric use of commutators.} At the top order of energy, we do not allow for an arbitrary combination of the commutator vector fields. Instead, we only allow for 
\begin{equation}\label{eq:intro.allowed.commutators}
\rd_x^\alp \rd_v^\bt Y^\om f,\quad \rd_{x_i} \rd_x^\alp \rd_v^\bt Y^\om f, \quad \rd_{v_i} \rd_x^\alp \rd_v^\bt Y^\om f,\quad \rd^2_{v_i v_j} \rd_x^\alp \rd_v^\bt Y^\om f
\end{equation}
for $0\leq |\alp| + |\bt| + |\om| \leq N_{max}$. (See the $\wtE_{\alp,\bt.\om}$ and $\wtD_{\alp,\bt,\om}$ norms in \eqref{eq:wtEabo}--\eqref{eq:wtDabo}.) Put differently,
\begin{itemize}
\item at the top level (with $N_{max}+2$ derivatives), at least two commutators have to be $\rd_v$;
\item at the penultimate level (with $N_{max}+1$ derivatives), at least one commutator has to be $\rd_x$ or $\rd_v$;
\item at lower levels (with $N_{max}$ derivatives or fewer), the commutators can be arbitrary combinations of $\rd_x$, $\rd_v$ and $Y$.
\end{itemize}

On the one hand, this is \emph{needed} because the nonlinear density estimates (unlike the energy estimates) lose derivatives, and thus to bound $\rd_x^\alp Y^\omega \rho$ requires estimates for $\rd_x^\alp \rd_v^{\bt} Y^\omega f$ for $|\bt|\leq 3$, which can only be obtained by commuting with two additional $\rd_v$ derivatives. On the other hand, this is \emph{possible} because commuting $\rd_v^2$ does not generate terms like $\rd_x^2 f$. \\

\textbf{Growth of top-order energy.} When controlling the energy for the terms \eqref{eq:intro.allowed.commutators} with $|\alp| + |\bt| + |\om| = N_{max}-1$ or $\,N_{max}$, we allow the energy to grow either in $t$ or $\nu^{-\f 13}$. The underlying reason is that the decay of $E$ by Landau damping is determined by the regularity. The decay at the highest level is thus slower, and ultimately the terms $2(E\cdot v) \sqrt{\mu}$ and $E\cdot \nab_v f$ in \eqref{eq:intro.Vlasov.f} cause the top-order energy to grow.

Nevertheless, importantly, even though the energies at the top two orders grow, the nonlinear analysis in the density estimates (see Section~\ref{sec:intro.density}) still allows one to prove a desired density estimate without loss at the top level. This allows the bootstrap argument to close.\\

\textbf{Different decay rates for the $k = 0$ and $k\not = 0$ modes.} As we have already discussed above (see Section~\ref{sec:intro.hypocoercivity}), enhanced dissipation is only seen for the spatial Fourier modes $k \not = 0$, i.e. the $k = 0$ mode decays slower. This in particular means that in the nonlinear analysis, we need to be careful of terms without derivatives, as they could potentially be more slowly decaying. In all cases, we show that there is an integration by parts giving bounds with the right decay; see Lemmas~\ref{lem-Edvmu} and \ref{lem-EdvY}.\\

\textbf{Handling some lowest order terms.} Finally, recall that the linearized Landau operator has a non-trivial kernel, which was dealt with in \cite{Guo12} by analyzing a separate system for the macroscopic quantities. The hypocoercive energy allows us to sidestep this complication; see a related observation in \cite{jBfW2020}. More precisely, the hypocoercive energy gives better bounds on the $\rd_x$ derivatives, so that we only need to control the $x$-mean of the contribution from the kernel, which in turn can be treated trivially using the conservation laws.

\subsection{Related works}\label{sec:related.works}

\subsubsection{Landau damping for the Vlasov--Poisson system}

Linear Landau damping for the Vlasov--Poisson system was first observed in Landau's seminal work \cite{lL1946}. A mathematical  breakthrough was achieved by Mouhot--Villani \cite{cMcV2011}, justifying Landau damping in a nonlinear setting under analyticity assumptions. This has been extended and simplified in \cite{jBnMcM2016, eGtNiR2020a}. More recently, the effect of plasma echoes have been further explored in \cite{jB2021, eGtNiR2020b}. See also \cite{eCcM1998, hjHjjlW2009} for earlier constructions of \emph{some} Landau damped solutions, \cite{jBnMcM2018, jBnMcM2020, rGjS1994, rGjS1995, dHKttNfR2019, dHKttNfR2020} for works on the whole space (instead of the torus), and \cite{bY2016} for the relativistic case.

\subsubsection{Nonlinear stability of global Maxwellians}

In the $\nu = 1$ case of \eqref{eq:VPL}--\eqref{eq:Poisson} (or its two-species analogue), the nonlinear asymptotic stability of global Maxwellians was first proven in Guo's \cite{Guo12} in a periodic box; see also \cite{dqD2021,rjDtYhjZ2011}. The corresponding stability result on $\mathbb R^3$ was proven in \cite{StZh13} (with alternative proofs in \cite{cHyjL2016, yjLljXhjZ2014, yjW2012}). See also the more recent \cite{rjDhjY2020} for stability of local Maxwellians representing rarefaction waves.

The work \cite{Guo12} can be viewed in the context of a larger program of stability of Maxwellians result using energy methods. This began with Guo's seminal work \cite{Guo02} for the Landau equation, and inspired many subsequent works; see \cite{AMUXY12.3, AMUXY12, AMUXY12.2, CaMi17, CaTrWu17, CaTrWuErratum17, GrSt11, Guo02.2, Guo03, Guo03.2, lHhjY2007, StGu04, StGu06, StGu08} and the references therein for further discussions.

\subsubsection{Related works in the physics literature} There have been many works in the physics literature studying the interaction of Landau damping and weak collisions, see \cite{boyd2003physics, DubNaz1994, GoldRuther1997, John1971, LenBern1958, MalmWhar1964, MaWhGoOn1968plasma, NgBhSk1999, NgBhSk2006, On1968, Ryu1999, Short2002, Stix1992, SuOb1968, vanneste1998strong, yu2002diocotron, yu2005phase} and the references therein.

\subsubsection{Hypocoercivity}\label{sec:hypo.related.works}

The method of hypocoercivity has roots in the theory of hypoelliptic operators \cite{lH1967, jjK1977}. The use of hypocoercivity method for decay estimates was pioneered Eckmann--Hairer \cite{jpEmH2003}, H\'erau--Nier \cite{fHfN2004} and Helffer--Nier \cite{bHfN2005}. See \cite{lDcV2005, fH2006, fH2007, lDfS2009, fH2018} for a small sample of further results, and see particularly for results in a weakly viscous setting \cite{jBmCZ2017, mCZtmEkW2020}. Finally, we refer the reader to \cite{lD2006, cV2009, fH2018} for systematic discussions.

\subsubsection{Weakly collisional regimes for kinetic equations}

Despite its physical importance, there are very few mathematical works on weakly collisional regimes for kinetic equations. This type of questions were raised in the mathematics literature for instance in \cite[IV.25.8.2]{PrincetonCompanion} and \cite[Section~8]{cV2013}. The only nonlinear result is the work of Bedrossian \cite{jB2017} on the Vlasov--Poisson--Fokker--Planck system that we already mentioned. This was predated slightly earlier by a linear analysis in \cite{iT2017}. More recently, the linear analysis was extended to include effects of a uniform background magnetic field \cite{jBfW2020}.

\subsubsection{Related models with vanishing dissipation}

Even though there are not many works on weakly collisional regimes for kinetic equations, there are closely related models, problems and results in fluid dynamics. See \cite{jBmCZ2017,jBjGnM2015, jBpGnM2017, jBpGnM2019, jBpGnM2020, jBsimH2020, jBnMvV2016, jBvVfW2018, qCtLdyWzfZ2020, mCZtmEkW2020, sjDzlL2020, tG2018, GNRS, Kelvin1887, LatBer2001, xL2020, nMwrZ2019, nMwrZ2020, Orr1907, dyWzjZ2018, dyWzfZwrZ2020, WeZhZh20,cZ2020} and the references therein for a sample of results. We in particular highlight the paper \cite{mCZtmEkW2020} for its use of the hypocoercive energy method.

\subsubsection{Commutating vector field method for kinetic models}\label{sec:vector.field.method.related.works}

The commutating vector field method, pioneered in \cite{sK1985} for quasilinear wave equations, has been very successful in capturing dispersion to prove  global stability for nonlinear evolution equations. Recently, it has likewise found many applications for collisionless kinetic equations. In particular, the stability of vacuum has been established in many different settings \cite{Bi17, FaJoSm17.1, Sm16, Wo18}, and the stability of the Minkowski spacetime for the Einstein--Vlasov system in general relativity has also been resolved \cite{lBdFjJjSmT2020, FaJoSm17, hLmT20, jS2018, Ta17}.  (See also \cite{BaDeGo84, BaDe85, GlSt87, GlSc88, Wa18.1, Wa18.3, Wa18.2} for related works on stability of vacuum type results for collisionless models.) For collisional models, recent works using the commutating vector field method give --- for the first time --- stability of vacuum results for collisional models with a long range interaction, first for the Landau equation \cite{jL2019, sC2020a}, and more recently for Boltzmann equation without angular cutoff \cite{sC2020b}. As for phase mixing, it has been successfully used for the linearized $\beta$-plane equation in \cite{WeZhZh20}.

\subsection{Discussions}\label{sec:discussion}

\subsubsection{Related models}

\begin{enumerate}
\item \textbf{Magnetic field.} Using the methods introduced here, one can potentially study the problem in the presence of a constant external magnetic field as in \cite{jBfW2020}, but now also with the Landau collision operator.
\item While the Landau collision operator is the most commonly used collision operator in plasma physics, there are other collision models for which the weakly collision regime is of interest:
\begin{enumerate}
\item \textbf{The Boltzmann operator.} One can consider both the case with or without cutoff. In either case, one expects the threshold to be different from the Landau case. See discussions in \cite{jB2017}.
\item \textbf{The Lenard--Balescu operator.} The Lenard--Balescu operator is significantly more complicated than the Landau operator and takes into account collective screening effects. Notice, however, that mathematical results for the Lenard--Balescu operator in the Coulomb case are so far confined to the linear setting \cite{ahMrlL1973, rmS2007} (see however \cite{mDrW2021}), and even a nonlinear result analogous to Theorem~\ref{thm:guo} appears to be out of reach.
\end{enumerate}
\end{enumerate}

\subsubsection{Sharp threshold}

We now discuss the conjectured threshold \eqref{eq:threshold.diagram}. While our paper concerns only initial data with high (but finite) Sobolev regularity, it is of interest to consider other function spaces, and it is expected that the sharp threshold may depend on the function space.
\begin{enumerate}
\item (High-regularity Sobolev spaces) It is conjectured by Bedrossian  that for the Vlasov--Poisson--Fokker--Planck system considered in \cite{jB2017}, $O(\ep\nu^{1/3})$ is the sharp threshold, possibly up to logarithms, due to the possible occurrence of plasma echoes. The same heuristics in \cite{jB2017} applies to our case (see the discussions before Theorem~\ref{thm:main.intro}) suggesting that the threshold in Theorem~\ref{thm:main.intro} may be sharp.
\item (Gevrey spaces) In Gevrey-$\f 1s$ spaces with $s>\f 13$, global stability is established for the Vlasov--Poisson system. This gives hope that in the weakly collisional regime, one can treat initial data of size $O(\de)$ in these Gevrey spaces, independently of the collisional parameter $\nu$.
\item (Low-regularity Sobolev spaces) Finally, recall that the $\nu$-independent decay rate by phase mixing depends on the regularity of the initial data. Thus in very low regularity spaces (e.g. those in \cite{rjDsqLsSrmS2021} so that global stability still holds for the Landau equation with $\nu = 1$), the stabilizing effect of phase mixing may be weaker. Nevertheless, it is still of interest to understand whether one can allow at least for $O(\de\nu^{1/2})$ data in a \emph{low-regularity} space . A similar question may also be studied in the case of bounded domain where one necessarily carry out low-regularity analysis due to boundary effects.
\end{enumerate}

\subsection{Outline of the paper}\label{sec:outline} 
The remainder of the paper is structured as follows.
\begin{itemize}
\item In \textbf{Section~\ref{sec:notations}}, we introduce the notation that will be in effect for the rest of the paper.\\

\item In \textbf{Section~\ref{sec:statement}}, we give a precise statement of the main theorem.\\

\item In \textbf{Section~\ref{sec:collision.op}}, we collect some facts about the Landau collisional operator.\\

\item In \textbf{Section~\ref{sec-EE}}, we set up the main energy estimate for the whole Vlasov--Poisson--Landau system. In particular, the precise energy and dissipation norms will be introduced. 
\\

\item In \textbf{Section~\ref{sec-linearLandau}}, we perform energy estimates for the linear Landau flow that are needed for closing the density estimates.\\

\item In \textbf{Section~\ref{sec-lineardensity}}, we provide pointwise resolvent bounds on density of the linearized Vlasov--Poisson--Landau system.\\

\item In \textbf{Section~\ref{sec:density}}, we establish the nonlinear density estimates under the bootstrap assumptions on the energy.\\

\item In \textbf{Section~\ref{s.closing_eng}}, we close the main nonlinear energy estimates for the Vlasov--Poisson--Landau system.\\

\item In \textbf{Section~\ref{sec:global}}, we prove global existence of solutions via a continuity argument.\\

\item In \textbf{Section~\ref{sec:exp-decay-density}}, we prove stretched exponential decay for the density at lower order.\\

\item In \textbf{Section~\ref{sec:exp-decay-energy}}, we prove stretched exponential decay for lower order energy.\\

\item In \textbf{Section~\ref{sec:putting-together}}, we put everything together and prove the main conclusions of the paper including the global existence, the stretched exponential decay, as well as the uniform Landau damping for the density.\\

\item Finally, in \textbf{Appendix~\ref{sec:appendix}}, we give two versions of Strain--Guo lemmas adapted to our setting.
\end{itemize}

\subsection*{Acknowledgments} S.~Chaturvedi~and J.~Luk~are supported by the NSF grant DMS-2005435. J.~Luk~also gratefully acknowledges the support of a Terman Fellowship. T.~Nguyen is partly supported by the NSF under grant DMS-1764119, an AMS Centennial fellowship, and a Simons fellowship.

\section{Notation}\label{sec:notations}

We first introduce a reformulation of the problem in terms of $f:= \f{1}{\sqrt{\mu}} (F-\mu)$, and then introduce some notations that will be used throughout the paper. 

\subsection{Reformulation in terms of $f$}\label{sec:f}

For the remainder of the paper, it is convenient to first rewrite the problem in terms of $f$ (see Section~\ref{sec:ideas.prelim}). Define $f$ via
\begin{equation}
F = \mu + \sqrt{\mu} f.
\end{equation}
In the remainder of the paper, we will solve \eqref{eq:VPL}--\eqref{eq:Poisson} with initial data $f_{\vert_{t=0}} = f_0$ that in particular satisfies $\strokedint_{\mathbb T^3} \int_{\mathbb R^3} f_0 \sqrt{\mu} \, \ud v\, \ud x=0$. The conservation of mass ensures that 
$$\strokedint_{\mathbb T^3} \int_{\mathbb R^3} f(t,x,v) \sqrt{\mu(v)}\, \ud v \, \ud x = 0.$$

Under this mean zero condition, it can be deduced that the Vlasov--Poisson--Landau system \eqref{eq:VPL}--\eqref{eq:Poisson} is equivalent to the following system for $f$:
\begin{subequations}
\begin{align}
\label{eq:Vlasov.f}
\rd_t f + v\cdot \nabla_x f + E\cdot \nabla_v f - E\cdot v f - 2 (E\cdot v) \sqrt{\mu} + \nu L f = \nu \Gamma(f,f), \\
\label{eq:Poisson.f}
E(t,x) = -\nabla \phi(t,x), \qquad - \Delta \phi  = \int_{\mathbb R^3} f(t,x,v) \sqrt{\mu(v)}\, \ud v,
\end{align}
\end{subequations}
where, following \cite[Lemma~1]{Guo02},
\begin{itemize}
\item the linear Landau operator $L$ admits a decomposition
\begin{equation}\label{def-L}L =-\mathcal K -\mathcal A ,\end{equation}
where $\mathcal A$ and $\mathcal K$ are given respectively by
\begin{align}
\label{eq:A.def}
\mathcal A g&:=\part_{v_i}(\sigma_{ij}\part_{v_j}g)-\sigma_{ij}v_iv_j g+\partial_{v_i}\sigma_i g,
\\
\label{eq:K.def}
\mathcal K g&:=-\mu^{-\frac{1}{2}}(v)\part_{v_i}\left\{\mu(v)\int_{\R^3}\Phi_{ij}(v-v')\sqrt{\mu}(v')[\partial_{v_j} g(v')+v_j'g(v')]\d v'\right\},
\end{align}
with 
\begin{equation}\label{eq:sigma}
\sigma_{ij}:= \Phi_{ij}\star \mu,\quad \sigma_i :=\Phi_{ij} \star(v_j \mu) = \sigma_{ij} v_j,
\end{equation}
for $\Phi_{ij}$ as in \eqref{eq:Phiij}, and $\star$ being the $v$-convolution,
\item and the nonlinear Landau collisional term $\Gamma(f,f)$ is given by
\begin{equation}\label{eq:Gamma.def}
\begin{split}
\Gamma(g_1,g_2)& := \partial_{v_i} \Big[ \Big(\Phi_{ij} \star (\mu^{1/2} g_1) \Big)\partial_{v_j} g_2 \Big] - \Big[ \Phi_{ij} \star \Big( v_i \mu^{1/2} g_1 \Big) \Big] \partial_{v_j} g_2
\\& \quad -  \partial_{v_i} \Big[ \Big(\Phi_{ij} \star (\mu^{1/2}\partial_{v_j}g_1) \Big)g_2 \Big]  +  \Big[ \Phi_{ij} \star \Big( v_i \mu^{1/2} \partial_{v_j}g_1 \Big) \Big] g_2 . 
\end{split}
\end{equation}
\end{itemize}
The rest of the paper deals with solutions $f$ to \eqref{eq:Vlasov.f}-\eqref{eq:Poisson.f}.

\subsection{Notations}

\textbf{Vector Field $Y$}. For any $t\geq 0$ and $i\in \{1,2,3\}$, we introduce the time-dependent vector field 
$$Y_i= \rd_{v_i} + t \rd_{x_i}.$$\\ 

\textbf{Multi-indices}. Given a multi-index $\alpha=(\alpha_1,\alpha_2,\alpha_3)\in (\N\cup \{0\})^3$, we define $\part_x^\alpha=\part^{\alpha_1}_{x_1}\part^{\alpha_2}_{x_2}\part^{\alpha_3}_{x_3}$ and similarly, $\part^\beta_v = \partial_{v_1}^{\beta_1}\partial_{v_2}^{\beta_2}\partial_{v_3}^{\beta_3}$ and $Y^\om = Y_1^{\om_1} Y_2^{\om_2} Y_3^{\om_3}$. Multi-indices are added according to the rule that if $\alpha'=(\alpha'_1,\alpha'_2,\alpha'_3)$ and $\alpha''=(\alpha''_1,\alpha''_2,\alpha''_3)$, then $\alpha'+\alpha''=(\alpha'_1+\alpha''_1,\alpha'_2+\alpha''_2,\alpha'_3+\alpha''_3)$. We also set $|\alpha|=\alpha_1+\alpha_2+\alpha_3$.\\

\textbf{Japanese brackets}. Given $w \in \mathbb R^n$, $n\in \mathbb N$, define $\la w \ra := (1+|w|^2)^{\f 12}$.   \\

\textbf{Velocity weights}. Fix $N_{max}\ge 9$, $M = N_{max} +30$ and $q_0 \in (0,1)$ (cf.~Theorem~\ref{t.main}). For any $\vartheta \in \{0,2\}$ and any triple of multi-indices $(\alpha,\beta,\omega)$ such that $|\alpha|+|\beta|+|\omega|\leq N_{max}$, we introduce velocity weights 
\begin{equation}\label{def-w}
w_{\alpha,\beta,\omega}=\jap{v}^{\ell_{\alpha,\beta,\omega}}e^{\frac{q|v|^\vartheta}{2}}\end{equation}
for $q = \begin{cases} q_0 & \mbox{if $\vartheta = 2$} \\ 0 & \mbox{if $\vartheta = 0$}\end{cases}$, and for the polynomially weighted index  
\begin{equation}\label{def-ell}
\ell_{\alpha,\beta,\omega}=2M-(2|\alpha|+2|\beta|+2|\om|)
\end{equation}
to be used throughout in the analysis. These velocity weights will be appropriately associated with norms for derivatives $\partial_x^\alpha\partial_v^\beta Y^\omega$ of the Vlasov--Poisson--Landau solutions. Note that in the applications below, when $\omega \not =0$, we take $\vartheta = 0$ in \eqref{def-w}: namely, only polynomial velocity weights will be used. \\

\textbf{$L^p$ spaces}. We will work with $L^p$ spaces with standard norm $\|\cdot \|_{L^p_x}$ or $\|\cdot \|_{L^p_v}$ for functions depending on $x$ or $v$, respectively. We also use mixed norms 
$$\norm{h}_{L^p_xL^q_v}:=(\int_{\T^3}(\int_{\R^3} |h|^q(x,v) \d v)^{\frac{p}{q}} \d x)^{\frac{1}{p}}$$
which reduce to $\|\cdot \|_{L^p_{x,v}}$ in the case when $p=q$.\\

\textbf{Weighted norms}. Fix $q_0 \in (0,1)$ for the remainder of the paper (cf.~Theorem~\ref{t.main}). 

For $\ell \in \R$ and $1\leq p\leq\infty$, we define the following weighted norms
$$
\begin{aligned}\norm{h}_{L^{p}_v(\ell,0)}&:=\norm{\jap{v}^{\ell}h}_{L^{p}_v}, \quad \norm{h}_{L^{p}_v(\ell,2)} & := \norm{\jap{v}^{\ell}e^{\frac{q_0 |v|^2}{2}}h}_{L^{p}_v},
\end{aligned}$$
where $q_0 \in (0,1)$ is the fixed constant above. Analogously, we introduce the following dissipation norms
\begin{equation}\label{eq:Delta.def}
\begin{aligned}
\norm{h}^2_{\Delta_{v}(\ell,0)} & :=\int_{\R^3}\jap{v}^{2\ell}\left[\part_{v_i}g(\sigma_{ij})\part_{v_j}g+\sigma_{ij}\frac{v_i}{2}\frac{v_j}{2}g^2\right]\d v,
\\
\norm{h}^2_{\Delta_{v}(\ell,2)} & := \int_{\R^3}\jap{v}^{2\ell} e^{q_0|v|^2}\left[\part_{v_i}g(\sigma_{ij})\part_{v_j}g+\sigma_{ij}\frac{v_i}{2}\frac{v_j}{2}g^2\right]\d v,
\end{aligned}
\end{equation}
for $q_0 \in (0,1)$ as above, and $\sigma_{ij}$ as in \eqref{eq:sigma}. We also use mixed norms  $\norm{h}_{L^{p_1}_xL^{p_2}_v(\ell,\vartheta)}$, $\norm{h}_{L^p_x\Delta_{v}(\ell,\vartheta)}$ and $\norm{h}_{\Delta_{x,v}(\ell,\vartheta)}:= \norm{h}_{L^2_x\Delta_{v}(\ell,\vartheta)}$ in an obvious manner, for $\vartheta\in \{0,2\}$. Using the properties of $\sigma_{ij}$ (see Lemma \ref{l.sigma_diff} below), we note that 
\begin{equation} \| h\|_{L_v^2(\ell - 1/2,\vartheta)} + \| \nabla_v h \|_{L_v^2 (\ell - 3/2,\vartheta)} \lesssim \| h\|_{\Delta_v(\ell,\vartheta)} .\end{equation}

For all of the above norms, we also define analogous norms, specified with a $'$, so that when $\vartheta = 2$, they have a weaker Gaussian weight in $v$, with $q_0$ replaced by $q_0/2$. More precisely, we define 
\begin{equation}\label{eq:Lp.'.1}
\norm{h}_{L^{p}_v(\ell,0)'} := \norm{h}_{L^{p}_v(\ell,0)},\quad \norm{h}_{\Delta_{v}(\ell,0)'} := \norm{h}_{\Delta_{v}(\ell,0)}
\end{equation} 
and
\begin{equation}\label{eq:Lp.'.2}
\norm{h}_{L^{p}_v(\ell,2)'}  := \norm{\jap{v}^{\ell}e^{\frac{q_0 |v|^2}{4}}h}_{L^{p}_v},\,\,\, \norm{h}^2_{\Delta_{v}(\ell,2)'}  := \int_{\R^3}\jap{v}^{2\ell} e^{\f{q_0|v|^2}2}\left[\sigma_{ij} \part_{v_i}g \part_{v_j}g+\sigma_{ij}\frac{v_i}{2}\frac{v_j}{2}g^2\right]\d v.
\end{equation}
\\


\section{Statement of the main theorem}\label{sec:statement}

The following is the precise version of our main theorem.

\begin{theorem}\label{t.main}
Let $q_0 \in (0,1)$ and $N_{max} \in \mathbb N$ with $N_{max} \geq 9$. Define $M = N_{max} +30$. There exist $\ep_0 = \ep_0(q_0, N_{max}) >0$ and $\nu_0 = \nu_0(q_0, N_{max}) > 0$ such that the following hold. 

Consider the Vlasov--Poisson--Landau system \eqref{eq:Vlasov.f}--\eqref{eq:Poisson.f} with collision parameter $\nu \in (0,\nu_0]$. Suppose that the initial function $f_0$ is smooth and satisfies
\begin{equation}\label{eq:thm.assumption.1}
\int_{\mathbb T^3 \times \R^3} f_0 \sqrt{\mu} \, \ud v\, \ud x = \int_{\mathbb T^3 \times \R^3} f_0 v_j \sqrt{\mu} \, \ud v\, \ud x = \int_{\mathbb T^3 \times \R^3} f_0 |v|^2 \sqrt{\mu} \, \ud v\, \ud x + \int_{\R^3} |E_0|^2 \, \ud x = 0,
\end{equation}
and for some $\ep\in (0,\ep_0]$, $f_0$ obeys the smallness bound
\begin{equation}\label{eq:thm.assumption.2}
\sum_{|\alp|+|\bt| \leq N_{max}+2} \|e^{q_0|v|^2} \bv^{2M} \rd_x^\alp \rd_v^\bt f_0 \|_{L^2_{x,v}} \leq \ep \nu^{1/3}.
\end{equation}

Then there exists a global-in-time smooth solution $f$ to \eqref{eq:Vlasov.f}--\eqref{eq:Poisson.f}  with $f_{|t=0} = f_0$. 
Moreover, there exist constants $C>0$ and $\de>0$ (depending only on $q_0$ and $N_{max}$, and in particular independent of $\ep$ and $\nu$) such that the following estimates hold for all $t \in [0,\infty)$:
\begin{enumerate}
\item (Boundedness of weighted energy)
\begin{subequations}
\begin{align}
 \sum_{|\alp| + |\bt| \leq N_{max}-1 } \nu^{|\bt|/3} \|e^{q_0|v|^2} \bv^{2M-2|\alp|-2|\bt|} \rd_x^\alp \rd_v^\bt f\|_{L^2_{x,v}}(t) \leq C \ep \nu^{1/3}, \label{eq:main.energy.lowest}\\
 \sum_{|\alp| + |\bt| + |\om| \leq N_{max}-2 } \nu^{|\bt|/3} \|\bv^{2M-2|\alp|-2|\bt|-2|\om|} \rd_x^\alp \rd_v^\bt Y^\om f\|_{L^2_{x,v}}(t) \leq C \ep \nu^{1/3}.
\label{eq:main.energy}
\end{align}
\end{subequations}
\item (Energy decay) 
\begin{equation}\label{eq:main.decay}
 \sum_{|\alp| + |\bt| + |\om| \leq N_{max} -2} \nu^{|\bt|/3} \|\rd_x^\alp \rd_v^\bt Y^\om f\|_{L^2_{x,v}}(t) \leq C \ep \nu^{1/3} e^{-\de (\nu t)^{\f 23} }.
 \end{equation}
 \item (Enhanced dissipation)
For $f_{\not = 0}(t,x,v) := f(t,x,v) - \strokedint_{\mathbb T^3} f(t,x,v) \, \ud x$, 
\begin{equation}\label{eq:main.decay.neq0}
 \sum_{|\alp| + |\bt| + |\om| \leq N_{max} -2} \nu^{|\bt|/3} \|\rd_x^\alp \rd_v^\bt Y^\om f_{\not = 0}\|_{L^2_{x,v}}(t)  \leq C \ep \nu^{1/3} \min \{e^{-\de (\nu^{\f 13} t)^{\f 13} }, e^{-\de (\nu t)^{\f 23} }\}.
\end{equation}
\item (Uniform polynomial decay rate) For $\rho(t,x) = \sum_{k \in \mathbb Z^3} \rho_k(t) e^{ik\cdot x}$, it holds that
\begin{equation}\label{eq:uniform.Landau.damping}
|\rho_{k}|(t) \leq C\ep \nu^{1/3} (1+|k|+|kt|)^{-N_{max}+1}  \min \{e^{-\de (\nu^{\f 13} t)^{\f 13} }, e^{-\de (\nu t)^{\f 23} }\}
\end{equation}
for every $k \in \mathbb N \setminus \{0\}$.
\end{enumerate}
Finally, the solution is the unique smooth global solutions obeying the bound \eqref{eq:main.energy.lowest}.
\end{theorem}

A few remarks are in order.

\begin{remark}[Local existence, uniqueness and continuation criterion]
We do not explicitly handle local existence and uniqueness in this paper, but they follow in essentially the same manner as \cite{cHsSaT2019}. Using their methods, we have a local existence and uniqueness result for initial data with 
\begin{equation}\label{eq:norm.for.LWP}
\sum_{|\alp| + |\bt| \leq 4} \| e^{\rho |v|^2} \rd_x^\alp \rd_v^\bt F \|_{L^2_{x,v}} < \infty
\end{equation}
 for any $\rho >0$. In particular, as long as one can guarantee the norm in \eqref{eq:norm.for.LWP} to be finite, we have existence and uniqueness of solutions. (Recalling that $F = \mu + \sqrt{\mu} f$, we note that the estimate \eqref{eq:main.energy.lowest} is much stronger than \eqref{eq:norm.for.LWP}. For this reason, in most of the proof, we will focus on proving the a priori estimates. See the proof of Theorem~\ref{thm:existence}.
\end{remark}

\begin{remark}[Some top-order bounds not stated]
Notice that some of the top-order bounds are not stated. In fact, the highest order energies will not be shown to be bounded by $C\ep \nu^{1/3}$, but instead has a loss in $\nu^{-1/3}$ or $\la t \ra$; see Theorem~\ref{theo-mainEE}.
\end{remark}

\begin{remark}[Exponential $v$-weights]
We only propagate the exponential weight $e^{q_0|v|^2}$ when there are no $Y$ derivatives, i.e.~when $|\om| = 0$. Note that the techniques of \cite{StGu08} require using the exponential weights in order to obtain the stretched exponential decay  in \eqref{eq:main.decay} and \eqref{eq:main.decay.neq0}. We will therefore first obtain the stretched exponential decay statement for $|\om| = 0$, and then deduce the full statement by interpolation.
\end{remark}

\begin{remark}[Stretched exponential decay]
Notice that in a manner similar to \cite{Guo12}, our ($\nu$-dependent) time decay is not exponential, but is instead only stretched exponential. For the $(\nu t)$-decay, we have $e^{-\de (\nu t)^{2/3}}$ decay, where the $2/3$-power is the same as \cite{Guo12}. On the other hand, for technical reasons concerning the $v$ weights in the hypocoercive energy, for the $(\nu^{1/3} t)$-decay, we have a slightly weaker exponent and only have $e^{-\de (\nu^{1/3} t)^{1/3}}$ decay.

As far as we are aware, it is not known whether this is sharp even for the linearized problem with $\nu = 1$.
\end{remark}

\begin{remark}[$\nu^{1/3}$ weights for $\rd_v$ derivatives]
Notice that in all estimates \eqref{eq:main.energy.lowest}--\eqref{eq:main.decay.neq0}, every $\rd_v$ derivative loses a power of $\nu^{-1/3}$. These estimates can be improved for short times so that $\nu^{-1/3}$ is replaced by $\min \{\nu^{-1/3}, \la t \ra)\}$. For this one only needs to perform the corresponding change in the energy estimates. We will not pursue the details.
\end{remark}

The remainder of the paper will be devoted to the proof of Theorem~\ref{t.main}. \textbf{From now on, we work under the assumptions of Theorem~\ref{t.main}. We will use the convention that, unless otherwise stated, all constants $C$ or implicit constants in $\ls$ will be allowed to depend on $q_0$ and $N_{max}$, but are not allowed to depend on $\ep$ or $\nu$, as long as $\ep_0$ and $\nu_0$ are sufficiently small.}

\section{Landau collision operator}\label{sec:collision.op}

In this section, we recall basic properties of the linear and quadratic Landau collision operators $Lf$ and $\Gamma(f,f)$ (recall \eqref{def-L}--\eqref{eq:Gamma.def}). Most of these results are proven by Guo \cite{Guo02, Guo12} or Strain--Guo \cite{StGu08}. 

\subsection{Basic properties}

\begin{lemma}[Lemma 3 in \cite{Guo02}]\label{l.sigma_diff}
The functions $\sigma_{ij}(v)$ and $\sigma_i(v)$ (see \eqref{eq:sigma}) are smooth and satisfy
$$|\part_v^\beta \sigma_{ij}(v)|+|\part_v^\beta \sigma_{i}(v)|\leq C_\beta\jap{v}^{-1-|\beta|},$$
$$\sigma_{ij}g_ig_j=\lambda_1(v)\{P_vg_i\}^2+\lambda_2(v)\{[I-P_v]g_i\}^2,$$
and $$\sigma_{ij}(v)v_iv_j g^2=\Phi_{ij}*\{v_iv_j\mu\}g^2=\lambda_1(v)|v|^2g^2,$$
where $P_vg = \frac{(v\cdot g)v}{|v|^2}$, the projection of vector $g$ onto $v$.  
The spectrum of $\sigma_{ij}(v)$ consists of a simple eigenvalue $\lambda_1(v)>0$ associated with the vector $v$ and a double eigenvalue $\lambda_2(v)>0$ associated with $v^\perp$. Moreover, there are constants $c_1>0$ and $c_2>0$ such that asymptotically, as $|v|\to \infty$, we have $$\lambda_1(v) \to c_1\jap{v}^{-3},\hspace{1em} \lambda_2(v) \to c_2\jap{v}^{-1}.$$
\end{lemma}

\subsection{Lower bounds for the linear Landau operator}

In this subsection we will collect estimates which show that the linear Landau operator $L$ is coercive up to lower order terms. 
While most bounds can be found in \cite{Guo02, Guo12, StGu08}, we need some small modifications when the vector field commutator $Y$ is involved.

\subsection*{Lower bounds from \cite{Guo02, Guo12, StGu08}} 
We give three lowers bounds for $L$: (1) a weighted lower bound with $\rd_v$ derivatives (Corollary~\ref{c.L_est_v}), (2) a weighted lower bound without $\rd_v$ derivatives (Lemma~\ref{l.lin_est_x_weight}), (3) an unweighted lower bound without $\rd_v$ derivatives (Lemma~\ref{lem:lowest.order.positivity}).

From now on, let us define 
\begin{equation}\label{eq:chi.bar}
\bar\chi_m(v) = \chi(\f {|v|}m), \quad \mbox{where $\chi: [0,\infty)\to [0,\infty)$ is smooth, $\chi(z) = \begin{cases} 1 & \mbox{if $z\leq 1$} \\ 0 & \mbox{if $z \geq 2$} \end{cases}$}.
\end{equation}

In order to give our first lower bound for $L$, we estimate each piece in the decomposition in \eqref{def-L}--\eqref{eq:K.def}. The estimates \eqref{eq:K.upper.bound.SG}--\eqref{eq:A.lower.bound.SG} were proven in \cite[Lemma~8]{StGu08}, while \eqref{eq:A.upper.bound.SG} follows from an easy adaptation of the proof of \eqref{eq:A.lower.bound.SG}. (We note that the exact statement in \cite{StGu08} may look slightly different: in \cite{StGu08}, only polynomial weights with negative powers are used, though the actual proof applies more generally to our setting. In fact, this slightly modified version was used in \cite[(93), (94)]{Guo12}.)

\begin{lemma}[Lemma 8 in \cite{StGu08}]\label{l.lin_est_v}
Let $|\beta|>0$, $\ell\in \R$, $\vartheta\in\{0,2\}$ and fix $0<q_0<1$. Define $w(\ell,\vartheta) = \bv^{\ell} e^{\f{q_0|v|^{\vartheta}}{2}}$. Then for any small $\eta>0$, there exists $C_\eta > 0$ such that
\begin{equation}\label{eq:K.upper.bound.SG}
\left|\int_{\R^3}w^2(\ell,\vartheta)\part^\beta_v[\mathcal K g_1]g_2\d v\right|\leq \left\{\eta\sum_{|\beta'|\leq |\beta|}\norm{\part^{\beta'}_v g_1}_{\Delta_{v}(\ell,0)} + C_\eta \norm{\bar \chi_{C_\eta}g_1}_{L^2_v(\ell,0)}\right\} \norm{g_2}_{\Delta_{v}(\ell,\vartheta)}.
\end{equation}
Further,
\begin{equation}\label{eq:A.lower.bound.SG}
\begin{aligned}
&\: -\int_{\R^3}w^2\part^\beta_v[\mathcal A g]\part_v^\beta g\d v \\
&\geq \norm{\part^\beta_v g}_{\Delta_{v}(\ell,\vartheta)}^2-\eta\sum_{|\beta'|=|\beta|}\norm{\part_v^{\beta'} g}^2_{\Delta_{v}(\ell,\vartheta)} -C_\eta \sum_{|\beta''|<|\beta|}\norm{\part_{v}^{\beta''}g}_{\Delta_{v}(\ell + 2|\bt| -2|\beta''|,\vartheta)}^2.
\end{aligned}
\end{equation}
and 
\begin{equation}\label{eq:A.upper.bound.SG}
\begin{split}
 \Big| \int_{\R^3}w^2\part^\beta_v[\mathcal A g_1]\part_v^\beta g_2 \d v \Big| 
\ls (\sum_{|\beta'| \leq |\beta|}\norm{\part_{v}^{\beta'}g_1}_{\Delta_{v}(\ell + 2|\bt| -2|\beta'|,\vartheta)}) \norm{\part_{v}^{\beta}g_2}_{\Delta_{v}(\ell,\vartheta)}.
\end{split}
\end{equation}
\end{lemma}

Using \eqref{def-L}--\eqref{eq:K.def}, the inequalities \eqref{eq:K.upper.bound.SG}--\eqref{eq:A.lower.bound.SG} easily imply the lower bound given in the following corollary. This is the content of the first part of \cite[Lemma~9]{StGu08}. 
\begin{corollary}[Lemma 9 in \cite{StGu08}]\label{c.L_est_v}
Let $|\beta|>0$, $\ell\in \R$, $\vartheta\in\{0,2\}$ and fix $0<q_0<1$. Define $w(\ell,\vartheta) = \bv^{\ell} e^{\f{q_0|v|^{\vartheta}}{2}}$. Then for any small $\eta>0$, there exists $C_\eta > 0$ such that
$$
\begin{aligned}
\int_{\R^3}w^2(\ell,\vartheta)\part^\beta_v[L g] \part_v^\beta g\d v 
&\geq \norm{\part^\beta_v g}_{\Delta_{v}(\ell,\vartheta)}^2-\eta\sum_{|\beta'|=|\beta|}\norm{\part_v^{\beta'} g}^2_{\Delta_{v}(\ell,\vartheta)}
\\&\quad -C_\eta \sum_{|\beta''|<|\beta|}\norm{\part_{v}^{\beta''}g}_{\Delta_{v}(\ell + 2|\bt| -2|\beta''|,\vartheta)}^2.
\end{aligned}$$
\end{corollary}

We now turn to a weighted lower bound without commutations, corresponding to the $|\bt| = 0$ case of Corollary~\ref{c.L_est_v}. This is the second part of \cite[Lemma~9]{StGu08}.
\begin{lemma}[Lemma 9 in \cite{StGu08}]\label{l.lin_est_x_weight}
Let  $\ell\in \R$, $\vartheta\in\{0,2\}$ and fix $0<q_0<1$. Define $w(\ell,\vartheta) = \bv^{\ell} e^{\f{q_0|v|^{\vartheta}}{2}}$. Then, for every $\eta>0$, there is $C_\eta>0$ such that
$$\int_{\R^3}w^2(\ell,\vartheta)[L g]g\d v\geq (1-q_0^2-\eta)^2\norm{g}^2_{\Delta_{v}(\ell,\vartheta)} - C_\eta \norm{\bar\chi_{C_\eta}g}^2_{L^2_v}.$$
\end{lemma}

Next, we state a lower bound without derivatives in an unweighted space, which can be viewed as a scalar version of the positivity lemma for $L$ in \cite[Lemma~2]{Guo12}.
\begin{lemma}[Lemma~2~in~\cite{Guo12}]\label{lem:lowest.order.positivity}
We have $\int_{\R^3}[L g]h\d v=\int_{\R^3}[L h]g\d v$, $\int_{\R^3}[L g] g \d v\geq 0$ and $Lg=0$ if and only if $g=\Pi g$ where $\Pi$ is the $L^2_v(\R^3)$ projection with respect to the $L^2_v$ inner product onto the null space of $L$, given by $\text{span}\left\{\sqrt{\mu},v_i\sqrt{\mu},|v|^2\sqrt{\mu}\right\},$
where 
$i\in \{ 1,2,3\}$. Furthermore, 
$$\int_{\R^3} [L g]g\d v\gtrsim  \norm{(I-\Pi) g}^2_{\Delta_{v}(0,0)}.$$
\end{lemma}

\subsection*{Lower bounds for the linear Landau operator when $Y$ commutations are involved}
We now turn to the analogue of Corollary~\ref{c.L_est_v} when the vector field commutator $Y = \nabla_v + t\nabla_x$ is involved. The main estimate is given in Corollary~\ref{c.L_est_Y} below.

Just as Corollary~\ref{c.L_est_v} is based on \lref{lin_est_v}, the lower bound in Corollary~\ref{c.L_est_v} is based on a similar lemma (see Lemma~\ref{l.lin_est_Y}). One difference between the bounds we prove here in Lemma~\ref{l.lin_est_Y} and the previous bounds where $Y$ commutations are not involved is that we do not use exponential weights in Lemma~\ref{l.lin_est_Y}. The proof of Lemma~\ref{l.lin_est_Y} is an adaptation of the ideas in the proof of Lemma~\ref{l.lin_est_v}.

\begin{lemma}\label{l.lin_est_Y}
Let $\ell \in \mathbb N$. Then
\begin{equation*}
\begin{split}
\Big| \int_{\T^3}\int_{\R^3}\jap{v}^{2\ell} \part^\beta_v Y^\omega[\mathcal A g_1] \part_v^\beta Y^\omega g_2\d v\d x \Big|
\ls &\: (\sum_{\substack{|\beta'|\leq |\beta|\\|\omega'|\leq |\omega|}}\norm{\part_v^{\beta'}Y^{\omega'} g_1}_{\Delta_{x,v}(\ell,0)} ) \norm{Y^{\omega'}\part_v^{\beta'} g_2}_{\Delta_{x,v}(\ell,0)},
\end{split}
\end{equation*}
and, for any small $\eta>0$, there exists $C_\eta > 0$ such that
\begin{align*}
\left|\int_{\R^3}\jap{v}^{2\ell}\part^\beta_vY^\omega[\mathcal K g_1] g_2\d v\right|\leq \Big( \eta\sum_{|\beta'|\leq |\beta|}\norm{\part^{\beta'}_v Y^\omega g_1}_{\Delta_{v}(\ell,0)}+C_\eta \sum_{|\omega'|\leq |\omega|}\norm{\mu Y^{\omega'} g_1}_{L^2_v(\ell,0)} \Big)
 \norm{g_2}_{\Delta_{v}(\ell,0)}.
\end{align*}
In addition, for any small $\eta>0$, there exists $C_\eta > 0$ such that
\begin{equation}\label{eq:A.lower.with.Y}
\begin{split}
 -\int_{\T^3}\int_{\R^3}\jap{v}^{2\ell} \part^\beta_vY^\omega[\mathcal A g]\part_v^\beta Y^\omega g\d v\d x \geq &\: \norm{\part^\beta_v Y^\omega g}_{\Delta_{x,v}(\ell,0)}^2-\eta\sum_{\substack{|\beta'|\leq |\beta|\\|\omega'|\leq |\omega|}}\norm{\part_v^{\beta'} Y^{\omega'} g}^2_{\Delta_{x,v}(\ell,0)}\\
&\: - C_\eta (\smashoperator{\sum_{\substack{|\beta'|\leq|\beta|\\|\omega'|\leq|\omega|\\|\beta'|+|\omega'|\leq |\beta|+|\omega|-1}}}\hspace{1em}\norm{\part_v^{\beta'}Y^{\omega'} g}^2_{\Delta_{x,v}(\ell,0)}-\smashoperator{\sum_{\substack{|\omega'|\leq|\omega|}}}  \norm{\mu Y^{\omega'}g}_{L^2_{x,v}}^2).
\end{split}
\end{equation}
\end{lemma}

\begin{proof}
We will only prove \eqref{eq:A.lower.with.Y} as the other estimates are similar, if not simpler.

Let $w=\jap{v}^{\ell}$. To lighten the notation, we write $L^2_v(\ell) = L^2_v(\ell,0)$, etc.~in this proof.

Using \eqref{eq:A.def} we get,
\begin{align}
- \int_{\T^3}\int_{\R^3} &w^2 \part_v^\beta Y^\omega(\mathcal A g)\part_v^\beta Y^\omega g\d v\d x\notag\\
&\quad\geq \norm{\part_v^\beta Y^\omega g}_{\Delta_{x,v}(\ell,\vartheta)}^2 \notag\\
&\qquad-C \sum_{\substack{|\beta'|+|\beta''|\leq|\beta|\\|\omega'|+|\omega''|\leq |\omega|\\|\beta'|+|\omega'|\geq 1}} \left|\int_{\T^3}\int_{\R^3} w^2\part_{v}^{\beta'}Y^{\omega'}\sigma_{ij}(\part_{v}^{\beta''}Y^{\omega''}\part_{v_j} g)\part_v^\beta Y^\omega\part_{v_i} g\d v\d x \right| \label{e.A_error_1}\\
&\qquad- C \sum_{\substack{|\beta'|+|\beta''|\leq|\beta|\\|\omega'|+|\omega''|\leq |\omega|}} \left|\int_{\T^3}\int_{\R^3} \part_{v_i}(w^2)\part_{v}^{\beta'}Y^{\omega'}\sigma_{ij}(\part_{v}^{\beta''}Y^{\omega''}\part_{v_j} g)\part_v^\beta Y^\omega g\d v\d x \right| \label{e.A_error_2}\\
&\qquad -C \sum_{\substack{|\beta'|+|\beta''|\leq|\beta|\\|\omega'|+|\omega''|\leq |\omega|\\|\beta'|+|\omega'|\geq 1}} \left|\int_{\T^3}\int_{\R^3} w^2\part_{v}^{\beta'}Y^{\omega'}(\sigma_{ij}v_iv_j)(\part_{v}^{\beta''}Y^{\omega''} g)\part_v^\beta Y^\omega g\d v\d x \right| \label{e.A_error_3}\\
&\qquad -C \sum_{\substack{|\beta'|+|\beta''|\leq|\beta|\\|\omega'|+|\omega''|\leq |\omega|}} \left|\int_{\T^3}\int_{\R^3} w^2\part_{v}^{\beta'}Y^{\omega'}\part_{v_i}\sigma_{i}(\part_{v}^{\beta''}Y^{\omega''} g)\part_v^\beta Y^\omega g\d v\d x \right|.\label{e.A_error_4}
\end{align}

\subsection*{Estimates for \eref{A_error_3} and \eref{A_error_4}} Since $\sigma_{ij}$ is independent of $x$, $Y$ acts as a purely velocity derivative. It follows from Lemma~\ref{l.sigma_diff} that
\begin{equation}\label{eq:A.error.3.4}
|\eref{A_error_3}|+|\eref{A_error_4}| \ls \sum_{\substack{|\beta''|\leq|\beta|\\|\omega''|\leq |\omega|}} \int_{\T^3} \int_{\R^3} w^2 \bv^{-2} |\rd_v^{\bt''} Y^{\om''} g| |\rd_v^\bt Y^\om g| \, \ud v \, \ud x.
\end{equation}

For every $m\geq 1$, let $\bar \chi_m$ be as in \eqref{eq:chi.bar}. Then,
\begin{align*}
\norm{\jap{v}^{-1}\part_v^{\beta''}Y^{\omega''} g}^2_{L^2_{x,v}(\ell)} \le  \norm{\bar\chi_{m}\part_v^{\beta''}Y^{\omega''} g}^2_{L^2_{x,v}(\ell)}+\norm{(1-\bar\chi_m)\jap{v}^{-1}\part_v^{\beta''}Y^{\omega''} g}^2_{L^2_{x,v}(\ell)}.\end{align*}
For the large velocity part, we use the extra $\jap{v}$ weights to get
\begin{align*}
\norm{(1-\bar\chi_m)\jap{v}^{-1}\part_v^{\beta''}Y^{\omega''} g}^2_{L^2_{x,v}(\ell)}&\lesssim  \frac{1}{m}\norm{\jap{v}^{-1/2}\part_v^{\beta''}Y^{\omega''} g}^2_{L^2_{x,v}(\ell)}
\lesssim \f 1m \norm{\part_v^{\beta''}Y^{\omega''} g}^2_{\Delta_{x,v}(\ell)}
\end{align*}
For the small velocity part, we interpolate between Sobolev spaces to get that for any $m>1$, $\eta'>0$, there is $C_{\eta',m}>0$ so that
\begin{equation*}
\begin{split}
&\: \norm{\bar\chi_{m}\part_v^{\beta''}Y^{\omega''} g}^2_{L^2_{x,v}(\ell)} \le \eta' \norm{\jap{v}^{-\frac{3}{2}}\part_{v_i}\part_v^{\beta''}Y^{\omega''} g}^2_{L^2_{x,v}(\ell)}+C_{\eta',m} \norm{\mu Y^{\omega''} g}^2_{L^2_{x,v}} \\
\ls &\: \eta' \|\part_v^{\beta''}Y^{\omega''} g \|_{\Delta_{v}(\ell)}^2 + C_{\eta',m} \norm{\mu Y^{\omega''} g}^2_{L^2_{x,v}}.
\end{split}
\end{equation*}
Hence,  in total we have,
\begin{align*}
\norm{\jap{v}^{-1}\part_v^{\beta''}Y^{\omega''} g}^2_{L^2_{x,v}(\ell)}
&\ls (\eta' + \f 1m) \norm{\part_v^{\beta''} Y^{\omega''} g}^2_{\Delta_{x,v}(\ell)}+C_{\eta',m}  \norm{\mu Y^{\omega''} g}^2_{L^2_{x,v}}.
\end{align*}

Choosing $1/m$ and $\eta'$ sufficiently small in terms of $\eta$, we obtain
$$|\eref{A_error_3}|+|\eref{A_error_4}|\leq \f{\eta}{10} \sum_{\substack{|\beta'|\leq|\beta|\\|\omega'|\leq |\omega|}}\norm{\part_v^{\beta'}Y^{\omega'} g}^2_{\Delta_{x,v}(\ell)}+C_\eta \sum_{|\omega'|\leq |\omega|}\norm{\mu Y^{\omega'} g}^2_{L^2_{x,v}}.$$ 

\subsection*{Estimates for \eref{A_error_2}} If $|\beta'|+|\omega'|\geq 1$ in \eref{A_error_2}, the corresponding terms are bounded by
$$\sum_{\substack{|\beta'|\leq|\beta|\\|\omega'|\leq |\omega|}} \int_{\T^3} \int_{\R^3} w^2 \bv^{-3} |\rd_{v_j}\rd_v^{\bt''} Y^{\om''} g| |\rd_v^\bt Y^\om g| \, \ud v \, \ud x,$$
which has enough $\bv$ decay for the argument above (with easy modifications) for \eref{A_error_3}, \eref{A_error_4}.

We thus only need to consider $|\beta'|+|\omega'|=0$, for which we integrate by parts in $\rd_{v_j}$ to get
\begin{align*}
\int_{\T^3}\int_{\R^3}\part_{v_i}(w^2)\sigma_{ij}&(\part_{v}^{\beta}Y^{\omega}\part_{v_j} g)\part_v^\beta Y^\omega g\d v\d x\\
&=-\frac{1}{2}\int_{\T^3}\int_{\R^3}[\part^2_{v_jv_i}(w^2)\sigma_{ij}+\part_{v_i}(w^2)\part_{v_j}\sigma_{ij}](\part_{v}^{\beta}Y^{\omega}g)^2\d v\d x.
\end{align*}
Now, by Lemma~\ref{l.sigma_diff}, this term is bounded above by
\begin{equation*}
 \int_{\T^3} \int_{\R^3} w^2 \bv^{-3} |\rd_v^{\bt} Y^{\om} g| |\rd_v^\bt Y^\om g| \, \ud v \, \ud x.
\end{equation*}
This is better than the term in \eqref{eq:A.error.3.4}, which can therefore be controlled in the same way.

\subsection*{Estimates for \eref{A_error_1}} If $|\beta'|+|\omega'|\geq 2$, then using the Cauchy--Schwarz and the Young inequalities, we can bound
\begin{align*}
&\sum_{\substack{|\beta''|\leq |\beta|\\|\omega''|\leq |\omega|\\ |\beta''|+|\omega''|\leq |\beta|+|\omega|-2}} \int_{\T^3}\int_{\R^3}w^2 \jap{v}^{-3} |\part_{v}^{\beta''}Y^{\omega''} \part_{v_j}g| |\part_v^\beta Y^\omega\part_{v_i} g| \d v\d x\\
\lesssim &\: \eta' \norm{\part_v^\beta Y^\omega g}^2_{\Delta_{x,v}(\ell)}+ C_{\eta'} \sum_{\substack{|\beta''|\leq |\beta|\\|\omega''|\leq |\omega|\\ |\beta''|+|\omega''|\leq |\beta|+|\omega|-2}}\norm{\part_v^{\beta''} Y^{\omega''} g}^2_{\Delta_{x,v}(\ell)}.
\end{align*}

If $|\beta'|+|\omega'|=1$, then we have two cases:\\
\emph{Case 1:} $|\beta'|=1$. In this case $\part_{v}^\beta Y^\omega=\part_{v_l}\part_v^{\beta''}Y^{\omega''}.$ Then integrating by parts in $\rd_{v_l}$, we get,
\begin{align*}
\int_{\T^3}\int_{\R^3} w^2 \part_{v_l} \sigma_{ij}&(\part_v^{\beta''}Y^\omega \part_{v_j}g)(\part_{v_l}\part_{v_i}\part_v^{\beta''}Y^\omega g)\d v\d x\\
&=-\frac{1}{2}\int_{\T^3}\int_{\R^3} \part_{v_l}(w^2 \part_{v_l} \sigma_{ij})(\part_v^{\beta''}Y^\omega \part_{v_j}g)^2\d v\d x.
\end{align*}
Since $|\part_{v_l}(w^2 \part_{v_l} \sigma_{ij})|\leq w^2 \jap{v}^{-3}$ (by Lemma~\ref{l.sigma_diff}), we have
$$\int_{\T^3}\int_{\R^3} \part_{v_l}(w^2 \part_{v_l} \sigma_{ij})(\part_v^{\beta''}Y^\omega \part_{v_j}g)^2\lesssim \sum_{\substack{|\beta''|<|\beta|}}\norm{\part_{v}^{\beta''} Y^{\omega} g}^2_{\Delta_{x,v}(\ell)}.$$
\emph{Case 2:} $|\omega'|=1$. In this case $\part_{v}^\beta Y^\omega=Y_l\part_v^{\beta}Y^{\omega''}.$ Then using integration by parts in both $x_l$ and $v_l$, we get,
\begin{align*}
\int_{\T^3}\int_{\R^3} w^2 \part_{v_l} \sigma_{ij}&(\part_v^{\beta}Y^{\omega''} \part_{v_j}g)(Y_l\part_{v_i}\part_v^{\beta}Y^{\omega''} g)\d v\d x\\
&=-\frac{1}{2}\int_{\T^3}\int_{\R^3} \part_{v_l}(w^2 \part_{v_l} \sigma_{ij})(\part_v^{\beta}Y^{\omega''} \part_{v_j}g)^2\d v\d x.
\end{align*}
We again get the required bound as above. 

Combining the estimates for \eqref{e.A_error_1}--\eqref{e.A_error_4}, and choosing $\eta'$ small enough in terms of $\eta$, we obtain \eqref{eq:A.lower.with.Y}. \qedhere
\end{proof}

Using the decomposition \eqref{def-L}--\eqref{eq:K.def}, the previous lemma immediately implies the following lower bound for $L$:
\begin{corollary}\label{c.L_est_Y}
Fix $\ell \in \mathbb N$. For any small $\eta>0$, there exists $C_\eta  > 0$ such that
\begin{equation*}
\begin{split}
&\: \int_{\T^3}\int_{\R^3}\jap{v}^{2\ell}\part^\beta_v Y^\omega[L g]Y^\omega\part_v^\beta g\d v\d x\\
\geq &\: \norm{Y^\omega\part^\beta_v g}_{\Delta_{x,v}(\ell,0)}^2-\eta\sum_{\substack{|\beta'|\leq |\beta|\\|\omega'|\leq |\omega|}}\norm{\part_v^{\beta'} Y^{\omega'} g}^2_{\Delta_{x,v}(\ell,0)}-C_\eta (\smashoperator{\sum_{\substack{|\beta'|\leq|\beta|\\|\omega'|\leq|\omega|\\|\beta'|+|\omega'|\leq |\beta|+|\omega|-1}}}\hspace{1em}\norm{\part_v^{\beta'} Y^{\omega'} g}^2_{\Delta_{x,v}(\ell,0)} + \smashoperator{\sum_{\substack{|\omega'|\leq|\omega|}}} \norm{\mu Y^{\omega'} g}_{L^2_{x,v}}^2).
\end{split}
\end{equation*}
\end{corollary}

\subsection{Upper bounds for the linear Landau operator}

Using Lemma~\ref{l.lin_est_v} and Lemma~\ref{l.lin_est_Y}, we also obtain the following upper bounds for the linear Landau operator.
\begin{corollary}\label{c.L_est_upper}
\begin{enumerate}
\item For $\ell \in \mathbb N$, $\vartheta \in \{0,2\}$, and $w(\ell,\vartheta) = \bv^{\ell} e^{\f{q_0|v|^{\vartheta}}{2}}$,
\begin{equation*}
\begin{split}
&\: \Big|\int_{\T^3}\int_{\R^3}w^2(\ell,\vartheta)\part^\beta_v[L g_1] g_2 \d v\d x \Big|
\ls (\smashoperator{\sum_{\substack{|\beta'|\leq |\beta|}}} \norm{\rd_v^{\bt'}  g_1}_{\Delta_{x,v}(\ell,\vartheta)}) \norm{ g_2}_{\Delta_{x,v}(\ell,\vartheta)}.
\end{split}
\end{equation*}
\item For any $\ell \in \mathbb N$, 
\begin{equation*}
\begin{split}
&\: \Big|\int_{\T^3}\int_{\R^3}\jap{v}^{2\ell}\part^\beta_vY^\omega[L g_1] g_2 \d v\d x \Big|
\ls (\smashoperator{\sum_{\substack{|\beta'|\leq |\beta|\\|\omega'|\leq |\omega|}}} \norm{\rd_v^{\bt'} Y^{\omega'} g_1}_{\Delta_{x,v}(\ell,0)}) \norm{ g_2}_{\Delta_{x,v}(\ell,0)}.
\end{split}
\end{equation*}
\end{enumerate}
\end{corollary}
\begin{proof}
Recalling \eqref{def-L}--\eqref{eq:K.def}, the first estimate follows from Lemma~\ref{l.lin_est_v}, while the second estimate follows from Lemma~\ref{l.lin_est_Y}.
\end{proof}

\subsection{Bounds for the nonlinear Landau operator}

We close this section with bounds for the nonlinear Landau operator (see \eqref{eq:Gamma.def}). We begin with the following estimate from \cite{StGu08}.
\begin{lemma}[Lemma 10 in \cite{StGu08}]\label{l.nonlin}
 Let $\vartheta\in\{0,2\}$, $\ell\geq 0$ and fix $0<q_0<1$. Define $w(\ell,\vartheta) = \bv^{\ell} e^{\f{q_0|v|^{\vartheta}}{2}}$. Then  for any $\ell'\in\R$ we have,
\begin{equation}\label{e.nonlin_1}
\begin{split}
&\left|\int_{\R^3}w^2(\ell,\vartheta)\rd_v^\bt Y^\om \Gamma(g_1,g_2)\rd_v^{\bt} Y^{\om} g_3\d v\right|\\
&\quad \lesssim \smash{\sum_{\substack{|\beta'|+|\beta''|\leq |\beta|\\|\omega'|+|\omega''|\leq |\omega|}}}\norm{\rd_v^\bt Y^\om g_3}_{\Delta_{v}( \ell,\vartheta)}\left[\norm{\rd_v^{\bt'} Y^{\om'} g_1}_{L^2_v}\norm{\rd_v^{\bt''} Y^{\om''} g_2}_{\Delta_{v}(\ell,\vartheta)}\right.\\
&\hspace{13em}\left.+\norm{\rd_v^{\bt'} Y^{\om'} g_1}_{\Delta_{v}(0,0)}\norm{\rd_v^{\bt''} Y^{\om''} g_2}_{L^2_v(\ell,\vartheta)}\right].
\end{split}
\end{equation}
\end{lemma}

\medskip

We also need the following more pessimistic estimates for $\Gamma(g_1,g_2)$, in which we do not exploit the divergence structure. (They will be relevant for controlling the inhomogeneous terms in the density estimates; see \eqref{eq:Gamma.error.in.density} and \eqref{eq:nonlinear.collision.in.density.decay}.) 
\begin{lemma}\label{lem:Gamma.trivial}
Let $\vartheta\in\{0,2\}$, $\ell\geq 0$ and fix $0<q_0<1$. Define $w(\ell,\vartheta) = \bv^{\ell} e^{\f{q_0|v|^{\vartheta}}{2}}$. Then
\begin{equation*}
\begin{split}
 \|w(\ell,\vartheta)  \rd_v^\bt Y^\om \Gamma(g_1, g_2) \|_{L^2_{v}} 
\ls &\:\smash{\sum_{\substack{|\beta'|+|\beta''|\leq |\beta|\\|\omega'|+|\omega''|\leq |\omega|}}} \quad \sum_{ |\widetilde\bt'| + |\widetilde\bt''| \leq 2} 
 \| \rd_v^{\bt'+\widetilde{\bt}'} Y^{\om'} g_1 \|_{L^2_{v}} \| w(\ell,\vartheta) \rd_v^{\bt''+\widetilde{\bt}''} Y^{\om''} g_2 \|_{L^2_{v}} .
\end{split}
\end{equation*}
\end{lemma}
\begin{proof}
Recalling \eqref{eq:Gamma.def}, we know that 
\begin{equation}\label{eq:Gamma.def.again}
\begin{split}
\Gamma(g_1,g_2)& = \partial_{v_i} \Big[ \Big(\Phi_{ij} \star (\mu^{1/2} g_1) \Big)\partial_{v_j} g_2 \Big] - \Big[ \Phi_{ij} \star \Big( v_i \mu^{1/2} g_1 \Big) \Big] \partial_{v_j} g_2
\\& \quad -  \partial_{v_i} \Big[ \Big(\Phi_{ij} \star (\mu^{1/2}\partial_{v_j}g_1) \Big)g_2 \Big]  +  \Big[ \Phi_{ij} \star \Big( v_i \mu^{1/2} \partial_{v_j}g_1 \Big) \Big] g_2 . 
\end{split}
\end{equation} 
Now, it is easy to check that $\||v|^{-1} \star h \|_{L^\i_v} \ls \| h\|_{L^1_v}^{1/3} \|h\|_{L^2_v}^{2/3}$ (for instance by adapting the proof of \cite[Lemma~5.1]{jL2019}). It follows that $\| \Phi_{ij} \star (\bv^\ell \mu^{1/2} h)\|_{L^\i_v} \ls \| \bv^{-\ell'} h \|_{L^2_v}$ for any $\ell' \geq 0$. 

Therefore, using H\"older's inequality and apply the above observation for $h$ being derivatives of $g_1$ or $g_2$, we obtain the required result. \qedhere

\end{proof}

\medskip

\section{Setting up the energy estimates}\label{sec-EE}

In this section, we set up the main energy estimates as well as introduce the global energy and dissipation norms for the full nonlinear Vlasov--Poisson--Landau system \eqref{eq:Vlasov.f}--\eqref{eq:Poisson.f}. Precisely, for a given electric field $E = -\nabla_x \phi$, we shall derive 
energy estimates for smooth solutions $f$ to the following Vlasov--Landau equation  
\begin{equation}\label{VPL-re}
D_t f   - E\cdot v f + \nu L f
=\mathcal{Q}
\end{equation}
where $D_t$ denotes the transport operator 
\begin{equation}\label{def-Dt}D_t = \part_{t} +v\cdot \nabla_x  +E\cdot \nabla_v.\end{equation} 
The transport-diffusion structure of \eqref{VPL-re} is clear, being transported by the electric field in the phase space and diffused through the Landau collision operator $L$. We note that a similar structure also holds for derivatives of $\partial_x^\alpha \partial_v^\beta Y^\omega f$ for any triple of multi-indices $(\alpha,\beta,\omega)$. The main result of this section will be given in Subsection \ref{sec-EEsetup} below.

\begin{remark} The equation \eqref{VPL-re} is exactly the Vlasov--Poisson--Landau equation \eqref{eq:Vlasov.f} with 
\begin{equation}\label{def-QQQ}\mathcal{Q} = 2 E\cdot v\sqrt{\mu} + \nu \Gamma(f,f) .\end{equation} 
Note that the first term in $\mathcal{Q} $ is linear in $f$ and thus it cannot in principle be treated as a remainder. However, this linear term is very localized both in velocity $v$ (through $\mu = e^{-|v|^2}$) and in frequency $\partial_v$ (through the Poisson equation), a fact that will play a role in our nonlinear analysis. 
\end{remark}

%
%

\subsection{Basic energy estimates} We start with basic energy estimates for the transport-diffusion equation \eqref{VPL-re}.

\begin{lemma}\label{lem-basicEE} Let $\ell\in \R$, $0< q_0<1$, and $\vartheta\in \{0,2\}$. Define $q = \begin{cases} q_0 & \mbox{if $\vartheta = 2$} \\ 0 & \mbox{if $\vartheta = 0$}\end{cases}$. Then, there is a positive constant $\theta = \theta(\vartheta, q_0)$ so that smooth solutions to \eqref{VPL-re} satisfy
\begin{equation}\label{est-basic}
\begin{aligned}
\frac{d}{dt} \| & e^{(q+1)\phi} f\|^2_{L^2_{x,v}(\ell,\vartheta)}  + \theta \nu  \| e^{(q+1)\phi}  f\|_{\Delta_{x,v}(\ell,\vartheta)}^2 
\ls
\nu \mathcal{R}^{L,\ell}_{0}+ \mathcal{R}^{T,\ell}_{0}+ \mathcal{R}^{Q,\ell}_{0},
\end{aligned}
\end{equation}
where the remainders are defined by
$$
\begin{aligned}
\mathcal{R}^{L,\ell}_{0} 
& = \norm{\mu e^{(q+1)\phi} f}_{L^2_{x,v}}^2,
\\ \mathcal{R}^{T,\ell}_{0} 
 &= (\|\partial_t\phi \|_{L^\infty_x} + \|E\|_{L^\infty_x}) \| e^{(q+1)\phi} f\|^2_{L^2_{x,v}(\ell,\vartheta)},
\\
 \mathcal{R}^{Q,\ell}_{0}
&= \Big| \iint_{\T^3\times \R^3} e^{2(q+1)\phi}w^2 f \mathcal{Q} \; \ud x\,\ud v \Big| .
\end{aligned}
$$
\end{lemma}

\begin{remark}
Observe that there are three contributions to the energy production of \eqref{VPL-re}: namely, the remainders from the transport dynamics $D_t$, the Landau operator $L$, and the source $\mathcal{Q}$. 
\end{remark}

\begin{proof} Directly from the transport structure of \eqref{VPL-re}, we compute 
$$
\begin{aligned}
 \frac{1}{2}\frac{d}{dt} \| e^{(q+1)\phi}wf\|^2_{L^2_{x,v}} & = \frac12 \iint_{\T^3\times \R^3}  |f|^2 (D_t + 2v \cdot E)[e^{2(q+1)\phi}w^2]   \; \ud v \, \ud x 
 \\& \quad +  \iint_{\T^3\times \R^3} \Big[
- \nu L f+ \mathcal{Q} \Big] e^{2(q+1)\phi}w^2f \; \ud v \, \ud x .
\end{aligned}
$$
Recalling that $E = -\nabla_x \phi$ and $w=\jap{v}^{\ell}e^{\frac{q|v|^\vartheta}{2}}$, we compute
\begin{equation}\label{dt-vE}
\begin{aligned}
\frac12 (D_t + 2v \cdot E)[e^{2(q+1)\phi}w^2] & =\Big[ (q+1)(\partial_t + v\cdot \nabla_x) \phi + E\cdot \nabla_v \log w + v\cdot E \Big] e^{2(q+1)\phi}w^2
\\&= \Big[ \frac{q}{2}(\vartheta |v|^{\vartheta -2} - 2 ) v\cdot E  + (q+1)\partial_t \phi+ \ell \langle v\rangle^{-2} v\cdot E \Big]e^{2(q+1)\phi}w^2,
\end{aligned}\end{equation}
in which we note that the first term vanishes, since either $\vartheta = 2$ or $q=0$ (when $\vartheta=0$). This proves 
$$
\begin{aligned}
 \iint_{\T^3\times \R^3}  |f|^2 (D_t + 2v \cdot E)[e^{2(q+1)\phi}w^2]   \; \ud v \, \ud x \le (2 \|\partial_t \phi\|_{L^\infty} + \ell\|E\|_{L^\infty_x}) \| e^{(q+1)\phi} f\|^2_{L^2_{x,v}(\ell,\vartheta)} .
\end{aligned}$$
Finally, using \lref{lin_est_x_weight} with $\eta=\frac{1-q^2}{2}$ and noting $\phi$ is independent of $v$, we get 
\begin{align*}
\iint_{\T^3\times \R^3} e^{2(q+1)\phi}w^2 f Lf \; \ud v \, \ud x&\geq \frac{1-q^2}{2} \| e^{(q+1)\phi}f\|_{\Delta_{x,v}(\ell,\vartheta)}^2- C_q\norm{\bar\chi_{C_q} e^{(q+1)\phi} f}_{L^2_{x,v}(\ell,0)}^2
\end{align*}
where $\bar\chi_{C_q}$ is a cut off function near the origin. 
The lemma follows. 
\end{proof}

\begin{remark}
Note that the basic energy estimate derived in Lemma \ref{lem-basicEE} uses only the equation \eqref{VPL-re} for a given electric field $E$ (i.e.~the Poisson equation was not used). 
\end{remark}

\subsection{Derivative energy estimates} Next, we obtain the following energy estimates for derivatives. 

\begin{lemma}\label{lem-basicEEdx} Let $\ell\in \R$, $0< q_0<1$, $\vartheta\in \{0,2\}$, and $(\alpha,\beta,\omega)$ be any triple of multi-indices. If $|\omega|>0$, we take $\vartheta=0$. Define $q = \begin{cases} q_0 & \mbox{if $\vartheta = 2$} \\ 0 & \mbox{if $\vartheta = 0$}\end{cases}$.

Then, there is a positive constant $\theta$ so that smooth solutions to \eqref{VPL-re} satisfy
\begin{equation}\label{basic-derv}
\begin{aligned}
\frac{d}{dt} \| & e^{(q+1)\phi} \partial_x^\alpha\partial_v^\beta Y^\omega f\|^2_{L^2_{x,v}(\ell,\vartheta)}  + \theta \nu  \| e^{(q+1)\phi} \partial_x^\alpha\partial_v^\beta Y^\omega f\|_{\Delta_{x,v}(\ell,\vartheta)}^2 
\lesssim  
\mathcal{R}^{T,\ell}_{\alpha,\beta,\omega} + \nu \mathcal{R}^{L,\ell}_{\alpha,\beta,\omega} +  \mathcal{R}^{Q,\ell}_{\alpha,\beta,\omega},
\end{aligned}
\end{equation}
 in which we have collected 

\begin{itemize}

\item the remainders $\mathcal{R}^{T,\ell}_{\alpha,\beta,\omega} $ due to the transport dynamics: 
$$
\begin{aligned}
\mathcal{R}^{T,\ell}_{\alpha,\beta,\omega} 
 &= \sum_{\substack{|\beta''|=|\beta|-1\\|\beta'|=1}}\| e^{(q+1)\phi}  \partial_x^{\alpha+\beta'}\partial_v^{\beta''} Y^\omega f\|_{L^2_{x,v}(\ell,\vartheta)} \| e^{(q+1)\phi}  \partial_x^\alpha\partial_v^\beta Y^\omega f\|_{L^2_{x,v}(\ell,\vartheta)} 
\\&\quad+  (\|\partial_t\phi \|_{L^\infty_x} + \|E\|_{L^\infty_x}) \| e^{(q+1)\phi} \partial_x^\alpha\partial_v^\beta Y^\omega f\|^2_{L^2_{x,v}(\ell,\vartheta)},
\end{aligned}
$$

\item the remainders $ \mathcal{R}^{L,\ell}_{\alpha,\beta,\omega}$ due to the linear Landau operator: 
 
$$
\begin{aligned}
 \mathcal{R}^{L,\ell}_{\alpha,\beta,\omega} 
 &=
\eta\sum_{\substack{|\beta'|\leq |\beta|\\|\omega'|\leq |\omega|}}\norm{e^{(q+1)\phi}\derv{}{'}{'} f}^2_{\Delta_{x,v}(\ell,\vartheta)}
+C_\eta \smashoperator{\sum_{\substack{|\beta'|\leq|\beta|\\|\omega'|\leq|\omega|\\|\beta'|+|\omega'|\leq |\beta|+|\omega|-1}}}\hspace{1em}\norm{e^{(q+1)\phi}\derv{}{'}{'} f}^2_{\Delta_{x,v}(\ell,\vartheta)}
\\
&\quad\quad+C_\eta \smashoperator{\sum_{\substack{|\omega'|\leq|\omega|}}}\norm{\mu e^{(q+1)\phi} \rd_x^\alp Y^{\om'} f}_{L^2_{x,v}}^2,
\end{aligned}
$$
for any small $\eta>0$,

\item the remainders $\mathcal{R}^{Q,\ell}_{\alpha,\beta,\omega} $ due to the source term (of the equations for derivatives): 
$$
\mathcal{R}^{Q,\ell}_{\alpha,\beta,\omega}  = \Big| \iint_{\T^3\times \R^3} e^{2(q+1)\phi}w^2\partial_x^\alpha \partial_v^\beta Y^\omega f \mathcal{Q}_{\alpha,\beta,\omega}\; \ud v \, \ud x  \Big|
$$
where $w=\jap{v}^{\ell}e^{\frac{q|v|^\vartheta}{2}}$, and
\begin{equation}\label{def-Qalpha}
\mathcal{Q}_{\alpha,\beta,\omega}: =  \partial_x^\alpha\partial_v^\beta Y^\omega\mathcal{Q} - [E\cdot \nabla_v - E\cdot v, \partial_x^\alpha \partial_v^\beta Y^\omega] f.
\end{equation}

\end{itemize}

\end{lemma}


\begin{proof} Directly from \eqref{VPL-re}, we observe that derivatives $\partial_x^\alpha \partial_v^\beta Y^\omega$ solve 
$$
\begin{aligned}
\Big[ D_t   - E\cdot v\Big] \partial_x^\alpha \partial_v^\beta Y^\omega f + \nu \partial_v^\beta Y^\omega [L \partial_x^\alpha  f] 
& = \partial_x^\alpha \partial_v^\beta Y^\omega \mathcal{Q} + [D_t - E\cdot v, \partial_x^\alpha\partial_v^\beta Y^\omega] f.\end{aligned}$$
Note that $ [\partial_t + v \cdot \nabla_x,\partial_x] = 0$  and $ [\partial_t + v \cdot \nabla_x,Y] = 0$. Hence, for $|\beta|>0$, we compute 
\begin{equation}\label{non-dv} [\partial_t + v \cdot \nabla_x, \partial_x^\alpha\partial_v^\beta Y^\omega] = - \sum_{\substack{|\beta''|=|\beta|-1\\|\beta'|=1}}\part_x^{\beta'}\part_{v}^{\beta''}\part_x^\alpha Y^\omega .\end{equation}
Thus, the lemma follows directly from performing a similar energy estimate as done in the previous lemma and using Corollary \ref{c.L_est_v} and Corollary \ref{c.L_est_Y} (which contribute precisely into the remainder $\mathcal{R}^{L,\ell}_{\alpha,\beta,\omega}$). 
\end{proof}

\begin{remark}
Note that the first term in $\mathcal{R}^{T,\ell}_{\alpha,\beta,\omega} $ is linear (due to \eqref{non-dv}), which reflects precisely the linear growth in $t$ of $v$-derivatives in the regime where the transport dynamics in \eqref{VPL-re} is dominant. 
\end{remark}

\subsection{Hypocoercivity estimates} We next derive hypocoercivity estimates that capture precisely the transport-diffusion structure of \eqref{VPL-re}. Precisely, we obtain the following key lemma.

\begin{lemma}\label{lem-cross}
Let $\ell\in \R$, $0< q_0<1$, $\vartheta\in \{0,2\}$, and $(\alpha,\beta,\omega)$ be any triple of multi-indices.  Define $q = \begin{cases} q_0 & \mbox{if $\vartheta = 2$} \\ 0 & \mbox{if $\vartheta = 0$}\end{cases}$ as before. If $|\omega|>0$, we take $q =0$ and $\vartheta=0$. 

Then, for $w=\jap{v}^{\ell}e^{\frac{q|v|^\vartheta}{2}}$, smooth solutions to \eqref{VPL-re} satisfy
\begin{equation}\label{est-cross}
\begin{aligned}
 \frac{d}{dt}&\iint_{\T^3\times \R^3} e^{2(q+1)\phi} w^2 \partial_{x_j} \partial_x^\alpha\partial_v^\beta Y^\omega f\partial_{v_j} \partial_x^\alpha\partial_v^\beta Y^\omega f \; \ud v \, \ud x + \|e^{(q+1)\phi} \partial_{x_j} \partial_x^\alpha\partial_v^\beta Y^\omega f\|^2_{L^2_{x,v}(\ell,\vartheta)} 
 \\& 
 \lesssim \mathcal{Z}^{T,\ell}_{\alpha,\beta,\omega} + \mathcal{Z}^{L,\ell}_{\alpha,\beta,\omega}+ \mathcal{Z}^{Q,\ell}_{\alpha,\beta,\omega}
 \end{aligned}\end{equation}
in which we have denoted

\begin{itemize}

\item by $\mathcal{Z}^{T,\ell}_{\alpha,\beta,\omega}$ the contribution from the transport dynamics: 
$$ 
\begin{aligned}
\mathcal{Z}^{T,\ell}_{\alpha,\beta,\omega} & = (\| \partial_t \phi\|_{L^\infty_x} + \|E\|_{L^\infty_x})\| e^{(q+1)\phi} \partial_x \partial_x^\alpha\partial_v^\beta Y^\omega  f\|_{L^2_{x,v}(\ell,\vartheta)}\| e^{(q+1)\phi} \partial_v \partial_x^\alpha\partial_v^\beta Y^\omega f\|_{L^2_{x,v}(\ell,\vartheta)}
  \\& \quad + \sum_{\substack{|\beta''|=|\beta|-1\\|\beta'|=1}} 
  \Big| \iint_{\T^3\times \R^3} e^{2(q+1)\phi} w^2 \partial_{x_j} \partial_x^\alpha\partial_v^\beta Y^\omega f \partial_{v_j}\partial_x^{\alpha+\beta'}\partial_v^{\beta''} Y^\omega f \; \ud v \, \ud x \Big|
  \\& \quad + \sum_{\substack{|\beta''|=|\beta|-1\\|\beta'|=1}} 
 \Big| \iint_{\T^3\times \R^3} e^{2(q+1)\phi} w^2 \partial_{v_j} \partial_x^\alpha\partial_v^\beta Y^\omega f \partial_{x_j}\partial_x^{\alpha+\beta'}\partial_v^{\beta''} Y^\omega f \; \ud v \, \ud x \Big|,
\end{aligned}$$

\item by $\mathcal{Z}^{L,\ell}_{\alpha,\beta,\omega}$ the contribution from the linear Landau operator:

$$\begin{aligned}
\mathcal{Z}^{L,\ell}_{\alpha,\beta,\omega} & = 
 \nu \Big| \iint_{\T^3\times \R^3}e^{2(q+1)\phi} w^2 \partial_{x_j} \partial_x^\alpha\partial_v^\beta Y^\omega f\partial_{v_j}\partial_x^\alpha\partial_v^\beta Y^\omega [Lf]  
 \; \ud v \, \ud x \Big|
\\
 &\quad +\nu \Big| \iint_{\T^3\times \R^3}e^{2(q+1)\phi} w^2 \partial_{v_j} \partial_x^\alpha\partial_v^\beta Y^\omega f\partial_{x_j}\partial_x^\alpha\partial_v^\beta Y^\omega [Lf] \; \ud v \, \ud x \Big|,
      \end{aligned}$$

\item by $\mathcal{Z}^{Q,\ell}_{\alpha,\beta,\omega}$ the contribution from the source: 

$$\begin{aligned}
\quad \mathcal{Z}^{Q,\ell}_{\alpha,\beta,\omega} & = \Big| \iint_{\T^3\times \R^3} e^{2(q+1)\phi} w^2 \Big[ \partial_{x_j} \partial_x^\alpha\partial_v^\beta Y^\omega f\mathcal{Q}_{\alpha,\beta+e_j,\omega} + \partial_{v_j} \partial_x^\alpha\partial_v^\beta Y^\omega f\mathcal{Q}_{\alpha+e_j,\beta,\omega}  \Big]  \; \ud v \, \ud x \Big|,
     \end{aligned}$$
recalling $\mathcal{Q}_{\alpha,\beta,\omega}$ defined as in \eqref{def-Qalpha}, with $\partial_x^{e_j} = \partial_{x_j}$ and $\partial_v^{e_j} = \partial_{v_j}$. 

\end{itemize}

\end{lemma}

\begin{remark}\label{rem-cross}
Note that the last integral term in the above remainders $\mathcal{Z}^{T,\ell}_{\alpha,\beta,\omega}$ are of the same order as the good term $\|e^{(q+1)\phi} \partial_{x_j} \partial_x^\alpha\partial_v^\beta Y^\omega f\|^2_{L^2_{x,v}(\ell,\vartheta)}$ on the left hand side! A crucial point here is that these last remainder terms vanish for $|\beta|=0$, while for $|\beta|>0$ they are controlled by the good terms for $|\beta|=0$ and the dissipation norms; see \eqref{claim-RSa} below. 
\end{remark}

\begin{proof} Recall that the derivatives satisfy 
\begin{equation}\label{eqs-dxdv}
\begin{aligned}
(D_t - E\cdot v)\partial_{x_j} f + \nu \partial_{x_j}[L f]
&=\partial_{x_j}\mathcal{Q} + [E\cdot \nabla_v - E\cdot v, \partial_{x_j}] f
\\
(D_t - E \cdot v)\partial_{v_j} f + \nu \partial_{v_j}[Lf]
&=- \partial_{x_j} f  + \partial_{v_j}\mathcal{Q}  + [E\cdot \nabla_v - E\cdot v, \partial_{v_j}] f,
\end{aligned}\end{equation}
in which we note that the first term on the right in the second equation plays a crucial role. Indeed, we compute 
$$ 
\begin{aligned}
(D_t &- 2E\cdot v)(\partial_{x_j} f\partial_{v_j} f) +  \nu ( \partial_{x_j} f\partial_{v_j}[Lf] +  \partial_{v_j} f\partial_{x_j}[Lf])
\\&=-|\partial_{x_j} f|^2 + \partial_{x_j} f\partial_{v_j}\mathcal{Q} + \partial_{v_j} f\partial_{x_j}\mathcal{Q} + \partial_{v_j} f  [E\cdot \nabla_v - E\cdot v, \partial_{x_j}] f + \partial_{x_j} f [E\cdot \nabla_v - E\cdot v, \partial_{v_j}] f. 
\end{aligned}$$
Therefore, multiplying the above equation by $e^{2(q+1)\phi} w^2$ and integrating the result, we get 
$$
\begin{aligned}
 &\frac{d}{dt}\iint_{\T^3\times \R^3} e^{2(q+1)\phi} w^2 \partial_{x_j} f\partial_{v_j} f \; \ud v \, \ud x + \iint_{\T^3\times \R^3} e^{2(q+1)\phi} w^2 |\partial_{x_j} f|^2 \; \ud v \, \ud x
 \\& = \iint_{\T^3\times \R^3} (\partial_{x_j} f\partial_{v_j} f) ( D_t + 2E \cdot v) [e^{2(q+1)\phi} w^2]\; \ud v \, \ud x
 \\& \quad + \iint_{\T^3\times \R^3} \Big[ - \nu  ( \partial_{x_j} f\partial_{v_j}[Lf] +  \partial_{v_j} f\partial_{x_j}[Lf]) + \partial_{x_j} f\partial_{v_j}\mathcal{Q} + \partial_{v_j} f\partial_{x_j}\mathcal{Q} \Big] e^{2(q+1)\phi} w^2 \; \ud v \, \ud x
 \\& \quad + \iint_{\T^3\times \R^3} \Big[  \partial_{v_j} f  [E\cdot \nabla_v - E\cdot v, \partial_{x_j}] f + \partial_{x_j} f [E\cdot \nabla_v - E\cdot v, \partial_{v_j}] f \Big] e^{2(q+1)\phi} w^2 \; \ud v \, \ud x .
 \end{aligned}$$
In view of \eqref{dt-vE}, we note 
$$
\begin{aligned}
\iint_{\T^3\times \R^3} & (\partial_{x_j} f\partial_{v_j} f) ( D_t + 2E \cdot v) [e^{2(q+1)\phi} w^2]\; \ud v \, \ud x 
\\&\lesssim (\| \partial_t \phi\|_{L^\infty_x} + \|E\|_{L^\infty_x})\| e^{(q+1)\phi} \partial_x f\|_{L^2_{x,v}(\ell,\vartheta)}\| e^{(q+1)\phi} \partial_vf\|_{L^2_{x,v}(\ell,\vartheta)} .
 \end{aligned}$$
This yields the lemma for $(\alpha,\beta,\omega)=0$. For any triple of $(\alpha,\beta,\omega)$, we simply observe that the derivatives $\partial_x^\alpha \partial_v^\beta Y^\omega f$ satisfy similar transport-diffusion equations to \eqref{eqs-dxdv}, upon noting that 
$$
\begin{aligned}
 [\partial_t + v \cdot \nabla_x, \partial_{x_j} \partial_{v}^{\beta} \part_x^\alpha Y^\omega] &= - \sum_{\substack{|\beta''|=|\beta|-1\\|\beta'|=1}} \partial_{x_j}\part_x^{\alpha+\beta'}\part_{v}^{\beta''} Y^\omega
 \\
  [\partial_t + v \cdot \nabla_x, \partial_{v_j} \partial_{v}^{\beta}\part_x^\alpha Y^\omega ] &= - \sum_{\substack{|\beta''|=|\beta|-1\\|\beta'|=1}} \partial_{v_j}\part_x^{\alpha+\beta'}\part_{v}^{\beta''} Y^\omega  -  \partial_{x_j}\partial_{v}^{\beta} \part_x^\alpha Y^\omega .
  \end{aligned} $$
The last term in the second equation above yields the crucial bound on $|\partial_{x_j} \part_x^\alpha\partial_{v}^{\beta} Y^\omega f|^2$ in \eqref{est-cross}. Collecting terms, we obtain the lemma. 
\end{proof}

%
%
%

\subsection{The hypocoercive energies}\label{sec-EEsetup}

 We are now ready to introduce the main energy estimates, which are an intricate combination of the energy estimates derived for $\partial_x^\alpha \partial_v^\beta Y^\omega f$ in the previous sections. In addition to the $\nu$-dependence that respects the hypocoercivity scaling of the Landau equations, the norms also reflect the weight loss in $v$ due to the Landau collision operator. 
 
\subsection*{The partial energy and dissipation norms} For each triple of multi-indices $(\alpha,\beta,\omega)$, we introduce the partial energy and dissipation norms  
\begin{equation}\label{eq:Eabo}
\begin{aligned} 
\| H \|^2_{\mathcal{E}_{\alpha,\beta,\omega}^{(\vartheta)}} &:= A_0 \sum_{|\alpha'|\le 1}\| e^{(q+1)\phi}\partial^{\alpha'}_xH \|^2_{L^2_{x,v}(\ell_{\alpha,\beta,\omega}-2|\alpha'|,\vartheta)} \\
&\quad+ \nu^{1/3}\langle e^{2(q+1)\phi}\nab_x H,  \nab_v H \rangle_{L^2_{x,v}(\ell_{\alpha,\beta,\omega}-2,\vartheta)} 
 +\nu^{2/3} \| e^{(q+1)\phi}\nab_v H \|^2_{L^2_{x,v}(\ell_{\alpha,\beta,\omega} - 2,\vartheta)},
\end{aligned}
\end{equation}
\begin{equation}\label{eq:Dabo}
\begin{aligned} 
\| H \|^2_{\mathcal{D}_{\alpha,\beta,\omega}^{(\vartheta)}} &:=  \nu^{2/3}A_0 \sum_{|\alpha'|\le 1}\| e^{(q+1)\phi} \partial^{\alpha'}_xH \|^2_{\Delta_{x,v}(\ell_{\alpha,\beta,\omega}-2|\alpha'|,\vartheta)} 
+ \|e^{(q+1)\phi} \nab_x H \|^2_{L^2_{x,v}(\ell_{\alpha,\beta,\omega}-2,\vartheta)} 
\\&\quad + \nu^{4/3}\| e^{(q+1)\phi}\nab_v H \|^2_{\Delta_{x,v}(\ell_{\alpha,\beta,\omega}-2,\vartheta)},
\end{aligned}
\end{equation}
which are used for derivatives $H = \nu^{|\beta|/3}\partial_x^\alpha \partial_v^\beta Y^\omega f$, with $\ell_{\alp,\bt,\om}$ as in \eqref{def-ell}, $q = \begin{cases} q_0 & \mbox{if $\vartheta = 2$} \\ 0 & \mbox{if $\vartheta = 0$}\end{cases}$ as above, and $A_0$ a (large) constant to be determined.  

%

A few comments on the choice of the energies are in order.
\begin{itemize}
\item The large constant $A_0$ will be chosen to ensure \eqref{eq:Eabo} is non-negative, and that the bulk terms also have a good sign. 
\item The $\nu$ powers in the norms above are exactly chosen so that every $\nab_v$ derivative is paired with a $\nu^{1/3}$ power. 
\begin{itemize}
\item This choice of norms is consistent with the facts (1) that for the linear transport equation, every $\partial_v$ gives rise to a $(1+t)$ growth, and (2) that enhanced dissipation acts on a times scale of $t \sim \nu^{-1/3}$.
\end{itemize}
\item Notice that since $q = \begin{cases} q_0 & \mbox{if $\vartheta = 2$} \\ 0 & \mbox{if $\vartheta = 0$}\end{cases}$, we only put in Gaussian $v$-weights when $|\om| = 0$ to avoid difficulties when $Y$ hits on those weights. 
\begin{itemize}
\item The Gaussian weights are used as in \cite{StGu08} to obtain stretched exponential ($\nu$-dependent) time decay in $e^{-\de (\nu^{1/3}t)^{1/3}}$ or $e^{-\de (\nu t)^{2/3}}$ (see \eqref{eq:main.decay}-\eqref{eq:uniform.Landau.damping}). 
\item As a result of not having Gaussian $v$-weights when $|\om|\neq 0$, as discussed in Section~\ref{sec:intro.decay}, at first we only obtain the polynomial $\nu$-dependent time decay when $|\om| \neq 0$, though at the end, we can obtain some stretched exponential decay for $|\om|\neq 0$ via an interpolation argument. 
\item We also mention that the stretched exponential decay is only proved starting in Section~\ref{sec:exp-decay-density} i.e.~one can in principle close the main bootstrap with just polynomial weights but that does not give the stretched exponential bounds.
\end{itemize}
\item These norms are weighted by $e^{(q+1)\phi} $, for a given electric potential $\phi$. In the nonlinear analysis, we shall bootstrap the nonlinear solution so that $\phi$ remains sufficiently small in $L^\infty_x$ (and in fact decays rapidly in time). Therefore, the weight is harmless. 
\end{itemize}

\subsection*{The top-order partial energy and dissipation norms}
We need a variation of the partial energy and dissipation norms $\mathcal{E}_{\alpha,\beta,\omega}^{(\vartheta)}$ and $\mathcal{D}_{\alpha,\beta,\omega}^{(\vartheta)}$ norms, which we denote by $\wtE_{\alpha,\beta,\omega}^{(\vartheta)}$ and $\wtD_{\alpha,\beta,\omega}^{(\vartheta)}$. The difference is that they include one more $\rd_v$ derivative, which is useful to handling the loss of derivative from the density estimates; see Section~\ref{sec:additional.difficulties}. More precisely, for $H = \nu^{|\beta|/3}\partial_x^\alpha \partial_v^\beta Y^\omega f$ as before, define
\begin{equation}\label{eq:wtEabo}
\begin{aligned} 
\| H \|^2_{\wtE_{\alpha,\beta,\omega}^{(\vartheta)}} &:= \| H \|^2_{\mathcal E_{\alpha,\beta,\omega}^{(\vartheta)}} + A_0^{-1}\sum_{|\bt'| = 2} \nu^{2|\bt'|/3} \| e^{(q+1)\phi}\partial_v^{\bt'} H \|^2_{L^2_{x,v}(\ell_{\alpha,\beta,\omega} - 4,\vartheta)},
\end{aligned}
\end{equation}
\begin{equation}\label{eq:wtDabo}
\begin{aligned} 
\| H \|^2_{\wtD_{\alpha,\beta,\omega}^{(\vartheta)}} &:=  \| H \|^2_{\mathcal D_{\alpha,\beta,\omega}^{(\vartheta)}} + A_0^{-1} \sum_{|\bt'| = 2} \nu^{(2+2|\bt'|)/3}\| e^{(q+1)\phi}\partial_v^{\bt'} H \|^2_{\Delta_{x,v}(\ell_{\alpha,\beta,\omega}-4,\vartheta)}.
\end{aligned}
\end{equation}

\subsection*{The combined energy and dissipation norms} Given any $\vartheta \in \{0,2\}$ and any quadruple $(N_{\alp}^{low}, N_{\alp,\bt}, N_{\bt}, N_{\om}) \in (\mathbb N \cup \{0\})^4$ with $N_{\alp,\bt}+N_\om \leq N_{max}$, $N_{\alp}^{low},\, N_{\bt} \leq N_{\alp,\bt}$, define the norms $\mathbb{E}_{N_{\alp}^{low}, N_{\alp,\bt}, N_{\bt}, N_{\om}}^{(\vartheta)}$ and $\mathbb{D}_{N_{\alp}^{low}, N_{\alp,\bt}, N_{\bt}, N_{\om}}^{(\vartheta)}$ by
\begin{equation}\label{eq:combined.norms}
\begin{split}
\| f\|_{\mathbb{E}_{N_{\alp}^{low}, N_{\alp,\bt}, N_{\bt}, N_{\om}}^{(\vartheta)}}^2 := &\: \qquad \smashoperator{\sum_{\substack{|\alp| \geq N_{\alp}^{low} ,\, |\bt| \leq N_{\bt} \\ |\alp| + |\bt| \leq N_{\alp,\bt}}}} \| \rd_x^\alp \rd_v^\bt  f\|_{\mathcal E_{\alp,\bt,0}^{(\vartheta)}}^2 + \qquad \smashoperator{\sum_{\substack{|\alp| \geq N_{\alp}^{low} ,\, |\bt| \leq N_{\bt} \\ |\alp| + |\bt| \leq N_{\alp,\bt} ,\, 1\leq |\omega|\leq N_\om}}} \| \rd_x^\alp \rd_v^\bt Y^\om f\|_{\mathcal E_{\alp,\bt,\om}^{(0)}}^2,\\
\| f\|_{\mathbb{D}_{N_{\alp}^{low}, N_{\alp,\bt}, N_{\bt}, N_{\om}}^{(\vartheta)}}^2 := &\: \qquad \smashoperator{\sum_{\substack{|\alp| \geq N_{\alp}^{low} ,\, |\bt| \leq N_{\bt} \\ |\alp| + |\bt| \leq N_{\alp,\bt}}}} \| \rd_x^\alp \rd_v^\bt  f\|_{\mathcal D_{\alp,\bt,0}^{(\vartheta)}}^2 + \qquad \smashoperator{\sum_{\substack{|\alp| \geq N_{\alp}^{low} ,\, |\bt| \leq N_{\bt} \\ |\alp| + |\bt| \leq N_{\alp,\bt} ,\, 1\leq |\omega|\leq N_\om}}} \| \rd_x^\alp \rd_v^\bt Y^\om f\|_{\mathcal D_{\alp,\bt,\om}^{(0)}}^2.
\end{split}
\end{equation}
We emphasize two points about the definition \eqref{eq:combined.norms}:
\begin{enumerate} 
\item $N_{\alp}^{low}$ is a lower bound, while the other are upper bounds.
\item Even though we may use $\vartheta = 2$ in the $\mathbb{E}_{N_{\alp}^{low}, N_{\alp,\bt}, N_{\bt}, N_{\om}}^{(\vartheta)}$ and $\mathbb{D}_{N_{\alp}^{low}, N_{\alp,\bt}, N_{\bt}, N_{\om}}^{(\vartheta)}$ norms, the exponential $v$-weight is only present when $|\om| =0$.
\end{enumerate}

We also explain the various parameters in the norms in \eqref{eq:combined.norms}:
\begin{itemize}
\item The parameter $N_{\alp,\bt}$ counts the maximum number of $\rd_x$ and $\rd_v$ derivatives, while $N_\bt$ only counts the $\rd_v$ derivatives. $N_\om$ separately counts the number of $Y$ derivatives. The point is that which controlling up to $N_\bt$ $\rd_v$-derivatives, the linear error terms involve at most $N_\bt -1$ $\rd_v$-derivatives (and similarly for $Y$ derivatives). Thus we can induct in $N_\bt$ and $N_\om$ to obtain estimates with the right constants. See for instance the step right after \eqref{eq:EE.Eabo.2}.
\item As we said above, $N_{\alp}^{low}$ is a lower bound. The important point is that we need to separate the $N_\alp^{low} = 0$ and $N_\alp^{low} >0$ cases since we only have enhanced dissipation when $N_\alp^{low}>0$; see discussions in Section~\ref{sec:additional.difficulties}.
\end{itemize}

\subsection*{The top-order combined energy and dissipations norms} Given $\vartheta$, $N_{\alp}^{low}$, $N_{\alp,\bt}$, $N_{\bt}$, $N_{\om}$ as above, we also define corresponding combined norms which include the extra $\rd_v$ derivatives as in \eqref{eq:wtEabo}, \eqref{eq:wtDabo},
\begin{equation}\label{eq:combined.norms.top}
\begin{split}
\| f\|_{\wtbE_{N_{\alp}^{low}, N_{\alp,\bt}, N_{\bt}, N_{\om}}^{(\vartheta)}}^2 := \qquad \smashoperator{\sum_{\substack{|\alp| \geq N_{\alp}^{low} ,\, |\bt| \leq N_{\bt} \\ |\alp| + |\bt| \leq N_{\alp,\bt}}}} \| \rd_x^\alp \rd_v^\bt  f\|_{\wtE_{\alp,\bt,0}^{(\vartheta)}}^2 + \qquad \smashoperator{\sum_{\substack{|\alp| \geq N_{\alp}^{low} ,\, |\bt| \leq N_{\bt} \\ |\alp| + |\bt| \leq N_{\alp,\bt} ,\, 1\leq |\omega|\leq N_\om}}} \| \rd_x^\alp \rd_v^\bt Y^\om f\|_{\wtE_{\alp,\bt,\om}^{(0)}}^2 , \\
 \| f\|_{\wtbD_{N_{\alp}^{low}, N_{\alp,\bt}, N_{\bt}, N_{\om}}^{(\vartheta)}}^2 := \qquad \smashoperator{\sum_{\substack{|\alp| \geq N_{\alp}^{low} ,\, |\bt| \leq N_{\bt} \\ |\alp| + |\bt| \leq N_{\alp,\bt}}}} \| \rd_x^\alp \rd_v^\bt  f\|_{\widetilde{\mathcal D}_{\alp,\bt,0}^{(\vartheta)}}^2 + \qquad \smashoperator{\sum_{\substack{|\alp| \geq N_{\alp}^{low} ,\, |\bt| \leq N_{\bt} \\ |\alp| + |\bt| \leq N_{\alp,\bt} ,\, 1\leq |\omega|\leq N_\om}}} \| \rd_x^\alp \rd_v^\bt Y^\om f\|_{\widetilde{\mathcal D}_{\alp,\bt,\om}^{(0)}}^2 .
 \end{split}
 \end{equation}

For brevity, we also introduce
\begin{equation}\label{eq:def.EN.DN}
\| f\|_{\wtbE_N^{(\vartheta)}}^2 := \sum_{N_{\alp,\bt}+N_\om \leq N} \| f\|_{\wtbE_{0, N_{\alp,\bt}, N_{\alp,\bt}, N_{\om}}^{(\vartheta)}}^2, \quad \| f\|_{\wtbD_N^{(\vartheta)}}^2 := \sum_{N_{\alp,\bt}+N_\om \leq N} \| f\|_{\wtbD_{0, N_{\alp,\bt}, N_{\alp,\bt}, N_{\om}}^{(\vartheta)}}^2.
\end{equation}

\subsection*{The primed energy and dissipation norms}

Finally, for each of the norms defined above, we introduce an analogous norm, labelled by $(\vartheta)'$ instead of $(\vartheta)$, which is defined so that when $\vartheta = 2$, the exponential $v$-weights $e^{q|v|^2}$ are replaced by $e^{q'|v|^2}$, where $q' = \frac12q$; cf.~\eqref{eq:Lp.'.1}--\eqref{eq:Lp.'.2}. In other words, starting from \eqref{eq:Eabo}, we define
\begin{equation}\label{eq:Eabo.p}
\begin{aligned} 
\| H \|^2_{\mathcal{E}_{\alpha,\beta,\omega}^{(\vartheta)'}} &:= A_0 \sum_{|\alpha'|\le 1}\| e^{(q+1)\phi}\partial^{\alpha'}_xH \|^2_{L^2_{x,v}(\ell_{\alpha,\beta,\omega}-2|\alpha'|,\vartheta)'} \\
&\quad+ \nu^{1/3}\langle e^{2(q+1)\phi}\nab_x H,  \nab_v H \rangle_{L^2_{x,v}(\ell_{\alpha,\beta,\omega}-2,\vartheta)'} 
 +\nu^{2/3} \| e^{(q+1)\phi}\nab_v H \|^2_{L^2_{x,v}(\ell_{\alpha,\beta,\omega} - 2,\vartheta)'},
\end{aligned}
\end{equation} 
for $H = \nu^{|\beta|/3}\partial_x^\alpha \partial_v^\beta Y^\omega f$, and make similar definitions for 
\begin{equation}\label{eq:p.norms.general}
\begin{split}
\| \nu^{|\beta|/3}\partial_x^\alpha \partial_v^\beta Y^\omega f \|^2_{\wtD_{\alpha,\beta,\omega}^{(\vartheta)'}},\, \| f\|_{\mathbb{E}_{N_{\alp}^{low}, N_{\alp,\bt}, N_{\bt}, N_{\om}}^{(\vartheta)'}}^2,\, \| f\|_{\mathbb{D}_{N_{\alp}^{low}, N_{\alp,\bt}, N_{\bt}, N_{\om}}^{(\vartheta)'}}^2,\\
\| f\|_{\wtbE_{N_{\alp}^{low}, N_{\alp,\bt}, N_{\bt}, N_{\om}}^{(\vartheta)'}}^2,\, \| f\|_{\wtbD_{N_{\alp}^{low}, N_{\alp,\bt}, N_{\bt}, N_{\om}}^{(\vartheta)'}}^2,\, \| f\|_{\wtbE_N^{(\vartheta)'}}^2,\,\| f\|_{\wtbD_N^{(\vartheta)'}}^2, 
\end{split}
\end{equation}
by modifying \eqref{eq:Dabo}, \eqref{eq:wtEabo}, \eqref{eq:wtDabo}, \eqref{eq:combined.norms} and \eqref{eq:def.EN.DN}.

\subsection{The main energy estimates}

In the following proposition, we estimate all the remainder terms in Lemma~\ref{lem-basicEEdx} and Lemma~\ref{lem-cross} except for the $\mathcal R^Q$ and $\mathcal Z^Q$ terms. For the full nonlinear solution, those terms will be treated in Sections~\ref{s.closing_eng} and \ref{sec:exp-decay-energy}.

\begin{proposition}\label{prop-mainEE} 
Let $\vartheta\in \{0,2\}$, and $(N_{\alp}^{low}, N_{\alp,\bt}, N_{\bt}, N_{\om}) \in (\mathbb N \cup \{0\})^4$ with $N_{\alp,\bt}+N_\om \leq N_{max}$, $N_{\alp}^{low},\, N_{\bt} \leq N_{\alp,\bt}$. Recall the definitions of $\mathbb{E}_{N_{\alp}^{low}, N_{\alp,\bt}, N_{\bt}, N_{\om}}^{(\vartheta)}$ and $\mathbb{D}_{N_{\alp}^{low}, N_{\alp,\bt}, N_{\bt}, N_{\om}}^{(\vartheta)}$ in \eqref{eq:combined.norms}.

There is a positive constant $\theta$ so that for $N_{\alp}^{low} \geq 1$, smooth solutions to \eqref{VPL-re} satisfy
\begin{equation}\label{est-mainEE}
 \begin{split}
 \frac{d}{dt}& \| f \|^2_{\mathbb{E}_{N_{\alp}^{low}, N_{\alp,\bt}, N_{\bt}, N_{\om}}^{(\vartheta)}}+ \theta \nu^{1/3} \| f \|^2_{\mathbb{D}_{N_{\alp}^{low}, N_{\alp,\bt}, N_{\bt}, N_{\om}}^{(\vartheta)}} \\
 &\ls  (\|\partial_t\phi \|_{L^\infty_x} + \|\phi \|_{W^{1,\infty}_x})  \| f \|^2_{\mathbb{E}_{N_{\alp}^{low}, N_{\alp,\bt}, N_{\bt}, N_{\om}}^{(\vartheta)}} + \qquad \smashoperator{\sum_{\substack{|\alp| \geq N_{\alp}^{low} ,\, |\bt| \leq N_{\bt} \\ |\alp| + |\bt| \leq N_{\alp,\bt} ,\, |\omega|\leq N_\om}}}\mathcal R_{\alp,\bt,\om},
 \end{split}
 \end{equation}
 and for $N_{\alp}^{low} = 0$, smooth solutions to \eqref{VPL-re} satisfy
 \begin{equation}\label{est-mainEE.0}
 \begin{split}
 \frac{d}{dt}& \| f \|^2_{\mathbb{E}_{0, N_{\alp, \bt}, N_{\bt}, N_\om}^{(\vartheta)}}+ \theta \nu^{1/3} \| f \|^2_{\mathbb{D}_{0, N_{\alp, \bt}, N_{\bt}, N_\om}^{(\vartheta)}} \\
 &\ls  \nu |(\bar{a},\bar{b},\bar{c})|^2 + (\|\partial_t\phi \|_{L^\infty_x} + \|\phi \|_{W^{1,\infty}_x})  \| f \|^2_{\mathbb E_{0, N_{\alp, \bt}, N_{\bt}, N_\om}^{(\vartheta)}} + \qquad \smashoperator{\sum_{\substack{|\bt| \leq N_\bt \\ |\alp| + |\bt| \leq N_{\alp,\bt} ,\, |\omega|\leq N_{\om}}}}\mathcal R_{\alp,\bt,\om},
 \end{split}
 \end{equation}
 where the remainders are calculated by 
\begin{equation}\label{def-Rabc}
\begin{aligned}
\mathcal{R}_{\alpha,\beta,\omega} =  \sum_{|\alpha'|+|\beta'|\le 1}\nu^{2(|\beta|+|\beta'|)/3} \mathcal{R}^{Q,\ell-2|\alpha'|-2|\beta'|}_{\alpha+\alpha',\beta+\beta',\omega} + \nu^{1/3}  \nu^{2|\beta|/3}\mathcal{Z}^{Q,\ell-2}_{\alpha,\beta,\omega} 
\end{aligned}
\end{equation}
with $\mathcal{R}^{Q,\ell}_{\alpha,\beta,\omega}$ and $\mathcal{Z}^{Q,\ell}_{\alpha,\beta,\omega}$ as introduced in Lemma \ref{lem-basicEEdx} and Lemma \ref{lem-cross}, and $|(\bar{a},\bar{b},\bar{c})|^2 := |\bar{a}|^2 + \sum_{j=1}^3 |\bar{b}_j|^2 + |\bar{c}|^2$, where
$$\bar{a} := \iint_{\T^3\times \R^3} f \sqrt{\mu} \, \ud v \,\ud x,\quad \bar{b}_j := \iint_{\T^3\times \R^3} f v_j \sqrt{\mu} \, \ud v \,\ud x,\quad \bar{c} := \iint_{\T^3\times \R^3} f |v|^2 \sqrt{\mu} \, \ud v \,\ud x. $$
\end{proposition}

\begin{remark}
We note that the ``remainders'' $\mathcal{R}_{\alpha,\beta,\omega}$ do contain linear terms (due to $\mathcal{Q}$ from \eqref{def-QQQ}), the control of which by the energy and dissipation norms is certainly not immediate; see Section \ref{s.closing_eng} for the full treatment of these and the other nonlinear terms. 
\end{remark}

\begin{proof} Let $(\alpha,\beta,\omega)$ be any triple of multi-indices, with $|\alpha|+|\beta|+|\omega| \leq N_{max}$, and let $\ell = \ell_{\alpha,\beta,\omega}$ and $w = w_{\alpha,\beta,\omega}$ be the weight functions defined as in \eqref{def-w}--\eqref{def-ell}. Recalling the partial energy and dissipation norms and appropriately combining Lemma \ref{lem-basicEEdx} and Lemma \ref{lem-cross}, we obtain 
\begin{equation}\label{est-partialEE}
\begin{aligned}
 \frac{d}{dt}& \| \partial_x^\alpha\partial_v^\beta Y^\omega f \|^2_{\mathcal{E}^{(\vartheta)}_{\alpha,\beta,\omega}}+ \theta \nu^{1/3} \| \partial_x^\alpha\partial_v^\beta Y^\omega f \|^2_{\mathcal{D}^{(\vartheta)}_{\alpha,\beta,\omega}} \lesssim \mathcal{R}'_{\alpha,\beta,\omega},
 \end{aligned}\end{equation}
where the remainders are calculated by 
\begin{equation}\label{eq:remainder.alp.bt.om}
\begin{aligned}
\mathcal{R}'_{\alpha,\beta,\omega}
& = A_0  \nu^{2|\beta|/3}\sum_{|\alpha'|\le 1}\Big[ \mathcal{R}^{T,\ell-2|\alpha'|}_{\alpha+\alpha',\beta,\omega} + \nu \mathcal{R}^{L,\ell-2|\alpha'|}_{\alpha+\alpha',\beta,\omega}+ \mathcal{R}^{Q,\ell-2|\alpha'|}_{\alpha+\alpha',\beta,\omega}\Big] 
\\ & \quad +  \nu^{2|\beta|/3}\sum_{|\beta'|\le 1} \nu^{2|\beta'|/3} \Big[  \mathcal{R}^{T,\ell-2|\beta'|}_{\alpha,\beta+\beta',\omega} + \nu \mathcal{R}^{L,\ell-2|\beta'|}_{\alpha,\beta+\beta',\omega}+  \mathcal{R}^{Q,\ell-2|\beta'|}_{\alpha,\beta+\beta',\omega}\Big]
\\&\quad +  \nu^{1/3}  \nu^{2|\beta|/3}\Big[ \mathcal{Z}^{T,\ell-2}_{\alpha,\beta,\omega}  + \mathcal{Z}^{L,\ell-2}_{\alpha,\beta,\omega}  + \mathcal{Z}^{Q,\ell-2}_{\alpha,\beta,\omega} \Big]
\end{aligned}
\end{equation}
where the remainders were introduced previously in Lemma \ref{lem-basicEEdx} and Lemma \ref{lem-cross}. 

\subsection*{Estimates on $\mathcal{R}^{T,\ell}_{\alpha,\beta,\omega}$}

Let us take care of the remainders arising due to the transport dynamics. 
We first consider the $\mathcal R^T$ terms in \eqref{eq:remainder.alp.bt.om} with $|\alp'|=0$ and $|\bt'|=0$. We claim that
\begin{equation}\label{claim-RTa.prelim}
\begin{split}
\nu^{2|\beta|/3}  \mathcal{R}^{T,\ell}_{\alpha,\beta,\omega}
\lesssim &\: A_0^{-1/2} \nu^{1/3} \sum_{\substack{ |\bt'| < |\bt| }} \|\rd_x^\alp \rd_v^{\bt'} Y^\om f\|_{\mathcal{D}^{(\vartheta)}_{\alp,\bt',\om}}^2 + (\|\partial_t\phi \|_{L^\infty_x} + \|E\|_{L^\infty_x})\| \rd_x^\alp \rd_v^\bt Y^\om f \|^2_{\mathcal{E}^{(\vartheta)}_{\alp,\bt,\om}}.
\end{split}
\end{equation}

Indeed, in view of the definition of $\mathcal{R}^{T,\ell}_{\alpha,\beta,\omega}$ from Lemma \ref{lem-basicEEdx}, the bound for the last term involving $\partial_t\phi$ and $E$ is immediate. As for the first term in $\mathcal R^{T,\ell}_{\alp,\bt,\om}$, noting that $|\beta|\ge 1$ and recalling the weight function \eqref{def-w}--\eqref{def-ell}, for $|\beta''|=1$ and $|\bt''|+|\bt'''| = |\bt|$, we bound  
\begin{equation}\label{eq:RT.main.error}
\begin{aligned}
\nu^{|\beta|/3} \| e^{(q+1)\phi} \partial_x^{\beta''} \partial_x^{\alpha}\partial_v^{\beta'''} Y^\omega f\|_{L^2_{x,v}(\ell_{\alpha,\beta,\omega},\vartheta)} 
&\lesssim \nu^{|\beta|/3}\|  e^{(q+1)\phi}\partial_x^{\beta''} \partial_x^{\alpha}\partial_v^{\beta'''} Y^\omega f\|_{L^2_{x,v}(\ell_{\alpha,\beta''',\omega}-2,\vartheta)} 
\\&\lesssim 
\nu^{1/3} \sum_{\substack{ |\beta'''| < |\beta|}} \| \partial_x^{\alp}\partial_v^{\beta'''} Y^\omega f\|_{\mathcal{D}^{(\vartheta)}_{\alp,\beta''',\omega}}.
\end{aligned}
\end{equation}
In addition, again noting $|\beta|\ge 1$, we bound 
\begin{equation}\label{dvL2}
\begin{aligned}
\nu^{|\beta|/3} \| e^{(q+1)\phi}  \partial_x^\alpha\partial_v^\beta Y^\omega f\|_{L^2_{x,v}(\ell_{\alpha,\beta,\omega},\vartheta)} 
&\lesssim 
\nu^{|\beta|/3} \sum_{\substack{|\beta''| < |\beta|}} \| e^{(q+1)\phi}  \partial_{v_j}\partial_x^\alpha\partial_v^{\beta''} Y^\omega f\|_{L^2_{x,v}(\ell_{\alpha,\beta'',\omega}-2,\vartheta)} 
\\&\lesssim  A_0^{-1/2} \sum_{\substack{|\beta''| < |\beta|}} \| \partial_x^{\alpha}\partial_v^{\beta''} Y^\omega f\|_{\mathcal{D}^{(\vartheta)}_{\alpha,\beta'',\omega}}.
\end{aligned}
\end{equation}
The claim \eqref{claim-RTa.prelim} follows from the above computations.

Now consider terms in \eqref{eq:remainder.alp.bt.om} involving $\mathcal{R}^{T,\ell-2|\alpha'|}_{\alpha+\alpha',\beta,\omega}$ and $\mathcal{R}^{T,\ell-2|\beta'|}_{\alpha,\beta+\beta',\omega}$ with either $|\alp'|=1$ or $|\bt'|=1$. Notice that the $\mathcal D_{\alp,\bt,\om}$ norm controls one additional $\rd_x$ or $\nu^{\f 13} \rd_v$ derivative with exactly the right weights so that the $\rd_t\phi$ and $E$ terms can be treated in the same way. Similar argument applies to \eqref{eq:RT.main.error} with an additional $\nu^{\f 13} \rd_v$ derivative and to \eqref{dvL2} with an additional $\rd_x$ or $\nu^{\f 13} \rd_v$ derivative. The only non-obvious term we need to consider is \eqref{eq:RT.main.error} with an additional $\rd_x$ derivative. We bound this term as follows, with $|\alp'| = 1$, $|\bt''| = 1$ and $|\bt''| + |\bt'''| = |\bt|$,
\begin{equation}\label{eq:RT.main.error.special}
\begin{aligned}
&\: \nu^{|\beta|/3} \| e^{(q+1)\phi} \partial_x^{\beta''} \rd_x^{\alp'}\partial_x^{\alpha}\partial_v^{\beta'''} Y^\omega f\|_{L^2_{x,v}(\ell_{\alpha,\beta,\omega}-2,\vartheta)} \\
\lesssim &\:  \nu^{|\beta|/3}\|  e^{(q+1)\phi}\partial_x^{\beta''} \rd_x^{\alp'} \partial_x^{\alpha}\partial_v^{\beta'''} Y^\omega f\|_{L^2_{x,v}(\ell_{\alpha+\alp',\beta''',\omega}-2,\vartheta)} 
\\ \lesssim &\: 
\nu^{1/3} \sum_{\substack{|\alp''|= |\alp|+1 \\ |\beta'''| < |\beta|}} \| \partial_x^{\alp''}\partial_v^{\beta'''} Y^\omega f\|_{\mathcal{D}^{(\vartheta)}_{\alp'',\beta''',\omega}}.
\end{aligned}
\end{equation}

Combining all the above considerations and relabelling the indices give
\begin{equation}\label{claim-RTa}
\begin{split}
&\: A_0 \nu^{2|\beta|/3} \sum_{|\alpha'|\le 1}  \mathcal{R}^{T,\ell-2|\alpha'|}_{\alpha+\alpha',\beta,\omega} +  \nu^{2|\beta|/3} \sum_{|\beta'|\le 1} \nu^{2|\beta'|/3} \mathcal{R}^{T,\ell-2|\beta'|}_{\alpha,\beta+\beta',\omega} \\
\ls &\: A_0 (\|\partial_t\phi \|_{L^\infty_x} + \|E\|_{L^\infty_x}) \| \rd_x^\alp \rd_v^\bt Y^\om f \|^2_{\mathcal{E}^{(\vartheta)}_{\alp,\bt,\om}} + A_0^{1/2} \nu^{1/3}  \sum_{\substack{ |\bt'| < |\bt| }} \|\rd_x^\alp \rd_v^{\bt'} Y^\om f\|_{\mathcal{D}^{(\vartheta)}_{\alp,\bt',\om}}^2 \\
&\: + A_0^{1/2} \nu^{\f 13} (\sum_{\substack{ |\bt'| < |\bt| }} \|\rd_x^\alp \rd_v^{\bt'} Y^\om f\|_{\mathcal{D}^{(\vartheta)}_{\alp,\bt',\om}}) (\sum_{\substack{|\alp'|= |\alp| + 1 \\ |\bt'| < |\bt| }} \|\rd_x^{\alp'} \rd_v^{\bt'} Y^\om f\|_{\mathcal{D}^{(\vartheta)}_{\alp',\bt',\om}}) \\
\ls &\: A_0 (\|\partial_t\phi \|_{L^\infty_x} + \|E\|_{L^\infty_x}) \| \rd_x^\alp \rd_v^\bt Y^\om f \|^2_{\mathcal{E}^{(\vartheta)}_{\alp,\bt,\om}} + A_0^{1/2} \nu^{1/3}  \qquad \smashoperator{\sum_{\substack{|\alp'| \geq |\alp|,\, |\bt'| < |\bt| \\ |\alp'| + |\bt'| \leq |\alp| + |\bt|}}} \|\rd_x^{\alp'} \rd_v^{\bt'} Y^\om f\|_{\mathcal{D}^{(\vartheta)}_{\alp',\bt',\om}}^2 .
\end{split}
\end{equation}


 
 \subsection*{Estimates on $\mathcal{Z}^{T,\ell}_{\alpha,\beta,\omega}$}
 
Next, let us give bounds on $\mathcal{Z}^{T,\ell-2}_{\alpha,\beta,\omega} $ defined as in Lemma \ref{lem-cross}. We claim that 
\begin{equation}\label{claim-RSa}
\begin{aligned}
\nu^{1/3}\nu^{2|\beta|/3} \mathcal{Z}^{T,\ell-2}_{\alpha,\beta,\omega} 
\lesssim &\: A_0^{-1/2}\nu^{1/3} \sum_{\substack{|\alp'| = |\alp| + 1\\ |\bt'|< |\bt|}}\| \partial_x^{\alpha'}\partial_v^{\beta''}  Y^\omega f\|_{\mathcal{D}^{(\vartheta)}_{\alp',\beta',\omega}} \| \partial_x^{\alpha}\partial_v^{\beta}  Y^\omega f\|_{\mathcal{D}^{(\vartheta)}_{\alpha,\beta,\omega}}  \\
&\:  +(\|\partial_t\phi \|_{L^\infty_x} + \|E\|_{L^\infty_x})\| \rd_x^\alp \rd_v^\bt Y^\om f \|^2_{\mathcal{E}^{(\vartheta)}_{\alp,\bt,\om}}.
\end{aligned}
\end{equation}
Note that the weight function is rightly indexed at $\ell_{\alpha,\beta,\omega}-2$. Again, the term involving $\partial_t\phi$ and $E$ is direct, contributing to the last term in the above estimate. Next, for $|\beta'|=1$ and $|\beta''|<|\beta|$, we bound 
the integral 
$$
\begin{aligned}
&\nu^{1/3}\nu^{2|\beta|/3}  \left| \iint_{\T^3\times \R^3} e^{2(q+1)\phi} w^2 \langle v\rangle^{-4} \partial_{v_j} \partial_x^\alpha\partial_v^\beta Y^\omega f \partial_{x_j}\partial_x^{\alpha+\beta'}\partial_v^{\beta''} Y^\omega f \; \ud v \, \ud x \right|
 \\
 &\lesssim \nu^{1/3}\nu^{2|\beta|/3} \| e^{2(q+1)\phi}\partial_{x_j}\partial_x^{\alpha+\beta'}\partial_v^{\beta''}  Y^\omega f\|_{L^2_{x,v}(\ell_{\alpha,\beta,\omega}-2,\vartheta)}  \| e^{2(q+1)\phi}\partial_{v_j} \partial_x^\alpha\partial_v^\beta  Y^\omega f\|_{L^2_{x,v}(\ell_{\alpha,\beta,\omega}-2,\vartheta)} 
 \\
 &\lesssim \nu^{1/3} ( \nu^{|\bt''|/3} \| \rd_{x_j}\partial_x^{\alpha+\beta'}\partial_v^{\beta''}  Y^\omega f\|_{L^2_{x,v}(\ell_{\alpha+\bt',\beta'',\omega}-2,\vartheta)})  (\nu^{(|\beta|+1)/3} \|  \partial_x^\alpha\partial_v^\beta  Y^\omega f\|_{\Delta_{x,v}(\ell_{\alpha,\beta,\omega},\vartheta)} ) \\
 &\lesssim A_0^{-1/2}\nu^{1/3}\| \partial_x^{\alpha+\beta'}\partial_v^{\beta''}  Y^\omega f\|_{\mathcal{D}^{(\vartheta)}_{\alpha+\beta',\beta'',\omega}} \| \partial_x^{\alpha}\partial_v^{\beta}  Y^\omega f\|_{\mathcal{D}^{(\vartheta)}_{\alpha,\beta,\omega}}    
.  \end{aligned}$$

Relabelling the multi-indices yields the bounds as claimed in \eqref{claim-RSa}. The last integral term in $ \mathcal{Z}^{T,\ell-2}_{\alpha,\beta,\omega} $ that involves $\partial_{x_j} \partial_x^\alpha\partial_v^\beta Y^\omega f \partial_{v_j}\partial_x^{\alpha+\beta'}\partial_v^{\beta''} Y^\omega f $ is treated similarly. This verifies the claim \eqref{claim-RSa}.

\subsection*{Estimates on $\mathcal{R}^{L,\ell}_{\alpha,\beta,\omega}$}

We go on with giving bounds on $\mathcal{R}^{L,\ell}_{\alpha,\beta,\omega}$ as introduced in Lemma \ref{lem-basicEEdx}.  
Recall from Lemma \ref{lem-basicEEdx} that $\mathcal R^{L,\ell}_{\alp,\bt,\om}$ has three contributions, which we label as 
$$\mathcal R^{L,\ell}_{\alp,\bt,\om} =: \mathcal R^{L,\ell,1}_{\alp,\bt,\om} + \mathcal R^{L,\ell,2}_{\alp,\bt,\om} + \mathcal R^{L,\ell,3}_{\alp,\bt,\om}.$$
 The first terms can be bounded directly using the definition of the dissipation norms as follows:
\begin{equation}\label{eq:RL1.est}
\begin{split}
&\: A_0 \nu \nu^{2|\beta|/3} \sum_{|\alpha'|\le 1} \mathcal{R}^{L,\ell-2|\alpha'|,1}_{\alpha+\alpha',\beta,\omega} + \nu \nu^{2|\beta|/3} \sum_{|\beta'|\le 1}\nu^{2|\beta'|/3} \mathcal{R}^{L,\ell-2|\beta'|,1}_{\alpha,\beta+\beta',\omega}
\\ 
\lesssim &\: 
\eta  \nu \nu^{2|\bt|/3} A_0 \sum_{\substack{|\alp'|\leq 1\\ |\beta''|\leq |\beta|\\|\omega''|\leq |\omega|}}\norm{e^{(q+1)\phi} \rd_x^{\alp+\alp'} \rd_v^{\bt''} Y^{\om''} f}^2_{\Delta_{x,v}(\ell_{\alpha+\alp',\beta'',\omega})} \\
&\: + \eta \nu \nu^{2|\bt|/3} A_0 \sum_{\substack{|\bt'|\leq 1\\ |\beta''|\leq |\beta|\\|\omega''|\leq |\omega|}} \nu^{2|\bt'|/3} \norm{e^{(q+1)\phi} \rd_x^{\alp} \rd_v^{\bt''+\bt'} Y^{\om''} f}^2_{\Delta_{x,v}(\ell_{\alpha,\beta''+\bt',\omega})} \\
\ls &\: \eta \nu^{\f 13} \smashoperator{\sum_{\substack{|\bt''|\leq |\bt|\\|\omega''|\leq|\omega|}}} \| \rd_x^\alp \rd_v^{\bt''} Y^{\om''} f\|_{\mathcal D^{(\vartheta)}_{\alp,\bt'',\om''}}^2.
\end{split}
\end{equation}

In a similar manner, the second terms can be bounded by 
\begin{equation}\label{eq:RL2.est}
\begin{split}
&\: A_0 \nu \nu^{2|\beta|/3} \sum_{|\alpha'|\le 1} \mathcal{R}^{L,\ell-2|\alpha'|,2}_{\alpha+\alpha',\beta,\omega} + \nu \nu^{2|\beta|/3} \sum_{|\beta'|\le 1}\nu^{2|\beta'|/3} \mathcal{R}^{L,\ell-2|\beta'|,2}_{\alpha,\beta+\beta',\omega} \\
\ls &\: C_\eta \nu^{\f 13} \smashoperator{\sum_{\substack{|\bt''|\leq |\bt|\\|\omega''|\leq|\omega| \\ |\bt''| + |\om''| \leq |\bt| + |\om| -1}}} \| \rd_x^\alp \rd_v^{\bt''} Y^{\om''} f\|_{\mathcal D^{(\vartheta)}_{\alp,\bt'',\om''}}^2.
\end{split}
\end{equation}

The third terms require slightly more work. First, we bound
\begin{equation}\label{eq:RL3.est}
\begin{split}
&\: A_0 \nu \nu^{2|\beta|/3} \sum_{|\alpha'|\le 1} \mathcal{R}^{L,\ell-2|\alpha'|,3}_{\alpha+\alpha',\beta,\omega} + \nu \nu^{2|\beta|/3} \sum_{|\beta'|\le 1}\nu^{2|\beta'|/3} \mathcal{R}^{L,\ell-2|\beta'|,3}_{\alpha,\beta+\beta',\omega}
\\ 
\lesssim &\: C_\eta A_0 \nu \nu^{2|\bt|/3} \smashoperator{\sum_{\substack{|\alp'|\leq 1\\ |\omega''|\leq|\omega|}}}\norm{\mu e^{(q+1)\phi} \rd_x^{\alp+\alp'} Y^{\om'} f}_{L^2_{x,v}}^2 \\
\end{split}
\end{equation}
We will analyze the RHS of \eqref{eq:RL3.est} further. The issue here is that if we control it directly with the dissipation norms, we would not have enough smallness. 
For any function $g$, decompose $g = g_{=0} + g_{\not =0}$, where $g_{=0}(t,v) := \int_{\T^3} g(t,x,v) \, \ud x$. For $|\alp'| \leq 1$ and $|\om'|\leq |\om|$, we bound
\begin{equation}\label{eq:low.term.in.L.est}
\begin{split}
&\: A_0 \nu \norm{\mu e^{(q+1)\phi} \rd_x^{\alp+\alp'} Y^{\omega'} f}_{L^2_{x,v}}^2 \\
= &\: A_0 \nu (\norm{\mu(\rd_x^{\alp+\alp'} Y^{\omega'} f)_{\not = 0}}_{L^2_{x,v}}^2 + \norm{\mu (\rd_x^{\alp+\alp'} Y^{\omega'} f)_{=0}}_{L^2_{x,v}}^2 +  \norm{\mu (e^{(q+1)\phi} -1)\rd_x^{\alp+\alp'} Y^{\omega'} f}_{L^2_{x,v}}^2).
\end{split}
\end{equation}
The last term is clearly bounded by $\nu \| \phi\|_{L^\infty}\| \rd_x^\alp Y^{\om'} f\|_{\mathcal{E}^{(\vartheta)}_{\alp,0,\om'}}^2$. As for the first term, if $|\alp'| = 1$, then we bound it directly by $\nu^{2/3} \nu^{1/3} \|\rd_x^\alp Y^{\om'} f\|_{\mathcal D^{(\vartheta)}_{\alp,0,\om'}}^2$ (noting the extra factor of $\nu^{2/3}$); while if $|\alp'| =0$, we use the Poincar\'e's inequality to obtain
$$A_0 \nu \norm{\mu(\rd_x^{\alp} Y^{\omega'} f)_{\not = 0}}_{L^2_{x,v}}\ls A_0 \nu \norm{\mu \nab_x \rd_x^\alp Y^{\omega'} f}_{L^2_{x,v}} \ls \nu^{2/3} \nu^{1/3} \| \rd_x^{\alp} Y^{\om'} f\|_{\mathcal{D}^{(\vartheta)}_{\alp,0,\om'}}^2. $$
Next, note that the zeroth mode (i.e.~second term in \eqref{eq:low.term.in.L.est}) is only non-vanishing when $|\alp| = |\alp'| =0$. Fixing $|\alp| = |\alp'| =0$, we further have two cases: if $|\om'| = 0$, we simply bound the term by $A_0 \nu \norm{\mu f_0}_{L^2_{x,v}}^2$; while if $|\om'|>0$, we write $Y^{\om'} = Y_j Y^{\om''}$ for some $j$ and $\om''$, and use the fact $(Y^{\om'} f)_{=0} = Y^{\om'} (f_{=0}) = (t \rd_{x_j} + \rd_{v_j}) Y^{\om''} f_{=0} = \rd_{v_j} Y^{\om''} f_{=0}$ to deduce 
\begin{equation}
A_0 \nu \norm{\mu (Y^{\omega'} f)_{=0}}_{L^2_{x,v}}^2 \ls A_0 \nu \norm{\mu \rd_{v_j} Y^{\om''} f_{=0}}_{L^2_{x,v}}^2 \ls \nu^{\f 13} \| Y^{\om''} f \|_{\mathcal D^{(\vartheta)}_{\alp,0,\om''}}^2.
\end{equation}
Hence, combining all the cases above, we obtain
\begin{equation}\label{eq:RL3.est.final}
\begin{split}
&\: \mbox{RHS of \eqref{eq:RL3.est}} \\
\ls &\:  
\begin{cases}
C_\eta \smashoperator{\sum_{\substack{|\bt'|\leq |\bt| \\ |\om'| \leq |\om|}}} (\nu^{2/3} \nu^{1/3}  \|\rd_x^{\alp} \rd_v^{\bt'} Y^{\om'} f\|_{\mathcal D^{(\vartheta)}_{\alp,\bt',\om'}}^2 +  \nu \| \phi\|_{L^\infty}\| \rd_x^\alp \rd_v^{\bt'} Y^{\om'} f\|_{\mathcal{E}^{(\vartheta)}_{\alp,\bt',\om'}}^2) &\mbox{if $|\alp| >0$} \\
 C_\eta \smashoperator{\sum_{\substack{|\bt'|\leq |\bt| \\ |\om'| \leq |\om|}}} (\nu^{2/3} \nu^{1/3}  \|\rd_x^{\alp} \rd_v^{\bt'} Y^{\om'} f\|_{\mathcal D^{(\vartheta)}_{\alp,\bt',\om'}}^2 +  \nu \| \phi\|_{L^\infty}\| \rd_x^\alp \rd_v^{\bt'} Y^{\om'} f\|_{\mathcal{E}^{(\vartheta)}_{\alp,\bt',\om'}}^2) & \\
 \qquad + C_\eta \smashoperator{\sum_{\substack{|\bt'|\leq |\bt| \\ |\om'| < |\om| }}} \nu^{1/3} \|\rd_x^{\alp} \rd_v^{\bt'} Y^{\om'} f\|_{\mathcal D^{(\vartheta)}_{\alp,\bt',\om'}}^2 + C_\eta  \nu \norm{\mu f_0}_{L^2_{x,v}}^2 & \mbox{if $|\alp| =0$}.
\end{cases}
\end{split}
\end{equation}

Putting together \eqref{eq:RL1.est}, \eqref{eq:RL2.est} and \eqref{eq:RL3.est.final}, we obtain
\begin{equation}\label{claim-RLa}
\begin{split}
&\: A_0 \nu \nu^{2|\beta|/3} \sum_{|\alpha'|\le 1} \mathcal{R}^{L,\ell-2|\alpha'|}_{\alpha+\alpha',\beta,\omega} + \nu \nu^{2|\beta|/3} \sum_{|\beta'|\le 1}\nu^{2|\beta'|/3} \mathcal{R}^{L,\ell-2|\beta'|}_{\alpha,\beta+\beta',\omega}
\\ 
\ls &\: \begin{cases}
(\eta + C_\eta \nu^{2/3}) \nu^{1/3} \smashoperator{\sum_{\substack{|\bt''|\leq |\bt|\\|\omega''|\leq|\omega|}}}\| \rd_x^\alp \rd_v^{\bt''} Y^{\om''} f\|_{\mathcal D^{(\vartheta)}_{\alp,\bt'',\om''}}^2 + C_\eta \nu^{1/3} \smashoperator{\sum_{\substack{|\bt''|\leq |\bt|\\|\omega''|\leq|\omega| \\ |\bt''| + |\om''| \leq |\bt| + |\om| -1}}} \| \rd_x^\alp \rd_v^{\bt''} Y^{\om''} f\|_{\mathcal D^{(\vartheta)}_{\alp,\bt'',\om''}}^2 & \\
\qquad + C_\eta \nu \| \phi\|_{L^\infty} \smashoperator{\sum_{\substack{|\bt''|\leq |\bt| \\ |\om''|\leq |\om|}}} \| \rd_x^\alp \rd_v^{\bt''} Y^{\om''} f\|_{\mathcal{E}^{(\vartheta)}_{\alp,\bt'',\om''}}^2 & \mbox{if $|\alp| >0$} \\
(\eta + C_\eta \nu^{2/3}) \nu^{1/3} \smashoperator{\sum_{\substack{|\bt''|\leq |\bt|\\|\omega''|\leq|\omega|}}}\| \rd_x^\alp \rd_v^{\bt''} Y^{\om''} f\|_{\mathcal D^{(\vartheta)}_{\alp,\bt'',\om''}}^2 + C_\eta \nu^{1/3} \smashoperator{\sum_{\substack{|\bt''|\leq |\bt|\\|\omega''|\leq|\omega| \\ |\bt''| + |\om''| \leq |\bt| + |\om| -1}}} \| \rd_x^\alp \rd_v^{\bt''} Y^{\om''} f\|_{\mathcal D^{(\vartheta)}_{\alp,\bt'',\om''}}^2 & \\
\qquad + C_\eta \nu \| \phi\|_{L^\infty} \smashoperator{\sum_{\substack{|\bt''|\leq |\bt| \\ |\om''|\leq |\om|}}} \| \rd_x^\alp \rd_v^{\bt''} Y^{\om''} f\|_{\mathcal{E}^{(\vartheta)}_{\alp,\bt'',\om''}}^2 + C_\eta  \nu \norm{\mu f_0}_{L^2_{x,v}}^2 & \mbox{if $|\alp| =0$}. 
\end{cases}
\end{split}
\end{equation}

 \subsection*{Estimates on $\mathcal{Z}^{L,\ell}_{\alpha,\beta,\omega}$} We now bound the remainder $\mathcal{Z}^{L,\ell}_{\alpha,\beta,\omega}$ introduced in Lemma \ref{lem-cross} and appeared on the right hand side of \eqref{est-partialEE}. 
Using 
Corollary \ref{c.L_est_upper}, 
we bound
\begin{align*}
|\int_{\R^3}&\jap{v}^{-4} w^2 e^{2(q+1)\phi} \part_v^{\beta}Y^\omega L[\part_{x_i}\part_x^\alpha f]\part_{v_i}\der f\d v|\\
&\lesssim \smashoperator{\sum_{\substack{|\beta'|\leq|\beta|\\|\omega'|\leq |\omega|}}}\norm{\part_{v_i}\der f}_{\Delta_{v}(\ell_{\alpha,\beta,\omega}-2,\vartheta)}\norm{\part_{x_i}\derv{}{'}{'}f}_{\Delta_{v}(\ell_{\alpha,\beta,\omega}-2,\vartheta)}
\end{align*}
\break
and 
\begin{align*}
|\int_{\R^3}&\jap{v}^{-4} w^2 e^{2(q+1)\phi} \part_{v_i}\part_v^{\beta}Y^\omega L[\part_x^\alpha f]\part_{x_i}\der f\d v|\\
&\lesssim \smashoperator{\sum_{\substack{|\beta'|\leq|\beta|+1\\|\omega'|\leq |\omega| }}}\norm{\part_{x_i}\der f}_{\Delta_{v}(\ell_{\alpha,\beta,\omega}-2,\vartheta)}\norm{\derv{}{'}{'}f}_{\Delta_{v}(\ell_{\alpha,\beta,\omega}-2,\vartheta)}.
\end{align*}
\break
Therefore, multiplying by $\nu^{4/3}\nu^{2|\beta|/3}$ and integrating over $x$, we get 
 \begin{equation}\label{claim-ZL}
\begin{split}
&\:  \nu^{4/3} \nu^{2|\bt|/3} \mathcal{Z}^{L,\ell-2}_{\alpha,\beta,\omega} \\
 \ls &\: \nu^{4/3}\nu^{2|\beta|/3}  \norm{\nab_v \der f}_{\Delta_{x,v}(\ell_{\alpha,\beta,\omega}-2,\vartheta)}( \smashoperator{\sum_{\substack{|\beta'|\leq|\beta|\\|\omega'|\leq |\omega| }}} \norm{\nab_x \rd_x^\alp \rd_v^{\bt'} Y^{\om'} f}_{\Delta_{x,v}(\ell_{\alpha,\beta,\omega}-2,\vartheta)}) \\
&\: + \nu^{4/3}\nu^{2|\beta|/3}  \norm{\nab_x \der f}_{\Delta_{x,v}(\ell_{\alpha,\beta,\omega}-2,\vartheta)} (\smashoperator{\sum_{\substack{|\beta'|\leq|\beta|+1\\|\omega'|\leq |\omega| }}} \norm{\rd_x^\alp \rd_v^{\bt'} Y^{\om'} f}_{\Delta_{x,v}(\ell_{\alpha,\beta,\omega}-2,\vartheta)} )\\
\lesssim &\: \nu^{1/3}A_0^{-1/2} \| \rd_x^\alp \rd_v^\bt Y^\om f\|_{\mathcal{D}^{(\vartheta)}_{\alp,\bt,\om}} \sum_{\substack{|\beta'|\leq|\beta|\\|\omega'|\leq |\omega| }} \| \rd_x^\alp \rd_v^{\bt'} Y^{\om'} f\|_{\mathcal{D}^{(\vartheta)}_{\alp,\bt',\om'}}.
\end{split}
\end{equation}

 \subsection*{Proof of \eqref{est-mainEE}}
We now put together the above estimates. The $|\alp| >0$ and $|\alp| = 0$ cases are treated slightly differently. Consider first $|\alp| >0$. Combining \eqref{est-partialEE}, \eqref{eq:remainder.alp.bt.om} with the bounds \eqref{claim-RTa}, \eqref{claim-RSa}, \eqref{claim-RLa} and \eqref{claim-ZL} for the remainder terms, 
 \begin{equation}\label{eq:EE.Eabo}
 \begin{split}
 \frac{d}{dt}& \| \partial_x^\alpha\partial_v^\beta Y^\omega f \|^2_{\mathcal{E}^{(\vartheta)}_{\alpha,\beta,\omega}}+ \theta \nu^{1/3} \| \partial_x^\alpha\partial_v^\beta Y^\omega f \|^2_{\mathcal{D}^{(\vartheta)}_{\alpha,\beta,\omega}} \\
 &\ls (\eta + C_\eta \nu^{2/3} + A_0^{-1/2}) \nu^{1/3} \smashoperator{\sum_{\substack{|\bt'|\leq |\bt|\\|\omega'|\leq|\omega|}}}\| \rd_x^\alp \rd_v^{\bt'} Y^{\om'} f\|_{\mathcal D^{(\vartheta)}_{\alp,\bt',\om'}}^2 \\
 &\: \qquad + C_\eta \nu^{1/3} \smashoperator{\sum_{\substack{|\bt'|\leq |\bt|\\|\omega'|\leq|\omega| \\ |\bt'| + |\om'| \leq |\bt| + |\om| -1}}} \| \rd_x^\alp \rd_v^{\bt'} Y^{\om'} f\|_{\mathcal D^{(\vartheta)}_{\alp,\bt'',\om''}}^2 + A_0^{1/2} \nu^{1/3}  \qquad \smashoperator{\sum_{\substack{|\alp'| \geq |\alp|,\, |\bt'| < |\bt| \\ |\alp'| + |\bt'| \leq |\alp| + |\bt|}}} \|\rd_x^{\alp'} \rd_v^{\bt'} Y^\om f\|_{\mathcal{D}^{(\vartheta)}_{\alp,\bt',\om}}^2\\
 & \qquad + C_\eta A_0 (\|\partial_t\phi \|_{L^\infty_x} + \|\phi \|_{W^{1,\infty}_x})  \smashoperator{\sum_{\substack{|\bt'|\leq |\bt|\\|\omega'|\leq|\omega|}}}\| \rd_x^\alp \rd_v^{\bt'} Y^{\om'} f \|^2_{\mathcal{E}^{(\vartheta)}_{\alp,\bt',\om'}} + \mathcal R_{\alp,\bt,\om} .
 \end{split}
 \end{equation}
 
Summing over $|\alp| \geq N_\alp^{low}$, $|\alp|+|\bt| \leq N_{\alp,\bt}$, $|\bt|\leq N_{\bt}$, $|\om| \leq N_{\om}$, and recalling \eqref{eq:combined.norms}, we obtain
 \begin{equation}
 \begin{split}
 \frac{d}{dt}& \| f \|^2_{\mathbb{E}^{(\vartheta)}_{N_\alp^{low},N_{\alp,\bt},N_\bt,N_\om}}+ \theta \nu^{1/3} \| f \|^2_{\mathbb{D}^{(\vartheta)}_{N_\alp^{low},N_{\alp,\bt},N_\bt,N_\om}} \\
 &\ls (\eta + C_\eta \nu^{2/3} + A_0^{-1/2}) \nu^{1/3} \| f\|_{\mathbb D^{(\vartheta)}_{N_\alp^{low},N_{\alp,\bt},N_\bt,N_\om}}^2 \\
 & \qquad + (C_\eta + A_0^{1/2}) \nu^{1/3} (\|  f\|_{\mathbb D^{(\vartheta)}_{N_\alp^{low},N_{\alp,\bt},N_\bt-1,N_\om}}^2 + \|  f\|_{\mathbb D^{(\vartheta)}_{N_\alp^{low},N_{\alp,\bt},N_\bt,N_\om-1}}^2) \\
 & \qquad + C_\eta A_0 (\|\partial_t\phi \|_{L^\infty_x} + \|\phi \|_{W^{1,\infty}_x})  \| f \|^2_{\mathbb E^{(\vartheta)}_{N_\alp^{low},N_{\alp,\bt},N_\bt,N_\om}} + \qquad \smashoperator{\sum_{\substack{|\alp| \geq N_\alp^{low} ,\, |\bt| \leq N_\bt \\ |\alp| + |\bt| \leq N_{\alp,\bt} ,\, |\omega|\leq N_\om}}}\mathcal R_{\alp,\bt,\om},
 \end{split}
 \end{equation}
where we have introduced the convention $\|f\|_{\mathbb D^{(\vartheta)}_{*,*,*,-1}} = 0$, etc. 
 
 Choosing first $A_0$ and $\eta$ small, and then choosing $\nu_0$ small, we can arrange $\eta + C_\eta \nu^{2/3} + A_0^{-1/2}$ to be small enough so that the first term on the RHS can be absorbed by the second term on the LHS. At this point, we fix $\eta$ and $A_0$ so that
 \begin{equation}\label{eq:EE.Eabo.2}
 \begin{split}
 \frac{d}{dt}& \| f \|^2_{\mathbb{E}^{(\vartheta)}_{N_\alp^{low},N_{\alp,\bt},N_\bt,N_\om}}+ \theta \nu^{1/3} \| f \|^2_{\mathbb{D}^{(\vartheta)}_{N_\alp^{low},N_{\alp,\bt},N_\bt,N_\om}} \\
 &\ls (C_\eta + A_0^{1/2}) \nu^{1/3} (\|  f\|_{\mathbb D^{(\vartheta)}_{N_\alp^{low},N_{\alp,\bt},N_\bt-1,N_\om}}^2 + \|  f\|_{\mathbb D^{(\vartheta)}_{N_\alp^{low},N_{\alp,\bt},N_\bt,N_\om-1}}^2) \\
 & \qquad + C_\eta A_0 (\|\partial_t\phi \|_{L^\infty_x} + \|\phi \|_{W^{1,\infty}_x})  \| f \|^2_{\mathbb E^{(\vartheta)}_{N_\alp^{low},N_{\alp,\bt},N_\bt,N_\om}} + \qquad \smashoperator{\sum_{\substack{|\alp| \geq N_\alp^{low} ,\, |\bt| \leq N_\bt \\ |\alp| + |\bt| \leq N_{\alp,\bt} ,\, |\omega|\leq N_\om}}}\mathcal R_{\alp,\bt,\om}.
 \end{split}
 \end{equation}
For fixed $N_{\alp,\bt} \geq N_{\alp}^{low} \geq 1$, we now perform an induction in $N_\bt$ and $N_\om$. The base case is $N_\bt = N_\om = 0$: since $\| f \|_{\mathbb D^{(\vartheta)}_{N_\alp^{low}, N_{\alp,\bt}, -1, N_\om}} = \| f \|_{\mathbb D^{(\vartheta)}_{N_\alp^{low}, N_{\alp,\bt}, N_\bt, -1}} = 0$, the desired conclusion is immediate. A simple induction, say, first in $N_\om$, and then in $N_\bt$, finishes the proof of the proposition in the case $N_\alp^{low} >0$.

 \subsection*{Proof of \eqref{est-mainEE.0}} Finally, we consider the case $N_\alp^{low} = 0$. Notice that when repeating the argument in the proof of \eqref{est-mainEE}, the only difference is that we obtain an extra term $\nu \norm{\mu f_{=0}}_{L^2_{x,v}}^2$ (coming from \eqref{eq:RL3.est.final} in the $|\alp| = 0$ case). Thus
  \begin{equation}\label{est-mainEE.0.prelim}
 \begin{split}
 \frac{d}{dt}& \| f \|^2_{\mathbb{E}^{(\vartheta)}_{0,N_{\alp,\bt},N_\bt,N_\om}}+ \theta \nu^{1/3} \| f \|^2_{\mathbb{D}^{(\vartheta)}_{0,N_{\alp,\bt},N_\bt,N_\om}} \\
 &\ls  \nu \norm{\mu f_{=0}}_{L^2_{x,v}}^2 + (\|\partial_t\phi \|_{L^\infty_x} + \|\phi \|_{W^{1,\infty}_x})  \| f \|^2_{\mathbb E^{(\vartheta)}_{0,N_{\alp,\bt},N_\bt,N_\om}} + \qquad \smashoperator{\sum_{\substack{|\bt| \leq N_\bt \\ |\alp| + |\bt| \leq N_{\alp,\bt} ,\, |\omega|\leq N_\om}}}\mathcal R_{\alp,\bt,\om}.
 \end{split}
 \end{equation}
 
To proceed, we write $f(t,x,v) = a(t,x)\sqrt{\mu} + b_j(t,x) v_j \sqrt{\mu} + c(t,x)|v|^2 \sqrt{\mu}  + (I - \Pi) f$, where $I$ is the identity, and $\Pi$ is the projection as in Lemma~\ref{lem:lowest.order.positivity}. Repeating now the basic energy estimate in Lemma~\ref{est-basic} with $\ell = \vartheta = 0$, we obtain 
$$ \f 12 \frac{d}{dt} \| e^{\phi} f\|^2_{L^2_{x,v}}  + \nu \iint_{\mathbb T^3\times \R^3} e^\phi f L (e^\phi f) \, \ud v\, \ud x \ls \mathcal R^{T,0}_{0,0,0} + \mathcal R^{Q,0}_{0,0,0}.$$
Applying Lemma~\ref{lem:lowest.order.positivity} for $\la e^\phi f, L(e^\phi f) \ra$ and controlling $\mathcal R^{T,0}_{0,0,0}$ by \eqref{claim-RTa.prelim}, we thus obtain 
\begin{equation}\label{eq:lowest.order.EE}
 \frac{d}{dt} \| e^{\phi} f\|^2_{L^2_{x,v}}  + \de^2 \nu \| e^{\phi} (I - \Pi)f\|^2_{\Delta_{x,v}} \ls (\|\partial_t\phi \|_{L^\infty_x} + \|E\|_{L^\infty_x})\|  f \|^2_{\mathcal{E}^{(0)}_{0,0,0}} + \mathcal R^{Q,0}_{0,0,0}.
\end{equation}
Notice that \eqref{eq:lowest.order.EE} implies 
\begin{equation}
\begin{split}
\nu \|\mu f_{=0} \|_{L^2_{x,v}}^2 \ls &\: \nu |(\bar{a},\bar{b},\bar{c})|^2 + \nu \| e^{\phi} (I - \Pi) f \|^2_{\Delta_{x,v}} + \nu \|\phi\|_{L^\i_x} \|\mu f_{=0}\|_{L^2_{x,v}}^2 \\
\ls &\: \nu |(\bar{a},\bar{b},\bar{c})|^2 + (\|\partial_t\phi \|_{L^\infty_x} + \|\phi \|_{L^\infty_x})\|  f \|^2_{\mathbb{E}^{(0)}_{0,N_{\alp,\bt},N_\bt,N_\om}} + \mathcal R^{Q,0}_{0,0,0}.
\end{split}
\end{equation} 
Plugging this into \eqref{est-mainEE.0.prelim} yields the desired conclusion.
 \qedhere
 
\end{proof}

\subsection{The main energy estimates including the top-order energy}

\begin{proposition}\label{prop:top.order.energy}
Fix $\vartheta \in \{0,2\}$. The estimates \eqref{est-mainEE} and \eqref{est-mainEE.0} in Proposition~\ref{prop-mainEE} both hold (for $N_{\alp}^{low} >0$ and $N_{\alp}^{low} = 0$ respectively) with $\mathbb E^{(\vartheta)}_{N_\alp^{low},N_{\alp,\bt},N_\bt,N_\om}$, $\mathbb D^{(\vartheta)}_{N_\alp^{low},N_{\alp,\bt},N_\bt,N_\om}$ and $\mathcal R_{\alp,\bt,\om}$ replaced by $\widetilde{\mathbb E}^{(\vartheta)}_{N_\alp^{low},N_{\alp,\bt},N_\bt,N_\om}$, $\widetilde{\mathbb D}^{(\vartheta)}_{N_\alp^{low},N_{\alp,\bt},N_\bt,N_\om}$ and $\widetilde{\mathcal R}_{\alp,\bt,\om}$ respectively, where $\widetilde{\mathbb E}^{(\vartheta)}_{N_\alp^{low},N_{\alp,\bt},N_\bt,N_\om}$ and $\widetilde{\mathbb D}^{(\vartheta)}_{N_\alp^{low},N_{\alp,\bt},N_\bt,N_\om}$ are as in \eqref{eq:combined.norms.top}, and 
$$\widetilde{\mathcal R}_{\alp,\bt,\om} = {\mathcal R}_{\alp,\bt,\om} + \sum_{|\beta'| = 2}\nu^{2(|\beta|+|\beta'|)/3} \mathcal{R}^{Q,\ell-2|\beta'|}_{\alpha,\beta+\beta',\omega},
$$
where $\mathcal{R}^{Q,\ell}_{\alpha,\beta,\omega}$ and ${\mathcal R}_{\alp,\bt,\om}$ are as introduced in Lemma~\ref{lem-basicEEdx} and Proposition~\ref{prop-mainEE}.
\end{proposition}
\begin{proof}
We repeat the argument in Proposition~\ref{prop-mainEE}, except that we now derive the energy estimates for $\wtE_{\alp,\bt,\om}$ and $\wtD_{\alp,\bt,\om}$ instead of $\mathcal E_{\alp,\bt,\om}$ and $\mathcal D_{\alp,\bt,\om}$. For this, we need to handle the additional terms
$$A_0^{-1} \nu^{2|\beta|/3}\sum_{|\beta'|= 2} \nu^{2|\beta'|/3} \Big[  \mathcal{R}^{T,\ell-2|\beta'|}_{\alpha,\beta+\beta',\omega} + \nu \mathcal{R}^{L,\ell-2|\beta'|}_{\alpha,\beta+\beta',\omega}+  \mathcal{R}^{Q,\ell-2|\beta'|}_{\alpha,\beta+\beta',\omega}\Big].$$

Now the $\sum_{|\beta'|= 2} \mathcal{R}^{Q,\ell-2|\beta'|}_{\alpha,\beta+\beta',\omega}$ term is part of $\widetilde{\mathcal R}_{\alp,\bt,\om}$ and does not need to be estimated for the purpose of this proposition. 

As for the other two terms, notice that while they contain one more $\rd_v$ derivative compared to their counterparts in Proposition~\ref{prop-mainEE}, the norms $\wtE_{\alp,\bt,\om}$ and $\wtD_{\alp,\bt,\om}$ also control the additional terms as indicated in \eqref{eq:wtEabo} and \eqref{eq:wtDabo}. It can be checked that the same energy estimates as in Proposition~\ref{prop-mainEE} can be obtained, as long as the $(\mathcal E, \mathcal D)$ norms are replaced by the $(\wtE, \wtD)$. We only consider in detail the following term from $\mathcal{R}^{T,\ell-2|\beta'|}_{\alpha,\beta+\beta',\omega}$ which requires modifications that are not completely obvious:
\begin{equation*}
\begin{split}
&\: \nu^{(2|\bt|+4)/3 } \sum_{|\bt'| = 2}\sum_{\substack{|\beta''|=|\beta|-1\\|\beta'''|=1}} \| e^{(q+1)\phi}  \partial_x^{\alpha+\beta'''}\partial_v^{\beta''+\bt'} Y^\omega f\|_{L^2_{x,v}(\ell_{\alp,\bt,\om}-4,\vartheta)} \| e^{(q+1)\phi}  \partial_x^\alpha\partial_v^{\beta+\bt'} Y^\omega f\|_{L^2_{x,v}(\ell_{\alp,\bt,\om}-4,\vartheta)} \\
\ls &\: \nu^{1/3} (\nu^{|\bt''|/3+2/3}  \smashoperator{\sum_{\substack{|\alp''| = |\alp| + 1 \\ |\bt'| = 2, |\bt''| = |\bt| -1}}} \| e^{(q+1)\phi} \rd_v^{\bt'} \partial_x^{\alp''}\partial_v^{\beta''} Y^\omega f\|_{L^2_{x,v}(\ell_{\alp,\bt,\om}-4,\vartheta)})\\
&\: \quad \times (\nu^{|\bt|/3+2/3} \smashoperator{\sum_{|\bt'| =2}} \|e^{(q+1)\phi}  \partial_x^\alpha\partial_v^{\beta+\bt'} Y^\omega f\|_{L^2_{x,v}(\ell_{\alp,\bt,\om}-4,\vartheta)}) \\
\ls &\: \nu^{1/3} 
(\nu^{|\bt''|/3+2/3}  \smashoperator{\sum_{\substack{|\alp''| = |\alp| + 1 \\ |\bt''| = |\bt| -1}}} \| e^{(q+1)\phi} \nab_v \partial_x^{\alp''}\partial_v^{\beta''} Y^\omega f\|_{\Delta_{x,v}(\ell_{\alp,\bt,\om}-2,\vartheta)})\\
&\: \quad \times (\nu^{|\bt|/3+2/3} \sum_{\substack{|\bt''| = |\bt| - 1 \\ |\bt'''| = 2}} \|e^{(q+1)\phi} \rd_v^{\bt'''} \partial_x^\alpha\partial_v^{\beta''} Y^\omega f\|_{\Delta_{x,v}(\ell_{\alp,\bt'',\om}-4,\vartheta)}) \\
\ls &\: \nu^{1/3} ( \smashoperator{\sum_{\substack{|\alp''| = |\alp| + 1 \\  |\bt''| = |\bt| -1}}} \|  \partial_x^{\alp''}\partial_v^{\beta''} Y^\omega f\|_{\wtD^{(\vartheta)}_{\alp'',\bt'',\om}} )\times (A_0^{1/2} \sum_{|\bt''| = |\bt| -1} \|  \partial_x^{\alp}\partial_v^{\beta''} Y^\omega f\|_{\wtD^{(\vartheta)}_{\alp,\bt'',\om}} ).
\end{split}
\end{equation*}

As a result, we can then complete the argument following the proof of Proposition~\ref{prop-mainEE}. \qedhere
\end{proof}

\section{Linear Landau equation}\label{sec-linearLandau}

In this section, we derive estimates on the semigroup of the linear Landau equation
\begin{equation}\label{lin-Landau}
\begin{aligned}
\part_{t} f & + v\cdot \nabla_x f  + \nu L f  = 0
\end{aligned}
\end{equation}
on $\T^3 \times \R^3$, with initial data $f(0,x,v) = f_0(x,v)$, where $L$ denotes the leading linear Landau operator as in \eqref{def-L}, \eqref{eq:A.def} and \eqref{eq:K.def}. Let $S(t)$ be the semigroup associated to \eqref{lin-Landau}, that is, for each $f_0(x,v)$, we set 
\begin{equation}\label{def-SLandau}
S(t)[f_0](x,v) := f(t,x,v)
\end{equation}
where $f(t,x,v)$ is the unique solution to \eqref{lin-Landau} with initial data $f_0(x,v)$. As $L$ is independent of $x$, the problem \eqref{lin-Landau} can be solved via the Fourier transform. Indeed, we can write 
$$ S(t)[f_0](x,v) = \sum_{k\in \Z^3} e^{ik\cdot x} S_k(t)[\hat{f}_{0k}](v) $$
in which, for each $k \in \Z^3$, $\hat{f}_{0k}(v)$ is the Fourier transform of $f(x,v)$ in variable $x$, and $S_k(t)[h_0]$ denotes the corresponding semigroup to the Fourier transform of \eqref{lin-Landau}: namely, $h(t) = S_k(t)[h_0] $ solves the following fixed mode linear Landau equation
\begin{equation}\label{lin-LandauF}
\begin{aligned}
\part_{t} h & + i k \cdot v h  +  \nu L h  = 0
\end{aligned}
\end{equation}
with initial data $h(0,v) = h_{0}(v)$. 

This section is devoted to deriving estimates for $\Big| \int_{\R^3} h \sqrt{\mu}\, \ud v\Big|$. We will prove both finite time bounds (Proposition~\ref{prop:Landau.finite.time.convergence}) and decay estimates (Proposition~\ref{prop-mixdecay}). We prove two types of decay estimates:
\begin{itemize}

\item Uniform phase mixing: decay in the variable $\la k t\ra$, uniformly in $\nu \ge 0$. 

\item Enhanced dissipation: decay in the variable $\la \nu^{1/3} t\ra$.

\end{itemize}

The precise decay estimates can be found in Proposition~\ref{prop-mixdecay} below. When $\nu =0$,  \eqref{lin-LandauF} becomes the free transport equation, whose semigroup reads $S_k(t)[h] = e^{-ikt \cdot v} h$. In that case the decay estimates in $\la k t\ra$ are thus direct. We shall prove the phase mixing for the linear Landau equations \eqref{lin-Landau} uniformly in $\nu \ge 0$. (In fact, we also prove a ``twisted'' estimate with decay in $\la kt + \eta \ra$ for $\eta \in \R^3$, which will be useful in the nonlinear density estimate.)

Next, using methods of \cite{Guo12}, it follows that the Landau diffusion dissipates energy at least at a rate of order $e^{-\de(\nu t)^{2/3}}$, which in particular becomes relevant at time of order $1/\nu$. Making use of the transport-diffusion structure of the Landau operator, we shall prove the enhanced dissipation in $\la \nu^{1/3} t\ra$, which takes place at a much earlier time of order $\nu^{-1/3}$, as $\nu$ is sufficiently small.

\subsection{Phase mixing and vector field bounds}

In this subsection, we prove that control of $Y_{k,\eta}$ derivatives (defined below) implies decay estimates for velocity averages.

\begin{proposition}\label{prop:Y.to.decay}
For $k \in \Z^3$, $\eta \in \R^3$, set $Y_{k,\eta} = \nab_v + i (\eta + kt)$, and $g:\R^3 \to \R$. Then, for any $N\ge 0$ and any $\ell' \geq 0$, there is a positive constant $C_{N,\ell'}$ so that 
\begin{equation}
\left| \int_{\R^3}  g \sqrt{\mu} \, \ud v \right| \leq C_{N,\ell'} \la kt+ \eta \ra^{-N} \sum_{|\om| \leq N} \|\bv^{-\ell'} Y^\omega_{k,\eta} g \|_{L^2(\R^3)}.
\end{equation}
\end{proposition}
\begin{proof}
If $|kt + \eta|\leq 1$, the desired estimate follows directly from the Cauchy--Schwarz inequality. Suppose that $|kt + \eta| > 1$. Take $j$ such that $|k_j t + \eta_j| \geq \f 1{\sqrt{3}} |kt + \eta|$. Then, writing $i(k_j t+ \eta_j) = Y_{k_j,\eta_j} - \partial_{v_j}$, we bound 
\begin{equation*}
\begin{split}
&\: |kt + \eta |^N \left| \int_{\R^3}  g \sqrt{\mu} \, \ud v \right| \leq 3^{\f N2} \left| \int_{\R^3}  (k_j t+ \eta_j)^N g \sqrt{\mu} \, \ud v \right| \\
\ls_N &\:  \sum_{N_1+ N_2 = N} \left| \int_{\R^3}  Y_{k_j, \eta_j}^{N_1} \rd_{v_j}^{N_2} g \sqrt{\mu} \, \ud v \right| \ls_{N,\ell'} \sum_{|\om| \leq N} \|\bv^{-\ell'} Y^\omega_{k,\eta} g \|_{L^2(\RR^3)},
\end{split}
\end{equation*}
where the final inequality is achieved by integrating by parts $N_2$ times in $\rd_{v_j}$.
\end{proof}

\subsection{Enhanced dissipation}

In this subsection, we prove the following enhanced dissipation estimates for the linear Landau equation \eqref{lin-Landau}, which are a direct consequence of energy estimates.  

\begin{proposition}\label{prop-enhancedSk} For $k\in \Z^3 \setminus \{0\}$ and $\eta \in \R^3$, set $Y_{k,\eta} = \nabla_v + i(\eta + k t)$ and let $S_k(t)$ be the semigroup of \eqref{lin-LandauF}. Then, there exists $\de' >0$ so that

\begin{equation}\label{decay-nuSk.poly} 
\begin{aligned}
\| Y^\omega_{k,\eta}S_k(t) [h_0] \|_{L^2_{v}}
\ls \la \nu^{1/3} t \ra^{-3/2}  \| h_0 \|_{\mathbb{E}^{(10,0)'}_{Landau, k,\eta,|\om|}}
\end{aligned}
\end{equation}
and
\begin{equation}\label{decay-nuSk.exp} 
\begin{aligned}
\| S_k(t) [h_0] \|_{L^2_{v}}
\ls \min\{ e^{-\de'(\nu^{1/3} t)^{1/3}},\, e^{-\de'(\nu t)^{2/3}} \} \| h_0 \|_{\mathbb{E}^{(2,2)'}_{Landau, k,\eta,0}}
\end{aligned}
\end{equation}
uniformly in $k \in  \Z^3\backslash \{0\}$, $\eta \in \R^3$, and $\nu\ge 0$, where for $\ell_* \in \R$ and $\vartheta \in \{0,2\}$, the linear Landau energy norm $ \| h_0 \|_{\mathbb{E}^{(\ell_*,\vartheta)'}_{Landau,k,\eta,|\om|}}$ is defined by 
\begin{equation}\label{def-ENq} 
\begin{split}
&\: \| h_0 \|_{\mathbb{E}^{(\ell_*, \vartheta)'}_{Landau, k,\eta,N}} \\
: =  &\: \|  \langle v\rangle^{2\ell_*} e^{\f{q' |v|^{\vartheta}}2}  h_0 \|_{L^2_v} + \sum_{1\leq |\omega|\le N} \|  \langle v\rangle^{2\ell_*} Y^{\omega}_{0,\eta} h_0 \|_{L^2_v} \\
&\: + \nu^{1/3} |k|^{-1} \sum_{|\bt'| = 1} \Big( \|   \bv^{2\ell_*-2} e^{\f{q' |v|^{\vartheta}}2} \partial_v^{\beta'} h_0 \|_{L^2_v} +  \sum_{1\leq |\omega|\le N} \|   \bv^{2\ell_*-2}  \partial_v^{\beta'} Y^{\omega}_{0,\eta} h_0 \|_{L^2_v}\Big), 
\end{split}
\end{equation}
and $q'$ is defined by $q' = \begin{cases} \frac12 q_0 & \mbox{if $\vartheta = 2$} \\ 0 & \mbox{if $\vartheta = 0$}\end{cases}$. 
\end{proposition}

\begin{remark}\label{rem-slowweight}
Note that $q' = \frac12 q$ with $q$ defined as in Section~\ref{sec-EE}. That is, the linear Landau energy norm $\| \cdot \|_{\mathbb{E}^{(\ell_*, \vartheta)'}_{Landau, k,\eta,N}} $ has slower Gaussian $v$-weights than do the corresponding energy and dissipation norms. In addition, it involves precisely the stationary vector field $Y_{0,\eta} = \nabla_v + i\eta$  (i.e.~independent of $t$).   
\end{remark}

\begin{proof} \textbf{Basic energy estimates.} Let $h(t) = S_k(t) [h_0]$. We note that $h(t)$ solves the linear Landau equation \eqref{lin-LandauF} with initial data $h_0$. As \eqref{lin-Landau} is a particular version of the full Landau equation \eqref{eq:Vlasov.f} without the electric field and nonlinear terms, we can thus apply to \eqref{lin-LandauF} the same energy estimates developed in Proposition~\ref{prop-mainEE} for $N_{\alp}^{low} = N_{\alp,\bt} = N_{\bt}  = 0$. Indeed, we claim that
\begin{equation}\label{eq:linear.Landau.main.energy}
\frac{d}{dt} \| h(t)\|_{\mathbb{E}_{0,0,0,N_{\om}}^{(\ell_*,\vartheta)'}}^2 + \theta \nu^{1/3} \| h(t)\|_{\mathbb{D}_{0,0,0,N_\om}^{(\ell_*,\vartheta)'}}^2 \le 0. 
\end{equation}
for any $\ell_* \in \R$  and $\vartheta \in \{0,2\}$. Here, in \eqref{eq:linear.Landau.main.energy}, the energy norm $\| h(t) \|_{\mathbb{E}_{0,0,0,N_{\om}}^{(\ell_*,\vartheta)'}}$ and the dissipation norm $\| h(t) \|_{\mathbb{D}_{0,0,0,N_{\om}}^{(\ell_*,\vartheta)'}}$ are defined by
$$
\begin{aligned} 
\| h(t) \|_{\mathbb{E}_{0,0,0,N_\om}^{(\ell_*,\vartheta)'}}^2 &=  \| h(t) \|_{\mathcal{E}_{0,0,0}^{(\ell_*,\vartheta)'}}^2 + \sum_{1\leq |\omega| \le N}  \| Y^\omega_{k,\eta}h(t) \|_{\mathcal{E}_{0,0,\omega}^{(\ell_*,0)'}}^2 
, \\
\| h(t) \|_{\mathbb{D}_{0,0,0,N_\om}^{(\ell_*,\vartheta)'}}^2 &= \| h(t) \|_{\mathcal{D}_{0,0,0}^{(\ell_*,\vartheta)'}}^2 + \sum_{1\leq  |\omega|\le N}  \| Y^\omega_{k,\eta} h(t) \|_{\mathcal{D}_{0,0,\omega}^{(\ell_*,0)'}}^2 ,
\end{aligned}$$ 
where for $H = Y^\omega_{k,\eta}h$, we set 
$$
\begin{aligned} 
\| H \|^2_{\mathcal{E}^{(\ell_*,\vartheta)'}_{0,0,\omega}} &= A_0\smashoperator{\sum_{|\alpha'|\le 1}}\| k^{\alpha'} H \|^2_{L^2_{v}(\ell_*-2|\alpha'|,\vartheta)'} + \nu^{1/3} \Re \int_{\R^3} i w^2 k \cdot (\nab_v H) \bar{H} \, \ud v + \nu^{2/3} \| \nab_v H \|^2_{L^2_{v}(\ell_*-2,\vartheta)'},
\\
\| H \|^2_{\mathcal{D}_{0,0,\omega}^{(\ell_*,\vartheta)'}} &= A_0 \nu^{2/3} \sum_{|\alpha'|\le 1}\| k^{\alpha'} H \|^2_{\Delta_{v}(\ell_* - 2|\alpha'|,\vartheta)'} + \| |k| H \|^2_{L^2_{v}(\ell_*-2,\vartheta)} 
+ \nu^{4/3} \| \nab_v H \|^2_{\Delta_{v}(\ell_*-2,\vartheta)'},
\end{aligned}$$ 
for $L^2_v(\ell, \vartheta)'$, $\Delta_v(\ell,\vartheta)'$ as in \eqref{eq:Lp.'.1}--\eqref{eq:Lp.'.2}, $A_0$ as in \eqref{eq:Eabo}--\eqref{eq:Dabo}, $w = \bv^{\ell_*-2} e^{\f{q'|v|^{\vartheta}}{2}}$, with $q' = \begin{cases} \frac12 q_0 & \mbox{if $\vartheta = 2$} \\ 0 & \mbox{if $\vartheta = 0$}\end{cases}$. Observe that these are exactly the norms in Proposition~\ref{prop-mainEE} adapted to the current setting, except with (1) $\phi \equiv 0$, (2) $Y$ is replaced by $Y_{k,\eta}$, (3) $2M$ replaced by $\ell_*$, (4) the Gaussian weights $e^{\f{q|v|^\vartheta}2}$ are replaced by $e^{\f{q'|v|^\vartheta}2}$, and (5) the polynomial $\bv$ weights depend only on the $k$ weights and $\rd_v$ derivatives, but not the $Y_{k,\eta}$ derivatives.

Now make the following observations:
\begin{itemize}
\item $Y_{k,0} = \nabla_v + ikt$ corresponds to the vector field $Y = \nabla_v + t \nabla_x$ in the physical space. Hence, the energy estimates for $Y_{k,0}$ are a Fourier transformed version of those in Section \ref{sec-EE}. Now, for $\eta \not =0$, we observe that the commutator of $Y_{k,\eta}$ with the linear Landau equation is identical to that of $Y_{k,0}$.
\item The argument in Proposition~\ref{prop-mainEE} goes through identically with $q$ replaced by $q'$ in the Gaussian $v$-weights, and with $\ell_{\alp,\bt,\om} = \ell_* - 2|\alp| - 2|\bt|$. (The choice of $\ell_{\alp,\bt,\om} = 2M - 2|\alp| -2|\bt| - 2|\om|$ will only be relevant in Section~\ref{s.closing_eng}; see for instance Lemma~\ref{lem-EdvY}.)
\end{itemize}
Therefore, the estimate \eqref{eq:linear.Landau.main.energy} can be obtained as in Proposition~\ref{prop-mainEE}, read off specifically for the linear Landau equations.

\subsection*{Polynomial decay.}

Define
\begin{equation}\label{eq:decay.fixed.mode.g}
g^2(t,v) = \sum_{|\om|\leq N} (A_0 \sum_{|\alpha'|\le 1} \bv^{4 - 4|\alp'|} |k^{\alp'} Y^\om_{k,\eta} h |^2 + \nu^{1/3} \Re (i k\cdot (\nab_j Y^\om_{k,\eta} h) \bar{Y^\om_{k,\eta} h} ) + \nu^{2/3} |\nab_v Y^\om_{k,\eta} h |^2)
\end{equation}
so that $\| h\|^2_{\mathbb E^{(2,0)'}_{0,0,0,N}} = \int_{\R^3} g^2 \, \ud v$. Notice that (using in particular $\|\langle v\rangle^{-3/2} (\nab_v \cdot) \|_{L_v^2} \le \| \cdot \|_{\Delta_{v}}$)
\begin{equation}\label{eq:D.lower.bound.fixed.mode}
\| h\|_{\mathbb D_{0,0,0,N}^{(2,0)'}}^2 \gtrsim \sum_{|\om|\leq N} (|k|^2 \| Y^\om_{k,\eta} h\|_{L^2_v(0,0)}^2 + \nu^{2/3}\|\bv^{-3/2} \nab_v Y^\om_{k,\eta} h\|_{L^2_v(0,0)}^2) \gtrsim \int_{\R^3} \la v\ra^{-4} g^2\, \ud v.
\end{equation}
That is, using \eqref{eq:linear.Landau.main.energy} with $(\ell_*, \vartheta)=(2,0)$, we get 
\begin{equation}\label{eq:linear.Landau.main.energy-g}
\frac{d}{dt}  \int_{\R^3} g^2 \, \ud v + \theta' \nu^{1/3} \int_{\R^3} \la v\ra^{-4} g^2\, \ud v\le 0,
\end{equation}
for some positive constant $\theta'$. Moreover, using \eqref{eq:linear.Landau.main.energy} again, we have 
$$\sup_{t\in [0,\infty)}\int_{\R^3} \bv^{16} g^2(t,v)\, \ud v \ls \sup_{t\in [0,\infty)}\| h(t) \|_{\mathbb{E}_{0,0,0,N}^{(10,0)'}}^2 \ls \| h(0) \|_{\mathbb{E}_{0,0,0,N}^{(10,0)}}^2 \ls |k|^2\| h_0 \|_{\mathbb{E}^{(10,0)'}_{Landau, k,\eta,N}}^2,$$
recalling the definition of $\|\cdot \|_{\mathbb{E}^{(10,0)'}_{Landau, k,\eta,N}}$ in \eqref{def-ENq} and noting $Y_{k,\eta}h(0) = Y_{0,\eta} h_0$. Therefore, applying Lemma~\ref{lem:SG.poly} to \eqref{eq:linear.Landau.main.energy-g} (with $\mfc \gtrsim \nu^{1/3}$, $\mfC \ls |k|^2\| h_0 \|_{\mathbb{E}^{(10,0)'}_{Landau, k,\eta,N}}^2$ and $\mfm = 4$), we obtain $\int_{\R^3} g^2(t,v)\, \ud v \ls \la \nu^{\f 13} t\ra^{-3} |k|^2 \| h_0 \|_{\mathbb{E}^{(10,0)'}_{Landau, k,\eta,N}}^2$. Noticing that $g^2 \gtrsim |k|^2 |Y^\om_{k,\eta} h|^2$ and dividing by $|k|^2$, we obtain \eqref{decay-nuSk.poly}.


\subsection*{Stretched exponential decay.} 

This requires only little modification from the previous case, except that we need to prove both $e^{-\de'(\nu^{1/3}t)^{1/3}}$ and $e^{-\de'(\nu t)^{2/3}}$ decay. 

Define $g$ as in \eqref{eq:decay.fixed.mode.g} but only for $N=0$, i.e.
$$g^2(t,v) = A_0 \sum_{|\alpha'|\le 1} \bv^{4 - 4|\alp'|} |k^{\alp'}  h |^2 + \nu^{1/3} \Re (i k\cdot (\nab_j h) \bar{ h} ) + \nu^{2/3} |\nab_v h |^2.$$
The equation \eqref{eq:D.lower.bound.fixed.mode} holds in the particular case $N=0$. This time, moreover, the initial bound on $\| h_0 \|_{\mathbb{E}^{(2,2)'}_{Landau, k,\eta,0}}
$ and \eqref{eq:linear.Landau.main.energy} (with $(\ell_*,\vartheta) = (2,2)$) give uniform in $t$ bounds for the Gaussian moments for $g^2$, i.e.
$$\int_{\R^3} e^{\frac12 q_0|v|^2} g^2 \, \d v \ls |k|^2 \| h_0 \|_{\mathbb{E}^{(2,2)'}_{Landau, k,\eta,0}}.$$ 

Therefore, 
applying Lemma~\ref{lem:SG} to \eqref{eq:linear.Landau.main.energy-g} (with $\mfc \gtrsim \nu^{1/3}$, $\mfC \ls |k|^2 \|h\|^2_{\mathbb{E}^{(2,2)'}_{Landau, k,\eta,0}}$, $\mfm = 4$), we obtain 
$$|k|^2\| h\|^2_{L^2_v}(t) \ls \int_{\R^3} g^2 \, \d v \ls e^{-\de (\nu^{1/3} t)^{1/3}} |k|^2 \|h\|^2_{{\mathbb{E}^{(2,2)'}_{Landau, k,\eta,0}}}.$$

Finally, to obtain the other, i.e.~the $e^{-\de(\nu t)^{2/3}}$, stretched exponential decay, note that we have (using $\|\langle v\rangle^{-1/2}  \cdot \|_{L_v^2} \le \| \cdot \|_{\Delta_{v}}$) the following bound, in addition to \eqref{eq:D.lower.bound.fixed.mode}:
$$\| h\|_{\mathbb D_{0,0,0,0}^{(2,0)'}}^2 \gtrsim \nu^{2/3}[\sum_{|\alp'|\leq 1} |k|^{2|\alp'|} \| h\|_{L^2_v(3/2-2|\alp'|,0)}^2 + \nu^{2/3}\|\bv^{-1/2} \nab_v h\|_{L^2_v(0,0)}^2 \gtrsim \nu^{2/3} \int_{\R^3} \la v\ra^{-1} g^2\, \ud v. $$
Remark that this features both the improved $\la v\ra^{-1}$ weight and the extra $\nu^{2/3}$ factor when compared  to \eqref{eq:D.lower.bound.fixed.mode}. Thus, an application of \eqref{eq:linear.Landau.main.energy} (with $(\ell_*,\vartheta) = (2,0)$ and $N_\om = 0$) yields 
\begin{equation*}
\frac{d}{dt}  \int_{\R^3} g^2 \, \ud v + \theta' \nu \int_{\R^3} \la v\ra^{-1} g^2\, \ud v\le 0,
\end{equation*}
for some $\th' >0$. Using Lemma~\ref{lem:SG} (with $\mfc \gtrsim \nu$, $\mfC \ls |k|^2 \|h\|^2_{\mathbb{E}^{(2,2)'}_{Landau, k,\eta,0}}$, $\mfm = 1$) thus gives
$$|k|^2\| h\|^2_{L^2_v}(t) \ls \int_{\R^3} g^2 \, \d v \ls e^{-\de (\nu t)^{2/3}} |k|^2 \|h\|^2_{{\mathbb{E}^{(2,2)'}_{Landau, k,\eta,0}}}.$$

Combining the two stretched exponential decay estimates above, and dividing by $|k|^2$, yield \eqref{decay-nuSk.exp}. \qedhere

\end{proof}

\subsection{Mixed decay estimates}

In the nonlinear analysis, we also need the following proposition, which is a direct combination of Proposition~\ref{prop:Y.to.decay} and Proposition \ref{prop-enhancedSk}. 

\begin{proposition}\label{prop-mixdecay} Fix $k\in \Z^3$ and $\eta \in \R^3$, and let $S_k(t)$ be the semigroup of \eqref{lin-LandauF}. Then, for any $N \ge 0$, there exist $C_N >0$ and $\de_N >0$ such that
\begin{equation}\label{decay-combineSk} \left| \int_{\mathbb R^3} S_k(t) [h_0] \sqrt{\mu} \, \ud v\right|  \le C_{N} \langle kt + \eta  \rangle^{-N} \la \nu^{1/3} t\ra^{-3/2}  \| h_0 \|_{\mathbb{E}_{Landau,k,\eta,N}^{(10,0)'}}, 
\end{equation}
\begin{equation}\label{decay-combineSk.low.nu.decay}
\left| \int_{\mathbb R^3} S_k(t) [h_0] \sqrt{\mu} \, \ud v\right|  \le C_0 \min \{ e^{-\de_0(\nu^{1/3} t)^{1/3}}, e^{-\de_0 (\nu t)^{2/3}} \}  \| h_0 \|_{\mathbb{E}^{(2,2)'}_{Landau,k,\eta,0}},
\end{equation}
and
\begin{equation}\label{decay-combineSk.nu.decay} 
\begin{aligned}
&\left| \int_{\mathbb R^3} S_k(t) [h_0] \sqrt{\mu} \, \ud v\right|  
\\&\le C_{N} \langle kt + \eta  \rangle^{-N} \min \{ e^{-\de_N (\nu^{1/3} t)^{1/3}}, e^{-\de_N(\nu t)^{2/3}} \} 
\Big[  \| h_0 \|_{\mathbb{E}_{Landau,k,\eta,N+1}^{(10,0)'}} + \| h_0 \|_{\mathbb{E}_{Landau,k,\eta,0}^{(2,2)'}} \Big],
\end{aligned}\end{equation}
uniformly in $k \in  \Z^3\backslash \{0\}$, $\eta \in \R^3$, and $\nu\ge 0$, where the norm $\|\cdot \|_{\mathbb{E}^{(\ell_*,\vartheta)'}_{Landau,k,\eta,N}}$ is defined as in Proposition~\ref{prop-enhancedSk}.
\end{proposition}
\begin{proof} 
Let $h(t) = S_k(t) h_0$. We combine Propositions~\ref{prop:Y.to.decay} and \ref{prop-enhancedSk} to obtain
$$\left| \int_{\mathbb R^3} h(t) \sqrt{\mu} \, \ud v\right| \ls \la kt+\eta \ra^{-N} \sum_{|\om|\leq N} \| Y^\om_{k,\eta} h \|_{L^2_v} \ls \langle kt + \eta  \rangle^{-N} \la \nu^{1/3} t\ra^{-3/2}  \| h_0 \|_{\mathbb{E}_{Landau,k,\eta,N}^{(10,0)'}}.$$
The proof of the exponential decay in \eqref{decay-combineSk.low.nu.decay} is similar (there recalling $N=0$). For \eqref{decay-combineSk.nu.decay} with $N>0$, we use Propositions~\ref{prop:Y.to.decay}, \ref{prop-enhancedSk} and interpolate (using Plancherel's theorem and H\"older's inequality), namely
\begin{equation*}
\begin{split}
&\: \left| \int_{\mathbb R^3} h(t) \sqrt{\mu} \, \ud v\right| \ls \la kt+\eta \ra^{-N} \sum_{|\om|\leq N} \| Y^\om_{k,\eta} h \|_{L^2_v} \\
\ls &\: \la kt+\eta \ra^{-N} (\sum_{|\om|\leq N+1} \| Y^\om_{k,\eta} h \|_{L^2_v})^{\f{N}{N+1}} \|h(t)\|_{L^2_v}^{\f 1{N+1}} \\
\ls &\: \langle kt + \eta  \rangle^{-N} \min \{ e^{-\de_N(\nu^{1/3} t)^{1/3}}, e^{-\de_N(\nu t)^{2/3}} \} 
\Big[  \| h_0 \|_{\mathbb{E}_{Landau,k,\eta,N+1}^{(10,0)'}} + \| h_0 \|_{\mathbb{E}_{Landau,k,\eta,0}^{(2,2)'}} \Big]. \qedhere
\end{split}
\end{equation*}
\end{proof}

\subsection{Finite time energy estimates and the $\nu \to 0$ limit}

In this subsection, we study the $\nu \to 0$ limit of $S_k(t)[h_0]$, where $t$ ranges over a finite time interval. To clarify the notations, in this subsection we use $S_k^{(\nu)}(t)$ to denote the semigroup associated to equation \eqref{lin-LandauF}.

\begin{proposition}\label{prop:Landau.finite.time.convergence}
Fix $T \in (0, +\infty)$. For any $h_0 \in \mathcal S(\R^3)$, 
$$\lim_{\nu \to 0} \sup_{t\in [0,T]} \left| \int_{\R^3} S^{(\nu)}_k(t) [h_0] \sqrt{\mu} \, \ud v - \int_{\R^3} S^{(0)}_k(t) [h_0] \sqrt{\mu} \, \ud v \right| = 0.$$
\end{proposition}
\begin{proof}
Let $f^{(\nu)} = S^{(\nu)}_k(t) [h_0]$. For every $\nu \in [0,1]$, standard \emph{finite time} energy estimates for \eqref{lin-LandauF} give that for $C_T>0$ (depending on $T$ but not on $\nu$)
\begin{equation}\label{eq:finite.time.local.energy}
\sup_{t \in [0,T]}\sum_{|\alp|+|\bt| \leq 2} \|k^\alp \rd_v^\bt f^{(\nu)}\|_{L^2_v} \leq C_T \sum_{|\alp|+|\bt| \leq 2} \|k^\alp \rd_v^\bt h_0 \|_{L^2_v}.
\end{equation}
(For finite time energy estimates, we simply use Gr\"onwall's inequality for many commutator terms, making them much easier than those in Proposition~\ref{prop-mainEE}.)

Consider now the equation $\rd_t (f^{(\nu)} - f^{(0)}) + i k\cdot v (f^{(\nu)} - f^{(0)}) = - \nu L f^{(\nu)}$. Multiplying by $f^{(\nu)} - f^{(0)}$, integrating in $v$ and then using \eqref{eq:finite.time.local.energy} gives
$$\f{d}{dt} \|f^{(\nu)} - f^{(0)}\|_{L^2_v}^2 \ls \nu  \|f^{(\nu)} - f^{(0)}\|_{L^2_v} \|Lf^{(\nu)} \|_{L^2_v} \ls C_T \nu  \|f^{(\nu)} - f^{(0)}\|_{L^2_v} \sum_{|\alp|+|\bt| \leq 2} \|k^\alp \rd_v^\bt h_0 \|_{L^2_v}.$$
This implies $\f{d}{dt} \|f^{(\nu)} - f^{(0)}\|_{L^2_v}\leq C\nu$, for some constant depending on $T$ and $h_0$. Noticing now that $\|f^{(\nu)} - f^{(0)}\|_{L^2_v}(0) = 0$, we then deduce that $\sup_{t\in [0,T]} \|f^{(\nu)} - f^{(0)}\|_{L^2_v}(t) \leq CT\nu$. The conclusion then follows from the Cauchy--Schwarz inequality in $v$. \qedhere
\end{proof}

\section{Linear density estimates}\label{sec-lineardensity}

The goal of this section is to derive decay estimates for the density of the following linear Vlasov--Poisson--Landau equation
\begin{equation}\label{lin-VPL}
\begin{aligned}
\part_{t} f & + v\cdot \nabla_x f  + \nu L f = 2E\cdot v\sqrt{\mu}  + \mathfrak N(t,x,v)
\end{aligned}
\end{equation}
for the linear Landau operator $L$ defined as in \eqref{def-L}, \eqref{eq:A.def} and \eqref{eq:K.def}. The equation is solved with initial data $f(0,x,v) = f_{0}(x,v)$ and a source $\mathfrak N(t,x,v)$, coupled with the Poisson equation $E = -\nabla_x (-\Delta_x)^{-1} \rho $, where the density is defined by  
$$\rho(t,x) =  \int_{\R^3} f(t,x,v) \sqrt\mu \d v.$$
For $k\in \Z^3\backslash \{0\}$, let $\hat\rho_k(t)$ be the Fourier transform of $\rho(t,x)$ with respect to variable $x$. We also denote by $\widehat{f_0}_k(v)$ and $\hat{\mathfrak N}_k(t,v)$ the Fourier transform of $f_0(x,v)$ and $\mathfrak N(t,x,v)$, respectively. 

\bigskip

The main result of this section is the following proposition. 

\begin{proposition}\label{prop-density} For any initial data $f_{0}(x,v)$ and any source term $\mathfrak N(t,x,v)$ in $L^2(\R^3 \times \R^3)$, the unique density solution $\rho(t,x)$ to \eqref{lin-VPL} satisfies the following representation 
\begin{equation}\label{exp-rho}
\hat\rho_k(t) = \mathcal{N}_k(t) + \int_0^tG_k(t-s)\mathcal{N}_k(s) \; ds
\end{equation}
for each Fourier mode $k\in \Z^3\backslash \{0\}$, where for any $N_0\ge 2$, there are $C_{N_0} >0$ and $\de''_{N_0} >0$ such that the kernel $G_k(t)$ satisfies
\begin{equation}\label{eq:G.est.main}
 |G_k(t)| \le C_{N_0} |k|^{-1}\langle kt \rangle^{-N_0+2} \min \{ e^{-\de''_{N_0} (\nu^{1/3} t)^{1/3}}, e^{-\de''_{N_0} (\nu t)^{2/3}} \}, \qquad \forall~t\ge 0 ,
 \end{equation}
uniformly in $k\not =0$ and $\nu \ge 0$, and the source $\mathcal{N}_k(t)$ is given by 
\begin{equation}\label{def-cNk}\mathcal{N}_k(t) = \int_{\R^3}S_k(t)[\widehat{f_0}_k(v)]\sqrt{\mu}\d v+\int_0^t\int_{\R^3}S_k(t-\tau)[ \hat{\mathfrak  N}_k(\tau,v)]\sqrt{\mu}\d v\d \tau
 \end{equation}
 where $S_k(t)$ is the semigroup of the linear Landau equation \eqref{lin-LandauF}.
  \end{proposition}

\begin{remark}
Proposition~\ref{prop-density} in particular shows uniform linear Landau damping for the linearized Vlasov--Poisson--Landau equation near the global Maxwellian $\mu = e^{-|v|^2}$. Indeed, combining with Proposition~\ref{prop-density} with Proposition~\ref{prop-mixdecay}, one deduces that for $h_k(t) := S_k(t) \hat{h_0}_k$, the corresponding density function $\hat{\rho}_k$ satisfies
$$|\hat\rho_k(t)| \ls_N \la kt\ra^{-N}\min \{ e^{-\de_N (\nu^{1/3} t)^{1/3}}, e^{-\de_N (\nu t)^{2/3}} \},$$
uniformly in $\nu\geq 0$, for sufficiently regular initial $h_0$.
\end{remark}

\begin{remark}
By comparison with the $\nu = 0$ case (see e.g., \cite{eGtNiR2020a}), one may expect that \eqref{eq:G.est.main} even holds with $|G_k(t)| \ls |k|^{-1} e^{-\de_0 |kt|}$. Proving this seems to require the technically involved task of deriving Proposition~\ref{prop-mixdecay} for $Y_{k,\eta}$ derivatives of all orders with almost-sharp constants, and has not been carried out. 
\end{remark}

\subsection{Equation for the density} We first derive an equation for the density from which the estimates are obtained.

\begin{lemma}\label{l.rho_k_set_up} Introduce the kernel 
\begin{equation}\label{e.kernel}
K_k(t)=  \frac{2}{|k|^2} \int_{\R^3} ik \cdot S_k(t)[v\sqrt{\mu}] \sqrt \mu\d v ,
\end{equation}
where $S_k(t)$ is the solution operator of \eqref{lin-LandauF}.  
Then, for each $k \in \Z^3\backslash \{0\}$, the density $\hat\rho_k(t)$ satisfies the following Volterra equation
\begin{equation}\label{e:den_eq}
\begin{split}
\hat\rho_k(t)+\int_0^t K_k(t-\tau)\hat{\rho}_k(\tau) \d \tau&= \mathcal{N}_k(t),
\end{split}
\end{equation}
where the nonlinear source term $\mathcal{N}_k(t)$ is as in \eqref{def-cNk}.
\end{lemma}

\begin{proof} The lemma is direct. Indeed, taking the Fourier transform in $x$ of the linear Vlasov--Poisson--Landau equation \eqref{lin-VPL}, we get
\begin{equation}\label{e.FT_x}
\begin{aligned}
\part_{t} \hat{f}_k & +i k\cdot v\hat{f}_k  + \nu L \hat{f}_k = 2\hat{E}_k\cdot v\sqrt{\mu}  +  \hat{\mathfrak N}_k(t,v).
\end{aligned}
\end{equation}
Let $S_k(t)$ be the semigroup of the linear Landau operator $\partial_t + ik \cdot v + \nu L$. Applying the Duhamel's principle to \eqref{e.FT_x}, we obtain 
\begin{equation}\label{e.f_k}
\begin{split}
\hat{f}_k(t,v)&=S(t)[\widehat{f_0}_k(v)] + 2 \int_0^t S_k(t-\tau)[\hat{E}_k(\tau)\cdot v\sqrt{\mu}]\d \tau+\int_0^tS_k(t-\tau) [ \hat{\mathfrak N}_k(\tau,v)]\d \tau .
\end{split}
\end{equation}
Note that $\hat{E}_k(t)$ is independent of $v$, and so 
$$ \int_0^t S_k(t-\tau)[\hat{E}_k(\tau)\cdot v\sqrt{\mu}]\d \tau = \int_0^t \hat{E}_k(\tau)\cdot  S_k(t-\tau)[v\sqrt{\mu}]\d \tau. $$
Recall that $\hat{E}_k(t) = -ik |k|^{-2} \hat{\rho}_k(t)$. Therefore, multiplying the equation \eqref{e.f_k} by $\sqrt \mu$ and integrating it over $\R^3$, we obtain the density equation 
$$
\hat\rho_k(t)+\int_0^t K_k(t-\tau)  \hat{\rho}_k(\tau) \d \tau =\int_{\R^3}S_k(t)[\widehat{f_0}_k(v)]\sqrt{\mu}\d v+\int_0^t\int_{\R^3}S_k(t-\tau)[ \hat{\mathfrak N}_k(\tau,v)]\sqrt{\mu}\d v\d \tau
$$
where the kernel $K_k(t)$ is defined as in \eqref{e.kernel}. Setting the right hand side to be $\mathcal{N}_k(t)$, which is the expression \eqref{def-cNk}, the lemma follows.
\end{proof}




\subsection{Kernel $K_k(t)$}

To solve the density equation \eqref{e:den_eq}, let us first study the kernel $K_k(t)$ defined as in \eqref{e.kernel}. We obtain the following.

\begin{lemma}\label{lem-Kk} For any $n,N \ge 0$, there exist constants $C_{N,n}>0$ and $\de_N>0$ so that 
\begin{equation}
|\partial_t^nK_k(t)| \le C_{N,n} |k|^{n-1}\langle kt \rangle^{-N} \min \{ e^{-\de_N(\nu^{1/3} t)^{1/3}}, e^{-\de_N(\nu t)^{2/3}} \}, \qquad \forall ~ t\ge 0,
 \end{equation}
uniformly in $\nu\ge 0$ and $k\not =0$. 
\end{lemma}

\begin{proof} Let $h_k(t,v) = 2|k|^{-2} ik \cdot S_k(t)[v\sqrt{\mu}]$, i.e.~that $h_k(t,v)$ solves the linear fixed mode Landau equation \eqref{lin-LandauF} with initial data 
$h(0,v) = 2 |k|^{-2} ik \cdot v \sqrt \mu$. By definition (see \eqref{e.kernel}), $K_k(t) =\int_{\R^3} h_k(t,v)\sqrt \mu\d v $. Hence, by \eqref{decay-combineSk.nu.decay} in Proposition~\ref{prop-mixdecay} with $\eta = 0$, we have 
\begin{equation*}
\begin{split} 
|K_k(t)| 
&\ls  C_N \la kt\ra^{-N} \min \{ e^{-\de_N(\nu^{1/3} t)^{1/3}}, e^{-\de_N(\nu t)^{2/3}} \} \| |k|^{-2} k \cdot v \sqrt \mu \|_{\mathbb E_{Landau,k,0,N+1}^{(2,2)'}} \\
&\ls C_N |k|^{-1} \la kt\ra^{-N} \min \{ e^{-\de_N(\nu^{1/3} t)^{1/3}}, e^{-\de_N(\nu t)^{2/3}} \}
\end{split}
\end{equation*}
for $k \not =0$, upon recalling that $\mu = e^{-|v|^2}$. As for derivatives, using \eqref{lin-LandauF} and integrating by parts in $v$, we compute 
$$
\begin{aligned}
\partial_t K_k(t) &= \int_{\R^3} \partial_t h_k(t,v)\sqrt \mu\d v 
= - \int_{\R^3} h_k(t,v) ( i k_j v_j +\nu L) \sqrt \mu \d v .
\end{aligned}$$  
Inductively, for $n\ge 1$, we have 
$$
\partial^n_t K_k(t) 
 = (-1)^n \int_{\R^3} h_k(t,v) (i k_j v_j  + \nu L)^n \sqrt \mu \d v. 
$$
The estimates for derivatives thus follow similarly, upon noting the loss of one factor of $|k|$ for each time derivative. 
This ends the proof of the lemma. 
\end{proof}

\subsection{Resolvent estimates}

We are now ready to solve the linear Volterra equation \eqref{e:den_eq} for the density  
\begin{equation}\label{lin-VPLk}
\begin{split}
\hat\rho_k(t)+\int_0^t K_k(t-\tau)\hat{\rho}_k(\tau) \d \tau&=\mathcal{N}_k(t)
\end{split}
\end{equation}
for the source term $\mathcal{N}_k(t)$ as in \eqref{def-cNk}, and thus give the proof of Proposition \ref{prop-density}.

\begin{proof}[Proof of Proposition \ref{prop-density}] 

\textbf{Taking the Laplace transform.} The linear  Volterra equation \eqref{lin-VPLk} is solved through its resolvent solution. Precisely, for any $F\in L^2(\R_+)$, let us introduce the Laplace transform 
$$ 
\mathcal{L}[F](\lambda) = \int_0^\infty e^{-\lambda t} F(t)\; dt 
$$
which is well-defined for any complex value $\lambda$ with $\Re \lambda >0$. Thus, taking the Laplace transform of \eqref{lin-VPLk}, we obtain the resolvent solution 
\begin{equation}\label{cL-rho} \mathcal{L}[\hat \rho_k](\lambda) = \frac{1}{1 + \mathcal{L}[K_k](\lambda)} \mathcal{L}[\mathcal{N}_k](\lambda) .\end{equation} 
The representation \eqref{exp-rho} follows from taking the inverse Laplace transform of \eqref{cL-rho} with the kernel $G_k(t)$ being the inverse Laplace transform of 
\begin{equation}\label{def-tG} \widetilde{G}_k(\lambda) =  - \frac{\mathcal{L}[K_k](\lambda)}{1 + \mathcal{L}[K_k](\lambda)} .\end{equation}

\textbf{Basic estimates for $\mathcal L[K_k](\lambda)$.} It remains to give estimates on the resolvent kernel $\widetilde{G}_k(\lambda)$. To simplify the exposition, we only prove the $e^{-\de_N (\nu^{1/3} t)^{1/3}}$ decay in \eqref{eq:G.est.main}; the $e^{-\de_N (\nu t)^{2/3}}$ decay can be proven in a completely analogous manner. By definition, we have 
$$  \mathcal{L}[K_k](\lambda) = \int_0^\infty e^{-\lambda t} K_k(t)\; dt $$
which is well-defined for any complex value $\lambda$ with $\Re \lambda >0$. Fix $N_0>1$. Using Lemma \ref{lem-Kk} with $N = N_0$, we bound 
\begin{equation}\label{Kk}
 | \mathcal{L}[K_k](\lambda)|  \le C_{N_0} |k|^{-1}\int_0^\infty \langle kt \rangle^{-N_0}\; dt \le C_{N_0} |k|^{-2} 
 \end{equation}
uniformly for any $\Re \lambda \ge 0$. Similarly, for any $0\le N <N_0-1$, we have 
\begin{equation}\label{derv-Kk} | \partial_\lambda^N\mathcal{L}[K_k](\lambda) |  \le C_{N_0} |k|^{-N-1}\int_0^\infty \langle kt \rangle^{N-N_0}\; dt \le C_{N_0} |k|^{-N-2}
\end{equation}
uniformly in $k\not =0$ and $\Re \lambda \ge 0$. 

For $N\geq N_0-1$, we use also the stretched exponential decay in Lemma \ref{lem-Kk} to obtain
\begin{equation}
\begin{split}
| \partial_\lambda^N\mathcal{L}[K_k](\lambda) |  \le C_{N_0} |k|^{-N_0+1}\int_0^\infty t^{N-N_0+2} e^{-\de_{N_0} (\nu^{1/3}t)^{1/3}} \langle kt \rangle^{-2}\; dt
\end{split}
\end{equation}
with a constant independent of $N$. Noticing that $\sup_{x\in [0,\infty)} x^M e^{-x^{1/3}} \leq (3M)^{3M}$, we have
\begin{equation}\label{exp-Kk}
\begin{split}
| \partial_\lambda^N\mathcal{L}[K_k](\lambda) | \le &\: C_{N_0} |k|^{-N_0+1} (\de_{N_0})^{-3(N-N_0+2)} \nu^{-(1/3)(N-N_0+2)} (3(N-N_0+2))^{3(N-N_0+2)} \\
\le &\: C_{N_0} |k|^{-N_0+1} [27 (\de_{N_0})^{-3} \nu^{-1/3} N^3]^{N},
\end{split}
\end{equation}
assuming, without loss of generality, $\de_{N_0} \nu \leq 1$.

\textbf{Checking the Penrose condition.} We now check the Penrose condition (see \eqref{Penrose} below) by comparing with the $\nu = 0$ case. To highlight the dependence on $\nu$, write $K^{(\nu)}_k = K_k$. 

First, by \eqref{Kk}, there exists $\mathbb K$ large such that $1+  \mathcal{L}[K_k^{(\nu)}](\lambda)\ge \frac12$ for for $|k| \geq \mathbb K$ and $\nu\geq 0$. On the other hand, it is classical \cite{cMcV2011} that the Penrose stability condition holds at $\nu=0$, i.e.~for any positive radial equilibria in $\R^3$, which in particular includes the Gaussian $\mu(v)$, there is $\kappa_0 \in (0, 1)$ such that
\begin{equation}\label{Penrose.limit} \inf_{\Re \lambda \ge 0} \inf_{k\in \R^3} |1+  \mathcal{L}[K_k^{(0)}](\lambda)| \ge \kappa_0 >0.\end{equation}

Now by the estimates in Lemma~\ref{lem-Kk}, it follows that there exists large $T>0$ such that $\int_T^\infty |K_k|(t) \, \ud t \leq \f{\kappa_0}4$ uniformly in $\nu \geq 0$ and $k\not = 0$. Moreover, fixing this $T$, Proposition~\ref{prop:Landau.finite.time.convergence} implies that $\lim_{\nu\to 0^+} \int_0^T |K_k^{(\nu)}-K_k^{(0)}|(t)\, \ud t = 0$ for every $k \not = 0$. It therefore follows from \eqref{Penrose.limit} that there exists $\nu_0>0$ such that $\inf_{\Re\lambda \geq 0} \inf_{|k|\leq \mathbb K} |1+ \mathcal L[K_k](\lambda)| \geq \f{\kappa_0}4$ for all $\nu \in [0,\nu_0]$. Together with the large $|k|$ estimates above, we have, for $\nu \in [0, \nu_0]$,
\begin{equation}\label{Penrose} \inf_{\Re \lambda \ge 0} \inf_{k\in \R^3} |1+  \mathcal{L}[K_k^{(\nu)}](\lambda)| \ge \f{\kappa_0}4 >0\end{equation}

\textbf{Basic estimates for $\widetilde{G}_k(\lambda)$.} Combining \eqref{derv-Kk} and \eqref{Penrose}, we obtain derivative bounds on the resolvent kernel, for $0\le N < N_0-1$,  
\begin{equation}\label{derv-tG} |\partial_\lambda^N \widetilde{G}_k(\lambda) | \le C_N |k|^{-N-2} \end{equation}
uniformly in $k\not =0$ and $\Re \lambda \ge 0$. 

Moreover, since $x\mapsto \f{x}{1-x}$ is real analytic on $[\kappa_0, \infty)$, using \eqref{exp-Kk}, \eqref{Penrose} and considering a power series expansion, we obtain that with $B_{N_0}$ independent of $N$,
\begin{equation}\label{exp-tG}
|\partial_\lambda^N \widetilde{G}_k(\lambda) | \le C_{N_0} |k|^{-N_0-2} [B_{N_0} \nu^{-1/3} N^3]^N
\end{equation}
uniformly in $k\not =0$, $\mathfrak R\lambda \geq 0$ and $N \geq N_0 -1$.

\textbf{Improved estimates for $\widetilde{G}_k(\lambda)$.} We need an improvement of \eqref{derv-tG} and \eqref{exp-tG} which incorporates decay in $\lambda$. More precisely, the kernel $G_k(t)$ is obtained through the inverse Laplace transform formula 
\begin{equation}\label{Lap-Gk} G_k(t) = \frac{1}{2\pi i} \int_{\{ \Re \lambda = \gamma_0\}} e^{\lambda t} \widetilde{G}_k(\lambda) \; d\lambda \end{equation}
for any $\gamma_0>0$. We stress that the estimates in \eqref{derv-tG} hold for $\Re \lambda =0$. Thus, to obtain decay in time, we need decay in $\mathfrak I \lambda$ independently of $\gamma_0$. To this end, for any $\lambda = \gamma_0 + i \tau $, we compute 
$$
\begin{aligned}
 (|k|^2 - \lambda^2) \mathcal{L}[K_k](\lambda) 
&= \int_0^\infty (|k|^2 - \partial_t^2)[e^{-\lambda t}] K_k(t)\; dt 
\\
&= e^{-\lambda t} \Big[\lambda K_k(t) + \partial_t K_k(t) \Big] \Big|_{t=0}^{t=\infty} +  \int_0^\infty e^{-\lambda t} (|k|^2  - \partial_t^2)K_k(t)\; dt .
\end{aligned}$$
In the above, the boundary term at $t = \infty$ vanishes, since (by Lemma~\ref{lem-Kk}) $K_k(t)$ and its derivatives decay rapidly in time. On the other hand, a direct calculation yields 
$$
\begin{aligned}
K_k(0) &= 2|k|^{-2} \int_{\R^3} ik \cdot v \mu\d v =0
\\
\partial_t K_k(0) &=  2|k|^{-2} \int_{\R^3} i \Big[ - i (k\cdot v)^2 \sqrt \mu  - \nu k_j L (v_j \sqrt \mu) \Big] \sqrt \mu  \d v. 
\end{aligned}$$  
Hence, $|\partial_t K_k(0)| \le C$, uniformly in $k$ for $\nu \leq 1$. Finally, using bounds from Lemma \ref{lem-Kk}, we obtain 
$$ |  (|k|^2 - \lambda^2) \mathcal{L}[K_k](\lambda) | \le C + C |k|\int_0^\infty  \langle kt \rangle^{-2}\; \d t \le C,$$ 
for $\lambda = \gamma_0 + i\tau$ and for some constant $C$ that is independent of $k, \gamma_0, \tau$.  This proves that 
$ |\mathcal{L}[K_k](\gamma_0+ i \tau) | \le \frac{ C }{|k|^2 + |\tau |^2 - |\gamma_0|^2 },$ giving
$$ |\widetilde{G}_k(\gamma_0 + i\tau)| \leq \f{C}{|k|^2 + |\tau|^2}$$
uniformly for all $\gamma_0 \in (0, 1/2)$. Similarly, repeating the proof leading to \eqref{derv-tG} and \eqref{exp-tG}, but incorporating the above integration by parts argument for additional $|\tau|^2$ decay, we obtain 
\begin{equation}\label{derv-tG1} |\partial_\lambda^N \widetilde{G}_k(\gamma_0 + i\tau) | \le  \frac{ C_N |k|^{-N} }{|k|^2 + |\tau |^2 } \end{equation}
for any $N < N_0 -1$, and, taking $B_{N_0}$ larger (but still independent of $N$) if necessary,
\begin{equation}\label{exp-tG1} |\partial_\lambda^N \widetilde{G}_k(\gamma_0 + i\tau) | \le  \frac{ C_{N_0} |k|^{-N_0} [B_{N_0} \nu^{-1/3} N^3]^N}{|k|^2 + |\tau |^2 } \end{equation}
for any $N \geq N_0-1$, where both estimates hold for any $k\not =0$, $\gamma_0 \in (0, 1/2)$ and $\tau \in \R$. 

\textbf{Estimating $G_k(t)$.} Thanks to the decay in $\tau$, we can take the $\gamma_0 \to 0^+$ limit in \eqref{Lap-Gk} with the dominated convergence theorem and perform repeated integrations by parts in $\tau$, yielding 
\begin{equation}\label{eq:inverse.Laplace.transform}
\begin{aligned}
G_k(t) 
&= \frac{1}{2\pi i} \int_{\{ \Re \lambda = 0 \}} e^{\lambda t}  \widetilde{G}_k(\lambda)\; d\lambda 
=  \frac{1}{2\pi } \int_{\R} e^{i\tau t}  \widetilde{G}_k(i\tau)\; d\tau 
\\
&=  \frac{1}{2\pi } \int_{\R} \frac{1}{it}\partial_\tau (e^{i\tau t})  \widetilde{G}_k(i\tau)\; d\tau = \frac{-1}{2\pi t} \int_{\R} e^{i\tau t} \partial_\lambda \widetilde{G}_k(i\tau)\; d\tau 
\\
&= \frac{(-1)^N}{2\pi t^N} \int_{\R} e^{i\tau t} \partial^N_\lambda  \widetilde{G}_k(i\tau)\; d\tau .
\end{aligned} 
\end{equation}
First, consider the case $t\leq 10^3 B_{N_0} e \nu^{-1/3}$ with $B_{N_0}$ as in \eqref{exp-tG1}. Using \eqref{derv-tG1} with $N=0$ and $N=N_0-2$, and plugging into \eqref{eq:inverse.Laplace.transform} (with the same $N$), we have
$$|G_k(t)| \ls \int_{\R} \f{\d \tau}{|k|^2 + |\tau|^2} \ls |k|^{-1}, \quad |G_k(t)| \ls_{N_0} \int_{\R} \f{|kt|^{-N_0+2} \, \d \tau}{|k|^2 + |\tau|^2} \ls_{N_0} |k|^{-1}|kt|^{-N_0+2}.$$
Therefore, for any $\de'' \leq 10^{-1} B_{N_0}^{-1} e^{-1}$,
\begin{equation}\label{eq:G.bd.small.t}
|G_k(t)| \ls_{N_0} |k|^{-1} \la kt\ra^{-N_0+2} e^{-\de'' (\nu^{1/3} t)^{1/3}}\quad \mbox{for $t\leq 10^3 B_{N_0} e \nu^{-1/3}$}.
\end{equation}

On the other hand, for $t \geq 10^3 B_{N_0} e \nu^{-1/3}$, we first use \eqref{exp-tG1} and \eqref{eq:inverse.Laplace.transform} with $N\geq N_0 -1$ to obtain
$$|G_k(t)| \ls_{N_0} (B_{N_0} \nu^{-1/3} N^3)^{N} t^{-N} \int_{\R} \f{|k|^{-N_0+1} \,\d \tau}{|k|^2 + |\tau|^2} \ls_{N_0} |k|^{-N_0+1} (B_{N_0} \nu^{-1/3} N^3)^{N} t^{-N} .$$
Given $t \geq 10^3 B_{N_0} e \nu^{-1/3}$, take $N = \lfloor ( B_{N_0}^{-1} e^{-1} \nu^{1/3} t)^{1/3} \rfloor$ so that $(B_{N_0}\nu^{-1/3} N^3) t^{-1} \leq \f 1e$ and $N\geq (1/2) ( B_{N_0}^{-1} \nu^{1/3} t)^{1/3}$. This implies, for $\de''_{N_0} \leq \f 12 B_{N_0}^{-1/3}$,
$$|G_k(t)| \ls_{N_0} |k|^{-N_0+1} e^{-N} \ls_{N_0} |k|^{-N_0+1} e^{-(1/2)B_{N_0}^{-1/3} (\nu^{1/3} t)^{1/3}} \ls_{N_0} |k|^{-N_0+1} e^{- \de''_{N_0} (\nu^{1/3} t)^{1/3}}.$$ 
A similar computation gives, after taking $\de''_{N_0}$ smaller
$$|G_k(t)| \ls_{N_0} |k|^{-1} |kt|^{-N_0+2} (B_{N_0} \nu^{-1/3} N^3)^{N} t^{-N+N_0-2} \ls_{N_0} |k|^{-1} |kt|^{-N_0+2} e^{- \de''_{N_0} (\nu^{1/3} t)^{1/3}}.$$
Combining the two estimates above gives
\begin{equation}\label{eq:G.bd.large.t}
|G_k(t)| \ls_{N_0} |k|^{-1} \la kt\ra^{-N_0+2} e^{-\de''_{N_0} (\nu^{1/3} t)^{1/3}}\quad \mbox{for $t\geq 10^3 B_{N_0} e \nu^{-1/3}$}.
\end{equation}

Combining \eqref{eq:G.bd.small.t} and \eqref{eq:G.bd.large.t} yields the $e^{-\de''_{N_0} (\nu^{1/3} t)^{1/3}}$ decay estimate in \eqref{eq:G.est.main}; the $e^{-\de''_{N_0}(\nu t)^{2/3}}$ decay can be proven similarly and is omitted. \qedhere

\end{proof}

\section{Nonlinear density estimates: bounds for all derivatives}\label{sec:density}

In this section, we derive density estimates for the full nonlinear Vlasov--Poisson--Landau equation \eqref{eq:Vlasov.f}--\eqref{eq:Poisson.f} under the bootstrap assumptions on $[0,T_B)$ for $N\le N_{max}$:
\begin{itemize}
\item For all $t\in [0,T_B)$, the nonlinear solution $f$ to \eqref{eq:Vlasov.f}--\eqref{eq:Poisson.f} satisfies 
\begin{equation}\label{bootstrap-f}
\begin{split}
 \norm{f(t)}^2_{\wtbE^{(\vartheta)}_N}+ \nu^{1/3}\int_0^{t} \norm{f(\tau)}^2_{\wtbD^{(\vartheta)}_N}\d \tau
\leq 
\ep \nu^{2/3}  \min\{\nu^{-1/3}, \la t \ra\}^{\max\{0, N - N_{max}+ 2 \}}
\end{split}
\end{equation}
for $\vartheta \in \{0,2\}$, where $\norm{f}^2_{\wtbE^{(\vartheta)}_N}$, $\|f \|_{\wtbD^{(\vartheta)}_N}^2$ are defined in \eqref{eq:def.EN.DN}.
\item The following holds for all $t\in [0, T_B)$ for $\phi := -\Delta_x^{-1}\rho_{\not =0}$ and $E:= -\nab_x\phi$:
\begin{equation}\label{bootstrap-phi}
 \|\rd_t \phi(t)\|_{L^\i_x} + \|\phi(t)\|_{W^{5,\i}_x} + \sum_{|\alp|+|\om|\leq 4} \|\rd_x^\alp Y^\om E(t) \|_{L^\i_x} \leq \ep ^{1/2}\nu^{1/3}\la t\ra^{-2}.
\end{equation}
\end{itemize}


\bigskip

The main result of this section is the following. 

\begin{theorem}\label{theo-density} Consider data as in Theorem~\ref{t.main}. Suppose there exists $T_B>0$ such that the solution $f$ to \eqref{eq:Vlasov.f}--\eqref{eq:Poisson.f} remains smooth in $[0,T_B)\times \T^3 \times \R^3$ and satisfies the bootstrap assumptions \eqref{bootstrap-f} and \eqref{bootstrap-phi}.  

Then, $\rho_{\not = 0}(t,x)$ satisfies
\begin{equation}\label{density-bound} 
\begin{split}
 \sum_{|\alp|+ |\om| \leq N_{max}} (\sup_{0\leq t < T_B} \norm{\rd_x^\alp Y^\omega \rho_{\not = 0}(t)}^2_{L^2_x} + \nu^{1/3} \int_0^{T_B} \norm{\rd_x^\alp Y^\omega \rho_{\not = 0}(\tau)}^2_{L^2_x}\, \d \tau)
\ls 
 \ep^2 \nu^{2/3}.
\end{split}
\end{equation}

\end{theorem}


The proof of Theorem \ref{theo-density} proceeds as follows. We write the nonlinear equation \eqref{eq:Vlasov.f} in the form of \eqref{lin-VPL}, which we recall 
$$ \part_{t} f  + v\cdot \nabla_x f  + \nu L f = 2E\cdot v\sqrt{\mu}  + \mathfrak N(t,x,v)$$
where the nonlinear source term $\mathfrak N(t,x,v)$ is computed by 
\begin{equation}\label{def-NNN} \mathfrak N(t,x,v) : =   E\cdot vf  - E\cdot \nabla_v f    +\nu \Gamma(f,f).\end{equation}
We can thus apply the linear theory developed in Proposition \ref{prop-density} to compute the density through the density representation \eqref{exp-rho}. Let us first give estimates on the source $\mathcal{N}_k(t)$ computed by \eqref{def-cNk}:
$$
\begin{aligned}
\mathcal{N}_k(t) 
&= \int_{\R^3}S_k(t)[\widehat{f_0}_k(v)]\sqrt{\mu}\d v+\int_0^t\int_{\R^3}S_k(t-\tau)[ \hat{\mathfrak N}_k(\tau,v)]\sqrt{\mu}\d v\d \tau
\end{aligned}$$
where $\widehat{f_0}_k(v)$ and $\hat{\mathfrak N}_k(t,v)$ are the Fourier transform of $f_0(x,v)$ and $\mathfrak N(t,x,v)$, respectively.

Precisely, we will prove the following proposition.

\begin{proposition}\label{prop-NNN} Define
\begin{equation}\label{def-zeta}
\zeta(t) := \sum_{N_1+N_2 \leq N_{max}} \Big[ \sup_{0\leq \tau \leq t} \sum_{l \not = 0} |l|^{2N_1} \la l\tau\ra^{2N_2} |\hat\rho_l(\tau)|^2 + \nu^{1/3} \int_0^t \sum_{l \not = 0} |l|^{2N_1} \la l\tau\ra^{2N_2} |\hat\rho_l(\tau)|^2\; \d \tau\Big].
\end{equation} 

Then, under the assumptions of Theorem~\ref{theo-density}, there holds 
\begin{equation}\label{bound-Nk}
\begin{split}
&\: \sum_{N_1+N_2 \leq N_{max}} (\sum_{k\not =0} | k |^{2N_1} \langle k t \rangle^{2N_2} |\mathcal{N}_k(t)|^2 + \nu^{1/3}\int_0^t \sum_{k\not =0} | k |^{2N_1} \langle k \tau \rangle^{2N_2} |\mathcal{N}_k(\tau)|^2 \,\d\tau )\\
\ls &\:  \ep^2 \nu^{2/3} +  \ep \zeta(t).
\end{split}
\end{equation}
\end{proposition}

We first prove that Proposition \ref{prop-NNN} gives Theorem \ref{theo-density}. 
Recalling the density representation \eqref{exp-rho}, with $|G_k(t)|\lesssim |k|^{-1}\langle kt\rangle^{-N_0+2} = |k|^{-1}\langle kt\rangle^{-2N_{max}-2}$ (choosing $N_0 = 2N_{max}+4$), we have the following bound for any $N_1+N_2 \leq N_{max}$
$$
\begin{aligned}
& \sum_{k\not =0}|k|^{2N_1} \la kt \ra^{2N_2} |\hat\rho_k(t)|^2 \le \sum_{k\not =0} |k|^{2N_1} \la kt \ra^{2N_2} |\mathcal{N}_k(t)|^2 + \sum_{k\not =0}|k|^{2N_1} \la kt \ra^{2N_2} \Big|\int_0^tG_k(t-s)\mathcal{N}_k(s) \; \d s\Big|^2
\\
\lesssim &\: \sum_{k\not =0}|k|^{2N_1} \la kt \ra^{2N_2}|\mathcal{N}_k(t)|^2 + \sum_{k\not =0} |k|^{2N_1} \la kt \ra^{2N_2}\int_0^t |G_k(t-s)| |\mathcal{N}_k(s)|^2 \; \d s \int_0^t |G_k(t-s)| \; \d s
\\
\lesssim  &\:\sum_{k\not =0}|k|^{2N_1} \la kt \ra^{2N_2} |\mathcal{N}_k(t)|^2 + \sum_{k\not =0} |k|^{-3}\int_0^t\langle k(t-s)\rangle^{-2} |k|^{2N_1} \langle ks \rangle^{2N_2} |\mathcal{N}_k(s)|^2 \; \d s,
\end{aligned}
$$
where at the very end we used $\f{t}2\leq \max\{t-s, s\}$ so that $\langle k(t-s)\rangle^{-2N_{max}-2} \la kt \ra^{2N_2} \ls \langle k(t-s)\rangle^{-2} \langle ks \rangle^{2N_2}$.

Now using \eqref{bound-Nk} and the fact that $\zeta(t)$ is monotone in $t$, we get 
$$
\begin{aligned}
\sum_{k\not =0} |k|^{2N_1} \la kt \ra^{2N_2} |\hat\rho_k(t)|^2 
&\ls \ep^2 \nu^{2/3} + \ep \zeta(t) + \sum_{k\not =0} |k|^{-3}\int_0^t\langle k(t-s)\rangle^{-2} \Big[ \ep^2 \nu^{2/3} +  \ep \zeta(s)\Big] \; \d s
\\
&\ls \ep^2 \nu^{2/3} + \ep \zeta(t) + \sum_{k\not =0} |k|^{-4} \Big[ \ep^2 \nu^{2/3} + \ep \zeta(t) \Big]
\\& \ls \ep^2 \nu^{2/3} + \ep \zeta(t),
\end{aligned}
$$
noting the summation over $\Z^3 \setminus \{0\}$ of $|k|^{-4}$ is finite. Similarly, using $L^2_t$ bounds in \eqref{bound-Nk}, we compute 
$$
\begin{aligned}
& \nu^{1/3}\int_0^t\sum_{k\not =0}|k|^{2N_1} \langle k\tau \rangle^{2N_2} |\hat\rho_k(\tau)|^2 \; \d\tau \\
&\lesssim  \nu^{1/3} \int_0^t \sum_{k\not =0}|k|^{2N_1} \langle k\tau \rangle^{2N_2}|\mathcal{N}_k(\tau)|^2\; \d\tau 
\\&\quad + \nu^{1/3} \sum_{k\not =0} |k|^{-3} \int_0^t \int_0^\tau \langle k(\tau-s)\rangle^{-2} |k|^{2N_1} \langle ks \rangle^{2N_2} |\mathcal{N}_k(s)|^2 \; \d s\,\d\tau
\\
&\ls \ep^2 \nu^{2/3} + \ep \zeta(t) + \sum_{k\not =0} |k|^{-4} \Big[  \ep^2 \nu^{2/3} +  \ep \zeta(t) \Big]
\ls \ep^2 \nu^{2/3} + \ep \zeta(t) .
\end{aligned}
$$
Combining and recalling \eqref{def-zeta}, we obtain 
$$\zeta(t)\ls \ep^2 \nu^{2/3}+ \ep \zeta(t) $$
which immediately yields Theorem \ref{theo-density}, upon taking $\epsilon$ sufficiently small and recalling that the Fourier transform of $\rd_x^\alp Y^\omega \rho(t)$ is precisely $(ik)^\alpha (ikt)^\omega \hat\rho_k(t)$.  

The remaining subsections are thus entirely devoted to prove Proposition \ref{prop-NNN}. In view of \eqref{def-NNN}, we write 
\begin{equation}\label{eq:I.II.III.def}
\begin{aligned}
\mathcal{N}_k(t) 
&= \int_{\R^3}S_k(t)[\widehat{f_0}_k]\sqrt{\mu}\ud v+\int_0^t\int_{\R^3}S_k(t-\tau)[ (\widehat{E\cdot vf})_k  - (\widehat{E\cdot \nabla_v f})_k](\tau)\sqrt{\mu}\ud v\ud \tau
\\&\quad + \nu \int_0^t\int_{\R^3}S_k(t-\tau) [ (\widehat{\Gamma(f,f)})_k(\tau) ]\sqrt{\mu}\d v\d \tau =: \mathbb I_k(t) + \mathbb{II}_k(t) + \mathbb{III}_k(t).
\end{aligned}
\end{equation}
We shall now prove \eqref{bound-Nk} for each term in the following subsections. For the remainder of the section, fix $N_1$, $N_2$ such that $N_1 + N_2 \leq N_{max}$.

\subsection{Initial data contribution}

In this section, we give estimates on  
$$
\mathbb{I}_k(t)=\int_{\R^3}S_k(t)[\widehat{f_0}_k(v)] \sqrt{\mu}\d v .$$
By \eqref{decay-combineSk.nu.decay} in Proposition~\ref{prop-mixdecay} with $\eta = k$, 
$$
\begin{aligned}
|\mathbb{I}_k(t)| \ls \langle k(t+1) \rangle^{-N_{max}-1} \min \{ e^{-(\de'/2)(\nu^{1/3} t)^{1/3}}, e^{-(\de'/2)(\nu t)^{2/3}} \} \| \widehat{f_0}_k\|_{\mathbb E^{(2,2)'}_{Landau, k, k, N_{max}+2}} .
\end{aligned}$$
Summing over $k$, and using the assumption \eqref{eq:thm.assumption.2} for the initial data,
\begin{equation}\label{claim-1k}
\begin{split}
\sum_{k\not =0}|k|^{2N_1} \langle kt \rangle^{2N_2} |\mathbb{I}_k(t)|^2 \ls &\: \langle t\rangle^{-2} \min \{ e^{-\de'(\nu^{1/3} t)^{1/3}}, e^{-\de'(\nu t)^{2/3}} \}  \sum_{k\in \Z^3}\| \widehat{f_0}_k\|_{\mathbb E^{(2,2)'}_{Landau, k, k, N_{max}+2}}^2 \\
\ls &\: \ep^2 \nu^{2/3} \langle t\rangle^{-2} \min \{ e^{-\de'(\nu^{1/3} t)^{1/3}}, e^{-\de'(\nu t)^{2/3}} \}.
\end{split}
\end{equation}
which in particular satisfies both the $L^\i_t$ and $L^2_t$ bounds required in Proposition \ref{prop-NNN}.

\subsection{Nonlinear interaction I}\label{sec:nonlinear.int.I}

In this section, we bound 
\begin{equation}\label{eq:IIk.def}
\begin{aligned}
\mathbb{II}_k(t) 
&= \int_0^t\int_{\R^3}S_k(t-\tau)[ (\widehat{E\cdot vf})_k(\tau)  - (\widehat{E\cdot \nabla_v f})_k(\tau)]\sqrt{\mu}\d v\d \tau =: \mathbb{II}_{k,1}(t) + \mathbb{II}_{k,2}(t)
\end{aligned}
\end{equation}
under the bootstrap assumption \eqref{bootstrap-f} on $f$. Precisely, 
we will prove that 
\begin{equation}\label{claim-2k}
\begin{split}
\sum_{k\not =0}& | k |^{2N_1} \langle kt \rangle^{2N_2} |\mathbb{II}_k(t)|^2 + \nu^{1/3}\int_0^t \sum_{k\not =0} | k |^{2N_1} \langle kt \rangle^{2N_2} |\mathbb{II}_k(\tau)|^2 \; d\tau \ls 
\ep \zeta(t),
\end{split}
\end{equation}
where $\zeta(t)$ is defined as in \eqref{def-zeta}. 

Clearly, the first term in \eqref{eq:IIk.def} involving $E\cdot vf$ can be treated similarly as $E\cdot \nabla_vf$. (In fact, it is better due to the absence of $\rd_v$ derivatives). We focus only the proof of the bounds involving the last term. Note that the semigroup $S_k(t-s)$ commutes with $\widehat{E}_l(s)$, as it is independent of $v$. Therefore, we have 
\begin{equation}\label{eq:IIk.written.out}
\begin{aligned}
\mathbb{II}_{k,2}(t) &= - \sum_{l \not =0} \int_0^t \widehat{E}_l(\tau) \cdot \int_{\R^3}S_k(t-\tau)[\widehat{\nabla_v f}_{k-l}(\tau)]\sqrt{\mu}\d v\d \tau .
\end{aligned}
\end{equation}
To prove \eqref{claim-2k}, we use \eqref{decay-combineSk} in Proposition \ref{prop-mixdecay} for the semigroup $S_k(t-\tau)$ with $\eta = (k-l)\tau$. Thus, for any $N_1',\, N_2' \ge 0$ with $N_1' + N_2' \leq N$, we bound 
\begin{equation}\label{eq:twisted.Landau}
\begin{aligned}
&\: \Big|  \int_{\R^3}S_k(t-\tau)[\widehat{\nabla_v f}_{k-l}(\tau)]\sqrt{\mu}\d v\Big| \\
\ls &\: \langle kt - l \tau \rangle^{-N_2'}  \langle \nu^{1/3}(t-\tau)\rangle^{-3/2} \| \widehat{\nabla_v f}_{k-l}(\tau)\|_{\mathbb{E}^{(10,0)'}_{Landau,k,(k-l)\tau,N_2'}} \\
\ls &\: \la k-l \ra^{-N_1'} \langle kt - l \tau \rangle^{-N_2'}  \langle \nu^{1/3}(t-\tau)\rangle^{-3/2} \smashoperator{\sum_{\substack{ |\alp|\leq N_1' ,\, |\om|\leq N_2' \\ 1\leq |\bt|\leq 2}}} \nu^{(|\bt|-1)/3} \|\la v \ra^{10} (\rd_x^\alp \rd_v^{\bt}  Y^\om f)\sphat_{k-l}(\tau)\|_{L^2_v} ,
\end{aligned}
\end{equation}
where we have used $Y_{0,(k-l)\tau}\hat{f}_{k-l}(\tau) = (\widehat{Y f})_{k-l}(\tau)$, recalling the vector fields $Y_{0,(k-l)\tau} = \nab_v + i (k-l)\tau$ and $Y = t\nab_x + \nab_v $.

To lighten the notation, define, for any $N \in \mathbb N$ and $k \in \mathbb Z^3$,
\begin{equation}\label{eq:wtG.def}
\| \hat{f}_{k} (\tau) \|_{\widetilde{\mathbb G}_N}:=  \smashoperator{\sum_{\substack{ |\alp|+ |\om|\leq N \\ 1\leq |\bt|\leq 2}}} \nu^{(|\bt|-1)/3} \|\la v \ra^{10} (\rd_x^\alp \rd_v^{\bt}  Y^\om f)\sphat_{k}(\tau)\|_{L^2_v}.
\end{equation}

Note that the $\wtG_N$ norm can be controlled by the $\wtbE_N^{(0)}$ norm (with an $\nu^{-1/3}$ weight) because it controls up to two $\rd_v$ derivatives (taking into account the $\nu^{1/3}$ power), and we have a lot of extra $\bv$-weights. In other words, for any $t\in [0,T_B)$,
\begin{equation}\label{eq:G.by.E}
 \sum_k \| \hat{f}_k(t) \|_{\wtG_N}^2 \ls \nu^{-2/3}\| f(t)\|_{\wtbE_N^{(0)}}^2 \ls \ep \min\{ \nu^{-1/3}, \la t \ra \}^{\max \{0, N-N_{max}+2\}},
 \end{equation}
where at the end we used the bootstrap assumption \eqref{bootstrap-f}.

We now plug the estimate \eqref{eq:twisted.Landau}, for any $N_1',\, N_2' \ge 0$, into \eqref{eq:IIk.written.out}, and recall that $\hat E_l = -il |l|^{-2} \hat\rho_l$, to deduce
\begin{equation}\label{bd-IIk}
\begin{aligned}
&\: |k|^{N_1} \langle kt \rangle^{N_2}|\mathbb{II}_{k,2}(t)| \\
\ls &\: \sum_{l \not = 0} \int_0^t |l|^{-1} | k |^{N_1} \langle kt \rangle^{N_2}  \la k-l \ra^{-N_1'} \langle kt - l \tau \rangle^{-N_2'}  \langle \nu^{1/3}(t-\tau)\rangle^{-3/2} |\hat\rho_l(\tau)| \| \hat f_{k-l}(\tau)\|_{\widetilde{\mathbb G}_N}  \; \d\tau \\
\ls &\: \sum_{l \not = 0} \int_0^t |l|^{-1}C_{k,l}^{N_1,N_2,\underline{N}_1,\underline{N}_2,N_1',N_2'}(t,\tau) |l|^{\uN_1}\la l\tau \ra^{\uN_2} |\hat\rho_l(\tau)| \| \hat f_{k-l}(\tau)\|_{\widetilde{\mathbb G}_N}  \; \d\tau,
\end{aligned}
\end{equation}
where we have set 
\begin{equation}\label{def-CCC} C_{k,l}^{N_1,N_2,\underline{N}_1,\underline{N}_2,N_1',N_2'}(t,\tau): = | k |^{N_1} | l |^{-\underline{N}_1} \langle kt \rangle^{N_2} \langle l\tau \rangle^{-\underline{N}_2}  \la k-l \ra^{-N_1'} \langle kt - l \tau \rangle^{-N_2'}  \langle \nu^{1/3}(t-\tau)\rangle^{-3/2}.\end{equation}
Here, $(\underline{N}_1,\underline{N}_2)$, $(N_1',N_2')$ and $N$ are arbitrary, as long as $\underline{N}_1 + \underline{N}_2\leq N_{max}$, $N_1' + N_2' \leq N$. The indexes are put for sake of flexibility, though only a certain pair of indexes is needed, as will be clear below.

\subsection*{Estimates for $\protect{C_{k,l}^{N_1,N_2,\protect\uN_1,\protect\uN_2,N_1',N_2'}(t,\tau)}$} 
Our next step is to estimate $C_{k,l}^{N_1,N_2,\uN_1,\uN_2,N_1',N_2'}(t,\tau)$.

We divide up the integration region in $\tau$ into $|l\tau|\leq |kt|/2$ and $|l\tau|>|kt|/2$. In the former case, we further split up the sum in $l$ to $|l| \leq |k|/2$ and $|l|>|k|/2$. In each case, we obtain the following bound:
\begin{itemize}
\item Case 1: $|l \tau |\le |kt|/2$ and $|l|\leq |k|/2$. In this case, $|kt - l \tau | \ge |k t|/2$ and $|k - l| \ge |k|/2$. We choose $(N_1', N_2') = (N_1,N_2)$, $N = N_{max}$ and $(\uN_1, \uN_2) = (2,3)$. Then
\begin{equation*}
\begin{split}
&\: \langle kt \rangle^{N_2} \langle l\tau \rangle^{-\underline{N}_2}  \langle kt - l \tau \rangle^{-N_2'} \ls \la l\tau\ra^{-3},\quad |k|^{N_1} |l|^{-\underline{N}_1} \la k-l\ra^{-N_1'} \ls |l|^{-2},
\end{split}
\end{equation*}
giving
\begin{equation}\label{C-case3}
C_{k,l}^{N_1,N_2,\uN_1,\uN_2,N_1',N_2'}(t,\tau) \lesssim \langle l\tau \rangle^{-3} |l|^{-2} \langle \nu^{1/3}(t-\tau)\rangle^{-3/2}.\end{equation}
\item Case 2: $|l \tau |\le |kt|/2$ and $|l|> |k|/2$. In this case $|kt - l \tau | \ge |k t|/2$. We choose $N=N_{max}$,
$$(\uN_1, \uN_2) = \begin{cases}
(N_1,N_2) \\
(N_1+2,3),
\end{cases}
\quad 
(N_1', N_2') = \begin{cases}
(2, N_2+3) & \mbox{ if $N_1\geq N_2$}\\
(0,N_2)& \mbox{ if $N_1<N_2 $}.
\end{cases}$$
Notice that our choice satisfies $\uN_1+ \uN_2\leq N_{max}$ and $N_1'+ N_2'\leq N= N_{max}$ (since $N_{max} \geq 9$). Whether $N_1\geq N_2$ or $N_1 <N_2$, it is straightforward to check that
\begin{equation*}
\begin{split}
&\: \langle kt \rangle^{N_2} \langle l\tau \rangle^{-\underline{N}_2}  \langle kt - l \tau \rangle^{-N_2'} \ls \la l\tau\ra^{-3},\quad |k|^{N_1} |l|^{-\underline{N}_1} \la k-l\ra^{-N_1'} \ls \max \{ |l|^{-2},\, \la k- l \ra^{-2} \},
\end{split}
\end{equation*}
giving
\begin{equation}\label{C-case1}
C_{k,l}^{N_1,N_2,\uN_1,\uN_2,N_1',N_2'}(t,\tau) \lesssim \langle l\tau \rangle^{-3} \langle \nu^{1/3}(t-\tau)\rangle^{-3/2} \max \{ |l|^{-2},\, \la k- l \ra^{-2} \}.\end{equation}
\item Case 3: $|l\tau| > |kt|/2$. In this case we must have $|l| > |k|/2$ (since $\tau \leq t$). Taking $(\uN_1, \uN_2) = (N_1, N_2)$, $(N_1',N_2') = (4,2)$ and $N =N_{max} -2 \geq 6$, we have
\begin{equation}\label{C-case2} 
\begin{split}
&\: C_{k,l}^{N_1,N_2,\uN_1,\uN_2,N_1',N_2'}(t,\tau) \lesssim \la k-l \ra^{-4} \langle kt - l \tau \rangle^{-2}.
\end{split}
\end{equation}

\end{itemize}

Define $\mathcal C_1$, $\mathcal C_2$ and $\mathcal C_3$ by
\begin{equation}\label{eq:def.Cj}
\begin{split}
\mathcal C_1:= \langle l\tau \rangle^{-3} |l|^{-2} \langle \nu^{1/3}&(t-\tau)\rangle^{-3/2},\quad \mathcal C_2 := \langle l\tau \rangle^{-3} \la k-l \ra^{-2} \langle \nu^{1/3}(t-\tau)\rangle^{-3/2},\\ 
&\mathcal C_3:= \la k-l \ra^{-4} \langle kt - l \tau \rangle^{-2},
\end{split}
\end{equation}
and define 
$$r_l(t) := \sum_{\uN_1+\uN_2 \leq N_{max}}  |l|^{2\uN_1} \la lt\ra^{2\uN_2} |\hat\rho_l(t)|^2,$$
(so that $\zeta(t) \ls   \sup_{0\leq \tau \leq t} \sum_{l\not = 0} r_l(\tau) + \nu^{1/3} \sum_{l\not = 0} \int_0^t r_l(\tau) \, \d\tau \ls \zeta(t)$).

Set
\begin{equation}\label{def-IIj} 
\mathcal{I}_j(t) := \sum_{k \not = 0}\Big[ \sum_{l \not = 0}\int_0^t |l|^{-1} \mathcal C_j r_l^{1/2}(\tau) \|\hat f_{k-l}(\tau) \|_{\widetilde{\mathbb G}_{N_{max}}}\, \ud \tau \Big]^2, \quad j = 1,2,
\end{equation}
and set
\begin{equation}\label{def-II3} 
\mathcal{I}_3(t) := \sum_{k \not = 0}\Big[ \sum_{l \not = 0}\int_0^t |l|^{-1} \mathcal C_3 r_l^{1/2}(\tau) \|\hat f_{k-l}(\tau) \|_{\widetilde{\mathbb G}_{N_{max}-2}}\, \ud \tau \Big]^2. \end{equation}

It follows from \eqref{bd-IIk}, \eqref{def-CCC} and the bounds for $C_{k,l}^{N_1,N_2,\uN_1,\uN_2,N_1',N_2'}(t,\tau)$ in \eqref{C-case3}, \eqref{C-case1}, \eqref{C-case2} that 
\begin{equation}\label{eq:reduction.II.to.IIj.1}
\begin{split}
&\: \sum_{k \not = 0} |k|^{2N_1} \langle kt \rangle^{2N_2}|\mathbb{II}_{k,2}(t)|^2 \ls \mathcal I_1 + \mathcal I_2 + \mathcal I_3.
\end{split}
\end{equation}

By the bound \eqref{eq:reduction.II.to.IIj.1}, in order to obtain the claim \eqref{claim-2k}, it suffices to show 
\begin{equation}\label{claim-2k-I}
\mathcal {I}_j(t) \ls \ep \zeta(t),\quad \int_0^t \mathcal I_j(\tau)\, \d \tau \ls \ep \nu^{-1/3} \zeta(t),\quad j = 1,2,3,
\end{equation} which will be achieved below.

\subsection*{$L^\infty_t$ bounds for $\mathcal{I}_1(t)$} To bound $\mathcal{I}_1$, we start with \eqref{def-IIj} and the definition of $\mathcal C_1$ in \eqref{eq:def.Cj}. Then, using the Cauchy--Schwarz inequality in $\tau$, and then the Young's convolution inequality for the sums, we bound 
\begin{equation}\label{beforebeforekeybound-I1}
\begin{aligned}
&\:\mathcal{I}_1(t) \\
\le &\:\sum_{k\not =0}
\Big[ \sum_{l \not =0}  \int_0^t|l|^{-3} \langle l\tau \rangle^{-3} \la \nu^{1/3} (t-\tau)\ra^{-3/2} r^{1/2}_l(\tau) \| \hat f_{k-l}(\tau)\|_{\wtG_{N_{max}}}\; \d\tau\Big]^2
\\
\ls &\:\sum_{k\not =0}
\Big[ \sum_{l \not =0}  (\int_0^t|l|^{-6} \langle l\tau \rangle^{-5/2} \la \nu^{1/3} (t-\tau)\ra^{-3} r_l(\tau ) \, \d \tau)^{1/2}  (\int_0^t \la \tau\ra^{-7/2} \| \hat f_{k-l}(\tau)\|_{\wtG_{N_{max}}}^2 \, \d\tau)^{1/2}\Big]^2 \\
\ls &\:\Big[ \sum_{l \not = 0} (\int_0^t|l|^{-6} \langle l\tau \rangle^{-5/2} \la \nu^{1/3} (t-\tau)\ra^{-3} r_l(\tau) \, \d \tau)^{1/2}\Big]^2 \Big[ \sum_{k } \int_0^t \la \tau\ra^{-7/2} \| \hat f_{k}(\tau)\|_{\wtG_{N_{max}}}^2 \; \d\tau\Big].
\end{aligned}
\end{equation}

By \eqref{eq:G.by.E} (allowing $\la \tau\ra^2$ growth),
\begin{equation}\label{eq:bd.top.wtG}
\sum_{k } \int_0^t \la \tau\ra^{-7/2} \| \hat f_{k}(\tau)\|_{\wtG_{N_{max}}}^2 \, \d\tau \ls \nu^{-2/3} \int_0^t \la \tau\ra^{-7/2} \| f(\tau)\|_{\wtbE^{(0)}_{N_{max}}}^2\, \d \tau \ls \ep \int_0^t \la \tau\ra^{-3/2} \, \d \tau \ls \ep.
\end{equation}
Thus, substituting this into \eqref{beforebeforekeybound-I1}, we obtain
\begin{equation}\label{beforekeybound-I1}
\begin{aligned}
\mathcal I_1(t) \ls \ep \Big[ \sum_{l \not = 0} (\int_0^t|l|^{-6} \langle l\tau \rangle^{-5/2} \la \nu^{1/3} (t-\tau)\ra^{-3} r_l(\tau) \, \d \tau)^{1/2}\Big]^2.
\end{aligned}
\end{equation}

To proceed, a direct computation using \eqref{def-zeta} shows
\begin{equation}
\begin{aligned}
&\Big[ \sum_{l \not = 0} (\int_0^t|l|^{-6} \langle l\tau \rangle^{-5/2} \la \nu^{1/3} (t-\tau)\ra^{-3} r_l(\tau) \, \d \tau)^{1/2}\Big]^2\ls  \Big[ \sum_{l \not = 0} (\int_0^t|l|^{-6} \langle l\tau \rangle^{-5/2} r_l(\tau) \, \d \tau)^{1/2}\Big]^2 \\
&\ls (\sum_{l \not = 0} |l|^{-7/2} )^2 (\sup_{l' \not = 0} \sup_{0\leq \tau \leq t} r_{l'}(\tau)) \ls \zeta(t).
\end{aligned}
\end{equation}
Plugging this into \eqref{beforekeybound-I1} proves the $L^\infty_t$ estimates in \eqref{claim-2k-I}. 

\subsection*{$L^1_t$ bounds for $\mathcal{I}_1(t)$} In view of \eqref{beforekeybound-I1}, to prove the $L^1_t$ bound for $\mathcal{I}_1(t)$, it suffices to understand 
$$\int_0^t \Big[ \sum_{l \not = 0} (\int_0^s |l|^{-6} \langle l\tau \rangle^{-5/2} \la \nu^{1/3} (s-\tau)\ra^{-3} r_l(\tau) \, \d \tau)^{1/2}\Big]^2 \, \d s.$$

We split the $\tau$-integration: when $\tau \geq \max\{\f s2, 1\}$, we have $\la l\tau\ra^{-5/2} \la \nu^{1/3} (s - \tau) \ra^{-3} \ls \la l\tau\ra^{-5/2} \ls \la l\tau\ra^{-5/4} \la ls\ra^{-5/4} \ls \la l\tau\ra^{-5/4} \la s\ra^{-5/4}$; while when $\tau \leq \max\{\f s2, 1\}$, we have $\la l\tau\ra^{-5/2} \la \nu^{1/3} (s - \tau) \ra^{-3} \ls \la l \tau \ra^{-5/2} \la \nu^{1/3} s \ra^{-3}$. Hence,
\begin{equation*}
\begin{split}
&\: \int_0^t \Big[ \sum_{l \not = 0} (\int_0^s |l|^{-6} \langle l\tau \rangle^{-5/2} \la \nu^{1/3} (s-\tau)\ra^{-3} r_l(\tau) \, \d \tau)^{1/2}\Big]^2 \, \d s \\
\ls &\: \int_0^t \Big[ \sum_{l \not = 0} (\int_0^s |l|^{-6} \langle l \tau \rangle^{-5/4} r_l(\tau) \, \d \tau)^{1/2}\Big]^2  \langle s \rangle^{-5/4} \, \d s  \\
&\: + \int_0^t \Big[ \sum_{l \not = 0} (\int_0^s |l|^{-6} \langle l\tau \rangle^{-5/2}  r_l(\tau) \, \d \tau)^{1/2}\Big]^2 \la \nu^{1/3} s \ra^{-3} \, \d s \\
\ls &\: (\sum_{l \not = 0} |l|^{-7/2} )^2 (\sup_{l' \not = 0} \sup_{0\leq \tau \leq t} r_{l'}(\tau)) \Big[\int_0^t \langle s \rangle^{-5/4} \, \d s + \int_0^t \langle \nu^{1/3} s \rangle^{-3} \, \d s \Big] \ls \nu^{-1/3} \zeta(t).
\end{split}
\end{equation*}
Combining this with \eqref{beforekeybound-I1} yields the desired conclusion in \eqref{claim-2k-I}.

\subsection*{$L^\infty_t$ bounds for $\mathcal{I}_2(t)$} 

This is similar to $\mathcal{I}_1$, except that we use $\la k- l \ra^{-2}$ instead of $|l|^{-2}$ for summability. More precisely, we argue as in \eqref{beforekeybound-I1} except for distributing the $\ell^1$ and $\ell^2$ sums differently in the application of Young's convolution inequality, to obtain
\begin{equation}\label{beforebeforekeybound-I2}
\begin{aligned}
&\mathcal{I}_2(t) \\
&\le \sum_{k\not =0}
\Big[ \sum_{l \not =0}  \int_0^t|l|^{-1}\la k-l\ra^{-2} \langle l\tau \rangle^{-3} \la \nu^{1/3} (t-\tau)\ra^{-3/2} r^{1/2}_l(\tau) \| \hat f_{k-l}(\tau)\|_{\wtG_{N_{max}}}\; \d\tau\Big]^2
\\
&\ls  \sum_{k\not =0}
\Big[ \sum_{l \not =0}  (\int_0^t|l|^{-2} \langle l\tau \rangle^{-5/2} \la \nu^{1/3} (t-\tau)\ra^{-3}r_l(\tau) \, \d \tau)^{1/2}  \\
&\qquad\qquad\qquad\qquad\qquad\qquad\times (\int_0^t \langle \tau \rangle^{-7/2} \la k-l\ra^{-4} \| \hat f_{k-l}(\tau)\|_{\wtG_{N_{max}}}^2 \; \d\tau)^{1/2}\Big]^2 \\
&\ls \Big[ \sum_{l \not = 0} \int_0^t|l|^{-2} \langle l\tau \rangle^{-5/2}  \la \nu^{1/3} (t-\tau)\ra^{-3}r_l(\tau) \, \d \tau \Big] \Big[ \sum_{k} \la k\ra^{-2}(\int_0^t \la \tau\ra^{-7/2}\| \hat f_{k}(\tau)\|_{\wtG_{N_{max}}}^2 \; \d\tau)^{1/2} \Big]^2 \\
&\ls \Big[ \sum_{l \not = 0} \int_0^t|l|^{-2} \langle l\tau \rangle^{-5/2}  \la \nu^{1/3} (t-\tau)\ra^{-3} r_l(\tau) \, \d \tau \Big]  (\sum_{k} \int_0^t \la \tau\ra^{-7/2}\| \hat f_{k}(\tau)\|_{\wtG_{N_{max}}}^2 \; \d\tau),
\end{aligned}
\end{equation}
where in the very end, we used the Cauchy--Schwarz inequality for the sum in $k$.

Using \eqref{eq:bd.top.wtG}, we thus obtain
\begin{equation}\label{beforekeybound-I2}
\begin{split}
\mathcal{I}_2(t) \ls \ep \sum_{l \not = 0} \int_0^t|l|^{-2} \langle l\tau \rangle^{-5/2}  \la \nu^{1/3} (t-\tau)\ra^{-3} r_l(\tau) \, \d \tau,
\end{split}
\end{equation}

Finally, we use \eqref{def-zeta} to bound
\begin{equation}
\begin{split}
 \sum_{l \not = 0} \int_0^t|l|^{-2} \langle l\tau \rangle^{-5/2}  \la \nu^{1/3} (t-\tau)\ra^{-3}r_l(\tau) \, \d \tau 
\ls  ( \int_0^t \langle \tau \rangle^{-5/2} \, \d \tau )(\sup_{0\leq \tau \leq t} \sum_{l\not = 0}r_{l}(\tau)) \ls \zeta(t).
\end{split}
\end{equation}
Plugging this into \eqref{beforekeybound-I2} gives the desired conclusion in \eqref{claim-2k-I}.

\subsection*{$L^1_t$ bounds for $\mathcal{I}_2(t)$} 
We bound the integral in \eqref{beforekeybound-I2} using Fubini's theorem:
\begin{equation*}
\begin{split}
&\: \sum_{l \not = 0} \int_0^t \int_0^s |l|^{-2} \langle l\tau \rangle^{-5/2}  \la \nu^{1/3} (s-\tau)\ra^{-3}r_l(\tau) \, \d \tau \, \d s \\
&\: \ls \sum_{l \not = 0} \int_0^t  \la \tau \ra^{-5/2} (\int_{\tau}^t \la \nu^{1/3} (s-\tau)\ra^{-3} \, \d s) \, r_l(\tau) \, \d \tau 
\ls  \nu^{-1/3} \sum_{l \not = 0} \int_0^t \la \tau\ra^{-5/2} r_l(\tau) \, \d \tau.
\end{split}
\end{equation*}
Then, using H\"older's inequality and \eqref{def-zeta}, we obtain
\begin{equation*}
\begin{split}
\nu^{-1/3} \sum_{l \not = 0} \int_0^t \la \tau\ra^{-5/2} r_l(\tau) \, \d \tau \ls \nu^{-1/3} ( \int_0^t \langle \tau \rangle^{-5/2} \, \d \tau )(\sup_{0\leq \tau \leq t} \sum_{l\not = 0}r_{l}(\tau)) \ls \nu^{-1/3} \zeta(t).
\end{split}
\end{equation*}
Combining these two estimates with \eqref{beforekeybound-I2} yields the desired bound in \eqref{claim-2k-I}.

\subsection*{$L^\infty_t$ bounds for $\mathcal{I}_3(t)$} 

Next, we bound $\mathcal{I}_3$ in $L^\i_t$. Starting with \eqref{eq:def.Cj}, \eqref{def-II3} and then using $\|\hat{f}_{k-l}(\tau)\|_{\wtG_{N_{max}-2}}^2 \ls \ep$ (by \eqref{eq:G.by.E}), we have
\begin{equation}\label{keybound-I3}
\begin{aligned}
\mathcal{I}_3(t)
&\ls 
 \sum_{k\not =0}
\Big[ \sum_{l \not =0}  \int_0^t |l|^{-1} \la k - l\ra^{-4} \langle kt - l\tau \rangle^{-2} r^{1/2}_l(\tau) \| \hat f_{k-l}(\tau)\|_{\wtG_{N_{max}-2}}\; \d\tau\Big]^2 \\
&\ls \ep (\sup_{l' \not = 0} \sup_{\tau' \in [0,t]} r_{l'}(\tau'))
\sum_{k\not =0}
\Big[ \sum_{l \not =0} |l|^{-1} \la k - l\ra^{-4} \int_0^t  \langle kt - l\tau \rangle^{-2}  \, \d\tau\Big]^2 \\
&\ls 
\ep (\sup_{l' \not = 0} \sup_{\tau' \in [0,t]} r_{l'}(\tau'))
\sum_{k\not =0}
\Big[ \sum_{l \not =0}  |l|^{-2}  \la k - l\ra^{-4}  \Big]^2.
\end{aligned}
\end{equation}
Hence, Young's convolution inequality gives
$$\sum_{k\not =0}
\Big[ \sum_{l \not =0}  |l|^{-2}  \la k - l\ra^{-4}  \Big]^2 \ls  \Big[ \sum_{l \not =0}  |l|^{-4}  \Big] \Big[ \sum_{k}\la k \ra^{-4} \Big]^2 \ls 1.$$
Plugging this back into \eqref{keybound-I3} and using \eqref{def-zeta} yields
$$ \mathcal I_3(t) \ls \ep \zeta(t).$$

\subsection*{$L^1_t$ bounds for $\mathcal{I}_3(t)$} 

To bound the $L^1_t$ norm for $\mathcal{I}_3(t)$, we first use $\|\hat{f}_{k-l}(\tau)\|_{\wtG_{N_{max}-2}}^2 \ls \ep$ (by \eqref{eq:G.by.E}) to obtain
\begin{equation}
\begin{split}
&\: \int_0^t \mathcal{I}_3(s) \, \d s \\
\ls &\: \sum_{k\not =0} \int_0^t 
\Big[ \sum_{l \not =0}  \int_0^s |l|^{-1} \la k - l\ra^{-4} \langle ks - l\tau \rangle^{-2} r^{1/2}_l(\tau) \| \hat f_{k-l}(\tau)\|_{\wtG_{N_{max}-2}}\; \d\tau\Big]^2 \, \d s \\
\ls &\: \ep \sum_{k\not =0} \int_0^t 
\Big[ \sum_{l \not =0}  \int_0^s |l|^{-1} \la k - l\ra^{-4} \langle ks - l\tau \rangle^{-2}  r^{1/2}_l(\tau) \; \d\tau\Big]^2 \, \d s.
\end{split}
\end{equation}

By Schur's test, 
\begin{equation}
\begin{split}
  \int_0^t \mathcal I_3(s) \, \d s 
\ls &\:  \ep \Big[ \sup_{k \not = 0}  \sup_s \Big( \int_0^t  \sum_{l\not = 0} |l|^{-1} \la k - l\ra^{-4} \langle ks - l\tau \rangle^{-2}   \, \d \tau \Big)\Big] \\
&\: \times  \Big[ \sup_{l \not = 0}  \sup_\tau \Big( \int_0^t  \sum_{k\not = 0} |l|^{-1} \la k - l\ra^{-4} \langle ks - l\tau \rangle^{-2}   \, \d s \Big)\Big] \sum_{l' \not = 0} \int_0^t r_{l'}(\tau') \, \d \tau'.
\end{split}
\end{equation}
Each of the integrals can be easily checked to be bounded, so that by \eqref{def-zeta} we have
$$\int_0^t \mathcal I_3(s) \, \d s \ls \ep \sum_{l\not = 0} \int_0^t r_l(\tau) \, \d \tau \ls \ep \nu^{-1/3} \zeta(t).$$

\subsection{Nonlinear interaction II}

In this section, under the bootstrap assumption \eqref{bootstrap-f} on $f$, we bound 
$$
\begin{aligned}
\mathbb{III}_k(t) 
&= \nu \int_0^t\int_{\R^3}S_k(t-\tau) [ (\widehat{\Gamma(f,f)})_k(\tau) ]\sqrt{\mu}\d v\d \tau .
\end{aligned}$$
We will prove that 
\begin{equation}\label{claim-3k}
\begin{split}
&\: \sum_{k\not =0} |k|^{2N_1}\langle kt \rangle^{2N_2}|\mathbb{III}_k(t)|^2 + \nu^{1/3}\int_0^t \sum_{k\not =0}|k|^{2N_1}\langle k\tau \rangle^{2N_2}|\mathbb{III}_k(\tau)|^2 \d\tau 
\ls  
\ep^2 \nu^{2/3}.
\end{split}
\end{equation}
To prove \eqref{claim-3k}, we use \eqref{decay-combineSk} in Proposition~\ref{prop-mixdecay} for the semigroup $S_k(t-\tau$) with $\eta = k\tau$. Thus, for any $N_1$, $N_2$ such that $N_1 + N_2 \leq N_{max}$,
\begin{equation}\label{eq:3k.first.bounds}
\begin{aligned}
\Big| &\int_{\R^3}S_k(t-\tau) [ (\widehat{\Gamma(f,f)})_k(\tau) ]\sqrt{\mu}\d v\Big| 
\\
&\ls  \langle kt \rangle^{-N_2} \langle \nu^{1/3} (t-\tau) \rangle^{-3/2}  \| (\widehat{\Gamma(f,f)})_k(\tau) \|_{\mathbb{E}^{(10, 0)'}_{Landau, k,k\tau,N_2}} \\
&\ls |k|^{-N_1} \langle kt \rangle^{-N_2} \langle \nu^{1/3} (t-\tau) \rangle^{-3/2}\qquad  \smashoperator{ \sum_{\substack{|\alp| = N_1,\, |\bt|\leq 1,\, |\omega|\le N_2}}} \qquad \nu^{|\bt|/3} \|  \langle v\rangle^{10}  [ \rd_x^{\alp} \rd_v^\bt Y^{\omega}(\Gamma(f,f))]\sphat_k(\tau)  \|_{L^2_v}
\end{aligned}
\end{equation}
in which we used $Y_{0,k\tau} \widehat \Gamma_k(\tau) = \widehat{(Y\Gamma)}_k(\tau)$, recalling the vector field $Y = t \nabla_x + \nabla_v$.

By Lemma~\ref{lem:Gamma.trivial}, 
\begin{equation}\label{eq:Gamma.error.in.density}
\begin{split}
&\: \sum_{k\not = 0} \quad \qquad  \smashoperator{ \sum_{\substack{|\alp| = N_1,\, |\bt|\leq 1,\, |\omega|\le N_2}}} \qquad \nu^{2|\bt|/3} \|  \langle v\rangle^{10}  [ \rd_x^{\alp} \rd_v^\bt Y^{\omega}(\Gamma(f,f))]\sphat_k(\tau)  \|_{L^2_v}^2 \\
\ls &\: \sum_{\substack{ |\alp'| + |\alp''| = N_1, \,|\om'| + |\om''| \leq N_2\\ 1\leq |\bt'|+|\bt''| \leq 2}} \Big\| \|\la v \ra^{10} \rd_x^{\alp'} \rd_v^{\bt'} Y^{\om'} f \|_{L^2_{v}}  \|\la v \ra^{10} \rd_x^{\alp''} \rd_v^{\bt''} Y^{\om''} f \|_{L^2_{v}} \Big\|_{L^2_x}^2 \\
&\: + \sum_{\substack{ |\alp'| + |\alp''| = N_1, \,|\om'| + |\om''| \leq N_2\\ |\bt'|+|\bt''| = 3 }} \nu^{2/3} \Big\| \|\la v \ra^{10} \rd_x^{\alp'} \rd_v^{\bt'} Y^{\om'} f \|_{L^2_{v}}  \|\la v \ra^{10} \rd_x^{\alp''} \rd_v^{\bt''} Y^{\om''} f \|_{L^2_{v}} \Big\|_{L^2_x}^2.
\end{split}
\end{equation}

The two terms in \eqref{eq:Gamma.error.in.density} are to be controlled in appropriate $\wtbE$ and $\wtbD$ norms. We point out three important observations:
\begin{itemize}
\item Note that the $\bv$-weights in \eqref{eq:Gamma.error.in.density} are significantly lower than what is encoded in the energy and dissipation norms. 
\item One term could have $N_{max} + 3$ derivatives (when, say, $\alp'=\bt'=\om'=0$), but in that case exactly three of the derivatives must be $\rd_v$, so it can be controlled with the $\wtbD^{(0)}_{N_{max}}$ norm. 
\item In each term of \eqref{eq:Gamma.error.in.density}, at least one factor has at least one $\rd_v$ derivative. We put that factor in the $\wtbD$ norm, and the other factor in the $\wtbE$ so as not to incur a loss of $\nu^{-1/3}$.
\end{itemize}

We only consider the second term in \eqref{eq:Gamma.error.in.density}; the first term is similar and slightly simpler. 
\begin{itemize}
\item Take $|\alp'| + |\alp''| = N_1$, $|\om'| + |\om''| \leq N_2$ and $|\bt'|+|\bt''| = 3$. 
\item After switching $(\alp',\bt',\om')$ with $(\alp'',\bt'',\om'')$ if necessary, we assume without loss of generality $|\alp'|+|\bt'|+|\om'| \leq \lfloor (N_{max}+2)/2 \rfloor$. We can apply Sobolev embedding in $x$ to the corresponding term, noting that since $N_{max} \geq 9$, we have $\lfloor (N_{max}+2)/2 \rfloor + 2 \leq N_{max}-2$.
\item By the pigeonhole principle, $|\bt'| \geq 1$ or $|\bt''| \geq 1$, i.e.~$\rd_v^{\bt'} = \rd_{v_j} \rd_v^{\bt'-e_j}$ or $\rd_v^{\bt''} = \rd_{v_j} \rd_v^{\bt''-e_j}$.
\end{itemize}
We use H\"older's inequality and then the Sobolev inequality in $x$ to obtain
\begin{equation}\label{eq:estimating.Gamma.1}
\begin{split}
&\: \nu^{2/3} \Big\| \|\la v \ra^{10} \rd_x^{\alp'} \rd_v^{\bt'} Y^{\om'} f \|_{L^2_{v}}  \|\la v \ra^{10} \rd_x^{\alp''} \rd_v^{\bt''} Y^{\om''} f \|_{L^2_{v}} \Big\|_{L^2_x}^2 \\
\ls &\: \nu^{-4/3} \Big[(\smashoperator{\sum_{\substack{ |\alp'''|\leq 2 }}} \nu^{2|\bt'|/3} \|\la v \ra^{10} \rd_x^{\alp'''} \rd_x^{\alp'} \rd_v^{\bt'} Y^{\om'} f \|_{L^2_{x,v}}^2)  \times \smashoperator{\sum_{\tilde{\bt}'' =\bt'' - e_j}}\nu^{2|\bt''|/3} \|\la v \ra^{10} \rd_x^{\alp''} \rd_v^{\widetilde{\bt}''} Y^{\om''} f \|_{\Delta_{x,v}}^2 \\
&\: \qquad\quad +  \sum_{\substack{|\alp'''|\leq 2 \\ \\ \tilde{\bt}' = \bt' - e_j}} \nu^{2|\bt'|/3} \|\la v \ra^{10} \rd_x^{\alp'''} \rd_x^{\alp'} \rd_v^{\bt'-e_j} Y^{\om'} f \|_{\Delta_{x,v}}^2 \nu^{2|\bt''|/3} \|\la v \ra^{10} \rd_x^{\alp''} \rd_v^{\bt''} Y^{\om''} f \|_{L^2_{x,v}}^2 \Big] \\
\ls &\: \nu^{-4/3} (\| f \|_{\wtbE^{(0)}_{N_{max}-2}}^2 \| f\|_{\wtbD^{(0)}_{N_{max}}}^2 + \| f \|_{\wtbD^{(0)}_{N_{max}-2}}^2 \| f \|_{\wtbE^{(0)}_{N_{max}}}^2) \\
\ls &\: \ep \nu^{-2/3} \| f\|_{\wtbD^{(0)}_{N_{max}}}^2 + \ep \nu^{-4/3} \| f \|_{\wtbD^{(0)}_{N_{max}-2}}^2,
\end{split}
\end{equation}
where in the very last line we used the bootstrap assumption \eqref{bootstrap-f}. The first term in \eqref{eq:Gamma.error.in.density} can be bounded similarly so that we have
\begin{equation}\label{eq:estimating.Gamma}
\begin{aligned}
&\sum_{k\not = 0} \quad \qquad  \smashoperator{ \sum_{\substack{|\alp| = N_1,\, |\bt|\leq 1,\, |\omega|\le N_2}}} \qquad \nu^{2|\bt|/3} \|  \langle v\rangle^{10}  [ \rd_x^{\alp} \rd_v^\bt Y^{\omega}(\Gamma(f,f))]\sphat_k(\tau)  \|_{L^2_v}^2 \\
& \ls \ep \nu^{-2/3} \| f\|_{\wtbD^{(0)}_{N_{max}}}^2 + \ep \nu^{-4/3} \| f \|_{\wtbD^{(0)}_{N_{max}-2}}^2.
\end{aligned}
\end{equation}

We now square \eqref{eq:3k.first.bounds}, multiply it by $|k|^{2N_1} \la kt \ra^{2N_2}$, sum over $k$, and plug in \eqref{eq:estimating.Gamma}. Using the Cauchy--Schwarz inequality for the $\tau$ integral, we bound 
\begin{equation}\label{eq:3k}
\begin{split}
 \sum_{k \not = 0} & |k|^{2N_1} \langle kt \rangle^{2N_2}|\mathbb{III}_{k}(t)|^2 \\
&\le \nu^2 \sum_{k \not = 0} |k|^{2N_1} \langle kt \rangle^{2N_2} \Big[ \int_0^t\int_{\R^3}S_k(t-\tau) [ (\widehat{\Gamma(f,f)})_k(\tau) ]\sqrt{\mu}\d v\d \tau\Big]^2
\\
&\ls \Big( \ep \nu^{4/3}  \int_0^t \| f (\tau)\|_{\wtbD^{(0)}_{N_{max}}}^2\, \d \tau  + \ep \nu^{2/3} \int_0^t \| f (\tau)\|_{\wtbD^{(0)}_{N_{max}-2}}^2\, \d \tau \Big) \int_0^t \langle \nu^{1/3} (t-\tau) \rangle^{-3} \, \d \tau \\
&\lesssim \ep \nu  \int_0^t \| f (\tau)\|_{\wtbD^{(0)}_{N_{max}}}^2\, \d \tau  + \ep \nu^{1/3} \int_0^t \| f (\tau)\|_{\wtbD^{(0)}_{N_{max}-2}}^2\, \d \tau \ls \ep \nu^{2/3},
\end{split}
\end{equation}
where we have also used the bootstrap assumption \eqref{bootstrap-f} at the end. This gives the desired $L^\i$ bounds in \eqref{claim-3k}. The $L^2$ estimates in \eqref{claim-3k} also follow from essentially the same computation as \eqref{eq:3k}, after also using Fubini's theorem, namely,
\begin{equation}
\begin{split}
 &\int_0^t  \sum_{k \not = 0}  |k|^{2N_1} \langle ks \rangle^{2N_2}|\mathbb{III}_{k}(s)|^2 \, \d s\\
\ls &\:\ep \Big(  \nu^{4/3}  \int_0^t \int_0^s \langle \nu^{1/3} (s-\tau) \rangle^{-3/2} \| f (\tau)\|_{\wtbD^{(0)}_{N_{max}}}^2 \d \tau \, \d s \\
&\quad+ \nu^{2/3} \int_0^t \int_0^s \langle \nu^{1/3} (s-\tau) \rangle^{-3/2} \| f (\tau)\|_{\wtbD^{(0)}_{N_{max}-2}}^2 \d \tau \, \d s\Big) \sup_{0\leq s' \leq t} \int_0^{s'} \langle \nu^{1/3} (s'-\tau) \rangle^{-3/2}  \d \tau \\
\lesssim &\: \ep \nu^{2/3}  \int_0^t \| f (\tau)\|_{\wtbD^{(0)}_{N_{max}}}^2\, \d \tau  +  \ep \int_0^t \| f (\tau)\|_{\wtbD^{(0)}_{N_{max}-2}}^2\, \d \tau \ls \ep^2 \nu^{1/3}.
\end{split}
\end{equation}

\bigskip

 This ends the proof of Proposition \ref{prop-NNN}, and so that of Theorem \ref{theo-density}.

\section{Nonlinear energy estimates}\label{s.closing_eng}

In this section, we derive energy estimates for the full nonlinear Vlasov--Poisson--Landau equation \eqref{eq:Vlasov.f}--\eqref{eq:Poisson.f} under the bootstrap assumptions \eqref{bootstrap-f} and \eqref{bootstrap-phi}.

\bigskip

The main result of this section is the following. 

\begin{theorem}\label{theo-mainEE} Consider data as in Theorem~\ref{t.main}. Suppose there exists $T_B>0$ such that the solution $f$ to \eqref{eq:Vlasov.f}--\eqref{eq:Poisson.f} remains smooth in $[0,T_B)\times \T^3 \times \R^3$ and satisfies the bootstrap assumptions \eqref{bootstrap-f} and \eqref{bootstrap-phi}.  

Then, for $\vartheta \in \{0,2\}$, $0\le t<T_B$, and $0\leq N \leq N_{max}$, the following energy estimate holds:
\begin{equation}\label{mainEE-close}
\begin{aligned}
 \sup_{0\leq \tau \le t} \| f(\tau) \|^2_{\wtbE^{(\vartheta)}_N}+  \nu^{1/3} \int_0^{t} \| f(\tau) \|^2_{\wtbD^{(\vartheta)}_N} \, \d \tau \quad&\lesssim\quad 
\ep^2 \nu^{2/3} \min \{ \nu^{-1/3}, \la t\ra \}^{\max\{0, N-N_{max}+2\}},
 \end{aligned}\end{equation}
 where $\|\cdot \|_{\wtbE^{(\vartheta)}_N}$ and $\|\cdot \|_{\wtbD^{(\vartheta)}_N}$ are the global energy and dissipation norms defined as in \eqref{eq:def.EN.DN}.  
\end{theorem}

The main step in the proof of Theorem~\ref{theo-mainEE} is the following estimates for the inhomogeneous terms $\widetilde{\mathcal R}_{\alp,\bt,\om}$:
\begin{proposition}\label{prop:energy.error}
Fix $\vartheta \in \{0,2\}$. Under the assumptions of Theorem~\ref{theo-mainEE}, for $|\alp|+ |\bt| + |\om|\leq N$,
\begin{equation}
\begin{split}
\widetilde{\mathcal R}_{\alp,\bt,\om} 
\ls &\: \ep^{1/2} \la t\ra^{-2} \| f(t)\|^2_{\wtbE^{(\vartheta)}_N} + \ep^{1/2} \nu^{2/3} \| f(t)\|_{\wtbD^{(\vartheta)}_N}^2 + \nu^{1/3} \|f(t)\|_{\wtbE^{(\vartheta)}_N} \|f(t) \|_{\wtbD^{(\vartheta)}_N} \|f(t) \|_{\wtbD^{(\vartheta)}_{N_{max}-2}} \\
&\: + \min\{ \| f(t)\|_{\wtbD^{(\vartheta)}_N},\, \|f(t) \|_{\wtbE^{(\vartheta)}_N} \}  \sum_{|\alp|+|\om|\leq N} \| \rd_x^\alp Y^\om \rho_{\not =0} (t)\|_{L^2_x},
\end{split}
\end{equation}
where $\widetilde{\mathcal R}_{\alp,\bt,\om}$ are as defined in Proposition~\ref{prop:top.order.energy}.
\end{proposition}

We now show that Proposition~\ref{prop:energy.error} implies Theorem~\ref{theo-mainEE}. 

\begin{proof}[Proof of Theorem~\ref{theo-mainEE} assuming Proposition~\ref{prop:energy.error}] 
Computing $\frac{d}{dt} \| f \|^2_{\wtbE_N^{(\vartheta)}}$ and using the main energy estimates in Proposition~\ref{prop:top.order.energy}, we obtain
\begin{equation}\label{eq:main.nonlinear.energy.prelim}
\begin{split}
&
\: \frac{d}{dt} (e^{1+\la t\ra^{-1}} \| f \|^2_{\wtbE^{(\vartheta)}_N} )+ \theta \nu^{1/3} e^{1+\la t\ra^{-1}}\| f \|^2_{\wtbD^{(\vartheta)}_N}  
\\
= &\: \sum_{N_{\alp,\bt} + N_{\om} \leq N} e^{1+\la t\ra^{-1}} \Big( \frac{d}{dt} \| f \|^2_{\widetilde{\mathcal{E}}_{0,N_{\alp,\bt},N,N_\om}}+ \theta \nu^{1/3} \| f \|^2_{\widetilde{\mathcal{D}}_{0,N_{\alp,\bt},N,N_\om}} \Big) - \la t\ra^{-2} e^{1+\la t\ra^{-1}} \| f \|^2_{\wtbE^{(\vartheta)}_N} 
\\
\le &\: -  \la t\ra^{-2} \| f \|^2_{\wtbE^{(\vartheta)}_N} + C \Big(\nu |(\bar{a},\bar{b},\bar{c})|^2 + (\|\rd_t \phi\|_{L^\i_x} + \|\phi\|_{W^{1,\infty}_x})  \| f \|^2_{\wtbE^{(\vartheta)}_N} + \sum_{|\alp|+|\bt|+|\om|\leq N} \widetilde{R}_{\alp,\bt,\om} \Big) ,
\end{split}
\end{equation}
noting $1\le e^{1+\la t\ra^{-1}} \le e^2$. We note that the $e^{1+\jap{t}^{-1}}$ weight is used instead of Gr\"{o}nwall's inequality to absorb the linear terms.

Since $f$ satisfies the nonlinear Vlasov--Poisson--Landau system, the conservation law and \eqref{eq:thm.assumption.1} imply
\begin{equation}\label{cv}
\iint_{\T^3 \times \R^3} f \sqrt \mu \; \ud v \, \ud x = \iint_{\T^3\times \R^3} v_j f \sqrt\mu \; \ud v \, \ud x =  \iint_{\T^3\times \R^3} |v|^2 f \sqrt\mu \; \ud v \, \ud x + \int_{\R^3} |E|^2 \, \ud x  =0.
\end{equation}
In other words, $\bar{a}=0$ and $\bar{b} = 0$. Moreover, using \eqref{bootstrap-phi}, we have $|\bar{c}| \leq  \int_{\R^3} |E|^2 \, \ud x \ls \ep \nu^{2/3}\la t\ra^{-2}$. Additionally, the bootstrap assumption \eqref{bootstrap-phi} implies that the third term in \eqref{eq:main.nonlinear.energy.prelim} is bounded $\ls \ep^{1/2} \nu^{1/3} \la t\ra^{-2}\| f \|^2_{\wtbE^{(\vartheta)}_N}$. Plugging in also the estimates for $\widetilde{R}_{\alp,\bt,\om}$ from Proposition~\ref{prop:energy.error}, we obtain
\begin{equation}\label{eq:main.nonlinear.energy.better}
\begin{split}
&\: \frac{d}{dt} (e^{1+\la t\ra^{-1}} \| f \|^2_{\wtbE^{(\vartheta)}_N} )+ \theta \nu^{1/3} \| f \|^2_{\wtbD^{(\vartheta)}_N}  + \la t\ra^{-2} \| f \|^2_{\wtbE^{(\vartheta)}_N} 
\\
\ls &\:\ep^2 \nu^{4/3} \la t\ra^{-4} +  \ep^{1/2} \la t\ra^{-2} \| f(t)\|^2_{\wtbE^{(\vartheta)}_N} + \ep^{1/2} \nu^{2/3}\| f(t)\|_{\wtbD^{(\vartheta)}_N}^2 \\
&\: + \nu^{1/3} \|f(t)\|_{\wtbE^{(\vartheta)}_N} \|f(t) \|_{\wtbD^{(\vartheta)}_N} \|f(t) \|_{\wtbD^{(\vartheta)}_{N_{max}-2}} \\
&\: + \min\{ \| f(t)\|_{\wtbD^{(\vartheta)}_N},\, \|f(t) \|_{\wtbE^{(\vartheta)}_N} \}  \quad \smashoperator{ \sum_{|\alp|+|\om|\leq N} }\| \rd_x^\alp Y^\om \rho_{\not =0} (t)\|_{L^2_x}.
\end{split}
\end{equation}
For $\ep$ sufficiently small, the second and third terms on the right hand side of \eqref{eq:main.nonlinear.energy.better} can be absorbed by the last two terms on the left hand side, i.e.
\begin{equation}\label{eq:main.nonlinear.energy.absorbed}
\begin{split}
&\: \frac{d}{dt} (e^{1+\la t\ra^{-1}} \| f \|^2_{\wtbE^{(\vartheta)}_N})+ \theta \nu^{1/3} \| f \|^2_{\wtbD^{(\vartheta)}_N} + \la t\ra^{-2} \| f \|^2_{\wtbE^{(\vartheta)}_N}\\
\ls &\: \ep^2 \nu^{4/3} \la t\ra^{-4} +   \nu^{1/3} \|f(t)\|_{\wtbE^{(\vartheta)}_N} \|f(t) \|_{\wtbD^{(\vartheta)}_N} \|f(t) \|_{\wtbD_{N_{max}-2}}  \\
&\: + \min\{ \| f(t)\|_{\wtbD^{(\vartheta)}_N},\, \|f(t) \|_{\wtbE^{(\vartheta)}_N} \}  \smashoperator{ \sum_{|\alp|+|\om|\leq N} }\| \rd_x^\alp Y^\om \rho_{\not =0} (t)\|_{L^2_x}.
\end{split}
\end{equation}

Define now
$$\mathcal F_N[f](t) := \|f(t)\|_{\wtbE^{(\vartheta)}_N}^2 + \nu^{1/3} \int_0^t \|f(\tau)\|^2_{\wtbD^{(\vartheta)}_N} \, \d \tau + \int_0^t \la \tau\ra^{-2} \|f(\tau)\|_{\wtbE^{(\vartheta)}_N}^2\, \d \tau.$$
Thus $\mathcal F_N[f](t)$ can be bounded in terms of the $t$-integral of the right hand side of \eqref{eq:main.nonlinear.energy.absorbed}. This will in turn be controlled below for different values of $N$.

\textbf{The case $N \leq N_{max}-2$.} Consider first the case $N\leq N_{max} -2$. Using \eqref{bootstrap-f}, we have $\| f\|_{\wtbE^{(\vartheta)}_{N_{max}-2}} \ls \ep^{1/2} \nu^{1/3}$. Thus, 
\begin{equation}\label{eq:error.Nmax-2.1}
\begin{split}
&\: \nu^{1/3} \int_0^t \|f(\tau)\|_{\wtbE^{(\vartheta)}_N} \|f(\tau) \|_{\wtbD^{(\vartheta)}_N} \|f(\tau) \|_{\wtbD^{(\vartheta)}_{N_{max}-2}} \, \d \tau \\
\ls &\: \ep \nu^{2/3} \int_0^t \|f(\tau) \|_{\wtbD^{(\vartheta)}_{N_{max}-2}}^2 \, \d \tau \ls \ep \nu^{1/3} \mathcal F_N[f](t).
\end{split}
\end{equation}

On the other hand, since $\rho_{\not = 0}$ is $v$-independent and has vanishing $x$-mean by definition, Poincar\'e's inequality implies that for $N \leq N_{max} -2$,
$$\smashoperator{ \sum_{|\alp|+|\om|\leq N} }\| \rd_x^\alp Y^\om \rho_{\not =0} (t)\|_{L^2_x} \ls \la t\ra^{-2}\sum_{|\alp'|=2} \sum_{|\alp|+|\om|\leq N} \|\la t\ra^{2}\rd_x^{\alp'} \rd_x^\alp Y^\om \rho_{\not =0} (t)\|_{L^2_x} \ls \la t\ra^{-2}\smashoperator{\sum_{|\alp|+|\om|\leq N_{max}}} \|\rd_x^\alp Y^\om \rho_{\not =0} (t)\|_{L^2_x}.$$
Thus, using \eqref{density-bound}  and Young's inequality, we obtain, for any $\eta>0$,
\begin{equation}\label{eq:error.Nmax-2.2}
\begin{split}
&\: \int_0^t \min\{ \| f(\tau)\|_{\wtbD^{(\vartheta)}_N},\, \|f(\tau) \|_{\wtbE^{(\vartheta)}_N} \}  \smashoperator{ \sum_{|\alp|+|\om|\leq N} }\| \rd_x^\alp Y^\om \rho_{\not =0} (\tau)\|_{L^2_x} \, \d \tau \\
\ls &\: \int_0^t  \|f(\tau) \|_{\wtbE^{(\vartheta)}_N} \la \tau \ra^{-2} \ep \nu^{1/3} \, \d \tau \ls \eta \mathcal F_N[f](t) + \eta^{-1} \ep^2 \nu^{2/3} \int_0^\infty \f{\d \tau}{\la \tau\ra^2} \\
\ls &\: \eta \mathcal F_N[f](t) + \eta^{-1} \ep^2 \nu^{2/3}.
\end{split}
\end{equation}

Plugging \eqref{eq:error.Nmax-2.1} and \eqref{eq:error.Nmax-2.2} into \eqref{eq:main.nonlinear.energy.absorbed}, and bounding the initial data term by \eqref{eq:thm.assumption.2}, we thus obtain
$$\mathcal F_N[f](t) \ls (\ep \nu^{1/3}+ \eta)\mathcal F_N[f](t) + \eta^{-1}\ep^2 \nu^{2/3}.$$
Choosing $\ep_0$, $\nu_0$ and $\eta$ sufficiently small, we can absorb the first term on the right to the left, giving the desired bound for $\mathcal F_N(t)$.

\textbf{The case $N_{max} -1 \leq N \leq N_{max}$.} We consider the case $N = N_{max}$; the case $N = N_{max}-1$ is similar. Note that we need to prove two estimates: one allowing for a loss in $\nu^{-2/3}$, and the other allowing for a growth in $\la t\ra^2$.

As above, we will bound the time-integral of the terms in \eqref{eq:main.nonlinear.energy.absorbed}, now for $N= N_{max}$. First, 
\begin{equation}\label{eq:error.Nmax.1}
\begin{split}
&\: \nu^{1/3} \int_0^t \|f(\tau)\|_{\wtbE^{(\vartheta)}_{N_{max}}} \|f(\tau) \|_{\wtbD^{(\vartheta)}_{N_{max}}} \|f(\tau) \|_{\wtbD^{(\vartheta)}_{N_{max}-2}} \, \d \tau \\
\ls &\: (\sup_{0\leq s<t} \|f(\tau)\|_{\wtbE^{(\vartheta)}_{N_{max}}}) (\nu^{1/3} \int_0^t \|f(\tau) \|_{\wtbD^{(\vartheta)}_{N_{max}}}^2 \, \d \tau)^{1/2} (\nu^{1/3} \int_0^t \|f(\tau) \|_{\wtbD^{(\vartheta)}_{N_{max}-2}}^2 \, \d \tau)^{1/2}\\
 \ls &\: \mathcal F_{N_{max}}[f](t) \mathcal F_{N_{max}-2}[f]^{1/2}(t) \ls \ep\nu^{1/3}\mathcal F_{N_{max}}[f](t),
\end{split}
\end{equation}
where in the final estimate, we used the bound for $\mathcal F_{N_{max}-2}[f](t)$ derived above.

We have two ways for bounding the other term. Using the $L^\i_t$ bound in \eqref{density-bound}, we have
\begin{equation}\label{eq:error.Nmax.2}
\begin{split}
&\: \int_0^t \min\{ \| f(\tau)\|_{\wtbD^{(\vartheta)}_{N_{max}}},\, \|f(\tau ) \|_{\wtbE^{(\vartheta)}_{N_{max}}} \} \quad  \smashoperator{ \sum_{|\alp|+|\om|\leq N_{max}} }\| \rd_x^\alp Y^\om \rho_{\not =0} (\tau)\|_{L^2_x} \, \d \tau \\
\ls &\: \int_0^t  \|f(\tau) \|_{\wtbE^{(\vartheta)}_{N_{max}}}  \ep \nu^{1/3} \, \d \tau  
\ls \eta \mathcal F_{N_{max}}[f](t) + \eta^{-1} \ep^2 \nu^{2/3} (\int_0^t \, \d \tau)^2 \\
\ls &\: \eta \mathcal F_{N_{max}}[f](t) + \eta^{-1} \ep^2 \nu^{2/3} \la t \ra^2.
\end{split}
\end{equation}
and, using instead the $L^2_t$ bound in \eqref{density-bound}, we obtain
\begin{equation}\label{eq:error.Nmax.3}
\begin{split}
&\: \int_0^t \min\{ \| f(\tau)\|_{\wtbD^{(\vartheta)}_{N_{max}}},\, \|f(\tau) \|_{\wtbE^{(\vartheta)}_{N_{max}}} \} \quad  \smashoperator{ \sum_{|\alp|+|\om|\leq N_{max}} }\| \rd_x^\alp Y^\om \rho_{\not =0} (\tau)\|_{L^2_x} \, \d \tau \\
\ls &\: \eta\nu^{1/3} \int_0^t  \|f(\tau) \|_{\wtbD^{(\vartheta)}_{N_{max}}}^2 \, \d \tau   + \eta^{-1} \nu^{-1/3} \int_0^t \qquad \smashoperator{ \sum_{|\alp|+|\om|\leq N_{max}} }\quad \| \rd_x^\alp Y^\om \rho_{\not =0} (\tau)\|_{L^2_x}^2 \, \d \tau \\
\ls &\: \eta \mathcal F_{N_{max}}[f](t) + \eta^{-1} \ep^2.
\end{split}
\end{equation}

Combining \eqref{eq:error.Nmax.1}--\eqref{eq:error.Nmax.3}, integrating \eqref{eq:main.nonlinear.energy.absorbed}, and controlling initial data by \eqref{eq:thm.assumption.2}, we thus obtain
$$\mathcal F_{N_{max}}[f](t) \ls (\ep\nu^{1/3} + \eta )\mathcal F_{N_{max}}[f](t) + \eta^{-1} \ep^2 \min \{ \nu^{2/3} \la t \ra^2, 1\}.$$
We can thus conclude as before by choosing $\ep_0$, $\nu_0$ and $\eta$ small. \qedhere \end{proof}

The remainder of this section is thus devoted to the proof of Proposition~\ref{prop:energy.error}, after some preliminary bounds on the electric field in the next subsection.

\subsection{Bounds on the electric field}

In this section, we give estimates on the electric field.  

\begin{lemma}\label{lem-E1} Let $0\leq N\leq N_{max}$. Then
\begin{equation}\label{bd-E.L2}
\sum_{|\alp| + |\omega|\le N+2} \|\rd_x^\alp Y^\omega \phi (t)\|_{L^2_x} + \sum_{|\alp| + |\omega|\le N+1} \| \partial_x^{\alp} Y^\omega E(t)\|_{L^2_x} \lesssim \sum_{|\alp| + |\omega|\le N}\| \rd_x^\alp Y^\omega \rho_{\not = 0}(t)\|_{L^2_x}
\end{equation}
and
\begin{equation}\label{bd-E.Li}
\sum_{|\alp| + |\omega|\le N}\|\rd_x^\alp Y^\omega \phi (t)\|_{L^\infty_x} + \sum_{|\alp| + |\omega|\le N-1}  \|\rd_x^\alp Y^{\omega} E(t)\|_{L^\infty_x} \ls \sum_{|\alp| + |\omega|\le N}\| \rd_x^\alp Y^\omega \rho_{\not = 0}(t)\|_{L^2_x}.
\end{equation} 
\end{lemma}
\begin{proof} The estimate \eqref{bd-E.L2} follows from the Poisson equation $\hat\phi_k = |k|^{-2} \hat\rho_k$ and the definition $\hat E_k = -ik \hat \phi_k$ for each Fourier mode $k \in \Z^3\backslash\{0\}$. The bound \eqref{bd-E.Li} then follows from Sobolev embedding. \qedhere
\end{proof}

\begin{lemma}\label{lem-E2} The electric potential $\phi$ satisfies 
\begin{equation}\label{bd-dtphi} \norm{\partial_t\phi}_{L^\infty_x}\lesssim \langle t\rangle^{-2}\sum_{|\alp| + |\omega|\le 3}\|\rd_x^{\alp} Y^\omega f\|_{L^2_{x,v}}.
\end{equation}
\end{lemma}
\begin{proof} 
Since $\int_{\T^3} \rd_t\phi \,\d x = 0$, we apply Poincar\'e's and Sobolev's inequalities to obtain
$$\| \rd_t \phi \|_{L^\i_x} \ls \sum_{|\alp|+|\om|\leq 2} \la t\ra^{-2} \| \rd_x^{\alp} Y^\om \rd_t \phi \|_{L^\i_x} \ls \la t \ra^{-2} \sum_{|\alp'|\leq 2} \sum_{|\alp|+|\om|\leq 2} \| \rd_x^{\alp+\alp'} Y^\om \rd_t \phi \|_{L^2_x}.$$
To control the final $L^2$ norm by the right hand side of \eqref{bd-dtphi}, we use the elliptic equation for $\partial_t\phi$:
$$ - \Delta \partial_t \phi = \partial_t \rho = -\nabla_x\cdot \int_{\R^3} v f\sqrt{\mu}\d v,$$ 
where the last identity is the conservation of mass. 
\end{proof}

\subsection{Estimates on $\mathcal{R}^{Q,\ell}_{\alpha,\beta,\omega}$}

In this subsection, we give the claimed bounds on $\widetilde{\mathcal{R}}^{Q,\ell}_{\alpha,\beta,\omega}$ that appear in Proposition~\ref{prop:top.order.energy} in terms of energy and dissipation norms, under the assumptions of Theorem~\ref{theo-mainEE}. (The bounds for $\mathcal{Z}^{Q,\ell}_{\alpha,\beta,\omega}$ will be derived later in Section~\ref{sec:Z}.) Precisely, we 
shall bound
$$
\begin{aligned}
\sum_{|\alp'|\leq 1} \nu^{2|\beta|/3} \mathcal{R}^{Q,\ell-2|\alpha'|}_{\alpha+\alpha',\beta,\omega} + \sum_{|\bt'|\leq 2} \nu^{2(|\beta|+|\beta'|)/3} \mathcal{R}^{Q,\ell-2|\beta'|}_{\alpha+,\beta+\beta',\omega} 
\end{aligned}$$
for $|\alpha|+|\beta|+|\omega| = N_{max}$. We recall from Lemma \ref{lem-basicEEdx} and \eqref{def-QQQ} that 
$$
\begin{aligned}
\mathcal{R}^{Q,\ell}_{\alpha,\beta,\omega} 
&= \iint_{\T^3\times \R^3} e^{2(q+1)\phi}w^2\partial_x^\alpha \partial_v^\beta Y^\omega f \Big[ \partial_x^\alpha\partial_v^\beta Y^\omega[E_j \partial_{v_j}\sqrt\mu] + [E\cdot \nabla_v - E\cdot v, \partial_x^\alpha \partial_v^\beta Y^\omega] f 
\\&\quad + \nu \partial_x^\alpha\partial_v^\beta Y^\omega \Gamma(f,f)
\Big]\; \ud v \, \ud x 
\\& =: \mathcal{R}^{Q,\ell,1}_{\alpha,\beta,\omega} +  \mathcal{R}^{Q,\ell,2}_{\alpha,\beta,\omega} 
 +  \mathcal{R}^{Q,\ell,3}_{\alpha,\beta,\omega} 
\end{aligned}
$$
in which $ \mathcal{R}^{Q,\ell,j}_{\alpha,\beta,\omega} $ correspond to the integral involving each term in the bracket. The claimed estimates on $\widetilde{\mathcal{R}}^{Q,\ell}_{\alpha,\beta,\omega}$ in Proposition~\ref{prop:energy.error} are thus a combination of Lemmas~\ref{lem-Edvmu}--\ref{lem-RGamma} below giving bounds on each of these integral terms. 

Before we proceed, let us remark that since the $\|\phi \|_{L^\i_x}\ls \ep^{1/2} \nu^{1/3}\la t\ra^{-2} \ls \ep^{1/2} $ by \eqref{bootstrap-phi}, we can replace any factors of $e^{(q+1)\phi}$ by $1$  (and vice versa) without changing the bounds.

\begin{lemma}\label{lem-Edvmu}
For $|\alpha|+|\beta|+|\omega| \leq N$, we have 
$$\begin{aligned}
\nu^{2(|\beta|+|\beta'|)/3} \mathcal{R}^{Q,\ell-2|\alpha'|-2|\beta'|,1}_{\alpha+\alpha',\beta+\beta',\omega}  \quad 
&\lesssim \quad 
\min\{\|f(t) \|_{\wtbE^{(\vartheta)}_N}, \| f(t)\|_{\wtbD^{(\vartheta)}_N} \}  \sum_{|\alp''| + |\om''|\leq N} \|\rd_x^{\alp''} Y^{\om''} \rho_{\not =0} \|_{L^2_x}
\\&\quad + \ep^{1/2} \la t \ra^{-2} \|f \|_{\wtbE^{(\vartheta)}_N}^2
\end{aligned}$$
when either (1) $|\alp'|\leq 1$ and $\bt'=0$, or (2) $\alp'=0$, $|\bt'|\leq 2$.
\end{lemma}

\begin{proof} Let us consider only the case $|\alp'|\leq 1$ and $\bt'=0$. The other case is similar after noting that the $\wtbD^{(\vartheta)}_N$ norm by definition controls the corresponding term with more $\rd_v$ derivative and that $\rho$ is independent of $v$.

We compute 
$$ \rd_x^{\alp'}\partial_x^\alpha\partial_v^\beta Y^\omega[E_j\partial_{v_j}\sqrt \mu]  = \sum_{\omega' + \omega'' = \omega} \rd_x^{\alp'} \partial_x^\alpha Y^{\omega'} E_j \partial_{v_j}\partial_v^{\beta + \omega''} \sqrt\mu.  $$
Notice that $\partial_{v_j}\partial_v^{\beta + \omega''} \sqrt\mu$ decays rapidly in $v$. 

\textbf{Bounding with the $\wtbE^{(\vartheta)}_N$ norm.} Using the fact that $E$ gains one derivative over $\rho_{\not = 0}$, the above computations and the Cauchy--Schwarz inequality implies
$$
\begin{aligned}
\nu^{2|\beta|/3}  \mathcal{R}^{Q,\ell - 2|\alpha'|,1}_{\alpha+\alp',\beta,\omega} & \le
\nu^{2|\beta|/3} \Big| \iint_{\T^3\times \R^3} e^{2(q+1)\phi}w^2 \partial_x^{\alpha+\alp'} \partial_v^\beta Y^\omega f \partial_x^{\alpha+\alp'}\partial_v^\beta Y^\omega[E_j \partial_{v_j}\sqrt \mu] \; \ud v \, \ud x \Big|
\\&\lesssim 
\nu^{2|\beta|/3} \|\partial_x^{\alpha+\alp'} \partial_v^\beta Y^\omega f \|_{L^2_{x,v}}  \sum_{|\alp'|\leq 1} \sum_{|\alp''|+ |\om''|\leq N}\| \partial_x^{\alpha''+\alp'} Y^{\omega''} E_j \|_{L^2_x} 
\\&\lesssim \| f\|_{\wtbE^{(\vartheta)}_N} \sum_{|\alp''|+ |\om''|\leq N}\| \partial_x^{\alpha''} Y^{\omega''} \rho_{\not =0}\|_{L^2_x}.
\end{aligned}
$$

\textbf{Bounding with the $\wtbD^{(\vartheta)}_N$ norm.} When $|\alp'| = 1$, we can thus use the $\wtbD^{(\vartheta)}_N$ norm to control $\rd_x^{\alp+\alp'} \rd_v^\bt Y^\om f$ so that
$$
\begin{aligned}
\nu^{2|\beta|/3}  \mathcal{R}^{Q,\ell-2,1}_{\alpha+\alp',\beta,\omega} & \le
\nu^{2|\beta|/3} \Big| \iint_{\T^3\times \R^3} e^{2(q+1)\phi}w^2 \partial_x^{\alpha+\alp'} \partial_v^\beta Y^\omega f \partial_x^{\alpha+\alp'}\partial_v^\beta Y^\omega[E_j \partial_{v_j}\sqrt \mu] \; \ud v \, \ud x \Big|
\\&\lesssim 
\nu^{2|\beta|/3} \|\partial_x^{\alpha+\alp'} \partial_v^\beta Y^\omega f \|_{L^2_{x,v}}  \sum_{|\alp'|\leq 1} \sum_{|\alp''|+ |\om''|\leq N}\| \partial_x^{\alpha''+\alp'} Y^{\omega''} E_j \|_{L^2_x} 
\\&\lesssim \| f\|_{\wtbD^{(\vartheta)}_N} \sum_{|\alp''|+ |\om''|\leq N}\| \partial_x^{\alpha''} Y^{\omega''} \rho_{\not =0}\|_{L^2_x},
\end{aligned}
$$
where we have used the definition of $\wtbD^{(\vartheta)}_N$ to bound $f$, and used \eqref{bd-E.L2} to bound $E$.

When $\alp' = 0$, note that directly bound the term with the $\wtbD^{(\vartheta)}_N$ norm would cause a loss of $\nu^{-1/3}$. Instead, we integrate by parts in $x$: recalling $E =- \nabla_x \phi$, we get 
$$
\begin{aligned}
\nu^{2|\beta|/3}  \mathcal{R}^{Q,\ell,1}_{\alpha,\beta,\omega} & = 
 \nu^{2|\beta|/3}  \iint_{\T^3\times \R^3} e^{2(q+1)\phi}w^2\partial_{x_j}\partial_x^\alpha \partial_v^\beta Y^\omega f \partial_x^{\alpha} \partial_v^\beta Y^\omega[\phi \partial_{v_j}\sqrt \mu] \; \ud v \, \ud x 
\\& \quad + 2(q+1)\nu^{2|\beta|/3}  \iint_{\T^3\times \R^3} \partial_{x_j} \phi e^{2(q+1)\phi}w^2 \partial_x^{\alpha} \partial_v^\beta Y^\omega f \partial_x^\alpha\partial_v^\beta Y^\omega[\phi \partial_{v_j}\sqrt \mu] \; \ud v \, \ud x .
\end{aligned}
$$
Using \eqref{bd-E.L2}, the bootstrap assumption \eqref{bootstrap-phi} on $\|\partial_x\phi\|_{L^\infty_x}$ and the fact that $\mu = e^{-|v|^2}$ decays rapidly, the second integral is clearly bounded by $\epsilon\nu^{1/3}  \langle t\rangle^{-2 } \| f(t)\|^2_{\mathcal{E}_N} $. As for the first integral term, we use the rapid decay in $\la v\ra$, H\"older's inequality and \eqref{bd-E.L2} to bound 
 $$
\begin{aligned}
\nu^{2|\beta|/3} &\Big| \iint_{\T^3\times \R^3} e^{2(q+1)\phi}w^2\partial_{x_j}\partial_x^\alpha \partial_v^\beta Y^\omega f \partial_x^\alpha\partial_v^\beta Y^\omega[\phi \partial_{v_j}\sqrt \mu] \; \ud v \, \ud x \Big|
\\&\lesssim 
\nu^{2|\beta|/3} \|\partial_{x_j} \partial_x^\alpha \partial_v^\beta Y^\omega f \|_{L^2_{x,v}}  \sum_{|\omega'|\le |\omega|}\| \partial_x^\alpha Y^{\omega'} \phi \|_{L^2_x} 
\\&\lesssim \nu^{|\beta|/3}\| f\|_{\wtbD^{(\vartheta)}_N}  \sum_{|\alp''|+ |\omega'|\le N} \| \rd_x^{\alp''}Y^{\omega''} \rho\|_{L^2_x},
\end{aligned}
$$
giving the lemma. 
\qedhere
\end{proof}

\begin{lemma}\label{lem-EdvY}
For $|\alpha|+|\beta|+|\omega| \leq N$, we have 
\begin{equation}\label{eq:EdvY}
\begin{aligned}
\nu^{2(|\beta|+|\beta'|)/3} \mathcal{R}^{Q,\ell-2|\alpha'|-2|\beta'|,2}_{\alpha+\alpha',\beta+\beta',\omega}  
\quad&\lesssim\quad 
\ep  \min \{ \|f(t) \|_{\wtbE^{(\vartheta)}_N} , \| f(t)\|_{\wtbD^{(\vartheta)}_N} \} \sum_{|\alp''| + |\om''|\leq N} \|\rd_x^{\alp''} Y^{\om''} \rho_{\not =0} \|_{L^2_x}
\\&\quad + \ep^{1/2} \la t \ra^{-2} \|f \|_{\wtbE^{(\vartheta)}_N}^2
\end{aligned}
\end{equation}
when either (1) $|\alp'|\leq 1$ and $\bt'=0$, or (2) $\alp'=0$, $|\bt'|\leq 2$.
\end{lemma}

\begin{proof} 
Take $|\alp| + |\bt| + |\om| \leq N$. To avoid notational confusion, we consider only the case $\alp'=0$ and $\bt'=0$; the other cases are almost identical upon using the higher derivative control of $\wtE_{\alp,\bt,\om}$ and $\wtD_{\alp,\bt,\om}$ and the fact that $E$ gains one $\rd_x$ derivative over $\rho_{\not =0}$ (see \eqref{bd-E.L2}). 

Let us start by estimating the integral involving $E \cdot \nabla_v$. By definition, we compute 
\begin{equation}\label{comm-E}
|[E_j\part_{v_j},\der]f | \ls 
 \sum_{\substack{|\alpha'''|+|\alpha''|=|\alpha|\\|\omega'''|+|\omega''|=|\omega|\\|\alpha'''|+|\omega'''|\geq 1}}|(\part_x^{\alpha'''} Y^{\omega'''} E_j) \part_{v_j}\derv{''}{}{''} f |.\end{equation}
Consider first the case when $|\alpha'''|+|\omega'''|\le 4$, for which the $L^\infty$ bounds on the electric field in the bootstrap assumption \eqref{bootstrap-phi} can be used. We have 
\begin{equation}\label{comm-E.1}
\begin{aligned}
\nu^{2|\beta|/3}&\Big|\iint_{\T^3\times \R^3} e^{2(q+1)\phi} w^2\partial_x^\alpha \partial_v^\beta Y^\omega f (\part_x^{\alpha'''} Y^{\omega'''} E_j) \part_{v_j}\derv{''}{}{''} f \; \ud v \, \ud x\Big|
\\
&\lesssim \nu^{2|\beta|/3} \| \partial_x^\alpha \partial_v^\beta Y^\omega f \|_{L^2_{x,v}(\ell_{\alpha,\beta,\omega},\vartheta)}
\| \part_x^{\alpha'''} Y^{\omega'''} E \|_{L^\infty_x} \|
\part_{v_j}\derv{''}{}{''} f\|_{L^2_{x,v}(\ell_{\alpha,\beta,\omega},\vartheta)}
\\ &\lesssim  (\ep^{1/2}\nu^{1/3} \la t\ra^{-2}) (\nu^{-1/3}\| f\|_{\wtbE^{(\vartheta)}_N}^2 ) \ls \ep^{1/2} \la t\ra^{-2} \| f\|_{\wtbE^{(\vartheta)}_N}^2,
\end{aligned}
\end{equation}
upon recalling $|\alpha''|+|\omega''| +1\le |\alpha|+|\omega| $. 

Next, we consider the case when $|\alpha'''|+|\omega'''|\ge 5$, which we will bound by the first term in \eqref{eq:EdvY}. In this case, we must have $|\alpha''|+|\omega''|\le |\alp| + |\om|-5$ (and in particular $|\alpha''|+|\bt|+|\omega''|\le N_{max}-5$), and so upon using Sobolev embedding in $x$ and the bootstrap assumption \eqref{bootstrap-f}, we have 
\begin{equation*}
\begin{split}
&\: \nu^{|\bt|/3}\|\part_{v_j}\derv{''}{}{''} f \|_{L^\i_x L^2_{v}(\ell_{\alp,\bt,\om},\vartheta)} \\
\ls &\: \sum_{|\alp'''|\leq 2} \nu^{|\bt|/3} \| \rd_x^{\alp'''}\part_{v_j}\derv{''}{}{''} f\|_{L^2_x L^2_{v}(\ell_{\alp,\bt,\om},\vartheta)}\ls \nu^{-1/3} \| f\|_{\wtbE^{(\vartheta)}_{N_{max}-2}} \ls \ep^{1/2}.
\end{split}
\end{equation*}
Therefore, using H\"older's inequality and Lemma~\ref{lem-E1}, we bound 
\begin{equation}\label{comm-E.2}
\begin{aligned}
\nu^{2|\beta|/3}&\Big|\iint_{\T^3\times \R^3} e^{2(q+1)\phi} w^2\partial_x^\alpha \partial_v^\beta Y^\omega f (\part_x^{\alpha'''} Y^{\omega'''} E_j) \part_{v_j}\derv{''}{}{''} f \; \ud v \, \ud x\Big|
\\&\lesssim \nu^{2|\beta|/3} \| \partial_x^\alpha \partial_v^\beta Y^\omega f \|_{L^2_{x,v}(\ell_{\alpha,\beta,\omega},\vartheta)}
\| \part_x^{\alpha'''} Y^{\omega'''} E_j \|_{L^2_x} \|
\part_{v_j}\derv{''}{}{''} f\|_{L^\infty_x L^2_{v}(\ell_{\alpha,\beta,\omega},\vartheta)}
\\ &\lesssim  \|f\|_{\wtbE^{(\vartheta)}_N} \| \part_x^{\alpha'''} Y^{\omega'''} E_j \|_{L^2_x} \ep^{1/2} 
\ls \ep^{1/2}  \|f\|_{\wtbE^{(\vartheta)}_N} \sum_{\substack{ |\alp'''|\leq |\alp|  \\ |\om'''| \leq |\om|}} \| \part_x^{\alpha'''} Y^{\omega'''} \rho_{\not = 0} \|_{L^2_x}.
\end{aligned}
\end{equation}
We also need a bound with $\wtbE^{(\vartheta)}_N$ above replaced by $\wtbD^{(\vartheta)}_N$. Noticing that a direct estimate with $\wtbD^{(\vartheta)}_N$ causes a loss of $\nu^{-1/3}$, we integrate by parts in $\rd_{v_j}$. (We remark that when $\alp'\not = 0$ or $\bt' \not = 0$, such an integration by parts is unnecessary, by definition of the $\wtbD^{(\vartheta)}_N$ norm.) After integration by parts, we argue as above with Sobolev embedding, noting also that since $|\alpha''|+|\omega''|\le |\alp| + |\om|-5$, we have additional weights in $\la v\ra$. In other words, we bound
\begin{equation}\label{comm-E.3}
\begin{aligned}
\nu^{2|\beta|/3}&\Big|\iint_{\T^3\times \R^3} e^{2(q+1)\phi} w^2\partial_x^\alpha \partial_v^\beta Y^\omega f (\part_x^{\alpha'''} Y^{\omega'''} E_j) \part_{v_j}\derv{''}{}{''} f \; \ud v \, \ud x\Big|
\\&\ls \nu^{2|\beta|/3}\Big|\iint_{\T^3\times \R^3} e^{2(q+1)\phi} w^2\rd_{v_j} \partial_x^\alpha \partial_v^\beta Y^\omega f (\part_x^{\alpha'''} Y^{\omega'''} E_j) \derv{''}{}{''} f \; \ud v \, \ud x\Big|
\\& \quad + \nu^{2|\beta|/3}\Big|\iint_{\T^3\times \R^3} e^{2(q+1)\phi} (\rd_{v_j} w^2) \partial_x^\alpha \partial_v^\beta Y^\omega f (\part_x^{\alpha'''} Y^{\omega'''} E_j) \derv{''}{}{''} f \; \ud v \, \ud x\Big|
\\&\ls \ep^{1/2}  \|f\|_{\wtbD^{(\vartheta)}_N} \sum_{\substack{ |\alp'''|\leq |\alp|  \\ |\om'''| \leq |\om|}} \| \part_x^{\alpha'''} Y^{\omega'''} \rho_{\not = 0} \|_{L^2_x}.
\end{aligned}
\end{equation}

Combining \eqref{comm-E.1}, \eqref{comm-E.2} and \eqref{comm-E.3}, we have thus proven the desired estimate corresponding to the commutator term in \eqref{comm-E}.

Similarly, we now treat the integral involving $E\cdot v$. We compute
$$|[E_jv_j,\der]f | \ls 
\smashoperator{\sum_{\substack{|\alpha'''|+|\alpha''|=|\alpha|\\|\omega'''|+|\omega''|=|\omega|\\|\bt'| = |\bt|-1 \\|\alpha'''|+|\omega'''|\geq 1}}} |(\part_x^{\alpha'''} Y^{\omega'''} E_j) \partial_x^{\alpha''} \partial_v^{\beta'}Y^{\omega''} f| + \smashoperator{\sum_{\substack{|\alpha'''|+|\alpha''|=|\alpha|\\|\omega'''|+|\omega''|=|\omega|\\|\alpha'''|+|\omega'''|\geq 1}}} |v_j(\part_x^{\alpha'''} Y^{\omega'''} E_j) \partial_x^{\alpha''}\rd_v^{\bt}Y^{\omega''} f|.$$

The first term is similar to previous terms in $[E_j\part_{v_j},\der]f$, and is in fact better because it has two fewer $\rd_v$ derivatives. The second term experiences a linear growth of $|v|$. This growth however causes no loss of $v$-weight, since $|\alpha''|+|\omega''| \le |\alpha|+|\omega| - 1$, gaining $\langle v\rangle^{-4}$ in the $v$-weight. The lemma follows. 
\end{proof}

\begin{lemma}\label{lem-RGamma}
For $|\alpha|+|\beta|+|\omega| \leq N$, 
we have
$$
\begin{aligned}
\nu^{2(|\beta|+|\beta'|)/3} \mathcal{R}^{Q,\ell-2|\alpha'|-2|\beta'|,3}_{\alpha+\alpha',\beta+\beta',\omega}  
\quad&\lesssim\quad 
\nu^{1/3} \| f\|_{\wtbE^{(\vartheta)}_{N_{max}-2}}\| f\|_{\wtbD^{(\vartheta)}_N}^2 + \nu^{1/3} \| f\|_{\wtbE_{N}}\| f\|_{\wtbD^{(\vartheta)}_N} \| f\|_{\wtbD_{N_{max}-2}} 
\end{aligned}$$
when either (1) $|\alp'|\leq 1$ and $\bt'=0$, or (2) $\alp'=0$, $|\bt'|\leq 2$.
\end{lemma}

\begin{proof} Using \lref{nonlin} with $\ell = \ell_{\alpha,\beta,\omega}$, we bound 
$$
\begin{aligned}
& \Big| \iint_{\T^3\times \R^3} e^{2(q+1)\phi}w^2\partial_x^\alpha \partial_v^\beta Y^\omega f 
\partial_x^\alpha\partial_v^\beta Y^\omega \Gamma(f,f)
\; \ud v \, \ud x \Big|
\\&\lesssim \smash{\sum_{\substack{|\alpha'|+|\alpha''|\leq |\alpha|\\ |\beta'|+|\beta''|\leq |\beta|\\|\omega'|+|\omega''|\leq |\omega|}}} 
 \norm{\der f}_{\Delta_{x,v}( \ell_{\alpha,\beta,\omega},\vartheta)}\left[\norm{\derv{'}{'}{'} f}_{L^2_{x,v}}\norm{\derv{''}{''}{''} f}_{\Delta_{x,v}( \ell_{\alpha,\beta,\omega},\vartheta)}\right.\\
&\hspace{15em}\left.+\norm{\derv{'}{'}{'} f}_{\Delta_{x,v}}\norm{\derv{''}{''}{''} f}_{L^2_v( \ell_{\alpha,\beta,\omega},\vartheta)}\right]
\end{aligned}
$$
noting the norms involving $\derv{'}{'}{'} f$ can have any weight in $v$. Therefore, by definition, we have 
$$
\begin{aligned}
\nu \nu^{2|\beta|/3} \mathcal{R}^{Q,\ell,3}_{\alpha,\beta,\omega}
&= \nu \nu^{2|\beta|/3} \Big| \iint_{\T^3\times \R^3} e^{2(q+1)\phi}w^2\partial_x^\alpha \partial_v^\beta Y^\omega f 
\partial_x^\alpha\partial_v^\beta Y^\omega \Gamma(f,f)
\; \ud v \, \ud x \Big|
\\&\lesssim \nu^{1/3} \| f\|_{\wtbE^{(\vartheta)}_{N_{max}-2}}\| f\|_{\wtbD^{(\vartheta)}_N}^2 + \nu^{1/3} \| f\|_{\wtbE_{N}}\| f\|_{\wtbD^{(\vartheta)}_N} \| f\|_{\wtbD_{N_{max}-2}}
.
\end{aligned}
$$
By definition of the energy and dissipation norms, the same bounds hold for $\mathcal{R}^{Q,\ell-2|\alpha'|-2|\beta'|,3}_{\alpha+\alpha',\beta+\beta',\omega}  $, 
upon assigning the respective $v$-weight and $\nu$-scaling. 
\end{proof}

%
%
%

\subsection{Estimates on $\mathcal{Z}^{Q,\ell}_{\alpha,\beta,\omega}$}\label{sec:Z}

Finally, in this section, we give bounds on $\mathcal{Z}^{Q,\ell-2}_{\alpha,\beta,\omega} $, defined as in Lemma \ref{lem-cross}, that appear in \eqref{def-Rabc}, noting the $v$-weight function is indexed at $\ell_{\alpha,\beta,\omega}-2$. 
Recalling the definition of $\mathcal{Z}^{Q,\ell-2}_{\alpha,\beta,\omega}$ from Lemma \ref{lem-cross}, we write  
$$ \mathcal{Z}^{Q,\ell-2}_{\alpha,\beta,\omega}  = \mathcal{Z}^{Q,\ell-2,1}_{\alpha,\beta,\omega}   +\mathcal{Z}^{Q,\ell-2,2}_{\alpha,\beta,\omega}   $$
where $\mathcal{Z}^{Q,\ell-2,1}_{\alpha,\beta,\omega}  $ is defined by 
$$ 
\begin{aligned}
\mathcal{Z}^{Q,\ell-2,1}_{\alpha,\beta,\omega}  &
=  2 \iint_{\T^3\times \R^3} e^{2(q+1)\phi} w^2 \langle v\rangle^{-4}(\partial_{x_j} \partial_x^\alpha\partial_v^\beta Y^\omega f)\partial_{v_j}\partial_x^\alpha\partial_v^\beta Y^\omega[E_j \partial_{v_j}\sqrt \mu]  \; \ud v \, \ud x
  \\& \quad + \iint_{\T^3\times \R^3} e^{2(q+1)\phi} w^2\langle v\rangle^{-4} (\partial_{x_j} \partial_x^\alpha\partial_v^\beta Y^\omega f) [E\cdot \nabla_v , \partial_{v_j}\partial_x^\alpha\partial_v^\beta Y^\omega] f \; \ud v \, \ud x
  \\& \quad - \iint_{\T^3\times \R^3} e^{2(q+1)\phi} w^2 \langle v\rangle^{-4}(\partial_{x_j} \partial_x^\alpha\partial_v^\beta Y^\omega f) [E\cdot v, \partial_{v_j}\partial_x^\alpha\partial_v^\beta Y^\omega] f \; \ud v \, \ud x 
  \\& \quad + 
 \nu \iint_{\T^3\times \R^3} e^{2(q+1)\phi} w^2 \langle v\rangle^{-4}(\partial_{x_j} \partial_x^\alpha\partial_v^\beta Y^\omega f) \partial_{v_j}\partial_x^\alpha\partial_v^\beta Y^\omega\Gamma(f,f) \; \ud v \, \ud x
 \end{aligned}$$
and $\mathcal{Z}^{Q,\ell-2,2}_{\alpha,\beta,\omega}  $ is defined in a symmetric way, switching $\partial_{x_j}$ and $\partial_{v_j}$ in each of the integrals above. Now observe that all the integral terms are estimated similarly, if not identically, as already done for the similar integral terms in Lemma \ref{lem-Edvmu} (the first term above), Lemma \ref{lem-EdvY} (the second and third), and Lemma \ref{lem-RGamma} (the last), respectively. This completes the proof of the claimed bounds on $\mathcal{Z}^{Q,\ell-2}_{\alpha,\beta,\omega} $, and hence the proof of Theorem \ref{theo-mainEE}. 

\section{Global existence of solutions}\label{sec:global}

\begin{theorem}\label{thm:existence}
Consider data as in Theorem~\ref{t.main}. Then the unique smooth solution arising from the given initial data is global in time. Moreover, the estimates \eqref{density-bound} and \eqref{mainEE-close} hold for $T_{B}$ replaced by $\infty$.
\end{theorem}
\begin{proof}
Using a standard local existence and uniqueness result (suitably adapting \cite{cHsSaT2019}), we can carry out a bootstrap argument. 

Suppose there exists $T_B>0$ such that the solution $f$ to \eqref{eq:Vlasov.f}--\eqref{eq:Poisson.f} remains smooth in $[0,T_B)\times \T^3 \times \R^3$ and satisfies the bootstrap assumptions \eqref{bootstrap-f} and \eqref{bootstrap-phi}. It suffices to show that in fact \eqref{bootstrap-f} and \eqref{bootstrap-phi} hold with $\ep$ replaced by $C\ep^2$, for some constant $C>0$ independent of $\ep$ and $\nu$.
\begin{itemize}
\item The improvement for \eqref{bootstrap-f} follows from Theorem~\ref{theo-mainEE}.
\item The improvement for \eqref{bootstrap-phi} is an immediate consequence of Lemmas~\ref{lem-E1}, \ref{lem-E2} and the bounds obtained in Theorem~\ref{theo-mainEE}.
\end{itemize}
This closes the bootstrap argument and implies that the solution is global and remains unique in the class of solutions obeying the bound \eqref{eq:main.energy.lowest}. Finally, since we have closed the bootstrap, the bounds \eqref{density-bound} and \eqref{mainEE-close} follow from Theorems~\ref{theo-density} and \ref{theo-mainEE}. \qedhere
\end{proof}

\section{Nonlinear density estimates: stretched exponential decay}\label{sec:exp-decay-density}
In the next two section, we will turn to the proof of the stretched exponential decay. Similarly as for the boundedness of the solution, the proof is split into two parts: the nonlinear density estimates are treated in this section, and the nonlinear energy decay estimates will be treated in Section~\ref{sec:exp-decay-energy}.

We first point out a few key points for the density estimates, especially in contrast to the bounds proven in Section~\ref{sec:density}:
\begin{enumerate}
\item In order to prove the stretched exponential decay, we need to prove a density estimate also with a stretched exponential decay factor; see $\mathbf e(t)$ factors in Theorem~\ref{theo-density-decay}.
\item We prove the estimate \eqref{density-bound-decay}, which is of the same size for all $p \in [2,\infty]$. This is in contrast with the boundedness estimates for $\rho_{\not = 0}$ in Section~\ref{sec:density}, where the $L^2_t$ estimate has a weaker $\nu$ power (see \eqref{density-bound}).

The main difference in the argument comes from the term $\mathbb{II}_k$, where we crucially rely on the fact that we are at a lower order, and that both the boundedness of the higher order density estimates and the higher order energy estimates were already established in the previous sections (see \eqref{density-bound} and \eqref{mainEE-close}).

\item We need a decomposition of $\rho_{\not = 0} = \rho_{\not = 0}^{(1)} + \rho_{\not = 0}^{(2)}$: the piece $\rho_{\not = 0}^{(2)}$ is better in terms of the size, and (its $\rd_x$ derivatives) obeys an $O(\ep^2 \nu)$ instead of an $O(\ep^2\nu^{2/3})$ bound; the piece $\rho_{\not = 0}^{(1)}$ only obeys an $O(\ep^2\nu^{2/3})$ upper bound, but importantly one can also take $L^1_t$ norm with the same upper bound. 

This decomposition is important for closing the energy decay estimates in Section~\ref{sec:exp-decay-energy}. 
\end{enumerate}

We put forth another bootstrap argument. For $\mathbf e(t) \in \{ e^{\de(\nu^{1/3}t)^{1/3}}, e^{\de(\nu t)^{2/3}} \}$, introduce the following bootstrap assumption:
\begin{equation}\label{bootstrap-f1}
\begin{split}
\sup_{0\leq t <T_B} \mathbf e(t) \norm{f(t)}^2_{\wtbE^{(2)'}_{1,1,0,0}}
 + \nu^{1/3}\int_0^{T_B} \mathbf e(\tau) \norm{f(\tau)}^2_{\wtbD^{(2)'}_{1,1,0,0}}\d \tau \leq \ep \nu^{2/3},
\end{split}
\end{equation}
where $\| \cdot \|_{\wtbE^{(2)'}_{1,1,0,0}}$ and $\|\cdot \|_{\wtbD^{(2)'}_{1,1,0,0}}$ are defined as in \eqref{eq:combined.norms.top} but with exponential weights $e^{\f{q'|v|^2}2}$ instead of $e^{\f{q|v|^2}2}$ (cf.~\eqref{eq:p.norms.general}). 

\bigskip

The main result of this section is the following.

\begin{theorem}\label{theo-density-decay} Consider data as in Theorem~\ref{t.main}. Suppose there exists $T_B>0$ such that the solution $f$ to \eqref{eq:Vlasov.f}--\eqref{eq:Poisson.f} satisfies the bootstrap assumption \eqref{bootstrap-f1} in $[0,T_B)\times \T^3 \times \R^3$. 

Then, for $\mathbf e(t) \in \{ e^{\de(\nu^{1/3}t)^{1/3}}, e^{\de(\nu t)^{2/3}} \}$, the following hold:
\begin{itemize}
\item For any $p \in [2,\infty]$, $\rho_{\not = 0}$ obeys the bound
\begin{equation}\label{density-bound-decay} 
\begin{split}
 \sum_{|\alp|\leq 1} \| \mathbf e^{1/2}(t) \rd_x^\alp  \rho_{\not = 0} \|^2_{L^p_t([0,T_B); L^2_x)} 
\ls \ep^2 \nu^{2/3}.
\end{split}
\end{equation}
\item $\rho_{\not = 0}(t,x)$ admits a decomposition $\rho_{\not = 0} = \rho_{\not = 0}^{(1)} + \rho_{\not = 0}^{(2)}$ such that
\begin{equation}\label{eq:rho1.bound}
\sum_{|\alp|\leq 1} \| \mathbf e^{1/2}(t) \rd_x^\alp  \rho^{(1)}_{\not = 0} \|^2_{L^1_t([0,T_B); L^2_x)} 
\ls \ep^2 \nu^{2/3},
\end{equation}
\begin{equation}\label{eq:rho2.bound}
\sum_{|\alp|\leq 1} \| \mathbf e^{1/2}(t) \rd_x^\alp  \rho^{(2)}_{\not = 0} \|^2_{L^2_t([0,T_B); L^2_x)} \ls \ep^2 \nu.
\end{equation}
\end{itemize}
\end{theorem}

As in Section~\ref{sec:density}, we split the density contribution into the terms $\mathbb I_k$, $\mathbb{II}_k$ and $\mathbb{III}_k$ as in \eqref{eq:I.II.III.def}. Define
\begin{equation}\label{eq:N1.def}
\mathcal N^{(1)}_k(t) := (\mathbb I_k + \mathbb{II}_k)(t),\quad \mathcal N^{(2)}_k(t) := \mathbb{III}_k(t),
\end{equation}
and, for $\mathbf e(\tau) \in \{e^{\de(\nu^{1/3} \tau)^{1/3}}, e^{\de(\nu \tau)^{2/3}} \}$ and $j \in \{1,2\}$, define
\begin{equation}\label{eq:M1.def}
\mathcal M^{(j)}_k(t) := \mathbf e(\tau)  \mathcal N^{(j)}_k(t).
\end{equation}
The density decomposition asserted in Theorem~\ref{theo-density-decay} is then defined as $\rho^{(j)}_{\not = 0} = \sum_{k \not =0} \rho^{(j)}_k e^{ik\cdot x}$, where
\begin{equation}\label{eq:rho1.def}
\begin{split}
\rho_k^{(j)}(t):= &\:\mathcal N^{(j)}_k(t) + \int_0^t G_k(t-s) \mathcal N^{(j)}_k(s) \, \d s.
\end{split}
\end{equation}
Similarly as done before, 
we introduce 
\begin{equation}\label{eq:zete.def}
\begin{split}
\zeta_e(t):= &\: \sup_{\tau \in [0,t]} \sum_{k \not =0} \mathbf e(\tau) |k|^2|\rho_k^{(1)}|^2(\tau) + 
\Big(\int_0^t [ \sum_{k \not =0} \mathbf e(\tau) |k|^2|\rho_k^{(1)}|^2(\tau)]^{1/2} \, \d \tau\Big)^2 \\
&\: +\nu^{-1/3} \sup_{\tau \in [0,t]} \sum_{k \not =0} \mathbf e(\tau) |k|^2|\rho_k^{(2)}|^2(\tau) + \nu^{-1/3} \int_0^t  \sum_{k \not =0} \mathbf e(\tau) |k|^2|\rho_k^{(2)}|^2(\tau) \, \d \tau.
\end{split}
\end{equation}

The following are the main estimates that will be used to prove Theorem~\ref{theo-density-decay}:
\begin{proposition}\label{prop:density.decay}
For $\mathbf e(\tau) \in \{e^{\de(\nu^{1/3} \tau)^{1/3}}, e^{\de(\nu \tau)^{2/3}} \}$, the following hold for all $t \in [0, T_B)$:
\begin{equation}\label{eq:density.decay.1}
\sup_{0\leq \tau \leq t} \sum_{k \not =0}  |k|^2|\mathcal M^{(1)}_k|^2(\tau) + \Big[ \int_0^t [\sum_{k \not = 0} |k|^2|\mathcal M^{(1)}_k|^2(\tau) ]^{1/2} \, \d \tau \Big]^2 \ls \ep^2 \nu^{2/3} + \ep \zeta_e(t)
\end{equation}
and 
\begin{equation}\label{eq:density.decay.2}
\sup_{0\leq \tau\leq t} \sum_{k \not =0}  |k|^2 |\mathcal M^{(2)}|^2(\tau) +  \sum_{k \not =0}  \int_0^t |k|^2 |\mathcal M^{(2)}_k|^2(\tau) \, \d \tau \ls \ep^2 \nu.
\end{equation}
\end{proposition}

\begin{proof}[Proof of Theorem~\ref{theo-density-decay} assuming Proposition~\ref{prop:density.decay}] First note that the bounds \eqref{eq:density.decay.1} and \eqref{eq:density.decay.2} for $|\mathcal M_k^{(j)}(\tau)|^2$ imply the same estimates for $\mathbf e(\tau)|\rho_k^{(j)}(\tau)|^2$ using \eqref{eq:rho1.def}. For instance, using the definitions of $\mathcal N_k^{(j)}$, $\mathcal M_k^{(j)}$ and $\rho_k^{(j)}$ in \eqref{eq:N1.def}--\eqref{eq:rho1.def},
\begin{equation}\label{eq:M1.to.rho1.1}
\begin{split}
\sum_{k \not = 0} \mathbf e(t) |k|^2|\rho_k^{(j)}(t)|^2 \ls  \sum_{k \not = 0} |k|^2 |\mathcal M_k^{(j)}(t)|^2 + \sum_{k\not = 0} \mathbf e(t) |k|^2 [ \int_0^t |G_k(t - \tau)| |\mathcal N^{(j)}_k|(\tau) \, \d \tau ]^2.
\end{split}
\end{equation}
Notice that by \eqref{eq:G.est.main}, choosing $\de$ small depending on $\de''$ so that  
\begin{equation}\label{et-bounds}\mathbf e(t) \min\{e^{-\de''(\nu^{1/3} (t-\tau))^{1/3}}, e^{-\de''(\nu (t-\tau))^{2/3}} \} \ls \mathbf e(\tau),\end{equation} 
we can bound the last term in \eqref{eq:M1.to.rho1.1} by 
\begin{equation}\label{eq:M1.to.rho1.2}
\begin{split}
&\: \sum_{k\not = 0} \mathbf e(t) |k|^2 [ \int_0^t |G_k(t - \tau)| |\mathcal N^{(j)}_k|(\tau) \, \d \tau ]^2 \ls \sum_{k\not = 0}[ \int_0^t  |k|^{-1} \la k(t - \tau)\ra^{-2} |k| |\mathcal M^{(j)}_k|(\tau)  \, \d \tau]^2 \\
\ls &\: \sum_{k\not = 0} |k|^{-2} \Big[\int_0^t   \la k(t - \tau)\ra^{-2} \, \d \tau \Big]^2 \sup_{k' \not = 0} \sup_{\tau' \in [0,t]}  |k'|^2 |\mathcal M^{(j)}_{k'}|^2(\tau')\\
\ls &\: [\sum_{k\not = 0} |k|^{-4}] \sup_{k' \not = 0} \sup_{\tau' \in [0,t]}  |k'|^2 |\mathcal M^{(j)}_{k'}|^2(\tau') \ls \sup_{k' \not = 0} \sup_{\tau' \in [0,t]}  |k'|^2 |\mathcal M^{(j)}_{k'}|^2(\tau').
\end{split}
\end{equation}
Combining \eqref{eq:M1.to.rho1.1} and \eqref{eq:M1.to.rho1.2}, and then using \eqref{eq:density.decay.1} and \eqref{eq:density.decay.2}, we obtain, for $j \in \{1,2\}$,
\begin{equation}\label{eq:rho.by.M.Li}
\sup_{\tau \in [0,t]} \sum_{k \not = 0} \mathbf e(\tau) |k|^2|\rho_k^{(1)}(\tau)|^2 \ls \ep^2 \nu^{2/3} + \ep \zeta_e(t), \quad\sup_{\tau \in [0,t]} \sum_{k \not = 0} \mathbf e(\tau) |k|^2|\rho_k^{(2)}(\tau)|^2 \ls \ep^2 \nu.
\end{equation}

We can control the $L^1_t \ell^2_k$ norm of $|k| |\mathcal \rho^{(1)}_k|$ and the $L^2_t \ell^2_k$ norm of $|k| |\mathcal \rho^{(2)}_k|$ in a similar manner. Indeed, using \eqref{eq:N1.def}--\eqref{eq:rho1.def} and then the bounds for $G_k$ in \eqref{eq:G.est.main},
\begin{equation*}
\begin{split}
&\: \Big[ \int_0^t [\sum_{k \not = 0} \mathbf e(\tau) |k|^2|\rho_k^{(1)}(\tau)|^2]^{1/2} \, \d \tau \Big]^2 \\
\ls &\: \Big[ \int_0^t [\sum_{k \not = 0} |k|^2 |\mathcal M_k^{(1)}(\tau)|^2]^{1/2} \, \d \tau \Big]^2  + \Big[ \int_0^t [\sum_{k\not = 0} \mathbf e(\tau) |k|^2  (\int_0^\tau |G_k(\tau - s)| |\mathcal N^{(1)}_k|(s) \, \d s)^2 ]^{1/2} \, \d \tau \Big]^2\\
\ls &\: \Big[ \int_0^t [\sum_{k \not = 0} |k|^2 |\mathcal M_k^{(1)}(\tau)|^2]^{1/2} \, \d \tau \Big]^2 + \Big[ \int_0^t [\sum_{k\not = 0}  [ \int_0^\tau |k|^{-1} \la k (\tau - s)\ra^{-2} |k| |\mathcal M^{(j)}_k|(s) \, \d s ]^2 ]^{1/2} \, \d \tau \Big]^2.
\end{split}
\end{equation*}
To control the final term, we use Minkowski's inequality to exchange the order of $L^1_s$ and $\ell^2_k$, and then use Fubini's theorem to exchange the order of $L^1_t$ and $L^1_s$ so as to obtain
\begin{equation*}
\begin{split}
&\: \Big[ \int_0^t [\sum_{k\not = 0}  [ \int_0^\tau |k|^{-1} \la k (\tau - s)\ra^{-2} |k| |\mathcal M^{(1)}_k|(s) \, \d s ]^2 ]^{1/2} \, \d \tau \Big]^2 \\
\ls &\: \Big[ \int_0^t \int_0^\tau [\sum_{k\not = 0} |k|^{-2} \la k (\tau-s) \ra^{-4}  |k|^2 |\mathcal M^{(1)}_k|^2(s)]^{1/2} \, \d s  \, \d \tau \Big]^2 \\ 
\ls &\: \Big[ \int_0^t \int_0^\tau \la \tau-s \ra^{-2} [\sum_{k\not = 0} |k|^{-2}   |k|^2 |\mathcal M^{(1)}_k|^2(s)]^{1/2} \, \d s  \, \d \tau \Big]^2 \\ 
\ls &\: \Big[ \int_0^t  [\sum_{k\not = 0} |k|^{-2}   |k|^2 |\mathcal M^{(1)}_k|^2(s)]^{1/2} \, \d s  \Big]^2 \ls \Big[ \int_0^t  [\sum_{k\not = 0} |k|^2 |\mathcal M^{(1)}_k|^2(s)]^{1/2} \, \d s  \Big]^2.
\end{split}
\end{equation*}
Thus, using also \eqref{eq:density.decay.1}, we have obtained 
\begin{equation}\label{eq:rho.by.M.L1}
\Big[ \int_0^t [\sum_{k \not = 0} \mathbf e(\tau) |k|^2|\rho_k^{(1)}(\tau)|^2]^{1/2} \, \d \tau \Big]^2 \ls \ep^2\nu^{2/3} + \ep\zeta_e(t).
\end{equation}

A similar but slightly simpler argument also allows us to bound the $L^2_t \ell^2_k$ norm of $\mathbf e^{1/2}(\tau) |k| |\mathcal \rho^{(2)}_k|$ using \eqref{eq:density.decay.2} so that we have
\begin{equation}\label{eq:rho.by.M.L2}
 \int_0^t \sum_{k \not = 0} \mathbf e(\tau) |k|^2 |\rho_k^{(2)}(\tau)|^2 \, \d \tau  \ls \ep^2\nu.
\end{equation}

Recalling the definition of $\zeta_e$ in \eqref{eq:zete.def} and using \eqref{eq:rho.by.M.Li}, \eqref{eq:rho.by.M.L1} and \eqref{eq:rho.by.M.L2}, we obtain
$$\zeta_e(t) \ls \ep^2\nu^{2/3} + \ep\zeta_e(t),$$
which implies
\begin{equation}\label{eq:main.zetae}
\zeta_e(t) \ls \ep^2 \nu^{2/3}.
\end{equation}

At this point, using \eqref{eq:main.zetae}, we easily conclude Theorem~\ref{theo-density-decay}: 
\begin{itemize} 
\item Plugging \eqref{eq:main.zetae} back into \eqref{eq:rho.by.M.L1} and \eqref{eq:rho.by.M.L2} yields the estimates \eqref{eq:rho1.bound} and \eqref{eq:rho2.bound}.
\item Plugging \eqref{eq:main.zetae} into \eqref{eq:rho.by.M.Li} yields the $p=\infty$ case of \eqref{density-bound-decay}.
\item Finally, the $p \in [2,\infty)$ cases of \eqref{density-bound-decay} can be obtained by interpolating between the $p=\infty$ case and the bounds \eqref{eq:rho1.bound} and \eqref{eq:rho2.bound}. \qedhere
\end{itemize}
\end{proof}

The remainder of this section will be devoted to the proof of Proposition~\ref{prop:density.decay}.

\subsection{Initial data contribution}
By \eqref{claim-1k} with $N_1 = N_2 = 1$, we have
\begin{equation*}
\begin{split}
\sum_{k\not =0}\mathbf e(t) |k|^{2} |\mathbb{I}_k(t)|^2 \ls  \ep^2 \nu^{2/3} \langle t\rangle^{-4},
\end{split}
\end{equation*}
which obeys the bounds required in \eqref{eq:density.decay.1}.

\subsection{Nonlinear interaction I}
Recall the decomposition in \eqref{eq:IIk.def}; as in Section~\ref{sec:nonlinear.int.I}, we only consider the term $\mathbb {II}_{k,2}$, as the term $\mathbb{II}_{k,1}$ is easier.

For the term $\mathbb {II}_{k,2}$, we prove below the $L^\i_t$ and $L^1_t$ according to \eqref{eq:density.decay.1}.

\subsection*{Proving the $L^\i$ bound} 
Arguing as in \eqref{bd-IIk}, with $(N_1,N_2) = (1,0)$ and $(N_1',N_2')=(5,2)$, and taking into the extra stretched exponential decay given by \eqref{decay-combineSk.nu.decay} to obtain
\begin{equation}
\begin{aligned}
 |k|  |\mathbb{II}_{k,2}(t)| 
\ls   \sum_{l \not = 0} &\int_0^t \:  |l|^{-1}| k |   \la k-l \ra^{-5} \langle kt - l \tau \rangle^{-2}  \\
&\: \times \min\{e^{-\de'(\nu^{1/3} (t-\tau))^{1/3}}, e^{-\de'(\nu (t-\tau))^{2/3}} \} |\hat\rho_l(\tau)| \| \hat f_{k-l}(\tau)\|_{\wtG'_{N_{max}-2}}  \; \d\tau
\end{aligned}
\end{equation}
where 
\begin{equation}\label{redef-GN}\| \hat{f}_{k} (\tau) \|_{\widetilde{\mathbb G}'_N} :=  \smashoperator{\sum_{\substack{ |\alp|+ |\om|\leq N \\ 1\leq |\bt|\leq 2}}} \nu^{(|\bt|-1)/3}  \Big[ \|\la v \ra^{10} (\rd_x^\alp \rd_v^{\bt}  Y^\om f)\sphat_{k}(\tau)\|_{L^2_v} +  \|\la v \ra^{4} e^{\frac14 q_0|v|^2} (\rd_x^\alp \rd_v^{\bt}  f)\sphat_{k}(\tau)\|_{L^2_v} \Big],\end{equation}
noting the last additional term (compared with \eqref{eq:wtG.def}) with the exponential weight $e^{\frac14 q_0|v|^2}$ slower than what is encoded in the energy and dissipation norms \eqref{eq:combined.norms.top}. See also Remark \ref{rem-slowweight}. 
In particular, we note that
\begin{equation}\label{eq:G.by.E-extra}
\sum_k \| \hat{f}_k(t) \|_{\wtG_{N_{max}-2}'}^2 \ls \nu^{-2/3}\| f(t)\|_{\wtbE_{N_{max}-2}^{(2)'}}^2 \ls \nu^{-2/3}\| f(t)\|_{\wtbE_{N_{max}-2}^{(2)}}^2 \ls \ep
\end{equation}
and
\begin{equation}\label{eq:G.by.E-extraextra}
\sum_{k \not = 0} \mathbf e(t) |k|^2\| \hat{f}_k(t) \|_{\wtG_0'}^2 \ls \nu^{-2/3}\mathbf e(t) \| f(t)\|_{\wtbE_{1,1,0,0}^{(2)'}}^2 \ls \ep,
\end{equation}
where we used, respectively, the energy bounds \eqref{mainEE-close} and the bootstrap assumption \eqref{bootstrap-f1}.


Next, use \eqref{et-bounds} with $\de''$ replaced by $\de'$, and notice that $|k||l|^{-1} \la k - l\ra^{-1} \ls 1$. We obtain
\begin{equation}\label{eq:IIk.decay.main}
\begin{aligned}
 \mathbf e^{1/2}(t) |k|  |\mathbb{II}_{k,2}(t)|
\ls &\:  \sum_{l \not = 0} \int_0^t \langle kt - l \tau \rangle^{-2} \mathbf e^{1/2}(\tau) |\hat\rho_l(\tau)| \la k-l\ra^{-4} \| \hat f_{k-l}(\tau)\|_{\wtG_{N_{max}-2}}  \; \d\tau.
\end{aligned}
\end{equation}

Using \eqref{eq:G.by.E-extra}, the Cauchy--Schwarz inequality in $\tau$, and the Young's convolution inequality for the sums, we obtain
\begin{equation}
\begin{aligned}
 &\: \sum_{k \not = 0} \mathbf e(t) |k|^{2}  |\mathbb{II}_{k,2}(t)|^2 \\
\ls &\:  \sum_{k \not = 0} \Big[ \sum_{l \not = 0} (\int_0^t \mathbf e(\tau) |\hat\rho_l(\tau)|^2 \, \d \tau)^{1/2} (\int_0^t  \la k-l\ra^{-8} \la kt - l\tau\ra^{-4} \| \hat f_{k-l}(\tau)\|_{\wtG_{N_{max}-2}}^2  \; \d\tau)^{1/2} \Big]^2 \\
\ls &\: \ep \sum_{k \not = 0} \Big[ \sum_{l \not = 0} (\int_0^t \mathbf e(\tau) |\hat\rho_l(\tau)|^2 \, \d \tau)^{1/2} (\int_0^t  \la k-l\ra^{-8} \la kt - l\tau\ra^{-4}   \; \d\tau)^{1/2} \Big]^2 \\
\ls &\:  \ep (\sum_{l\not =0} \int_0^t   \mathbf e(\tau) |\hat\rho_l(\tau)|^2 \d \tau )[\sum_{k} \la k\ra^{-4} ]^2 \ls \ep (\sum_{l\not =0} \int_0^t   \mathbf e(\tau) |\hat\rho_l(\tau)|^2 \d \tau ).
\end{aligned}
\end{equation}
It remains to check that 
\begin{equation}\label{eq:L2.by.zetae}
\sum_{l\not =0} \int_0^t   \mathbf e(\tau) |l|^2 |\hat\rho_l(\tau)|^2 \d \tau \ls \zeta_e(t).
\end{equation}
Indeed, the bound for $\sum_{l \not = 0} \int_0^t  \mathbf e(\tau)|l|^2  |\hat\rho_l^{(2)}(\tau)|^2  \; \d\tau$ is immediate from the definition \eqref{eq:zete.def}, while that for $\sum_{l \not = 0} \int_0^t  \mathbf e(\tau) |l|^2 |\hat\rho^{(1)}_l(\tau)|^2  \; \d\tau$ follows from interpolating between the $L^1_t$ and the $L^\i_t$ bounds in \eqref{eq:zete.def}.

\subsection*{Proving the $L^1$ bound} To obtain the $L^1_t$ bound, we need to estimate
\begin{equation}\label{eq:main.term.L1t.II.decay}
 \int_0^t \Big[ \sum_{k \not = 0} \Big( \sum_{l \not =0} \int_0^s \mathbf e^{1/2}(s)|k| |l|^{-1}| \hat{\rho}_l|(\tau) \Big|\int_{\R^3}S_k(s-\tau)[\widehat{\nabla_v f}_{k-l}(\tau)]\sqrt{\mu}\d v \Big|\d \tau\Big)^2 \Big]^{1/2} \, \d s.
 \end{equation}

To estimate \eqref{eq:main.term.L1t.II.decay}, we split the $\tau$-integral into $|l\tau| < |ks|/2$ and $|\l\tau| \geq |ks|/2$. In the latter integral, we further split the sums into the $l \not =k$ and the $l =k$ parts. 

First, consider the case $|l\tau|< |ks|/2$. We estimate the integrand as in \eqref{eq:IIk.decay.main}, i.e. 
\begin{equation}\label{eq:decay.integrand.1}
\begin{split}
&\: \mathbf e^{1/2}(s)|k| |l|^{-1}| \hat{\rho}_l|(\tau) \Big|\int_{\R^3}S_k(s-\tau)[\widehat{\nabla_v f}_{k-l}(\tau)]\sqrt{\mu}\d v \Big|\\
\ls &\: \langle ks - l \tau \rangle^{-2} \mathbf e^{1/2}(\tau) |\hat\rho_l(\tau)| \la k-l\ra^{-4} \| \hat f_{k-l}(\tau)\|_{\wtG_{N_{max}-2}'}
\end{split}
\end{equation}
with $\| \hat f_{k-l}(\tau)\|_{\wtG_{N_{max}-2}'}$ defined as in \eqref{redef-GN}. 
Since we imposed $|l\tau|< |ks|/2$, we have $\la ks - l\tau \ra^{-2} \ls \la ks\ra^{-2} \ls \la ks\ra^{-5/4} \la l\tau\ra^{-3/4} \ls \la s\ra^{-5/4} \la \tau\ra^{-3/4}$. It thus suffices to bound
$$\Big[ \int_0^t [\sum_{k \not = 0} (\sum_{l \not = 0} \int_0^s \langle \tau \rangle^{-3/4} \mathbf e(\tau) |\hat\rho_l(\tau)| \la k-l\ra^{-4} \| \hat f_{k-l}(\tau)\|_{\wtG_{N_{max}-2}'}  \; \d\tau)^2 ]^{1/2} \la s\ra^{-5/4} \, \d s \Big]^2.$$
Note that $ \| \hat f_{k-l}(\tau)\|_{\wtG_{N_{max}-2}'} \ls \ep^{1/2}$ by \eqref{eq:G.by.E-extra}. 
Then, integrating the $\la s\ra^{-5/4}$ factor, and using the Cauchy--Schwarz inequality in $\tau$ and the Young's convolution inequality for the sums, 
\begin{equation}
\begin{split}
&\: \Big[ \int_0^t [\sum_{k \not = 0} (\sum_{l \not = 0} \int_0^s \langle \tau \rangle^{-3/4} \mathbf e^{1/2}(\tau) |\hat\rho_l(\tau)| \la k-l\ra^{-4} \| \hat f_{k-l}(\tau)\|_{\wtG_{N_{max}-2}'}  \; \d\tau)^2 ]^{1/2} \la s\ra^{-5/4} \, \d s \Big]^2\\
\ls &\:  \ep \sum_{k \not = 0} (\sum_{l \not = 0} \int_0^t \la \tau \ra^{-3/4} \mathbf e^{1/2}(\tau) |\hat\rho_l(\tau)| \la k-l\ra^{-4}   \; \d\tau)^2   \\
\ls &\: \ep \sum_{k \not = 0} \Big(\sum_{l \not = 0} (\int_0^t  \mathbf e(\tau) |\hat\rho_l(\tau)|^2  \; \d\tau)^{1/2} (\int_0^t \la \tau \ra^{-3/2} \la k-l\ra^{-8}  \; \d\tau)^{1/2} \Big)^2   \\
\ls &\: \ep \Big(\sum_{l \not = 0} \int_0^t  \mathbf e(\tau) |\hat\rho_l(\tau)|^2  \; \d\tau\Big) \Big(\sum_k \la k\ra^{-4} \Big)^2 \ls \ep \sum_{l \not = 0} \int_0^t  \mathbf e(\tau) |\hat\rho_l(\tau)|^2  \; \d\tau \ls \ep \zeta_e(t),
\end{split}
\end{equation}
where at the end we used \eqref{eq:L2.by.zetae}.

Next, we turn to the case $|l\tau|\geq |ks|/2$ and $k \not = l$. Here, we bound the integrand differently: starting with \eqref{bd-IIk} but taking instead $(N_1,N_2) = (4,3)$ and $(N_1',N_2')=(1,0)$, we have
\begin{equation}
\begin{split}
&\: \mathbf e^{1/2}(s)|k| |l|^{-1}| \hat{\rho}_l|(\tau) \Big|\int_{\R^3}S_k(s-\tau)[\widehat{\nabla_v f}_{k-l}(\tau)]\sqrt{\mu}\d v \Big|\\
\ls &\: |l|^{-5} |k| \la k - l\ra^{-1} \langle l\tau \rangle^{-3} \mathbf e^{1/2}(\tau) (|l|^4 \la l\tau \ra^3 |\hat\rho_l(\tau)|) (|k-l| \| \hat f_{k-l}(\tau)\|_{\wtG_0'}) ,
\end{split}
\end{equation}
where $\| \hat f_{k-l}(\tau)\|_{\wtG_{0}'}$ is defined as in \eqref{redef-GN} with $N=0$. Observe that $|k| |l|^{-1} \la k - l \ra^{-1} \ls 1$, and that since $|l\tau|\geq |ks|/2$, it also holds that $\la l\tau\ra^{-3} \ls \la ks\ra^{-3/2}\la l\tau\ra^{-3/2} \ls \la s \ra^{-3/2} \la \tau \ra^{-3/2}$. Hence, it suffices to bound the following term:
\begin{equation*}
\begin{split}
&\: \Big[ \int_0^t [\sum_{k \not = 0} (\sum_{\substack{ l \not = 0, k}} \int_0^s |l|^{-4} \langle s \rangle^{-3/2} \la \tau \ra^{-3/2} \mathbf e^{1/2}(\tau) (|l|^4 \la l\tau \ra^3 |\hat\rho_l(\tau)|) |k-l| \| \hat f_{k-l}(\tau)\|_{\wtG_0'}  \; \d\tau)^2 ]^{1/2} \, \d s \Big]^2\\
\ls &\:  \sum_{k \not = 0} (\sum_{\substack{ l \not = 0, k }} \int_0^t |l|^{-4} \la \tau\ra^{-3/2} \mathbf e^{1/2}(\tau) (|l|^4 \la l\tau \ra^3 |\hat\rho_l(\tau)|) |k-l| \| \hat f_{k-l}(\tau)\|_{\wtG_0'}  \, \d\tau)^2 (\int_0^t \la s\ra^{-3/2} \, \d s)^2\\
\ls &\: \sum_{k \not = 0} \Big[ \sum_{\substack{ l \not = 0, k }} (\int_0^t |l|^{-8} \la \tau\ra^{-3/2} ( |l|^8 \la l\tau \ra^6 |\hat{\rho}_l(\tau)|^2 ) \, \d\tau)^{1/2} (\int_0^t \la \tau\ra^{-3/2} \mathbf e(\tau) |k-l|^2 \| \hat f_{k-l}(\tau)\|_{\wtG_0'}^2  \, \d\tau)^{1/2} \Big]^2 \\
\ls &\: \sum_{k \not = 0} \Big[ \sum_{\substack{ l \not = 0, k }}  |l|^{-4} (\sup_{l'\not = 0} \sup_{\tau'\in [0,t]} |l'|^4 \la l'\tau\ra^3 |\hat{\rho}_{l'}(\tau')|) (\int_0^t \la \tau\ra^{-3/2} \mathbf e(\tau) |k-l|^2 \| \hat f_{k-l}(\tau)\|_{\wtG_0'}^2  \, \d\tau)^{1/2} \Big]^2 \\
\ls &\: \ep \nu^{1/3} \Big(\sum_{l \not = 0} |l|^{-4} \Big)^2 \Big( \sum_{k \not = 0} \int_0^t \la \tau\ra^{-3/2}\mathbf e(\tau) |k|^2 \| \hat f_{k}(\tau)\|_{\wtG_0'}^2  \, \d\tau \Big) \\
\ls &\: \ep \nu^{1/3} (\sup_{\tau'\in [0,t]} \sum_{k \not = 0}\mathbf e(\tau') |k|^2 \| \hat f_{k}(\tau')\|_{\wtG_0'}^2) \int_0^t \la \tau\ra^{-3/2}   \, \d\tau \\
\ls &\: \ep \nu^{1/3} \sup_{\tau'\in [0,t]} \sum_{k \not = 0}\mathbf e(\tau') |k|^2 \| \hat f_{k}(\tau')\|_{\wtG_0'}^2 \ls \ep,
\end{split}
\end{equation*}
where we used the Cauchy--Schwarz and the Young's convolution inequalities, respectively, for the $\tau$-integral and for the sums, as well as bounded $\sup_{l'\not = 0} \sup_{\tau'\in [0,t]}  |l'|^4 \la l'\tau\ra^3 |\hat{\rho}_{l'}(\tau')| \ls \ep \nu^{1/3}$ using \eqref{density-bound}. Finally, we used \eqref{eq:G.by.E-extraextra} in the very last inequality.

The case where $|l\tau| \geq |ks|/2$ with $k = l$ has to be treated differently, since in this case $\hat{f}_{k-l}$ corresponds to the zeroth mode and does not experience enhanced dissipation. We bound the integrand using \eqref{eq:decay.integrand.1}. Considering only $k = l$, it thus suffices to bound
$$\Big[ \int_0^t \Big( \sum_{k \not = 0} (  \int_0^s |k|^{-1}   \langle k(s - \tau) \rangle^{-2}  \mathbf e^{1/2}(\tau) | k ||\hat\rho_k(\tau)| \| \hat f_{0}(\tau)\|_{\wtG_{N_{max}-2}'}  \; \d\tau )^2 \Big)^{1/2}\, \d s \Big]^2.$$
We use Minkowski's inequality so that the $\ell^2$ sum in $k$ is taken first, and then use Fubini's theorem to integrate out the $\la s-\tau\ra^{-2}$ factor. More precisely,
\begin{equation}
\begin{split}
&\: \Big[ \int_0^t \Big( \sum_{k \not = 0} (  \int_0^s |k|^{-1}   \langle k(s - \tau) \rangle^{-2}  \mathbf e^{1/2}(\tau) | k ||\hat\rho_k(\tau)| \| \hat f_{0}(\tau)\|_{\wtG_{N_{max}-2}'}  \; \d\tau )^2 \Big)^{1/2}\, \d s \Big]^2\\
\ls &\: \Big[\int_0^t \int_0^s \Big( \sum_{k \not = 0}   |k|^{-2}   \langle k(s - \tau) \rangle^{-4}  \mathbf e(\tau) | k |^2|\hat\rho_k(\tau)|^2 \| \hat f_{0}(\tau)\|_{\wtG_{N_{max}-2}'}^2 \Big)^{1/2}\, \d \tau\, \d s \Big]^2\\
\ls &\: \Big[\int_0^t \Big( \sum_{k \not = 0}   |k|^{-2} \mathbf e(\tau) | k |^2|\hat\rho_k(\tau)|^2 \| \hat f_{0}(\tau)\|_{\wtG_{N_{max}-2}'}^2 \Big)^{1/2} (\int_{\tau}^t \langle s - \tau \rangle^{-2} \, \d s ) \, \d \tau \Big]^2\\
\ls &\: \Big[\int_0^t [\sum_{k \not = 0}  \mathbf e(\tau) | k |^2|\hat\rho_k(\tau)|^2 ]^{1/2} \| \hat f_{0}(\tau)\|_{\wtG_{N_{max}-2}'} \, \d \tau \Big]^2.
\end{split}
\end{equation}
To proceed, we split $\hat{\rho}_k = \hat{\rho}_k^{(1)} + \hat{\rho}_k^{(2)}$ as in \eqref{eq:rho1.def}, so that by using H\"older's inequality, \eqref{eq:G.by.E-extra}, and \eqref{eq:zete.def}, we have
\begin{equation}
\begin{split}
&\: \Big[\int_0^t [\sum_{k \not = 0}  \mathbf e(\tau) | k |^2|\hat\rho_k(\tau)|^2 ]^{1/2} \| \hat f_{0}(\tau)\|_{\wtG_{N_{max}-2}'} \, \d \tau \Big]^2 \\
\ls &\: \Big(\int_0^t [\sum_{k \not = 0}  \mathbf e(\tau) | k |^2|\hat\rho_k^{(1)}(\tau)|^2 ]^{1/2} \, \d \tau \Big)^2 \Big( \sup_{\tau' \in [0,t]} \| \hat f_{0}(\tau')\|_{\wtG_{N_{max}-2}'} \Big)^2 \\
&\: +  \Big(\int_0^t \sum_{k \not = 0}  \mathbf e(\tau) | k |^2|\hat\rho_k^{(2)}(\tau)|^2 \d \tau \Big)\Big(\int_0^t \| \hat f_{0}(\tau)\|_{\wtG_{N_{max}-2}'}^2 \, \d \tau \Big) 
\ls  \ep \zeta_e(t).
\end{split}
\end{equation}

Combining all the above cases, we have thus proven the bound \eqref{eq:density.decay.1} for the term \eqref{eq:main.term.L1t.II.decay}.

\subsection{Nonlinear interaction II} Finally, we prove the bounds for $\mathbb{III}_k(t)$ corresponding to those required in \eqref{eq:density.decay.2}.

We argue as in \eqref{eq:3k.first.bounds} with $N_1 = 1$, $N_2 = 0$, but also take into account the stretched exponential decay given by \eqref{decay-combineSk.low.nu.decay} to obtain
\begin{equation*}
\begin{aligned}
&|k|\Big|  \int_{\R^3}S_k(t-\tau) [ (\widehat{\Gamma(f,f)})_k(\tau) ]\sqrt{\mu}\d v\Big| 
\\
&\ls   \min \{ e^{-\de'(\nu^{1/3}(t-\tau))^{1/3}}, e^{-\de'(\nu(t-\tau))^{2/3}}\} \qquad  \smashoperator{ \sum_{\substack{|\alp| = 1,\, |\bt|\leq 1}}} \qquad \nu^{|\bt|/3} \|  \langle v\rangle^{2} e^{\frac14 q_0 |v|^2}  [ \rd_x^{\alp} \rd_v^\bt (\Gamma(f,f))]\sphat_k(\tau)  \|_{L^2_v}.
\end{aligned}
\end{equation*}

To control the $\Gamma(f,f)$ term, we argue as in \eqref{eq:estimating.Gamma.1}, \eqref{eq:estimating.Gamma}, with the help of Lemma \ref{lem:Gamma.trivial}, except for noticing that, importantly, there is exactly one factor with a $\rd_x$ derivative. As in \eqref{eq:estimating.Gamma.1}, \eqref{eq:estimating.Gamma}, we still put a factor with at least one $\rd_v$ derivative in the $\wtD$ norm, and another factor in the $\wtE$ norm. We then put the factor with the $\rd_x$ derivative in $L^2_x$, and the other factor will be bounded in $L^\i_x$ together with Sobolev embedding. The factor with exactly one $\rd_x$ derivative can then by put into either that $\wtbE^{(2)'}_{1,1,0,0}$ or the $\wtbD^{(2)'}_{1,1,0,0}$ norm. Hence, we have
\begin{equation}\label{eq:nonlinear.collision.in.density.decay}
\begin{split}
&\: \sum_{k\not =0} \sum_{\substack{|\alp| = 1,\, |\bt|\leq 1}}  \nu^{2|\bt|/3} \|  \langle v\rangle^{2} e^{\frac14 q_0 |v|^2} [ \rd_x^{\alp} \rd_v^\bt (\Gamma(f,f))]\sphat_k(\tau)  \|_{L^2_v}^2 \\
\ls &\: \nu^{-4/3} (\| f(\tau)\|_{\wtbE^{(2)'}_{N_{max}-2}}^2 \|f(\tau) \|_{\wtbD^{(2)'}_{1,1,0,0}}^2 + \| f(\tau)\|_{\wtbD^{(2)'}_{N_{max}-2}}^2 \|f(\tau) \|_{\wtbE^{(2)'}_{1,1,0,0}}^2).
\end{split}
\end{equation}
Therefore, taking $\mathbf e(t) \in \{ e^{\de(\nu^{1/3} t)^{1/3}}, e^{\de(\nu t)^{2/3}}\}$, noting  
$$\mathbf e(t) \min \{ e^{-\de'(\nu^{1/3}(t-\tau))^{1/3}}, e^{-\de'(\nu(t-\tau))^{2/3}}\}  \ls \mathbf e(\tau) \min \{ e^{-(\de'/2)(\nu^{1/3}(t-\tau))^{1/3}}, e^{-(\de'/2)(\nu(t-\tau))^{2/3}}\},$$
and using the Cauchy--Schwarz inequality in $\tau$, we obtain
\begin{equation*}
\begin{split}
&\: \sum_{k \not = 0}  \mathbf e(t) |k|^{2} |\mathbb{III}_{k}(t)|^2 = \nu^2 \sum_{k \not = 0} \mathbf e(t) |k|^2 \Big| \int_0^t\int_{\R^3}S_k(t-\tau) [ (\widehat{\Gamma(f,f)})_k(\tau) ]\sqrt{\mu}\d v\d \tau \Big|^2 
\\
\ls &\:\nu^2 \sum_{k \not = 0} \mathbf e(t)  |k|^{2} \int_0^t e^{(\de'/2)(\nu^{1/3}(t-\tau))^{1/3}} \Big| \int_{\R^3}S_k(t-\tau) [ (\widehat{\Gamma(f,f)})_k(\tau) ]\sqrt{\mu}\d v\Big|^2 \d \tau \\
&\: \quad \times \int_0^t e^{-(\de'/2)(\nu^{1/3}(t-\tau))^{1/3}}\, \ud \tau 
\\
\ls &\:\nu^{5/3} \sum_{k \not = 0} \mathbf e(t)  |k|^{2} \int_0^t e^{(\de'/2)(\nu^{1/3}(t-\tau))^{1/3}} \Big| \int_{\R^3}S_k(t-\tau) [ (\widehat{\Gamma(f,f)})_k(\tau) ]\sqrt{\mu}\d v\Big|^2 \d \tau \\
\ls&\: \nu^{1/3} \sup_{\tau' \in [0,t]} \| f(\tau')\|_{\wtbE^{(2)'}_{N_{max}-2}}^2 \int_0^t  \mathbf e(\tau)  e^{-(\de'/2)(\nu^{1/3}(t-\tau))^{1/3}}
\| f (\tau)\|^2_{\wtbD^{(2)'}_{1,1,0,0}} \, \d \tau   \\
&+ \nu^{1/3} \sup_{\tau' \in [0,t]} \mathbf e(\tau')  \| f(\tau')\|_{\wtbE^{(2)'}_{1,1,0,0}}^2 \int_0^t   e^{-(\de'/2)(\nu^{1/3}(t-\tau))^{1/3}}
 \| f (\tau)\|^2_{\wtbD^{(2)'}_{N_{max}-2}} \, \d \tau .
\end{split}
\end{equation*}
By the energy bound \eqref{mainEE-close} and then the bootstrap assumption \eqref{bootstrap-f1}, this implies 
\begin{equation*}
\begin{split}
\: \sum_{k \not = 0}   \mathbf e(t) |k|^{2} |\mathbb{III}_{k}(t)|^2 
\ls&\: \epsilon\nu \int_0^t  \mathbf e(\tau)  \| f (\tau)\|^2_{\wtbD^{(2)'}_{1,1,0,0}} \, \d \tau  
+ \epsilon \nu^{2/3} \sup_{\tau' \in [0,t]} \mathbf e(\tau') \| f(\tau')\|_{\wtbE^{(2)'}_{1,1,0,0}}^2
\\\ls & \: \epsilon^2\nu^{4/3} .
\end{split}
\end{equation*}
%
An identical argument, using additionally Fubini's theorem, gives the desired $L^2_t$ bound:
\begin{equation*}
\begin{split}
&\int_0^t \sum_{k \not = 0}    \mathbf e(s) |k|^{2} |\mathbb{III}_{k}(s)|^2 \, \ud s\\
&\ls \ep \nu  \int_0^t (\int_{\tau}^t e^{-(\de'/2) (\nu^{1/3}(s-\tau))^{1/3}} \, \ud s) [ \mathbf e(\tau) \| f (\tau)\|_{\wtbD^{(2)'}_{1,1,0,0}}^2 +  \| f (\tau)\|^2_{\wtbD^{(2)'}_{N_{max}-2}} ] \, \d \tau 
 \\
& \ls \ep \nu^{2/3}  \int_0^t [ \mathbf e(\tau) \| f (\tau)\|_{\wtbD^{(2)'}_{1,1,0,0}}^2+ \| f (\tau)\|_{\wtbD^{(2)'}_{N_{max}-2}}^2] \, \d \tau 
\\& \ls \ep^2 \nu.
\end{split}
\end{equation*}
This ends the proof of Proposition~\ref{prop:density.decay}, and thus of Theorem~\ref{theo-density-decay}.  

\section{Nonlinear energy decay}\label{sec:exp-decay-energy}

In this section, we establish the nonlinear energy decay estimates for the full nonlinear Vlasov--Poisson--Landau equation \eqref{eq:Vlasov.f}--\eqref{eq:Poisson.f}. Throughout this section, we shall use primed energy and dissipation norms $\| \cdot \|_{\wtbE^{(2)'}_{*,*,0,0}}$ and $\|\cdot \|_{\wtbD^{(2)'}_{*,*,0,0}}$, which are defined as in \eqref{eq:combined.norms.top} with the primed exponential weights $e^{q'|v|^2}$ for $q' = \frac12 q$ (cf.~\eqref{eq:p.norms.general}).

The main result of this section is the following. 

\begin{theorem}\label{thm:energy.decay} Consider data as in Theorem~\ref{t.main}. Then, the following hold:
\begin{enumerate}
\item The energy of $f$ decays with the following stretched exponential rate:
$$\sup_{0\leq \tau < \infty} e^{\de (\nu \tau)^{2/3}} \| f(\tau) \|^2_{\wtbE^{(2)'}_{0,0,0,0}}+  \nu^{1/3} \int_0^{\infty} e^{\de (\nu \tau)^{2/3}} \| f(\tau) \|^2_{\wtbD^{(2)'}_{0,0,0,0}} \, \d \tau \ls \ep^2 \nu^{2/3}.$$
\item The energy of $\nab_x f$ decays with the following enhanced stretched exponential rate:
$$
\sup_{0\leq \tau <\infty} \mathbf e(\tau) \norm{f(\tau)}^2_{\wtbE^{(2)'}_{1,1,0,0}}
 + \nu^{1/3}\int_0^{\infty} \mathbf e(\tau) \norm{f(\tau)}^2_{\wtbD^{(2)'}_{1,1,0,0}}\d \tau \leq \ep \nu^{2/3},
$$
for $\mathbf e(t) \in \{ e^{\de(\nu^{1/3}t)^{1/3}}, e^{\de(\nu t)^{2/3}} \}$.
\end{enumerate}
\end{theorem}

\subsection{Preliminary energy estimates}

The first step in the proof of Theorem~\ref{thm:energy.decay} is the following energy estimates for the lowest order energies. 

\begin{proposition}\label{prop:EE.decay}
The following energy estimates hold for $\mathbf e(t) \in \{e^{\de(\nu^{1/3} t)^{1/3}}, e^{\de(\nu t)^{2/3}} \}$:
\begin{equation}
\begin{aligned}
 \frac{d}{dt} ( e^{1+\la t\ra^{-1}}\|f(t) \|_{\wtbE^{(2)'}_{0,0,0,0}}^{2} ) & + \theta \nu^{1/3} \| f (t) \|^2_{\wtbD^{(2)'}_{0,0,0,0}} +  \la t\ra^{-2} \| f(t)\|^2_{\wtbE^{(2)'}_{0,0,0,0}}
 \\ \quad&\lesssim\quad 
 \|\rho_{\not = 0}(t)\|_{L^2_x}^4 + \nu^{1/3} \|f(t)\|_{\wtbE^{(2)'}_{0,0,0,0}} \|f(t) \|_{\wtbD^{(2)'}_{0,0,0,0}} \|f(t) \|_{\wtbD_{N_{max}-2}}  
\\
&\: \quad+  
\min\{ \| f(t)\|_{\wtbE^{(2)'}_{0,0,0,0}} , \; \|f(t) \|_{\wtbD^{(2)'}_{0,0,0,0}}\}   \| \rho_{\not =0} (t)\|_{L^2_x},
 \end{aligned}\end{equation} 
 and 
\begin{equation}
\begin{aligned}
 \frac{d}{dt} ( e^{1+\la t\ra^{-1}}\|f(t) \|_{\wtbE^{(2)'}_{1,1,0,0}}^{2}) &+ \theta \nu^{1/3} \| f (t) \|^2_{\wtbD^{(2)'}_{1,1,0,0}} + \la t\ra^{-2} \| f(t)\|^2_{\wtbE^{(2)'}_{1,1,0,0}}
 \\ \quad&\lesssim\quad 
\|\rho_{\not = 0}(t)\|_{L^2_x}^4 + \nu^{1/3} \|f(t)\|_{\wtbE^{(2)'}_{1,1,0,0}} \|f(t) \|_{\wtbD^{(2)'}_{1,1,0,0}} \|f(t) \|_{\wtbD_{N_{max}-2}} 
\\
&\: \quad+  \min\{ \| f(t)\|_{\wtbE^{(2)'}_{1,1,0,0}} , \;\|f(t) \|_{\wtbD^{(2)'}_{1,1,0,0}}\}  \sum_{|\alp|\leq 1} \| \rd_x^\alp \rho_{\not =0} (t)\|_{L^2_x}.
 \end{aligned}\end{equation} 
 \end{proposition} 
 \begin{proof} 
 The proof is similar to that of Theorem~\ref{theo-mainEE} in Section~\ref{s.closing_eng}, except that we use different bounds for $|c|^2$ and for the remainder terms $\widetilde{R}_{\alp,0,0}$.

For the $c$ term, we simply use \eqref{cv} and that $E = \nabla \Delta^{-1} \rho_{\not = 0}$ to bound
$$|c|(t)\leq \int_{\R^3} |E|^2(t,x) \, \ud x \ls \|\rho_{\not = 0}(t)\|_{L^2_x}^2.$$
 

We then turn to the bounds for $\widetilde{\mathcal R}_{\alp,\bt,\om}$. Here, we take $\widetilde{\mathcal R}_{\alp,\bt,\om}$ to be as in the $\vartheta = 2$ case in Proposition~\ref{prop:top.order.energy}, except that the $e^{q|v|^\vartheta}$ weights are replaced by $e^{q'|v|^\vartheta}$ with $q' = \f 12 q$.
 
In view of the proof of Proposition \ref{prop:energy.error},  with $N=0$ and $N=1$, 
we first note the following bounds on the inhomogeneous terms $\widetilde{R}_{\alp,0,0}$: 
\begin{equation*}
\begin{split}
\widetilde{\mathcal R}_{0,0,0} 
\ls &\: \ep^{1/2} \la t\ra^{-2} \| f(t)\|^2_{\wtbE^{(2)'}_{0,0,0,0}}+ \ep^{1/2} \nu^{2/3} \| f(t)\|_{\wtbD^{(2)'}_{0,0,0,0}}^2 + \nu^{1/3} \|f(t)\|_{\wtbE^{(2)'}_{0,0,0,0}} \|f(t) \|_{\wtbD^{(2)'}_{0,0,0,0}} \|f(t) \|_{\wtbD^{(2)'}_{N_{max}-2}} \\
&\: + 
\min\{ \| f(t)\|_{\wtbE^{(2)'}_{0,0,0,0}} , \; \|f(t) \|_{\wtbD^{(2)'}_{0,0,0,0}}\}   \| \rho_{\not =0} (t)\|_{L^2_x},
\end{split}
\end{equation*}
and
\begin{equation*}
\begin{split}
\sum_{|\alp| = 1} \widetilde{\mathcal R}_{\alp,0,0} 
\ls &\: \ep^{1/2} \la t\ra^{-2} \| f(t)\|^2_{\wtbE^{(2)'}_{1,1,0,0}}+ \ep^{1/2} \nu^{2/3} \| f(t)\|_{\wtbD^{(2)'}_{1,1,0,0}}^2 \\
&\: + \nu^{1/3} \|f(t)\|_{\wtbE^{(2)'}_{1,1,0,0}} \|f(t) \|_{\wtbD^{(2)'}_{1,1,0,0}} \|f(t) \|_{\wtbD^{(2)'}_{N_{max}-2}} \\
&\: + \min\{ \| f(t)\|_{\wtbE^{(2)'}_{1,1,0,0}} , \;\|f(t) \|_{\wtbD^{(2)'}_{1,1,0,0}}\}  \sum_{|\alp|\leq 1} \| \rd_x^\alp \rho_{\not =0} (t)\|_{L^2_x}.
\end{split}
\end{equation*}

Using the above estimates, the proposition thus follows in a similar manner as in deriving \eqref{eq:main.nonlinear.energy.absorbed}. 
\end{proof}

 \subsection{Decay estimates}
 
We now give the proof of Theorem \ref{thm:energy.decay}. We shall only prove the enhanced decay rate, part (2) in the theorem; the other part is similar, if not simpler. 

 
\subsection*{Applying Lemma~\ref{lem:SG}.} We proceed by a bootstrap argument. Assume that there is $T_B >0$ such that the bootstrap assumption \eqref{bootstrap-f1} holds. In particular, we can use the bounds derived in Theorem~\ref{theo-density-decay}.

We will use the Strain--Guo type estimate in Lemma~\ref{lem:SG}. For the remainder of the proof, fix either $\mathbf e(t) = e^{\de(\nu t)^{2/3}}$ or $\mathbf e(t) = e^{\de (\nu^{1/3} t)^{2/3}}$. Define $g$ and $h$ so that
\begin{equation}\label{eq:g.h.for.decay}
\int_{\R^3} g^2(t,v) \, \d v = e^{1+\la t\ra^{-1}} \| f(t)\|_{\wtbE^{(2)'}_{1,1,0,0}}^2,\quad \int_{\R^3} h^2(t,v)\, \d v = e^{1+\la t\ra^{-1}}\| f(t) \|_{\wtbD^{(2)'}_{1,1,0,0}}^2,
\end{equation}
noting that the factor $e^{1+\la t\ra^{-1}}$ is harmless. Note also that the primed exponential weights $e^{q'|v|^2}$ are used. Specifically,
\begin{equation*}
\begin{split}
g^2(t,v) := &\: \sum_{|\alp| =1} \Big[ A_0 \sum_{|\alp'|\leq 1} \int_{\T^3} \la v \ra^{4M-4|\alp'|} |\rd_x^{\alp+ \alp'}f|^2\,\d x + \nu^{1/3} \int_{\T^3}\la v\ra^{4M-4} \la \nab_x \rd_x^\alp f, \nab_v \rd_x^\alp f \ra \, \d x \\
&\:\qquad  + \sum_{|\bt'| =1,2} \nu^{2|\bt'|/3} \int_{\T^3} \la v \ra^{4M-4|\bt'|} |\rd_x^\alp \rd_v^{\bt'} f|^2 \, \d x\Big]  e^{1+\la t\ra^{-1}} e^{\frac12 q_0|v|^2} ,
\end{split}
\end{equation*}
noting the exponential weight $e^{\frac12 q_0|v|^2} $ inserted above. A similar definition is introduced for $h^2(t,v)$ to satisfy \eqref{eq:g.h.for.decay}.
 
By definition, we note that 
$$\nu \int_{\R^3} \la v\ra^{-1} g^2(t,v) \, \d v \ls \nu^{1/3} \| f\|_{\wtbD^{(2)'}_{1,1,0,0}}^2, \quad \nu^{1/3} \int_{\R^3} \la v\ra^{-4} g^2(t,v) \, \d v \ls \nu^{1/3} \| f\|_{\wtbD^{(2)'}_{1,1,0,0}}^2,$$
where Poincar\'e's inequality was used in obtaining the second inequality, upon noting that $\rd_x^\alp f$ has zero $x$-mean with $|\alpha|=1$.

Therefore, after taking $\theta$ smaller if necessary, Proposition~\ref{prop:EE.decay} and the definitions of $\wtbE^{(2)'}_{1,1,0,0}$ and $\wtbD^{(2)'}_{1,1,0,0}$ imply that the differential inequality \eqref{eq:SG.lem.2} holds with $\mfc = \theta\nu^{1/3}$, $\mathfrak b = \theta\nu^{1/3}$, $\mfm = 4$, i.e.
\begin{equation}
\begin{split}
\f{d}{dt} \int_{\R^3} g^2(t,v) \, \ud v + \theta \nu^{1/3} \int_{ \R^3} \bv^{-4} g^2(t,v) \, \ud v + \theta \nu^{1/3} \int_{ \R^3} h^2(t,v) \, \ud v \ls \mathfrak F(t),
\end{split}
\end{equation}
and with $\mfc = \theta\nu$, $\mathfrak b = \theta\nu^{1/3}$, $\mfm = 1$, i.e.
\begin{equation}
\begin{split}
\f{d}{dt} \int_{\R^3} g^2(t,v) \, \ud v + \theta \nu \int_{ \R^3} \bv^{-1} g^2(t,v) \, \ud v + \theta \nu^{1/3} \int_{ \R^3} h^2(t,v) \, \ud v \ls \mathfrak F(t),
\end{split}
\end{equation}
where $\mathfrak F(t)$ is given by
\begin{equation}\label{eq:F.for.decay}
 \begin{split}
\mathfrak F(t) :=  &\:  \|\rho_{\not = 0}(t)\|_{L^2_x}^4 + \nu^{1/3} \|f(t)\|_{\wtbE^{(2)'}_{1,1,0,0}} \|f(t) \|_{\wtbD^{(2)'}_{1,1,0,0}} \|f(t) \|_{\wtbD^{(2)'}_{N_{max}-2}} \\
&\: +  \min \{ \|f(t) \|_{\wtbE^{(2)'}_{1,1,0,0}} , \| f(t) \|_{\wtbD^{(2)'}_{1,1,0,0}} \} \sum_{|\alp| \leq 1} \| \rd_x^\alp \rho_{\not = 0} (t) \|_{L^2_x}.
 \end{split}
 \end{equation}
Let $\mathbf e(\tau)  \in \{ e^{\de (\nu t)^{2/3}}, e^{\de (\nu^{1/3}t)^{1/3}} \}$. We will prove below that for any $\eta >0$, the following holds uniformly for all $T\in (0,T_B)$:
\begin{equation}\label{eq:E.F.int.goal}
\begin{split}
&\: \int_0^T \mathbf e(t) \mathfrak F(t) \, \ud t \\
\ls &\: \eta^{-1} \ep^2 \nu^{2/3} + (\ep \nu^{1/3}+ \eta)\Big[ \sup_{0\leq t < T} \mathbf e(t) \| f(t) \|^2_{\wtbE^{(2)'}_{1,1,0,0}}+  \nu^{1/3} \int_0^{T} \mathbf e(\tau) \| f(\tau) \|^2_{\wtbD^{(2)'}_{1,1,0,0}} \, \d \tau \Big].
\end{split}
\end{equation}

Note also that in view of the primed energy and dissipation norms, the boundedness of the corresponding unprimed norms yields the boundedness of exponential moments for $g$ (as is needed by \eqref{eq:SG.lem.1} in Lemma~\ref{lem:SG}). Namely, using Theorem~\ref{theo-mainEE}, we have  
$$ \int_{\R^3} e^{\frac12 q_0|v|^2} g^2(t,v) \, \ud v \ls \| f(t)\|_{\wtbE^{(2)}_{1,1,0,0}}^2 \ls \ep^2 \nu^{2/3}.$$

Therefore, we can apply Lemma~\ref{lem:SG}, using \eqref{eq:E.F.int.goal} and recalling \eqref{eq:g.h.for.decay}, to deduce 
\begin{equation}
\begin{split}
&\: \sup_{0\leq t < T} \mathbf e(t)  \| f(t) \|^2_{\wtbE^{(2)'}_{1,1,0,0}}+  \nu^{1/3} \int_0^{T} \mathbf e(\tau) \| f(\tau) \|^2_{\wtbD^{(2)'}_{1,1,0,0}} \, \d \tau \\
\ls &\: \eta^{-1} \ep^2 \nu^{2/3} + (\ep \nu^{1/3}+ \eta)\Big[ \sup_{0\leq t < T} \mathbf e(t) \| f(t) \|^2_{\wtbE^{(2)'}_{1,1,0,0}}+  \nu^{1/3} \int_0^{T} \mathbf e(\tau) \| f(\tau) \|^2_{\wtbD^{(2)'}_{1,1,0,0}} \, \d \tau \Big].
\end{split}
\end{equation}
Taking $\ep_0$, $\nu_0$ and $\eta$ sufficiently small, we can absorb the final term to the LHS, which yields 
\begin{equation}\label{eq:exp.decay.final}
\begin{split}
&\: \sup_{0\leq t < T} \mathbf e(t)  \| f(t) \|^2_{\wtbE^{(2)'}_{1,1,0,0}}+  \nu^{1/3} \int_0^{T} \mathbf e(\tau) \| f(\tau) \|^2_{\wtbD^{(2)'}_{1,1,0,0}} \, \d \tau 
\ls \ep^2 \nu^{2/3},
\end{split}
\end{equation}
after fixing $\eta>0$. This then improves the bootstrap assumption \eqref{bootstrap-f1}. In particular, this closes the bootstrap argument, and show that \eqref{eq:exp.decay.final} holds for all $t\in (0,\infty)$, which implies the desired estimate in Theorem~\ref{thm:energy.decay}.

It thus remains to prove the claim \eqref{eq:E.F.int.goal}, under the bootstrap assumption \eqref{bootstrap-f1}.

\subsection*{Controlling $\mathfrak F(t)$.} To prove \eqref{eq:E.F.int.goal}, we control each of the three terms in \eqref{eq:F.for.decay}. For the first term, we use \eqref{density-bound}  and \eqref{density-bound-decay} to obtain 
\begin{equation}\label{eq:F.bound.0}
 \int_0^T  \mathbf e(t) \|\rho_{\not = 0}\|_{L^2_x}^4(t) \, \ud t \ls  \|\mathbf e(t) \rho_{\not =0}(t)\|_{L^2_t([0,T];L^2_x)}^2 \|\rho_{\not = 0}(t)\|_{L^\infty_t([0,T];L^2_x)}^2 \ls \ep^4 \nu^{4/3}.
\end{equation}

For the second term, by H\"older's inequality,
 \begin{equation}\label{eq:F.bound.1}
 \begin{split}
&\: \nu^{1/3} \int_0^T \mathbf e(t) \|f(t)\|_{\wtbE^{(2)'}_{1,1,0,0}} \|f(t) \|_{\wtbD^{(2)'}_{1,1,0,0}} \|f(t) \|_{\wtbD^{(2)'}_{N_{max}-2}} \, \d t \\
\ls &\: (\sup_{0\leq \tau <T} \mathbf e^{1/2}(\tau)\|f(\tau)\|_{\wtbE^{(2)'}_{1,1,0,0}}) (\nu^{1/3} \int_0^T \mathbf e(t) \|f(t) \|_{\wtbD^{(2)'}_{1,1,0,0}}^2 \, \d t )^{1/2} (\nu^{1/3} \int_0^T  \|f(t) \|_{\wtbD^{(2)'}_{N_{max}-2}}^2 \, \d t )^{1/2} \\
\ls &\: \ep \nu^{1/3}\Big[ \sup_{0\leq t < T} \mathbf e(t) \| f(t) \|^2_{\wtbE^{(2)'}_{1,1,0,0}}+  \nu^{1/3} \int_0^{T} \mathbf e(\tau) \| f(\tau) \|^2_{\wtbD^{(2)'}_{1,1,0,0}} \, \d \tau \Big],
 \end{split}
 \end{equation}
 where we have used the estimate established in \eqref{mainEE-close} for $\nu^{1/3} \int_0^T  \|f(t) \|_{\wtbD^{(2)'}_{N_{max}-2}}^2 \, \d t$.
 
For the remaining term, we decompose $\rho_{\not = 0} = \rho_{\not = 0}^{(1)} + \rho_{\not = 0}^{(2)}$ according to Theorem~\ref{theo-density-decay} so that
 \begin{equation}
 \begin{split}
 &\: \mathbf e(t) \min \{ \|f(t) \|_{\wtbE^{(2)'}_{1,1,0,0}} , \| f(t) \|_{\wtbD^{(2)'}_{1,1,0,0}} \} \sum_{|\alp| = 1} \| \rd_x^\alp \rho_{\not = 0} (t) \|_{L^2_x} \\
 \ls &\: \mathbf e(t)   \|f(t) \|_{\wtbE^{(2)'}_{1,1,0,0}} \sum_{|\alp| = 1} \| \rd_x^\alp \rho_{\not = 0}^{(1)} (t) \|_{L^2_x} + \mathbf e(t) \| f(t) \|_{\wtbD^{(2)'}_{1,1,0,0}} \sum_{|\alp| = 1} \| \rd_x^\alp \rho_{\not = 0}^{(2)} (t) \|_{L^2_x}
 \end{split}
 \end{equation}
Thus, using \eqref{eq:rho1.bound} and \eqref{eq:rho2.bound} respectively, as well as H\"older's and Young's inequality, we have
\begin{equation}
\begin{split}
&\: \int_0^{T} \mathbf e(t)  \|f(t) \|_{\wtbE^{(2)'}_{1,1,0,0}} \sum_{|\alp| = 1} \| \rd_x^\alp \rho_{\not = 0}^{(1)} (t) \|_{L^2_x}  \, \d t \\
\ls &\: \eta (\sup_{0\leq t < T} \mathbf e^{1/2}(t) \|f(t) \|_{\wtbE^{(2)'}_{1,1,0,0}})^2 + \eta^{-1} [\int_0^{T} \mathbf e^{1/2}(t) \sum_{|\alp| = 1} \| \rd_x^\alp \rho_{\not = 0}^{(1)} (t) \|_{L^2_x}  \, \d t ]^2 \\
\ls &\: \eta \sup_{0\leq t < T} \mathbf e(t) \|f(t) \|^2_{\wtbE^{(2)'}_{1,1,0,0}} + \eta^{-1} \ep^2 \nu^{2/3},
 \end{split}
 \end{equation}
 as well as
\begin{equation}\label{eq:F.bound.4}
\begin{split}
&\: \int_0^{T} \mathbf e(t) \|f(t) \|_{\wtbD^{(2)'}_{1,1,0,0}} \sum_{|\alp| = 1} \| \rd_x^\alp \rho_{\not = 0}^{(2)} (t) \|_{L^2_x}  \, \d t \\
\ls &\: \eta  \nu^{1/3} \int_0^{T} \mathbf e(t)   \|f(t) \|_{\wtbD^{(2)'}_{1,1,0,0}}^2 \, \d t + \eta^{-1} \nu^{-1/3}  \int_0^{T} \mathbf e(t)  \sum_{|\alp| = 1} \| \rd_x^\alp \rho_{\not = 0}^{(2)} (t) \|^2_{L^2_x}  \, \d t \\
\ls &\: \eta \int_0^{T} \mathbf e(t)   \|f(t) \|_{\wtbD^{(2)'}_{1,1,0,0}}^2 \, \d t + \eta^{-1} \ep^2 \nu^{2/3}.
 \end{split}
 \end{equation} 
Combining \eqref{eq:F.bound.1}--\eqref{eq:F.bound.4}, and recalling \eqref{eq:F.for.decay}, we have thus obtained \eqref{eq:E.F.int.goal}. 
This ends the proof of Theorem \ref{thm:energy.decay}. 

\qedhere

\section{Putting everything together}\label{sec:putting-together}

The main theorem, Theorem \ref{t.main}, now follows straightforwardly. Indeed, 

\begin{itemize}
\item Global existence of smooth solutions follows from Theorem~\ref{thm:existence}.
\item The estimates \eqref{eq:main.energy.lowest} and \eqref{eq:main.energy} follows from \eqref{mainEE-close}. 
\item The bounds \eqref{eq:main.decay} and \eqref{eq:main.decay.neq0} follow from interpolating Theorem~\ref{thm:energy.decay} with \eqref{eq:main.energy}.
\item Finally, for the uniform Landau damping statement \eqref{eq:uniform.Landau.damping}, we bound, using Parseval's theorem, interpolation, \eqref{density-bound} and \eqref{density-bound-decay}:
\begin{equation*}
\begin{split}
|\hat{\rho}_k|(t) \ls &\: \la k (t+1) \ra^{-N_{max}+1} \sum_{|\alp| + |\om| \leq N_{max}-1} \|\rd_x^\alp Y^\om \rho_{\not = 0} (t)\|_{L^2_x} \\
\ls &\: \la k (t+1) \ra^{-N_{max}+1} \|\rho_{\not = 0}\|_{L^2_x}^{1/N_{max}} (\sum_{|\alp| + |\om| \leq N_{max}} \|\rd_x^\alp Y^\om \rho_{\not = 0} (t)\|_{L^2_x})^{(N_{max}-1)/N_{max}}  \\
\ls &\: \ep\nu^{1/3} \la k (t+1) \ra^{-N_{max}+1} \min \{e^{-\de(\nu^{1/3} t)^{1/3}}, e^{-\de (\nu t)^{2/3}} \},
\end{split}
\end{equation*}
after taking $\de$ smaller. \qedhere
\end{itemize}
This completes the proof of Theorem \ref{t.main}.

\appendix

\section{Strain--Guo type lemmas}\label{sec:appendix}

\begin{lemma}\label{lem:SG}
Let $T\in (0,\infty]$ and $g:[0,T) \times \mathbb R^3 \to \mathbb R$ be a smooth function. Suppose there exist $\mfC>0$, $\mfc>0$, $\mathfrak b>0$, $\mathfrak m \geq 0$, $\mathfrak q \in (0, 2)$ and $\mathfrak p \in (0,\f{\mathfrak q}{2})$ such that the following holds:
\begin{enumerate}
\item There is a uniform bound of Gaussian moments:
\begin{equation}\label{eq:SG.lem.1}
\sup_{t\in [0,T)} \int_{\R^3} e^{\mathfrak q |v|^2} g^2(t,v) \, \ud v \leq \mfC.
\end{equation}
\item The following differential inequality holds for all $t\in [0,\infty)$:
\begin{equation}\label{eq:SG.lem.2}
\begin{split}
\f{d}{dt} \int_{\R^3} g^2(t,v) \, \ud v + &\: \mfc \int_{ \R^3} \bv^{-\mathfrak m} g^2(t,v) \, \ud v + \mathfrak b \int_{ \R^3} h^2(t,v) \, \ud v \leq \mathfrak F(t),
\end{split}
\end{equation}
for some function $h:[0,T)\times \R^3\to \mathbb R$, and some function $\mathfrak F: [0,T)\to \R$ satisfying
\begin{equation}\label{eq:SG.F}
\int_0^T e^{\mathfrak p (\mfc t)^{\f 2{2+\mathfrak m}}} \mathfrak F(t) \, \d t \leq \mfC.
\end{equation}
\end{enumerate}

Then, there exists $C_{\mathfrak q, \mathfrak m}>0$ (depending only on $\mathfrak q$ and $\mathfrak m$) such that
\begin{equation}\label{eq:SG.lem.goal}
\sup_{t\in [0,T)} e^{\mathfrak p (\mfc t)^{\f 2{2+\mathfrak m}}}\int_{\R^3} g^2(t,v) \, \ud v + \mathfrak b \int_0^T e^{\mathfrak p (\mfc t)^{\f 2{2+\mathfrak m}}} \int_{\R^3} h^2(t,v) \, \ud v\, \ud t  \leq C_{\mathfrak q, \mathfrak m} \mfC.
\end{equation}
\end{lemma}
\begin{proof}
We compute using \eqref{eq:SG.lem.2} that
\begin{equation}\label{eq:SG.lem.main}
\begin{split}
&\: \f{d}{dt} (e^{\mathfrak p (\mfc t)^{\f 2{2+\mathfrak m}}} \int_{ \R^3} g^2(t,v) \, \ud v ) + \mathfrak b e^{\mathfrak p (\mfc t)^{\f 2{2+\mathfrak m}}}\int_{ \R^3} h^2(t,v) \, \ud v \\
\leq &\: e^{\mathfrak p (\mfc t)^{\f 2{2+\mathfrak m}}} ( \f {2\mathfrak p \mfc^{\f{2}{2+\mathfrak m}}}{(2+\mathfrak m)  t^{\f{\mathfrak m}{2+\mathfrak m}}} \int_{\R^3} g^2 \, \ud v - \mfc \int_{\R^3} \bv^{-\mathfrak m} g^2 \, \ud v) + e^{\mathfrak p (\mfc t)^{\f 2{2+\mathfrak m}}} \mathfrak F(t). 
\end{split}
\end{equation}

We control the first term in \eqref{eq:SG.lem.main}. Splitting into $\bv \leq (\mfc t)^{\f 1{2+\mathfrak m}}$ and $\bv \geq (\mfc t)^{\f 1{2+\mathfrak m}}$, we bound the low velocity by $\int_{\R^3} \bv^{-\mfm} g^2(t,x,v) \, \ud v\, \ud x$ and the high velocity using \eqref{eq:SG.lem.1}:
\begin{equation}\label{eq:SG.lem.wrong.sign.term}
\begin{split}
 & \:  \f {2\mathfrak p \mfc^{\f{2}{2+\mathfrak m}}}{(2+\mathfrak m)  t^{\f{\mathfrak m}{2+\mathfrak m}}} \int_{\R^3} g^2(t,v) \, \ud v \leq  \f { \mfc^{\f{2}{2+\mathfrak m}}}{ t^{\f{\mathfrak m}{2+\mathfrak m}}} (\int_{\{v| \bv \leq (\mfc t)^{\f 1{2+\mathfrak m}} \}} + \int_{\{v| \bv \geq (\mfc t)^{\f 1{2+\mathfrak m}} \}}) g^2(t,v) \, \ud v\\
 \leq &\: \mfc \int_{\R^3} \bv^{-\mathfrak m} g^2 \, \ud v\, \ud x + \f { \mfc^{\f{2}{2+\mathfrak m}}}{ t^{\f{\mathfrak m}{2+\mathfrak m}}} e^{\mathfrak q} e^{- \mathfrak q (\mfc t)^{\f 2{2+\mathfrak m}}} \int_{\R^3} e^{\mathfrak q |v|^2} g^2 \,\ud v \\
 \leq &\: \mfc \int_{\R^3} \bv^{-\mathfrak m} g^2 \, \ud v\, \ud x + \f {\mfc^{\f{2}{2+\mathfrak m}}}{ t^{\f{\mathfrak m}{2+\mathfrak m}}} \mathfrak C e^{\mathfrak q}  e^{- \mathfrak q (\mfc t)^{\f 2{2+\mathfrak m}}} .
 \end{split}
 \end{equation}
We plug \eqref{eq:SG.lem.wrong.sign.term} into \eqref{eq:SG.lem.main} and use $\mathfrak p \leq \f{\mathfrak q}2$. Note that the $\int_{\R^3} \bv^{-\mathfrak m} g^2(t,v) \, \ud v$ terms cancel.
\begin{equation}
\begin{split}
&\: \f{d}{dt} (e^{\mathfrak p (\mfc t)^{\f 2{2+\mathfrak m}}} \int_{\R^3} g^2(t,v) \, \ud v ) + \mathfrak b e^{\mathfrak p (\mfc t)^{\f 2{2+\mathfrak m}}}\int_{\R^3} h^2(t,v) \, \ud v \\
 \leq &\: e^{\mathfrak p (\mfc t)^{\f 2{2+\mathfrak m}}} \mathfrak F(t) + \f {\mfc^{\f{2}{2+\mathfrak m}}}{ t^{\f{\mathfrak m}{2+\mathfrak m}}} \mathfrak C e^{\mathfrak q}  e^{- \f{\mathfrak q}2 (\mfc t)^{\f 2{2+\mathfrak m}}}.
\end{split}
\end{equation}
Integrating, using \eqref{eq:SG.lem.1} to bound the initial term $\int_{\R^3} g^2(0,v) \, \ud v$, and using \eqref{eq:SG.F} to bound the $L^1_t$ norm of $e^{\mathfrak p (\mfc t)^{\f 2{2+\mathfrak m}}} \mathfrak F(t)$, we have
\begin{equation}\label{eq:SG.almost}
\begin{split}
&\: \sup_{t\in [0,T)} e^{\mathfrak p (\mfc t)^{\f 2{2+\mathfrak m}}} \int_{\R^3} g^2(t,v) \, \ud v  + \mathfrak b \int_0^T e^{\mathfrak p (\mfc t)^{\f 2{2+\mathfrak m}}}\int_{\R^3} h^2(t,v) \, \ud v\, \ud t\\
\leq &\: 2\mfC + \mfC \int_0^\infty \f {\mfc^{\f{2}{2+\mathfrak m}}}{ t^{\f{\mathfrak m}{2+\mathfrak m}}} e^{\mathfrak q}  e^{- \f{\mathfrak q}2 (\mfc t)^{\f 2{2+\mathfrak m}}} \, \ud t.
\end{split}
\end{equation}
To bound the integral in \eqref{eq:SG.almost}, split the integration domain into $[0,\mfc^{-1}]$ and $[\mfc^{-1},\infty)$ so that
\begin{equation}\label{eq:SG.integral}
\begin{split}
&\: \int_0^\infty  \f {\mfc^{\f{2}{2+\mathfrak m}}}{ t^{\f{\mathfrak m}{2+\mathfrak m}}} e^{\mathfrak q}  e^{- \f{\mathfrak q}2 (\mfc t)^{\f 2{2+\mathfrak m}}}  \, \ud t 
\leq  e^{\mathfrak q} (\mfc^{\f{2}{2+\mathfrak m}} \int_0^{\mfc^{-1}} \f{\ud t}{ t^{\f{\mathfrak m}{2+\mathfrak m}}} + \int_{\mfc^{-1}}^\infty e^{- \f{\mathfrak q}2 (\mfc t)^{\f 2{2+\mathfrak m}}} \, \ud (\mfc t) ) \leq C'_{\mathfrak q,\mfm}
\end{split}
\end{equation}
for some $C'_{\mathfrak q,\mathfrak m}$. Plugging \eqref{eq:SG.integral} back into \eqref{eq:SG.almost} yields the conclusion.
 \qedhere
\end{proof}

\begin{lemma}\label{lem:SG.poly}
Let $g:[0,\infty) \times \mathbb R^3 \to \mathbb R$ be a smooth function. Suppose there exist $\mfC>0$ and $\mfc>0$ such that \begin{enumerate}
\item There is a uniform bound of the $4\mfm$-th moments:
\begin{equation}\label{eq:SG.lem.poly.1}
\sup_{t\in [0,\infty)} \int_{\R^3} \bv^{4\mfm} g^2(t,v) \, \ud v \leq \mfC.
\end{equation}
\item The following differential inequality holds for all $t\in [0,\infty)$:
\begin{equation}\label{eq:SG.lem.poly.2}
\begin{split}
&\: \f{d}{dt} \int_{\R^3} g^2(t,v) \, \ud v + \mfc \int_{\R^3} \bv^{-\mfm} g^2(t,v) \, \ud v
\leq 0.
\end{split}
\end{equation}
\end{enumerate}

Then
\begin{equation}\label{eq:SG.lem.goal.poly}
\int_{\R^3} g^2(t,v) \, \ud v \leq (\f{3^5 \pi}{2}+1) \mfC  \la \mfc t\ra^{-3}.
\end{equation}
\end{lemma}
\begin{proof}
We compute using \eqref{eq:SG.lem.poly.2} that
\begin{equation}\label{eq:SG.lem.poly.main}
\begin{split}
\f{d}{dt} (\la \mfc t\ra^3 \int_{\R^3} g^2(t,v) \, \ud v ) 
\leq &\: \la \mfc t\ra^3 ( \f {3\mfc^2 t}{\la \mfc t\ra^2} \int_{\R^3} g^2(t,v) \, \ud v - \mfc \int_{\R^3} \bv^{-\mfm} g^2(t,v) \, \ud v). 
\end{split}
\end{equation}

To bound the first term in \eqref{eq:SG.lem.poly.main}, we split into $\bv^{\mfm} \leq \f 13\la \mfc t\ra$ and $\bv^{\mfm}  \geq \f13\la \mfc t\ra$, then bound the low velocity by $\int_{\R^3} \bv^{-\mfm} g^2(t,v) \, \ud v$ and the high velocity using \eqref{eq:SG.lem.poly.1}:
\begin{equation}\label{eq:SG.lem.wrong.sign.term.poly}
\begin{split}
 & \:  \f {3\mfc^2 t}{\la \mfc t\ra^2} \int_{\R^3} g^2(t,v) \, \ud v = \f {3\mfc^2 t}{\la \mfc t\ra^2} (\int_{\{v| \bv^{\mfm} \leq \f 13\la \mfc t\ra \}} + \int_{\{v| \bv^{\mfm} \geq \f 13\la \mfc t\ra \}}) g^2(t,v) \, \ud v\\
 \leq &\: \mfc \int_{\R^3} \bv^{-\mfm} g^2 \, \ud v + \f {3^5 \mfc^2 t}{\la \mfc t\ra^6}  \int_{\R^3} \la v\ra^{4\mfm} g^2 \,\ud v 
 \leq \mfc \int_{\R^3} \bv^{-\mfm} g^2 \, \ud v + \f {3^5 \mfC\mfc^2 t}{\la \mfc t\ra^6}.
 \end{split}
 \end{equation}
We plug \eqref{eq:SG.lem.wrong.sign.term.poly} into \eqref{eq:SG.lem.poly.main}, noting that the $\int_{\R^3} \bv^{-\mfm} g^2(t,v) \, \ud v$ terms cancel. So
\begin{equation}
\begin{split}
 \f{d}{dt} (\la \mfc t\ra^3 \int_{\R^3} g^2(t,v) \, \ud v ) \leq  \f {3^5 \mfC\mfc^2 t}{\la \mfc t\ra^3} \leq \f {3^5 \mfC\mfc}{\la \mfc t\ra^2}.
\end{split}
\end{equation}
Integrating, and using \eqref{eq:SG.lem.poly.1} for the $t=0$ term yield the conclusion. \qedhere
\end{proof}

\def\cprime{$'$} \def\cprime{$'$} \def\cprime{$'$}


\begin{thebibliography}{100}

\bibitem{AMUXY12.3}
R.~Alexandre, Y.~Morimoto, S.~Ukai, C.-J. Xu, and T.~Yang.
\newblock The {B}oltzmann equation without angular cutoff in the whole space:
  {II}, {G}lobal existence for hard potential.
\newblock {\em Anal. Appl. (Singap.)}, 9(2):113--134, 2011.

\bibitem{AMUXY12}
R.~Alexandre, Y.~Morimoto, S.~Ukai, C.-J. Xu, and T.~Yang.
\newblock Global existence and full regularity of the {B}oltzmann equation
  without angular cutoff.
\newblock {\em Comm. Math. Phys.}, 304(2):513--581, 2011.

\bibitem{AMUXY12.2}
R.~Alexandre, Y.~Morimoto, S.~Ukai, C.-J. Xu, and T.~Yang.
\newblock The {B}oltzmann equation without angular cutoff in the whole space:
  {I}, {G}lobal existence for soft potential.
\newblock {\em J. Funct. Anal.}, 262(3):915--1010, 2012.

\bibitem{BaDe85}
C.~Bardos and P.~Degond.
\newblock Global existence for the {V}lasov-{P}oisson equation in {$3$} space
  variables with small initial data.
\newblock {\em Ann. Inst. H. Poincar\'e Anal. Non Lin\'eaire}, 2(2):101--118,
  1985.

\bibitem{BaDeGo84}
C.~Bardos, P.~Degond, and F.~Golse.
\newblock A priori estimates and existence results for the {V}lasov and
  {B}oltzmann equations.
\newblock In {\em Nonlinear systems of partial differential equations in
  applied mathematics, {P}art 2 ({S}anta {F}e, {N}.{M}., 1984)}, volume~23 of
  {\em Lectures in Appl. Math.}, pages 189--207. Amer. Math. Soc., Providence,
  RI, 1986.

\bibitem{Beck-Wayne2013}
Margaret Beck and C.~Eugene Wayne.
\newblock Metastability and rapid convergence to quasi-stationary bar states
  for the two-dimensional {N}avier-{S}tokes equations.
\newblock {\em Proc. Roy. Soc. Edinburgh Sect. A}, 143(5):905--927, 2013.

\bibitem{jB2017}
Jacob Bedrossian.
\newblock Suppression of plasma echoes and {L}andau damping in {S}obolev spaces
  by weak collisions in a {V}lasov-{F}okker-{P}lanck equation.
\newblock {\em Ann. PDE}, 3(2):Paper No. 19, 66, 2017.

\bibitem{jB2021}
Jacob Bedrossian.
\newblock Nonlinear echoes and {L}andau damping with insufficient regularity.
\newblock {\em Tunis. J. Math.}, 3(1):121--205, 2021.

\bibitem{jBmCZ2017}
Jacob Bedrossian and Michele Coti~Zelati.
\newblock Enhanced dissipation, hypoellipticity, and anomalous small noise
  inviscid limits in shear flows.
\newblock {\em Arch. Ration. Mech. Anal.}, 224(3):1161--1204, 2017.

\bibitem{jBjGnM2015}
Jacob Bedrossian, Pierre Germain, and Nader Masmoudi.
\newblock Dynamics near the subcritical transition of the 3d {C}ouette flow
  {II}: {A}bove threshold case.
\newblock {\em arXiv:1506.03721, preprint}, 2015.

\bibitem{jBpGnM2017}
Jacob Bedrossian, Pierre Germain, and Nader Masmoudi.
\newblock On the stability threshold for the 3{D} {C}ouette flow in {S}obolev
  regularity.
\newblock {\em Ann. of Math. (2)}, 185(2):541--608, 2017.

\bibitem{jBpGnM2019}
Jacob Bedrossian, Pierre Germain, and Nader Masmoudi.
\newblock Stability of the {C}ouette flow at high {R}eynolds numbers in two
  dimensions and three dimensions.
\newblock {\em Bull. Amer. Math. Soc. (N.S.)}, 56(3):373--414, 2019.

\bibitem{jBpGnM2020}
Jacob Bedrossian, Pierre Germain, and Nader Masmoudi.
\newblock Dynamics near the subcritical transition of the 3{D} {C}ouette flow
  {I}: {B}elow threshold case.
\newblock {\em Mem. Amer. Math. Soc.}, 266(1294):v+158, 2020.

\bibitem{jBsimH2020}
Jacob Bedrossian and Siming He.
\newblock Inviscid damping and enhanced dissipation of the boundary layer for
  2{D} {N}avier-{S}tokes linearized around {C}ouette flow in a channel.
\newblock {\em Comm. Math. Phys.}, 379(1):177--226, 2020.

\bibitem{jBnMcM2016}
Jacob Bedrossian, Nader Masmoudi, and Cl\'{e}ment Mouhot.
\newblock Landau damping: paraproducts and {G}evrey regularity.
\newblock {\em Ann. PDE}, 2(1):Art. 4, 71, 2016.

\bibitem{jBnMcM2018}
Jacob Bedrossian, Nader Masmoudi, and Cl\'{e}ment Mouhot.
\newblock Landau damping in finite regularity for unconfined systems with
  screened interactions.
\newblock {\em Comm. Pure Appl. Math.}, 71(3):537--576, 2018.

\bibitem{jBnMcM2020}
Jacob Bedrossian, Nader Masmoudi, and Cl\'{e}ment Mouhot.
\newblock Linearized wave-damping structure of {V}lasov--{P}oisson in {$\mathbb
  R^3$}.
\newblock {\em arXiv:2007.08580, preprint}, 2020.

\bibitem{jBnMvV2016}
Jacob Bedrossian, Nader Masmoudi, and Vlad Vicol.
\newblock Enhanced dissipation and inviscid damping in the inviscid limit of
  the {N}avier-{S}tokes equations near the two dimensional {C}ouette flow.
\newblock {\em Arch. Ration. Mech. Anal.}, 219(3):1087--1159, 2016.

\bibitem{jBvVfW2018}
Jacob Bedrossian, Vlad Vicol, and Fei Wang.
\newblock The {S}obolev stability threshold for 2{D} shear flows near
  {C}ouette.
\newblock {\em J. Nonlinear Sci.}, 28(6):2051--2075, 2018.

\bibitem{jBfW2020}
Jacob Bedrossian and Fei Wang.
\newblock The linearized {V}lasov and {V}lasov-{F}okker-{P}lanck equations in a
  uniform magnetic field.
\newblock {\em J. Stat. Phys.}, 178(2):552--594, 2020.

\bibitem{Bi17}
L\'eo Bigorgne.
\newblock Asymptotic properties of small data solutions of the vlasov-maxwell
  system in high dimensions.
\newblock {\em arXiv:1712.09698, preprint}, 2017.

\bibitem{lBdFjJjSmT2020}
L\'eo Bigorgne, David Fajman, J\'er\'emie Joudioux, Jacques Smulevici, and
  Maximilian Thaller.
\newblock Asymptotic {S}tability of {M}inkowski {S}pace-{T}ime with
  non-compactly supported massless {V}lasov matter.
\newblock {\em arXiv:2003.03346, preprint}, 2020.

\bibitem{boyd2003physics}
TJM Boyd and JJ~Sanderson.
\newblock {\em The physics of plasmas}.
\newblock Cambridge University Press, 2003.

\bibitem{eCcM1998}
E.~Caglioti and C.~Maffei.
\newblock Time asymptotics for solutions of {V}lasov-{P}oisson equation in a
  circle.
\newblock {\em J. Statist. Phys.}, 92(1-2):301--323, 1998.

\bibitem{CaMi17}
K.~Carrapatoso and S.~Mischler.
\newblock Landau equation for very soft and {C}oulomb potentials near
  {M}axwellians.
\newblock {\em Ann. PDE}, 3(1):Art. 1, 65, 2017.

\bibitem{CaTrWu17}
Kleber Carrapatoso, Isabelle Tristani, and Kung-Chien Wu.
\newblock Cauchy problem and exponential stability for the inhomogeneous
  {L}andau equation.
\newblock {\em Arch. Ration. Mech. Anal.}, 221(1):363--418, 2016.

\bibitem{CaTrWuErratum17}
Kleber Carrapatoso, Isabelle Tristani, and Kung-Chien Wu.
\newblock Erratum to: {C}auchy problem and exponential stability for the
  inhomogeneous {L}andau equation [ {MR}3483898].
\newblock {\em Arch. Ration. Mech. Anal.}, 223(2):1035--1037, 2017.

\bibitem{sC2020a}
Sanchit Chaturvedi.
\newblock Stability of vacuum for the {L}andau equation with hard potentials.
\newblock {\em arXiv:2001.07208, preprint}, 2020.

\bibitem{sC2020b}
Sanchit Chaturvedi.
\newblock Stability of vacuum for the {B}oltzmann equation with moderately soft
  potentials.
\newblock {\em Ann. PDE}, 7(2):Paper No. 15, 104, 2021.

\bibitem{qCtLdyWzfZ2020}
Qi~Chen, Te~Li, Dongyi Wei, and Zhifei Zhang.
\newblock Transition threshold for the 2-{D} {C}ouette flow in a finite
  channel.
\newblock {\em Arch. Ration. Mech. Anal.}, 238(1):125--183, 2020.

\bibitem{mCZtmEkW2020}
Michele Coti~Zelati, Tarek~M. Elgindi, and Klaus Widmayer.
\newblock Enhanced dissipation in the {N}avier-{S}tokes equations near the
  {P}oiseuille flow.
\newblock {\em Comm. Math. Phys.}, 378(2):987--1010, 2020.

\bibitem{dqD2021}
Dingqun Deng.
\newblock Smoothing estimates of the {V}lasov--{P}oisson--{L}andau system.
\newblock {\em arXiv:2103.04114, preprint}, 2021.

\bibitem{PrincetonCompanion}
Mark~R. Dennis, Paul Glendinning, Paul~A. Martin, Fadil Santosa, and Jared
  Tanner, editors.
\newblock {\em The {P}rinceton companion to applied mathematics}.
\newblock Princeton University Press, Princeton, NJ, 2015.

\bibitem{lDcV2005}
L.~Desvillettes and C.~Villani.
\newblock On the trend to global equilibrium for spatially inhomogeneous
  kinetic systems: the {B}oltzmann equation.
\newblock {\em Invent. Math.}, 159(2):245--316, 2005.

\bibitem{lD2006}
Laurent Desvillettes.
\newblock Hypocoercivity: the example of linear transport.
\newblock In {\em Recent trends in partial differential equations}, volume 409
  of {\em Contemp. Math.}, pages 33--53. Amer. Math. Soc., Providence, RI,
  2006.

\bibitem{lDfS2009}
Laurent Desvillettes and Francesco Salvarani.
\newblock Asymptotic behavior of degenerate linear transport equations.
\newblock {\em Bull. Sci. Math.}, 133(8):848--858, 2009.

\bibitem{sjDzlL2020}
Shijin Ding and Zhilin Lin.
\newblock Enhanced dissipation and transition threshold for the $2-{D}$ plane
  {P}oiseuille flow via resolvent estimate.
\newblock {\em arXiv:2008.10057, preprint}, 2020.

\bibitem{rjDsqLsSrmS2021}
Renjun Duan, Shuangqian Liu, Shota Sakamoto, and Robert~M. Strain.
\newblock Global mild solutions of the {L}andau and non-cutoff {B}oltzmann
  equations.
\newblock {\em Comm. Pure Appl. Math.}, 74(5):932--1020, 2021.

\bibitem{rjDtYhjZ2011}
Renjun Duan, Tong Yang, and Huijiang Zhao.
\newblock Global solutions to the {V}lasov--{P}oisson--{L}andau system.
\newblock {\em arXiv:1112.3261, preprint}, 2011.

\bibitem{rjDhjY2020}
Renjun Duan and Hongjun Yu.
\newblock The {V}lasov-{P}oisson-{L}andau system near a local {M}axwellian.
\newblock {\em Adv. Math.}, 362:106956, 83, 2020.

\bibitem{DubNaz1994}
B~Dubrulle and S~Nazarenko.
\newblock On scaling laws for the transition to turbulence in uniform-shear
  flows.
\newblock {\em Europhysics Letters (EPL)}, 27(2):129–134, 1994.

\bibitem{mDrW2021}
Mitia Duerinckx and Raphael Winter.
\newblock Well-posedness of the lenard-balescu equation with smooth
  interactions.
\newblock {\em arXiv:2111.13320, preprint}, 2021.

\bibitem{jpEmH2003}
J.-P. Eckmann and M.~Hairer.
\newblock Spectral properties of hypoelliptic operators.
\newblock {\em Comm. Math. Phys.}, 235(2):233--253, 2003.

\bibitem{FaJoSm17}
David Fajman, J\'er\'emie Joudioux, and Jacques Smulevici.
\newblock The stability of the {M}inkowski space for the {E}instein--{V}lasov
  system.
\newblock {\em arXiv:1707.06141, preprint}, 2017.

\bibitem{FaJoSm17.1}
David Fajman, J\'er\'emie Joudioux, and Jacques Smulevici.
\newblock A vector field method for relativistic transport equations with
  applications.
\newblock {\em Anal. PDE}, 10(7):1539--1612, 2017.

\bibitem{tG2018}
Thierry Gallay.
\newblock Enhanced dissipation and axisymmetrization of two-dimensional viscous
  vortices.
\newblock {\em Arch. Ration. Mech. Anal.}, 230(3):939--975, 2018.

\bibitem{GlSc88}
R.~T. Glassey and J.~W. Schaeffer.
\newblock Global existence for the relativistic {V}lasov-{M}axwell system with
  nearly neutral initial data.
\newblock {\em Comm. Math. Phys.}, 119(3):353--384, 1988.

\bibitem{rGjS1994}
Robert Glassey and Jack Schaeffer.
\newblock Time decay for solutions to the linearized {V}lasov equation.
\newblock {\em Transport Theory Statist. Phys.}, 23(4):411--453, 1994.

\bibitem{rGjS1995}
Robert Glassey and Jack Schaeffer.
\newblock On time decay rates in {L}andau damping.
\newblock {\em Comm. Partial Differential Equations}, 20(3-4):647--676, 1995.

\bibitem{GlSt87}
Robert~T. Glassey and Walter~A. Strauss.
\newblock Absence of shocks in an initially dilute collisionless plasma.
\newblock {\em Comm. Math. Phys.}, 113(2):191--208, 1987.

\bibitem{GoldRuther1997}
R.~J. Goldston and P.~H. Rutherford.
\newblock {\em Introduction to plasma physics}.
\newblock Institute of Physics Pub., 1997.

\bibitem{eGtNiR2020a}
Emmanuel Grenier, Toan~T. Nguyen, and Igor Rodnianski.
\newblock Landau damping for analytic and {G}evrey data.
\newblock {\em arXiv:2004.05979, preprint}, 2020.

\bibitem{eGtNiR2020b}
Emmanuel Grenier, Toan~T. Nguyen, and Igor Rodnianski.
\newblock Plasma echoes near stable {P}enrose data.
\newblock {\em SIAM J. Math. Anal.}, 54(1):940--953, 2022.

\bibitem{GNRS}
Emmanuel Grenier, Toan~T. Nguyen, Fr\'{e}d\'{e}ric Rousset, and Avy Soffer.
\newblock Linear inviscid damping and enhanced viscous dissipation of shear
  flows by using the conjugate operator method.
\newblock {\em J. Funct. Anal.}, 278(3):108339, 27, 2020.

\bibitem{GrSt11}
Philip~T. Gressman and Robert~M. Strain.
\newblock Global classical solutions of the {B}oltzmann equation without
  angular cut-off.
\newblock {\em J. Amer. Math. Soc.}, 24(3):771--847, 2011.

\bibitem{Guo02}
Yan Guo.
\newblock The {L}andau equation in a periodic box.
\newblock {\em Comm. Math. Phys.}, 231(3):391--434, 2002.

\bibitem{Guo02.2}
Yan Guo.
\newblock The {V}lasov-{P}oisson-{B}oltzmann system near {M}axwellians.
\newblock {\em Comm. Pure Appl. Math.}, 55(9):1104--1135, 2002.

\bibitem{Guo03.2}
Yan Guo.
\newblock Classical solutions to the {B}oltzmann equation for molecules with an
  angular cutoff.
\newblock {\em Arch. Ration. Mech. Anal.}, 169(4):305--353, 2003.

\bibitem{Guo03}
Yan Guo.
\newblock The {V}lasov-{M}axwell-{B}oltzmann system near {M}axwellians.
\newblock {\em Invent. Math.}, 153(3):593--630, 2003.

\bibitem{Guo12}
Yan Guo.
\newblock The {V}lasov-{P}oisson-{L}andau system in a periodic box.
\newblock {\em J. Amer. Math. Soc.}, 25(3):759--812, 2012.

\bibitem{dHKttNfR2019}
Daniel Han-Kwan, Toan~T. Nguyen, and Fr\'ed\'eric Rousset.
\newblock Asymptotic stability of equilibria for screened {V}lasov--{P}oisson
  systems via pointwise dispersive estimates.
\newblock {\em arXiv:1906.05723, preprint}, 2019.

\bibitem{dHKttNfR2020}
Daniel Han-Kwan, Toan~T. Nguyen, and Fr\'ed\'eric Rousset.
\newblock On the linearized {V}lasov-{P}oisson system on the whole space around
  stable homogeneous equilibria.
\newblock {\em arXiv:2007.07787, preprint}, 2020.

\bibitem{cHyjL2016}
Cong He and Yuanjie Lei.
\newblock One-species {V}lasov-{P}oisson-{L}andau system for soft potentials in
  {$\Bbb{R}^3$}.
\newblock {\em J. Math. Phys.}, 57(12):121502, 25, 2016.

\bibitem{bHfN2005}
Bernard Helffer and Francis Nier.
\newblock {\em Hypoelliptic estimates and spectral theory for {F}okker-{P}lanck
  operators and {W}itten {L}aplacians}, volume 1862 of {\em Lecture Notes in
  Mathematics}.
\newblock Springer-Verlag, Berlin, 2005.

\bibitem{cHsSaT2019}
Christopher Henderson, Stanley Snelson, and Andrei Tarfulea.
\newblock Local existence, lower mass bounds, and a new continuation criterion
  for the {L}andau equation.
\newblock {\em J. Differential Equations}, 266(2-3):1536--1577, 2019.

\bibitem{fH2006}
Fr\'{e}d\'{e}ric H\'{e}rau.
\newblock Hypocoercivity and exponential time decay for the linear
  inhomogeneous relaxation {B}oltzmann equation.
\newblock {\em Asymptot. Anal.}, 46(3-4):349--359, 2006.

\bibitem{fH2007}
Fr\'{e}d\'{e}ric H\'{e}rau.
\newblock Short and long time behavior of the {F}okker-{P}lanck equation in a
  confining potential and applications.
\newblock {\em J. Funct. Anal.}, 244(1):95--118, 2007.

\bibitem{fH2018}
Fr\'{e}d\'{e}ric H\'{e}rau.
\newblock Introduction to hypocoercive methods and applications for simple
  linear inhomogeneous kinetic models.
\newblock In {\em Lectures on the analysis of nonlinear partial differential
  equations. {P}art 5}, volume~5 of {\em Morningside Lect. Math.}, pages
  119--147. Int. Press, Somerville, MA, 2018.

\bibitem{fHfN2004}
Fr\'{e}d\'{e}ric H\'{e}rau and Francis Nier.
\newblock Isotropic hypoellipticity and trend to equilibrium for the
  {F}okker-{P}lanck equation with a high-degree potential.
\newblock {\em Arch. Ration. Mech. Anal.}, 171(2):151--218, 2004.

\bibitem{lH1967}
Lars H\"{o}rmander.
\newblock Hypoelliptic second order differential equations.
\newblock {\em Acta Math.}, 119:147--171, 1967.

\bibitem{lHhjY2007}
Ling Hsiao and Hongjun Yu.
\newblock On the {C}auchy problem of the {B}oltzmann and {L}andau equations
  with soft potentials.
\newblock {\em Quart. Appl. Math.}, 65(2):281--315, 2007.

\bibitem{hjHjjlW2009}
Hyung~Ju Hwang and Juan J.~L. Vel\'{a}zquez.
\newblock On the existence of exponentially decreasing solutions of the
  nonlinear {L}andau damping problem.
\newblock {\em Indiana Univ. Math. J.}, 58(6):2623--2660, 2009.

\bibitem{John1971}
George~L. Johnston.
\newblock Dominant effects of {C}oulomb collisions on maintenance of {L}andau
  damping.
\newblock {\em Physics of Fluids}, 14(12):2719, 1971.

\bibitem{Kelvin1887}
Lord Kelvin.
\newblock Stability of fluid motion: rectilinear motion of viscous fluid
  between two parallel plates.
\newblock {\em Phil. Mag}, 24(5):188--196, 1887.

\bibitem{sK1985}
Sergiu Klainerman.
\newblock Uniform decay estimates and the {L}orentz invariance of the classical
  wave equation.
\newblock {\em Comm. Pure Appl. Math.}, 38(3):321--332, 1985.

\bibitem{jjK1977}
J.~J. Kohn.
\newblock Lectures on degenerate elliptic problems.
\newblock In {\em Pseudodifferential operator with applications ({B}ressanone,
  1977)}, pages 89--151. Liguori, Naples, 1978.

\bibitem{lL1946}
L.~Landau.
\newblock On the vibrations of the electronic plasma.
\newblock {\em Akad. Nauk SSSR. Zhurnal Eksper. Teoret. Fiz.}, 16:574--586,
  1946.

\bibitem{LatBer2001}
Marco Latini and Andrew~J. Bernoff.
\newblock Transient anomalous diffusion in poiseuille flow.
\newblock {\em Journal of Fluid Mechanics}, 441:399–411, 2001.

\bibitem{yjLljXhjZ2014}
Yuanjie Lei, Linjie Xiong, and Huijiang Zhao.
\newblock One-species {V}lasov-{P}oisson-{L}andau system near {M}axwellians in
  the whole space.
\newblock {\em Kinet. Relat. Models}, 7(3):551--590, 2014.

\bibitem{LenBern1958}
Andrew Lenard and Ira~B Bernstein.
\newblock Plasma oscillations with diffusion in velocity space.
\newblock {\em Physical Review}, 112(5):1456, 1958.

\bibitem{hLmT20}
Hans Lindblad and Martin Taylor.
\newblock Global stability of {M}inkowski space for the {E}instein-{V}lasov
  system in the harmonic gauge.
\newblock {\em Arch. Ration. Mech. Anal.}, 235(1):517--633, 2020.

\bibitem{jL2019}
Jonathan Luk.
\newblock Stability of vacuum for the {L}andau equation with moderately soft
  potentials.
\newblock {\em Ann. PDE}, 5(1):Paper No. 11, 101, 2019.

\bibitem{xL2020}
Xiang Luo.
\newblock The {S}obolev stability threshold of 2{D} hyperviscosity equations
  for shear flows near {C}ouette flow.
\newblock {\em Math. Methods Appl. Sci.}, 43(10):6300--6323, 2020.

\bibitem{MalmWhar1964}
JH~Malmberg and CB~Wharton.
\newblock Collisionless damping of electrostatic plasma waves.
\newblock {\em Physical Review Letters}, 13(6):184, 1964.

\bibitem{MaWhGoOn1968plasma}
JH~Malmberg, CB~Wharton, RW~Gould, and TM~O'neil.
\newblock Plasma wave echo experiment.
\newblock {\em Physical Review Letters}, 20(3):95, 1968.

\bibitem{nMwrZ2019}
Nader Masmoudi and Weiren Zhao.
\newblock Stability threshold of the $2{D}$ {C}ouette flow in {S}obolev spaces.
\newblock {\em arXiv:1908.11042, preprint}, 2019.

\bibitem{nMwrZ2020}
Nader Masmoudi and Weiren Zhao.
\newblock Enhanced dissipation for the 2{D} {C}ouette flow in critical space.
\newblock {\em Comm. Partial Differential Equations}, 45(12):1682--1701, 2020.

\bibitem{ahMrlL1973}
Allan~H. Merchant and Richard~L. Liboff.
\newblock Spectral properties of the linearized {B}alescu-{L}enard operator.
\newblock {\em J. Mathematical Phys.}, 14:119--129, 1973.

\bibitem{cMcV2010}
C.~Mouhot and C.~Villani.
\newblock Landau damping.
\newblock {\em J. Math. Phys.}, 51(1):015204, 7, 2010.

\bibitem{cMcV2011}
Cl\'{e}ment Mouhot and C\'{e}dric Villani.
\newblock On {L}andau damping.
\newblock {\em Acta Math.}, 207(1):29--201, 2011.

\bibitem{NgBhSk1999}
CS~Ng, A~Bhattacharjee, and F~Skiff.
\newblock Kinetic eigenmodes and discrete spectrum of plasma oscillations in a
  weakly collisional plasma.
\newblock {\em Physical review letters}, 83(10):1974, 1999.

\bibitem{NgBhSk2006}
CS~Ng, A~Bhattacharjee, and F~Skiff.
\newblock Weakly collisional {L}andau damping and three-dimensional
  {B}ernstein-{G}reene-{K}ruskal modes: {N}ew results on old problems.
\newblock {\em Physics of plasmas}, 13(5):055903, 2006.

\bibitem{On1968}
Thomas~M O'Neil.
\newblock Effect of coulomb collisions and microturbulence on the plasma wave
  echo.
\newblock {\em The physics of fluids}, 11(11):2420--2425, 1968.

\bibitem{Orr1907}
William~M'F Orr.
\newblock The stability or instability of the steady motions of a perfect
  liquid and of a viscous liquid. part ii: A viscous liquid.
\newblock In {\em Proceedings of the Royal Irish Academy. Section A:
  Mathematical and Physical Sciences}, volume~27, pages 69--138. JSTOR, 1907.

\bibitem{Ryu1999}
DD~Ryutov.
\newblock Landau damping: half a century with the great discovery.
\newblock {\em Plasma physics and controlled fusion}, 41(3A):A1, 1999.

\bibitem{Short2002}
RW~Short and A~Simon.
\newblock Damping of perturbations in weakly collisional plasmas.
\newblock {\em Physics of Plasmas}, 9(8):3245--3253, 2002.

\bibitem{Sm16}
Jacques Smulevici.
\newblock Small data solutions of the {V}lasov-{P}oisson system and the vector
  field method.
\newblock {\em Ann. PDE}, 2(2):Art. 11, 55, 2016.

\bibitem{jS2018}
Jacques Smulevici.
\newblock The stability of the {M}inkowski space for the {E}instein {V}lasov
  system.
\newblock In {\em S\'{e}minaire {L}aurent {S}chwartz---\'{E}quations aux
  d\'{e}riv\'{e}es partielles et applications. {A}nn\'{e}e 2017--2018}, pages
  Exp. No. XV, 15. Ed. \'{E}c. Polytech., Palaiseau, 2018.

\bibitem{Stix1992}
Thomas~H Stix.
\newblock {\em Waves in plasmas}.
\newblock Springer Science \& Business Media, 1992.

\bibitem{rmS2007}
Robert~M. Strain.
\newblock On the linearized {B}alescu-{L}enard equation.
\newblock {\em Comm. Partial Differential Equations}, 32(10-12):1551--1586,
  2007.

\bibitem{StGu04}
Robert~M. Strain and Yan Guo.
\newblock Stability of the relativistic {M}axwellian in a collisional plasma.
\newblock {\em Comm. Math. Phys.}, 251(2):263--320, 2004.

\bibitem{StGu06}
Robert~M. Strain and Yan Guo.
\newblock Almost exponential decay near {M}axwellian.
\newblock {\em Comm. Partial Differential Equations}, 31(1-3):417--429, 2006.

\bibitem{StGu08}
Robert~M. Strain and Yan Guo.
\newblock Exponential decay for soft potentials near {M}axwellian.
\newblock {\em Arch. Ration. Mech. Anal.}, 187(2):287--339, 2008.

\bibitem{StZh13}
Robert~M. Strain and Keya Zhu.
\newblock The {V}lasov-{P}oisson-{L}andau system in {$\mathbb{R}^3_x$}.
\newblock {\em Arch. Ration. Mech. Anal.}, 210(2):615--671, 2013.

\bibitem{SuOb1968}
CH~Su and C~Oberman.
\newblock Collisional damping of a plasma echo.
\newblock {\em Physical Review Letters}, 20(9):427, 1968.

\bibitem{Ta17}
Martin Taylor.
\newblock The global nonlinear stability of {M}inkowski space for the massless
  {E}instein-{V}lasov system.
\newblock {\em Ann. PDE}, 3(1):Art. 9, 177, 2017.

\bibitem{iT2017}
Isabelle Tristani.
\newblock Landau damping for the linearized {V}lasov {P}oisson equation in a
  weakly collisional regime.
\newblock {\em J. Stat. Phys.}, 169(1):107--125, 2017.

\bibitem{vanneste1998strong}
J~Vanneste, PJ~Morrison, and T~Warn.
\newblock Strong echo effect and nonlinear transient growth in shear flows.
\newblock {\em Physics of Fluids}, 10(6):1398--1404, 1998.

\bibitem{cV2009}
C\'{e}dric Villani.
\newblock Hypocoercivity.
\newblock {\em Mem. Amer. Math. Soc.}, 202(950):iv+141, 2009.

\bibitem{cV2013}
C\'{e}dric Villani.
\newblock ({I}r)reversibility and entropy.
\newblock In {\em Time}, volume~63 of {\em Prog. Math. Phys.}, pages 19--79.
  Birkh\"{a}user/Springer Basel AG, Basel, 2013.

\bibitem{Wa18.2}
Xuecheng Wang.
\newblock Decay estimates for the 3{D} relativistic and non-relativistic
  {V}lasov--{P}oisson systems.
\newblock {\em arXiv:1805.10837, preprint}, 2018.

\bibitem{Wa18.3}
Xuecheng Wang.
\newblock Propagation of regularity and long time behavior of the 3{D} massive
  relativistic transport equation {I}: {V}lasov--{N}ordstr\"om system.
\newblock {\em arXiv:1804.06560, preprint}, 2018.

\bibitem{Wa18.1}
Xuecheng Wang.
\newblock Propagation of regularity and long time behavior of the 3{D} massive
  relativistic transport equation {II}: {V}lasov--{M}axwell system.
\newblock {\em arXiv:1804.06566, preprint}, 2018.

\bibitem{yjW2012}
Yanjin Wang.
\newblock Global solution and time decay of the {V}lasov-{P}oisson-{L}andau
  system in {$\Bbb R^3$}.
\newblock {\em SIAM J. Math. Anal.}, 44(5):3281--3323, 2012.

\bibitem{dyWzjZ2018}
Dongyi Wei and Zhifei Zhang.
\newblock Transition threshold for the $3{D}$ {C}ouette flow in {S}obolev
  space.
\newblock {\em arXiv:1803.01359, preprint, to appear in
  Comm.~Pure~Appl.~Math.}, 2018.

\bibitem{dyWzfZwrZ2020}
Dongyi Wei, Zhifei Zhang, and Weiren Zhao.
\newblock Linear inviscid damping and enhanced dissipation for the {K}olmogorov
  flow.
\newblock {\em Adv. Math.}, 362:106963, 103, 2020.

\bibitem{WeZhZh20}
Dongyi Wei, Zhifei Zhang, and Hao Zhu.
\newblock Linear inviscid damping for the {$\beta$}-plane equation.
\newblock {\em Comm. Math. Phys.}, 375(1):127--174, 2020.

\bibitem{Wo18}
Willie Wai~Yeung Wong.
\newblock A commuting-vector-field approach to some dispersive estimates.
\newblock {\em Arch. Math. (Basel)}, 110(3):273--289, 2018.

\bibitem{bY2016}
Brent Young.
\newblock Landau damping in relativistic plasmas.
\newblock {\em J. Math. Phys.}, 57(2):021502, 68, 2016.

\bibitem{yu2002diocotron}
JH~Yu and CF~Driscoll.
\newblock Diocotron wave echoes in a pure electron plasma.
\newblock {\em IEEE transactions on plasma science}, 30(1):24--25, 2002.

\bibitem{yu2005phase}
JH~Yu, CF~Driscoll, and TM~O’Neil.
\newblock Phase mixing and echoes in a pure electron plasma.
\newblock {\em Physics of plasmas}, 12(5):055701, 2005.

\bibitem{cZ2020}
Christian Zillinger.
\newblock On enhanced dissipation for the {B}oussinesq equations.
\newblock {\em arXiv:2004.08125, preprint}, 2020.

\end{thebibliography}
\end{document}